\documentclass{compositio}
 \usepackage{amsmath,amssymb}
  \usepackage{eucal}
   \usepackage{mathtools}
\usepackage{mathrsfs}
\usepackage{leftidx}

 \usepackage[all,cmtip]{xy}

\usepackage{pdflscape}

 \usepackage{graphicx,  leftindex}
 \usepackage{pigpen}
 \usepackage[dvipsnames]{xcolor}

 \usepackage{tikz,tikz-3dplot}
 \usepackage{tikz-cd}
\usetikzlibrary{arrows, patterns, decorations.pathreplacing, decorations.markings}
\usetikzlibrary{math,angles,quotes,intersections,calc,through}
\usetikzlibrary{graphs,decorations.pathmorphing}

\tikzstyle{vecArrow} = [thick, decoration={markings,mark=at position
   1 with {\arrow[semithick]{open triangle 60}}},
   double distance=1.4pt, shorten >= 5.5pt,
   preaction = {decorate},
   postaction = {draw,line width=1.4pt, white,shorten >= 4.5pt}]
\tikzstyle{innerWhite} = [semithick, white,line width=1.4pt, shorten >= 4.5pt]

 \usepackage[colorlinks,hyperindex,pdfpagemode=None,bookmarksopen,linkcolor=blue,citecolor=blue,urlcolor=blue]{hyperref}

\newtheorem{thm}{Theorem}[section]
\newtheorem{theorem}[thm]{Theorem}
\newtheorem{cor}[thm]{Corollary}
\newtheorem{lemma}[thm]{Lemma}
\newtheorem{def/lemma}[thm]{Definition/Lemma}

\newtheorem{prop}[thm]{Proposition}
\newtheorem{claim}[thm]{Claim}

\theoremstyle{definition}
\newtheorem{example}[thm]{Example}
\newtheorem{definition}[thm]{Definition}
\newtheorem{remark}[thm]{Remark}
 \newtheorem{prop/def}[thm]{Proposition/Definition}

\newcommand{\cC}{\mathcal{C}}
\newcommand{\cD}{\mathcal{D}}
\newcommand{\cE}{\mathcal{E}}
\newcommand{\cF}{\mathcal{F}}
\newcommand{\cG}{\mathcal{G}}

\newcommand{\cM}{\mathcal{M}}
\newcommand{\cN}{\mathcal{N}}
\newcommand{\cO}{\mathcal{O}}
\newcommand{\cP}{\mathcal{P}}

\newcommand{\cR}{\mathcal{R}}
\newcommand{\cS}{\mathcal{S}}
\newcommand{\cT}{\mathcal{T}}
\newcommand{\cU}{\mathcal{U}}
\newcommand{\cV}{\mathcal{V}}
\newcommand{\cZ}{\mathcal{Z}}
\newcommand{\Slch}{\mathrm{S}_{\mathrm{LCH}}}

\newcommand{\Spc}{\mathcal{S}\mathrm{pc}}

\newcommand{\OneCat}{1\text{-}\mathcal{C}\mathrm{at}}
\newcommand{\bOneCat}{\mathbf{1}\text{-}\mathcal{C}\mathbf{at}}
\newcommand{\TwoCat}{2\text{-}\mathcal{C}\mathrm{at}}

\newcommand{\SSet}{\mathcal{S}\mathrm{et}_{\Delta}}

\newcommand{\bC}{\mathbf{C}}
\newcommand{\bD}{\mathbf{D}}

\newcommand{\bL}{\mathbf{L}}
\newcommand{\bN}{\mathbf{N}}
\newcommand{\bP}{\mathbf{P}}

\newcommand{\bR}{\mathbf{R}}
\newcommand{\bS}{\mathbf{S}}
\newcommand{\bX}{\mathbf{X}}
\newcommand{\bZ}{\mathbf{Z}}

\newcommand{\bc}{\mathbf{c}}

\newcommand{\bbD}{\mathbb{D}}
\newcommand{\bT}{\mathbb{T}}

\newcommand{\Mod}{\mathrm{Mod}}
\newcommand{\Fun}{\mathrm{Fun}}
\newcommand{\Map}{\mathrm{Map}}
\newcommand{\Hom}{\mathrm{Hom}}
\newcommand{\Maps}{\mathrm{Maps}}
\newcommand{\cMod}{\mathcal{M}\mathrm{od}}

\newcommand{\Pic}{\mathrm{Pic}}

\newcommand{\st}{\mathit{st}}

\newcommand{\Fin}{\mathrm{Fin}}

\newcommand{\Corr}{\mathrm{Corr}}
\newcommand{\bCorr}{\mathbf{Corr}}
\newcommand{\Seq}{\mathrm{Seq}}
\newcommand{\Grid}{\mathrm{Grid}}
\newcommand{\dgnl}{\mathrm{dgnl}}
\newcommand{\all}{\mathsf{all}}
\newcommand{\isom}{\mathsf{isom}}

\newcommand{\Sq}{\mathrm{Sq}}
\newcommand{\ord}{\mathrm{ord}}

\newcommand{\Cart}{\mathrm{Cart}}

\newcommand{\rightlax}{\mathrm{R}\text{-}\mathrm{lax}}
\newcommand{\leftlax}{\mathrm{L}\text{-}\mathrm{lax}}

\newcommand{\colim}{\mathrm{colim }}
\newcommand{\Factor}{\mathrm{Factor}}
\newcommand{\dense}{\mathrm{dense}}

\newcommand{\Seg}{\mathrm{Seg}}

\newcommand{\Comm}{\mathrm{Comm}}

\newcommand{\ShvSp}{\mathrm{ShvSp}}

\renewcommand{\SS}{\mathit{SS}}

\newcommand{\Vect}{\mathrm{Vect}}

\newcommand{\Brane}{\mathrm{Brane}}
\newcommand{\Alg}{\mathrm{Alg}}

\numberwithin{equation}{subsection}

\newcommand{\pt}{\mathit{pt}}

\newcommand{\Loc}{\mathrm{Loc}}
\newcommand{\bk}{\mathbf{k}}

\newcommand{\Sp}{\mathrm{Sp}}

\newcommand{\Shv}{\mathrm{Shv}}

\newcommand{\Cone}{\mathrm{Cone}}
\newcommand{\proj}{\mathrm{proj}}

\newcommand{\fri}{\mathfrak{i}}

\newcommand{\bV}{\mathbf{V}}

\newcommand{\cW}{\mathcal{W}}

\newcommand{\cL}{\mathcal{L}}

\newcommand{\bq}{\mathbf{q}}
\newcommand{\bp}{\mathbf{p}}
\newcommand{\PrL}{\mathcal{P}\mathrm{r}^{\mathrm{L}}}

\newcommand{\Open}{\mathrm{Open}}
\newcommand{\Gray}{\text{\circled{$\star$}}}
\newcommand{\bMaps}{{\cM\bf\text{aps}}}

\newcommand{\PrstL}{\mathcal{P}\mathrm{r}_{\mathrm{st}}^{\mathrm{L}}}
\newcommand{\PrstR}{\mathcal{P}\mathrm{r}_{\mathrm{st}}^{\mathrm{R}}}
\newcommand{\bPrstR}{\mathbf{Pr}_{\mathrm{st}}^{\mathrm{R}}}

\newcommand{\fS}{\mathfrak{S}}
\newcommand{\fL}{\mathfrak{L}}

\newcommand{\CAlg}{\mathrm{CAlg}}
\newcommand{\bPrstL}{\mathbf{Pr}_{\mathrm{st}}^{\mathrm{L}}}
\newcommand{\propfib}{\mathsf{propfib}}
\newcommand{\Sets}{\mathrm{Sets}}

\newcommand{\fib}{\mathsf{fib}}

\newcommand{\Graph}{\mathrm{graph}}

\newcommand{\po}{\ar@{}[dr]|{\text{\pigpenfont R}}}
\newcommand{\pb}{\ar@{}[dr]|{\text{\pigpenfont J}}}

\newcommand{\dagg}{\dagger}
\newcommand{\lng}{\langle}
\newcommand{\rng}{\rangle}

\newcommand{\const}{\mathrm{const}}

\newcommand{\propmap}{\mathsf{prop}}
\newcommand{\open}{\mathsf{open}}

\newcommand{\Pair}{\mathrm{Pair}}

\newcommand{\defGrid}{\mathrm{defGrid}}
\newcommand{\STop}{\mathrm{S}_{\mathrm{Top}}}

\newcommand{\module}{\mathrm{module}}

\newcommand{\Assoc}{\mathrm{Assoc}}

\newcommand{\Top}{\mathrm{Top}}
\newcommand{\Gpd}{\mathrm{Gpd}}

\newcommand{\inert}{\mathsf{inert}}
\newcommand{\torsors}{\text{-}\mathsf{torsors}}
\newcommand{\act}{\mathsf{active}}

\newcommand{\oneop}{1\text{-}op}
\newcommand{\twoop}{2\text{-}op}
\newcommand{\onetwoop}{1\& 2\text{-}\mathrm{op}}

\newcommand{\Diff}{\mathrm{Diff}}

\newcommand{\VSlch}{\mathrm{V}\Slch}

\newcommand{\inv}{\mathrm{invt}}

\newcommand{\VTop}{\mathrm{VTop}}
\newcommand{\Fib}{\mathrm{Fib}}
\newcommand{\Cof}{\mathrm{Cof}}
\newcommand{\ordn}{\mathrm{ordn}}
\newcommand{\vertop}{\mathrm{vert}\text{-}\mathrm{op}}
\newcommand{\OblvSubcat}{\mathrm{OblvSubcat}}
\newcommand{\Trpl}{\mathrm{Trpl}}
\newcommand{\hpf}{\mathsf{hpf}}
\newcommand{\CcoAlg}{\mathrm{CcoAlg}}
\newcommand{\perf}{\mathsf{perf}}

\newcommand{\QHam}{\mathsf{QHam}}

\newcommand*\circled[1]{\tikz[baseline=(char.base)]{
            \node[shape=circle,draw,inner sep=0.05pt] (char) {#1};}}

\renewcommand{\contentsname}{Contents\\{\footnotesize\normalfont(A table
of contents should normally not be included)}}

\begin{document}

\title[A Hamiltonian $\coprod\limits_n BO(n)$-action]{A Hamiltonian $\coprod\limits_n BO(n)$-action, stratified Morse theory and the $J$-homomorphism}
\author[X. Jin]{Xin Jin}
\email{xin.jin@bc.edu}
\address{Math Department,
Boston College,
Chestnut Hill, MA 02467, USA.}


\begin{abstract}
We use sheaves of spectra to quantize a Hamiltonian $\coprod\limits_n BO(n)$-action on $\varinjlim\limits_{N}T^*\mathbf{R}^N$ that naturally arises from Bott periodicity. We employ the category of correspondences developed in \cite{Nick} to give an enrichment of stratified Morse theory by the $J$-homomorphism. This provides a key step in the following work \cite{J-J} on the proof of a claim in \cite{JT}: the classifying map of the local system of brane structures on an (immersed) exact Lagrangian submanifold $L\subset T^*\mathbf{R}^N$ is given by the composition of the stable Gauss map $L\rightarrow U/O$ and the delooping of the $J$-homomorphism $U/O\rightarrow B\mathrm{Pic}(\mathbf{S})$. 

We put special emphasis on the functoriality and (symmetric) monoidal structures of the categories involved, and as a byproduct, we produce several concrete constructions of (commutative) algebra/module objects and (right-lax) morphisms between them in the (symmetric) monoidal $(\infty, 2)$-category of correspondences, generalizing the construction out of Segal objects in \cite{Nick}, which might be of interest by its own. 
\end{abstract}

\maketitle


\section{Introduction}
Let $X$ be a smooth manifold whose cotangent bundle $T^*X$ is equipped with the standard Liouville 1-form $\alpha=-\bp d\bq$.  Let $L$ be any (immersed) Lagrangian submanifold in $T^*X$ exact with respect to $\alpha$, i.e. $-\alpha|_L$ is an exact 1-form. Let $\bS$ denote the sphere spectrum. In joint work with D. Treumann \cite{JT},  we use microlocal sheaf theory to construct a local system of stable $\infty$-categories on $L$ with fiber equivalent to the $\infty$-category of spectra, which we call the \emph{sheaf of brane structures} and denote by $\Brane_L$. The sheaf of brane structures admits a classifying map $L\rightarrow B\Pic(\bS)$, where the target is the delooping of the $E_\infty$-groupoid of (suspensions of) $\bS$-lines. In \cite{JT}, we make the following claim (without proof):

\begin{claim}\label{claim: 1}
The classifying map of $\Brane_L$ factors as
\begin{align*}
L\overset{\gamma}{\rightarrow} U/O\overset{BJ}{\rightarrow} B\Pic(\bS),
\end{align*}
in which $\gamma$ is the stable Gauss map\footnote{More precisely, one should compose $\gamma$ with the canonical involution on $U/O$.} into the stable Lagrangian Grassmannian $U/O$ and $BJ$ is the delooping of the $J$-homomorphism 
\begin{align*}
J:\bZ\times BO\rightarrow \Pic(\bS)\simeq \bZ\times BGL_1(\bS)
\end{align*}
as an $E_\infty$-map. 
\end{claim}

The goal of this paper is three-fold. First, it will combine with the following work \cite{J-J}  to give a proof of the claim. Secondly, it enhances the Main theorem of stratified Morse theory with the rich structures of the $J$-homomorphism.  Lastly but not least, this work employs the category of correspondences developed in \cite{Nick}, and we have produced several concrete constructions of (commutative) algebra/module objects and (right-lax) morphisms between them 
in the symmetric monoidal $(\infty,2)$-category of correspondences, which we hope could lead to more applications. 

In the following, we summarize the key ingredients and main results. We also sketch the remaining steps in  \cite{J-J} to complete the proof of the above claim. 

\subsection{A Hamiltonian $\coprod\limits_n BO(n)$-action and Bott periodicity}\label{sec: Ham action}
In $T^*\bR^N$, for any quadratic form $A$ on $\bR^N$, we consider the quadratic Hamiltonian function $F_A(\bq,\bp)=A(\bp)$.  Here we identify $(\bR^N)^*$ with $\bR^N$ using the standard inner product on $\bR^N$. The time-1 flow $\varphi_A^1$ of the Hamiltonian vector field $X_{F_A}$ of $F_A$, defined as $\omega(X_{F_A},-)=dF_A$, sends $(\bq,\bp)$ to $(\bq+\partial_{\bp}A(\bp), \bp)$. In the following, we assume that $A$ as a symmetric matrix is idempotent, i.e. there exists a subspace $E_A\subset \bR^N$ such that $A(\bp)=|\proj_{E_A}\bp|^2$. If we start with the linear Lagrangian $L_0$ that is the graph of the differential of $-\frac{1}{2}|\bq|^2$, and do $\varphi_A^1$ to it, then the resulting Lagrangian $L_A$ is the graph of the differential of $-\frac{1}{2}|\bq|^2+A(\bq)$. Moreover, the Maslov index of the loop starting from  $L_0$, going towards $L_A$ via the Hamiltonian flow of $F_A$ and finally going back to $L_0$ through the path by decreasing the eigenvalue of $A$ from $1$ to $0$ is exactly the rank of $A$ (up to a negative sign). We can conjugate this loop by the path connecting the zero-section to $L_0$ through the graphs of differentials of $-\frac{1}{2}t|\bq|^2, t\in [0,1]$ so then the base point is always at the zero-section. 

If we take the stabilization $\varinjlim\limits_NT^*\bR^N$, then the above picture tells us a way to assign a based loop in the stable Lagrangian Grassmannian $U/O$ from an element in $\coprod\limits_{n}Gr(n,\bR^\infty)=\coprod\limits_{n}BO(n)$, which is exactly the content of (a part of) Bott periodicity that the latter after taking group completion is the based loop space of the former.
Now it is natural to assemble the Hamiltonian map for different choices of $A$ as a Hamiltonian $\coprod\limits_{n}BO(n)$-action on $\varinjlim\limits_NT^*\bR^N$. Here we take the standard commutative topological monoid structure on $\coprod\limits_{n}BO(n)$ as in \cite{Harris}, where the addition operation is only defined for two perpendicular finite dimensional subspaces in $\varinjlim\limits_N\bR^N$.

\subsection{Quantization of the Hamiltonian $\coprod\limits_n BO(n)$-action and a natural ``module"}
We use microlocal sheaf theory to study the Hamiltonian $\coprod\limits_n BO(n)$-action described above. 
The standard way to do this is to find a sheaf over prescribed coefficient with singular support in a conic Lagrangian lifting of the moment Lagrangian of the Hamiltonian action, called \emph{sheaf quantization} (cf. \cite{GKS}, \cite{Tamarkin2}). Here we take the coefficient to be the universal one, the sphere spectrum $\bS$, so sheaves are taking values in the stable $\infty$-category of spectra $\Sp$ (see \cite{JT} for basic properties of microlocal sheaves of spectra). The commutative monoid structure on $\coprod\limits_n BO(n)$ endows the associated sheaf category with a symmetric monoidal convolution structure (we will make this precise by using the category of correspondence developed in \cite{Nick}), and it is natural to ask the sought-for sheaf to be a commutative algebra object in it. Moreover, to establish a connection of the  Hamiltonian action with Maslov index and Bott periodicity, etc., we need to quantize a ``module" of it generated by (the stabilization of) $L_0$ as well. We will make this more precise below. 

For any smooth manifold $X$, let $T^{*,<0}(X\times \bR_t)$ be the negative half of $T^*(X\times\bR_t)$,  where every covector has strictly negative component in $dt$. 
For any exact Lagrangian submanifold $L$ in $T^*X$, a \emph{conic lifting} $\bL$ of $L$ in $T^{*,<0}(X\times \bR_t)$ (with the Liouville 1-form $-\tau dt+\alpha$) is any conic Lagrangian, i.e. invariant under the $\bR_+$-scaling on the fibers, determined by the properties that 
\begin{itemize}
\item the projection to $T^*X$ sends $\bL\cap \{\tau=-1\}$ isomorphically\footnote{We use projects ``isomorphically" rather than just ``onto" to make sure there are no multiple components that differ by some (locally) constant translations in the $t$-direction. If $L$ and $\bL$ are both connected, then there is no such distinction between the two conditions.} to $L$. 
\end{itemize}
By the conic property, the 1-form $-\tau dt+\alpha$ vanishes on $\bL$, which is equivalent to saying that $\bL\cap \{\tau=-1\}$ is a Legendrian submanifold in the contact hypersurface $\left(\{\tau=-1\}\cong T^*X\times \bR_t, dt+\alpha\right)$ that projects to $L$ isomorphically. The latter is usually referred as a Legendrian lifting of $L$. 
In other words, by writing $-\alpha|_{L}=df$, its Legendrian lifting (for the choice of $f$) is just the graph of $t=f(\bq, \bp)$ with $(\bq, \bp)\in L$, and $\bL$ is its conification. For more details about Legendrian liftings and their conification, we refer the reader to \cite[\S 3.1-3.5]{JT} (note in \cite{JT} there is some sign difference: we used $\alpha=\bp d\bq$ instead of $-\bp d\bq$).    
For any exact Lagrangian in general position,  $\bL$ is completely determined by the front projection of $\bL$ in $X\times \bR_t$, i.e. its image under the projection $T^{*,<0}(X\times \bR_t)\rightarrow X\times \bR_t$ to the base. Moreover, if $L$ is closed and the front projection of $\bL$ is a smooth submanifold, then $\bL$ is just half of the conormal bundle of the submanifold with negative $dt$-component, and we say it is the \emph{negative conormal bundle} of the smooth submanifold.

A (stabilized) conic Lagrangian lifting of the graph of the Hamiltonian action has front projection in $\coprod\limits_n BO(n)\times \varinjlim\limits_M\bR^M\times \varinjlim\limits_M\bR^M\times\bR_t$ consisting of points
\begin{align}\label{eq: point in graph}
&(A, \bq, \bq+\partial_{\bp}A(\bp), t=A(\bp)), \bq,\bp\in \varinjlim\limits_M\bR^M,
\end{align}
which is a smooth subvariety. 
Here and after, by some abuse of notation, we will use $A$ to denote its eigenspace $E_A$ of $1$ as well, so then $A$ also represents an element in $\coprod\limits_n BO(n)$. 

Note that we can equally represent a point in (\ref{eq: point in graph}) by
\begin{align*}
(A,\bq, \bq+2\bp, t=|\bp|^2), \bq\in \varinjlim\limits_M\bR^M, \bp\in A.
\end{align*}
Then modulo the roles of $\bq$ and $t$, which contribute trivially to the topology, the front projection is just the tautological vector bundle on $\coprod\limits_n BO(n)$.  

Let 
\begin{align}\label{eq: intro G_N, VG_N}
\nonumber&G_N=\coprod\limits_n Gr(n,\bR^N),\ VG_N=\text{the tautological vector bundle on } G_N,\\
&G=\varinjlim\limits_{N}G_N=\coprod\limits_n BO(n),\ VG=\varinjlim\limits_{N}VG_N. 
\end{align}
A quantization of the Hamiltonian $G$-action is then given by a commutative algebra object in 
\begin{align}\label{eq: intro LocVG}
&\Loc(VG;\Sp):=\varprojlim\limits_{N}\Loc(VG_{N};\Sp)
\end{align}
equipped with the convolution symmetric monoidal structure, whose restriction to the unit $(A=0, \bp=0)$ of $VG$ is the constant sheaf $\bS$ on a point. Here and after $\Loc(X;\Sp)$ consists of locally constant \emph{cosheaves}, and the limit of above is taken for the $!$-pullback along the natural embeddings 
$VG_{N}\hookrightarrow VG_{N'}$ for $N\leq N'$ (cf. Remark \ref{remark:locconstcoshv} below). A canonical candidate of such a quantization is the dualizing sheaf $\varpi_{VG}$.

Now what do we mean by to quantize a ``module" of the Hamiltonian action generated by the stabilization of $L_0$? 
First, we can quantize $L_0\subset T^*\bR^M$ by the $*$-extension of a local system on
\begin{align*}
Q^0_M:=\{s<-\frac{1}{2}|\bq|^2\}\subset \bR^M_\bq\times\bR_s
\end{align*}
to $ \bR^M_\bq\times\bR_s$, whose boundary has negative conormal bundle equal to a conic lifting of $L_0$. Second, we can assemble the images of (the stabilization of) $L_0$ under the Hamiltonian $G$-action, i.e. $L_A=\varphi_A^1(L_0)$, into a single Lagrangian $\widehat{L}$ in the stable sense,  and we can present a conic lifting of it as the stabilization of the negative conormal bundle of the boundary of the following domain
\begin{align}\label{eq: intro half-space}
&\widehat{Q}_{N,M}:=\{s<-\frac{1}{2}|\bq|^2+A(\bq),\ A\in G_N\}\subset G_N\times \bR^M_{\bq}\times \bR_s
\end{align}
for $M\geq N$.

Now the point is that given any quantization of $L_0$, represented by an object in $\Loc(Q^0;\Sp):= \varprojlim\limits_M \Loc(Q_M^0;\Sp)$, and the canonical quantization of the Hamiltonian $G$-action $\varpi_{VG}$,  we should automatically get a quantization of 
$\widehat{L}$, which is represented by an object in 
\begin{align}\label{eq: intro, LochatQ}
\Loc(\widehat{Q};\Sp):=\varprojlim\limits_{N,M} \Loc(\widehat{Q}_{N,M};\Sp).
\end{align}
This can be made precise using correspondences for sheaves. If we view $\widehat{L}$ as a ``module" of the Hamiltonian action generated by the stabilization of $L_0$,  then the resulting quantizations of $\widehat{L}$ as above are the ones that we are interested in. A main question that we will answer is a characterization of these quantizations of $\widehat{L}$, for example, what are the monodromies of the corresponding local systems on $\widehat{Q}$. 

\subsection{Stratified Morse theory}\label{subsec: SMT}
Let $X$ be a smooth manifold and $\fS=\{S_\alpha\}$ be a Whitney stratification of $X$. Let 
\begin{align*}
\Lambda_\fS=\bigcup\limits_{S_\alpha\in\fS}T^*_{S_\alpha}X
\end{align*}
be the union of the conormal bundles of the strata, which is a closed conic Lagrangian in $T^*X$. Now given any $\fS$-constructible sheaf $\cF$  and a covector $(x,p)$ in the \emph{smooth} locus of $\Lambda_\fS$, denoted by $\Lambda_\fS^{sm}$, one can define the \emph{local Morse group} or \emph{microlocal stalk} of $\cF$ at $(x,p)$. The definition depends on a choice of a local function $f$ near $x$ whose differential at $x$ is $p$, and which is a Morse function restricted to the stratum containing $x$. Such a function is called a \emph{local stratified Morse function}. Roughly speaking, the microlocal stalk measures how the local sections of $\cF$ change when we move across $x$ in the way directed by $f$. One uses the microlocal stalks to define the \emph{singular support} of $\cF$, denoted by $\SS(\cF)$, which is a closed conic Lagrangian contained in $\Lambda_\fS$. The microlocal stalks and singular support play an essential role in microlocal sheaf theory and the theory of perverse sheaves. For example, a central question in microlocal sheaf theory is to calculate the microlocal sheaf category associated to a conic Lagrangian (or equivalently a Legendrian), which is roughly speaking a localization of the category of sheaves whose singular support are contained in the conic Lagrangian.

The Main theorem in stratified Morse theory \cite{stratified-Morse-theory} states (in part) that the microlocal stalk defined above is independent of the choice of the local stratified Morse functions $f$, in the sense that the microlocal stalks at $(x,p)$ for two different $f_1, f_2$ are isomorphic up to a shift of degree by the difference of their Morse indices along the stratum containing $x$. 

We will enhance the main theorem in stratified Morse theory by exhibiting the (stabilized) monodromies of the microlocal stalks along the space of all choices of local stratified Morse functions $f$ (cf. Theorem \ref{thm: intro SMT} below). This is an application of the quantization of the Hamiltonian $G$-action and its ``module" described above.

\subsection{The $J$-homomorphism}
The $J$-homomorphism is an $E_\infty$-map
\begin{align*}
J: \coprod\limits_{n}BO(n)\rightarrow \Pic(\bS)\simeq \bZ\times BGL_1(\bS),
\end{align*}
where $\Pic(\bS)$ is the classifying space of stable sphere bundles. If one takes the group completion of $\coprod\limits_{n}BO(n)$, $\bZ\times BO$, then one can uniquely extend $J$ (up to a contractible space of choices) to an $\Omega^\infty$-map from $\bZ\times BO$ to $\Pic(\bS)$. By the Thom construction, every vector bundle gives rise to a stable sphere bundle, and $J$ is the map between the classifying space of stable vector bundles to the classifying space of stable sphere bundles.  

There are several equivalent models of $J$ to express its $E_\infty$-structure. A model of $J$ as an infinite loop space map from $G$ to $\Pic(\bS)$ is given as follows (cf. \cite{Lurie-circle} for the complex $J$-homomorphism). Let $\Vect_{\bR}^{\simeq}$ be the topologically enriched category of real finite dimensional vector spaces, with morphisms being linear isomorphisms. The direct sum of vector spaces makes $N(\Vect_{\bR}^{\simeq})$ into a symmetric monoidal $\infty$-groupoid. For any $V\in \Vect_{\bR}^{\simeq}$, let $V^c$ be the one-point compactification of $V$ (with the marked point at $\infty$). The functor $V\mapsto \Sigma^\infty V^c$ of forming $\bS$-lines determines the symmetric monoidal functor 
\begin{align}\label{eq: def J}
J: N(\Vect_{\bR}^{\simeq})^\otimes\rightarrow \Pic(\bS)^\otimes. 
\end{align}

Here we give an equivalent model of $J$ using correspondences for sheaves. Let 
\begin{align*}
\pi_{VG,N}: VG_N\rightarrow G_N
\end{align*} 
be the vector bundle projection, where $G_N, VG_N$ are defined in (\ref{eq: intro G_N, VG_N}). Since $G$ and $VG$ are commutative topological monoids, their local system categories $\Loc(G;\Sp)$ and $\Loc(VG;\Sp)$ are equipped with the convolution symmetric monoidal structures 
from the symmetric monoidal functor 
\begin{align*}
\Fun(-, \Sp): \Spc^{op}&\to \PrstL\\
X&\mapsto \Fun(X,\Sp)\simeq \Loc(X,\Sp). 
\end{align*}
This will agree with the symmetric monoidal structure defined in 
Subsection \ref{subsec: intro, application, correspondences} using correspondences in the 1-category of locally compact Hausdorff spaces and then taking inverse limits (\ref{eq: intro LocVG}). Then $J$ can be viewed as a commutative algebra object in $\Loc(G;\Sp)$ whose costalks are isomorphic to suspensions of the sphere spectrum. We have the following result proved in Subsection \ref{Appendix subsec: proof of J}. 

\begin{prop}[(See Proposition \ref{prop: equiv model of J})]
The functor
\begin{align*}
(\pi_{VG})_*=\varprojlim\limits_{N}(\pi_{VG,N})_*: \Loc(VG;\Sp)\rightarrow \Loc(G;\Sp)
\end{align*}
is symmetric monoidal, and the local system $(\pi_{VG})_*\varpi_{VG}$ is the commutative algebra object in $\Loc(G;\Sp)$ classified by $J$. 
\end{prop}

Here by the dualizing sheaf on a locally compact Hausdorff space $X$, we mean the constant cosheaf with cofiber $\bS$, through Lurie's Verdier duality for sheaves of spectra (cf. \cite[Theorem 5.5.5.1]{higher-algebra})
\begin{equation*}
\bbD: \Shv(X;\Sp)\overset{\sim}{\to} \Shv(X;\Sp^{op})^{op}. 
\end{equation*}
Then $\varpi_{VG}$ is the object corresponding to $\varpi_{VG_N}, N\geq 0$, under the inverse limit. 
We emphasize that the functor $(\pi_{VG})_*$ is not well defined if we regard $\pi_{VG}$ as a morphism in the $\infty$-category of spaces $\Spc$, for it is not invariant under homotopies (note that $(\pi_{VG})_!$ is well defined but gives trivial information). This is a place where we need to work in the ordinary 1-category of locally compact Hausdorff spaces, and the validity of $(\pi_{VG})_*$ (which follows from base change) and its symmetric monoidal structure leads us to employ the category of correspondences. 

\subsection{Statement of the main results}
Now we are ready to state our main results. The space $\widehat{Q}=\varinjlim\limits_{N,M}\widehat{Q}_{N,M}$ (see (\ref{eq: intro half-space}) for the definition of $\widehat{Q}_{N,M}$) is naturally homotopy equivalent to $G$, so we can view $\widehat{Q}$ as a $G$-torsor in $\Spc$ (i.e. a free $G$-module generated by a point), and $\Loc(\widehat{Q};\Sp)\simeq \Loc(G;\Sp)$ is naturally a module of  $\Loc(G;\Sp)$. Let $\Loc(\widehat{Q};\Sp)^J$ be the stable $\infty$-category of \emph{$J$-equivariant local systems} on $\widehat{Q}$, which is a full subcategory of $J$-modules in $\Loc(\widehat{Q};\Sp)$. Here we use $J$ to represent the local system on $G\simeq \widehat{Q}$ that it classifies.  Since $\widehat{Q}$ is a $G$-torsor, $\Loc(\widehat{Q};\Sp)^J$ consists of objects $J\underset{\bS}{\otimes}\cR$, $\cR\in \Sp$, and the mapping spectrum from $J\underset{\bS}{\otimes}\cR_1$ to $J\underset{\bS}{\otimes}\cR_2$ is isomorphic to the mapping spectrum from $\cR_1$ to $\cR_2$ (see Lemma \ref{lemma: chi, fully faithful}). There is a general definition of $\Loc(Y;\Sp)^\chi$ for any commutative topological monoid $H$ acting on a space $Y$ and $\chi: H\rightarrow \Pic(\bS)$ an $E_\infty$-map called a \emph{character} (see Definition \ref{def: chi-equivariant}).

Consider the map, which is a $VG_N$-module map (cf. Subsection \ref{subsubsec: VG module widehatQ})
\begin{align*}
\psi_{N,M}: &VG_N\times Q^0_M\rightarrow \widehat{Q}_{N,M},\\
&(A,\bp; \bq,s)\mapsto (A,\bq+2\bp,s+A(\bp)),\ \bp\in A,
\end{align*}
whose stabilization represents the Hamiltonian $G$-action that sends $L_0$ to $L_A, A\in G$. 

\begin{thm}\label{thm: intro 1}
The family of correspondences 
\begin{align*}
\xymatrix{&VG_N\times Q_M^0\ar[dr]^{\pi_{Q_M^0}}\ar[dl]_{\psi_{N,M}}&\\
\widehat{Q}_{N,M}&&Q^0_M
}
\end{align*}
induces a canonical equivalence 
\begin{align*}
\psi_*\pi^!_{Q^0}:=\varprojlim\limits_{N,M}(\psi_{N,M})_*\pi^!_{Q_M^0}:\Loc(Q^0;\Sp)\overset{\sim}{\longrightarrow} \Loc(\widehat{Q};\Sp)^{J}
\end{align*}
\end{thm}
Intuitively, the theorem says that the $\varpi_{VG}$-modules in $\Loc(\widehat{Q};\Sp)$ (as a module of $\Loc(VG;\Sp)$) that are freely generated by local systems on $Q^0$ are characterized precisely as the $J$-equivariant local systems on $\widehat{Q}$ (cf. Corollary \ref{cor: equiv J-equivariant} for more details). 

Now we state an application of the above theorem to stratified Morse theory. Although stratified Morse theory is aimed at constructible sheaves over ordinary rings, there is a direct generalization of it over ring spectra (cf. \cite{JT}). Let $\Shv_\fS(X;\Sp)$ be the full subcategory of $\Shv(X;\Sp)$ consisting of sheaves whose restriction to each stratum $S_\alpha$ is locally constant. The notion of microlocal stalk can be directly generalized to take values in spectra. Following the notations in Subsection \ref{subsec: SMT}, we can stabilize $X$ and $\fS$ as $X\times \bR_{y_1}\times \bR_{y_2}\times \cdots$, $\{S_\alpha\times\bR_{y_1}\times\cdots\}$, and the space of local stratified Morse functions associated to $(x,p)$, where $x\in S_\beta$ for some $\beta$, can be stabilized by adding $-\frac{1}{2}(y_1^2+y_2^2+\cdots)$ in the newly added variables $y_1,y_2,\cdots$. Then the space of stabilized local stratified Morse functions for $(x,p)$ is canonically homotopy equivalent to $G$, through the map taking $f$ to the sum of the positive eigenspaces of the Hessian of $f|_{S_\beta\times\bR\times\cdots}$ at $x$ in $T_xS_\beta\times\bR\times\cdots$. We can view this space as a $G$-torsor in $\Spc$. 

\begin{thm}\label{thm: intro SMT}
For any $\cF\in \Shv_{\fS}(X;\Sp)$, the process of taking microlocal stalks at  $(x,p)\in\Lambda_{\fS}^{sm}$ under stabilization canonically gives rise to a $J$-equivariant local system on the space of stabilized local stratified Morse functions associated to $(x,p)$. 

Moreover, this builds a canonical equivalence between the microlocal sheaf category associated to a small conic open neighborhood of $(x,p)$ in $\Lambda_\fS^{sm}$ and the category of $J$-equivariant local systems of spectra on the space of stabilized local stratified Morse functions associated to $(x,p)$.
\end{thm}

The process of taking microlocal stalks has a natural translation to correspondences via the notion of \emph{contact transformations} defined in \cite[A.2]{KS}. A contact transformation is just a local conic symplectomorphism that transforms a local piece of conic Lagrangian to another piece of conic Lagrangian, and the local stratified Morse functions correspond to a generic class among them (modulo some obvious equivalences that act trivially on the microlocal stalks), which we call \emph{Morse transformations}.\footnote{In \cite{JT}, we called this generic class of contact transformations simply contact transformations.} These relations are established in Subsection \ref{subsec: Morse transformations}, \ref{subsec: relationtoStratifiedMorse} and \ref{subsec: relevantLocalization}. This enables us to reveal the interesting structures underlying the ``local system of microlocal stalks'' on the space 
of stabilized local stratified Morse functions, by applying the category of correspondences and using Theorem \ref{thm: intro 1}. In particular, Theorem \ref{thm: intro SMT} directly follows from Theorem \ref{thm: L, J-equivariant}. Moreover, there is a generalization of Theorem \ref{thm: intro SMT} from $(x,p)\in \Lambda_{\fS}^{sm}$ to the case that $(x,p)\in \Lambda^{sm}$ for \emph{any} conic Lagrangian $\Lambda$ (see Remark \ref{remark: Morse transf, stratified Morse}).

\subsection{Expected strategy of the proof of Claim \ref{claim: 1}}

The idea of the proof is easier to express if $G=\coprod\limits_nBO(n)$ were a group. Let us first assume that this were the case. Let $L$ be an (immersed) exact Lagrangian in $T^*\bR^N$ and let $\bL$ be a conic lifting of $L$ in $T^{*,<0}(\bR^N\times\bR_t)$. The argument holds for $L$ being non-exact and can be generalized to the case when $\bR^N$ is replaced by a smooth manifold via a standard treatment  as in \cite{AbouzaidKragh}. Now we cover $L$ by small open sets $\{\Omega_i\}_{i\in I}$ whose finite intersections are all contractible if not empty (this is called a good cover), which gives a sufficiently fine covering of the front projection of $\bL$. In \cite{JT}, we show that for each finite intersection $\bigcap\limits_{s=1}^k\Omega_{i_s}$, the associated microlocal sheaf category with coefficient $\bS$ is (non-canonically) equivalent to the $\infty$-category of spectra, which defines a local system of stable $\infty$-categories on $L$ with fiber equivalent to $\Sp$ denoted by $\Brane_L$. Such an equivalence of categories follows from the correspondence given by a choice of Morse transformation, and Theorem \ref{thm: intro SMT} implies that if we consider the Morse transformations altogether (again modulo some obvious equivalences), then the microlocal sheaf category associated to $\bigcap\limits_{s=1}^k\Omega_i$ is canonically equivalent to the category of $J$-equivariant local systems of spectra on the space of  Morse transformations associated to $\bigcap\limits_{s=1}^k\Omega_i$ which is a $G$-torsor. 

Let $\check{C}_\bullet L$ be the $\check{\text{C}}$ech nerve of the covering $\{\Omega_i\}_{i\in I}$, and let $G\torsors$ denote the $E_\infty$-groupoid of $G$-torsors in $\Spc$, which is equivalent to $BG$ (assuming $G$ were a group). 
The above implies that the classifying map for $\Brane_L$ factors as 
\begin{align*}
\check{C}_\bullet L\overset{\check{\gamma}}{\longrightarrow} G\torsors\overset{\Loc(-;\Sp)^J}{\longrightarrow} B\Pic(\bS)
\end{align*}
where the first map $\check{\gamma}$ is taking the space of  Morse transformations and the latter is taking a $G$-torsor $M$ to $\Loc(M;\Sp)^J$. Here we use a compatibility result of Morse transformations under the restriction from a small open set to a smaller one in the $\check{\text{C}}$ech cover. One can show that $\check{\gamma}$ is homotopic to the stable Gauss map after passing to the geometric realization $|\check{C}_\bullet L|\simeq L$, and $\Loc(-;\Sp)^J$ exactly models the delooping of the $J$-homomorphism (up to the canonical involution on $BG$). This proves Claim \ref{claim: 1} under the assumption that $G$ were a group. 
 
For the actual situation in which $G=\coprod\limits_nBO(n)$ is not a group, some more work needs to be done, which involves a more detailed study of the space $U/O$ and the compatibility of Morse  transformations under restrictions (see Subsection \ref{subsection: compatibility A_S}). For example, one technical step in \cite{J-J} is to replace the universal principal $\bZ\times BO$-bundle $E(\bZ\times BO)\to U/O$ by an $\infty$-category of quadruples $(\cU, Q_\flat, \leftindex^{\backprime}{G}, M)$, denoted by $\QHam(U/O)$, over $\Open(U/O)^{op}$, where $\cU$ is an open subset of $U/O$, $Q_\flat$ is a (stabilized) quadratic form that plays a similar role as $A_S$ in (\ref{diagram: composite corr}) with $T_{(0,0)}L$ replaced by $\ell\in \cU$, $\leftindex^{\backprime}{G}$ is a commutative algebra object in $\Corr(\Slch)$ (see Subsection \ref{subsec: intro correspondences} below) that plays a similar role as $G$ (or $G^{A^\perp}$ (\ref{eq: G_N, Aperp, n}) depending on the choice of $A_S$), and $M$ is a free module of $\leftindex^{\backprime}{G}$ generated by $Q_\flat$ (parametrizing a space of Morse transformations that work for all $\ell\in \cU$).  A morphism between two objects $(\cU_i, Q^{(i)}_\flat, \leftindex^{\backprime}{G}_{(i)}, M_{(i)})$ is roughly given by $\cU_2\hookrightarrow \cU_1$, an inclusion $\leftindex^{\backprime}{G}_{(1)}\hookrightarrow \leftindex^{\backprime}{G}_{(2)}$ that is a commutative algebra homomorphism, and a $\leftindex^{\backprime}{G}_{(1)}$-module map $M_{(1)}\hookrightarrow M_{(2)}$ that sends $Q_\flat^{(1)}$ to some appropriate element in $M_{(2)}$. The key point of introducing $\QHam(U/O)$ is that the quantization results contained in this paper will be assembled into a natural functor $F: \QHam(U/O)\to \Fun(\Delta^1, \PrstL)$ that sends $(\cU, Q_\flat, \leftindex^{\backprime}{G}, M)$ to an equivalence of categories analogous to that in Theorem \ref{thm: intro 1} 
(more precisely see \cite[Proposition 4.13]{J-J}), and it further induces a functor $\Open(U/O)^{op}\to \Fun(\Delta^1, \PrstL)$ by right Kan extension. 
To give an appropriate definition of $\QHam(U/O)$ that works for our purposes, we must allow $Q_\flat$ and $\leftindex^{\backprime}{G}$ to be more flexible than $A_S$ and $G^{A^\perp}$ that are considered in Subsection \ref{subsection: compatibility A_S}. For example, since our $\ell$ is moving in $\cU$, we cannot expect (\ref{diagram: composite corr}) to hold for all $\ell\in \cU$. There are other more subtle reasons that impose both flexibility and technical conditions on $Q_\flat$, $\leftindex^{\backprime}{G}$ and the morphisms between the objects in $\QHam(U/O)$  (cf. \cite[\S 4.1]{J-J}).

In the following, we give a brief overview of the category of correspondences, the results we have obtained and how we apply these results to the quantization process.

\subsection{The category of correspondences and the functor of taking sheaves of spectra}\label{subsec: intro correspondences}
For any ordinary 1-category $\cC$ that admits fiber products, one can define an ordinary 2-category $\bCorr(\cC)$ (cf. \cite{Nick}) with the same objects as $\cC$ and whose set of morphisms from an object $X$ to $Y$ consists of correspondences
\begin{align}\label{eq: intro corr diagram}
\xymatrix{W\ar[r]^f\ar[d]_g&X\\
Y&}.
\end{align}
A 2-morphism between 1-morphisms is presented by the left one of the following commutative diagrams  (\ref{diagram: 2-morphism}), and a composition of two 1-morphisms is presented by the outer correspondence on the right. 
\begin{align}\label{diagram: 2-morphism}
\begin{tikzpicture}
\draw[->] (0.2,0) node[left] {$W_1$}--(2,0) node[right] {$X$};
\draw[->] (0.15,-0.15)--(0.5,-0.5) node[right] {$W_2$};
\draw[->] (1,-0.5)--(2,-0.1);
\draw[->] (0,-0.2)--(0,-2) node[below]{$Y$};
\draw[->] (0.7, -0.7)--(0.1, -2);
\end{tikzpicture}
\hspace{1in}
\begin{tikzpicture}
\draw[->] (0.2,0) node[left] {$W_1\underset{Y}{\times}W_2$}--(1,0) node[right] {$W_1$};
\draw[->] (1.6,0)--(2.4,0) node[right] {$X$};
\draw[->] (1.3,-0.1)--(1.3,-1) node[below] {$Y$};
\draw[->] (-0.3,-0.2)--(-0.3,-1) node[below] {$W_2$};
\draw[->] (0.2,-1.2)--(1,-1.2);
\draw[->] (-0.3,-1.5)--(-0.3,-2.4) node[below] {$Z$};
\end{tikzpicture}
\end{align}

There is a higher categorical analogue of this where $\cC$ is an $\infty$-category and $\bCorr(\cC)$ is an $(\infty,2)$-category. One can also forget the non-invertible 2-morphisms, and get an $\infty$-category $\Corr(\cC)$.  If $\cC$ is (symmetric) monoidal then $\bCorr(\cC)$ is (symmetric) monoidal, and so is $\Corr(\cC)$. 

In \cite{Nick}, the authors introduce the $(\infty,2)$-category of correspondences and use it to give a systematic treatment of the six-functor formalism for dg-categories of quasi/ind-coherent sheaves. A similar treatment without using $(\infty,2)$-categories has been taken in \cite{Yifeng2}. The approach of \cite{Nick} can be adapted to the case of sheaves of spectra, given the foundations on $\infty$-topoi \cite{higher-topoi} and stable $\infty$-categories \cite{higher-algebra}. Let $\Slch$ be the ordinary 1-category of locally compact Hausdorff spaces, equipped with the Cartesian symmetric monoidal structure, i.e. the tensor product is just the Cartesian product. The starting point is the symmetric monoidal functor (cf. \cite{higher-topoi})
\begin{align*}
\Shv\cS:& (\Slch)^{op}\rightarrow \PrL\\
&X\mapsto \Shv(X;\cS)\\
&(f: X\rightarrow Y)\mapsto (f^*: \Shv(Y;\cS)\rightarrow\Shv(X;\cS)),
\end{align*}
where $\Shv(X;\cS)$ denotes the (presentable) $\infty$-category of sheaves of spaces on $X$ and $\PrL$ denotes the $\infty$-category of presentable $\infty$-categories with continuous functors. 
Taking stabilizations of $\Shv(X;\cS)$ to $\Shv(X;\Sp)$ and running the approach in \cite{Nick}, one gets all six functors, and one has the following statement that is crucial for our purposes mentioned above:
\begin{thm}\label{thm: ShvSp}
There is a canonical symmetric monoidal functor
\begin{align*}
\ShvSp: \Corr(\Slch)\rightarrow \PrstL
\end{align*}
sending $X$ to $\Shv(X;\Sp)$ and any correspondence (\ref{eq: intro corr diagram}) to the functor $g_!f^*: \Shv(X;\Sp)\rightarrow \Shv(Y;\Sp)$. 
\end{thm}

Here $\PrstL$ is the $\infty$-category of presentable stable $\infty$-categories. The theorem will be proved in Section \ref{sec: ShvSp}. Now to define a symmetric monoidal structure on $\Shv(X;\Sp)$ for some $X\in \Slch$, it suffices to endow $X$ with the structure of a commutative algebra object in $\Corr(\Slch)$. This is exactly the motivation for us to develop general and concrete constructions of (commutative) algebra objects in Section \ref{section: background}, together with their modules, (right-lax) homomorphisms between algebras, etc. in $\bCorr(\cC)$, for a Cartesian symmetric monoidal $\infty$-category $\cC$. 

Let  $\PrstR\simeq (\PrstL)^{op}$ be the $\infty$-category of presentable stable $\infty$-categories with limit preserving functors. There is a canonical identification $\Corr(\Slch)\simeq \Corr(\Slch)^{op}$ by flipping the correspondences (cf. \cite[Chapter 9, 2.2]{Nick}). Hence by taking opposite categories on the source and target of $\ShvSp$, we get
\begin{cor}\label{cor: ShvSp!*}
There is a canonical symmetric monoidal functor
\begin{align*}
\ShvSp^!_*: \Corr(\Slch)\rightarrow \PrstR
\end{align*}
sending $X$ to $\Shv(X;\Sp)$ and any correspondence (\ref{eq: intro corr diagram}) to the functor $g_*f^!: \Shv(X;\Sp)\rightarrow \Shv(Y;\Sp)$. 
\end{cor}

\begin{remark}\label{remark:locconstcoshv}
\begin{itemize}
\item[(i)]
For any $X\in \Slch$, let $\Loc(X;\Sp)$ be the category of locally constant \emph{cosheaves} on $X$ (also referred as local systems). Let $\Loc'(X;\Sp)$ be the category of locally constant sheaves on $X$. The former is identified, via the Verdier duality functor, with the full subcategory of $\Shv(X;\Sp)$ consisting of $\cF$ with $\Gamma_c(-,\cF)$ locally constant. For a smooth manifold $X$, $\Loc(X;\Sp)$ is identified with $\Loc'(X;\Sp)$. Throughout the paper, the spaces in $\Slch$ that we are dealing with are smooth, e.g. $G_N, VG_N$, so we will not distinguish $\Loc$ and $\Loc'$ for them\footnote{On the other hand, for any locally contractible topological space $X$, one can identify $\Fun(X, \Sp)$ with both $\Loc(X;\Sp)$ and $\Loc'(X;\Sp)$. In particular, a locally constant cosheaf $F$ on $X$ satisfies $\Gamma_c(U, F)\cong \Sigma^\infty_+U\otimes_{\bS} F_{x}^!$, where $F_x^!$ is the costalk of $F$ at $x$ and $U\ni x$ is an open subset contained in a contractible neighborhood of $x$. However, this identification between locally constant cosheaves and sheaves is, in general, completely different from the Verdier duality functor.}. 

\item[(ii)] The preference for locally constant cosheaves over locally constant sheaves emerges naturally. First, in Chapter \ref{sec: ShvSp} and a good portion of Chapter \ref{sec: quantization}, we are dealing with sheaf categories (not just local system categories), and we later focus on $\Shv(VG;\Sp)$ and $\Shv(G;\Sp)$ (and a few others) which are defined, in the standard way, as the colimit categories of $\Shv(VG_N;\Sp)$ and $\Shv(G_N;\Sp)$ over $N\in \bZ_{\geq 0}$ in $\PrstL$, respectively, through $!$-pushforward (equivalently $*$-pushforward since all inclusions are proper) for the sequence of inclusions (\ref{eq: intro G_N, VG_N}). It is well known that these colimit categories are equivalent to the 
limit category $\varprojlim_{N}\Shv(VG_N;\Sp)$ 
and $\varprojlim_{N}\Shv(G_N;\Sp)$ in $\PrstR$, respectively, through $!$-pullbacks (by taking the right adjoints), and the latter is much easier to calculate since the limits are equivalently taken in $\OneCat$. Under this setting, we have the dualizing sheaf $\varpi_{VG}$ (resp. $\varpi_G$), rather than the constant sheaf, as a well defined object in $\Shv(VG;\Sp)$ (resp. $\Shv(G;\Sp)$) and as the monoidal unit of $\overset{!}{\otimes}$ on $\Shv(VG;\Sp)$ (resp. $\Shv(G;\Sp)$). Second, it is convenient to characterize the dualizing sheaf $\varpi_X$ as the constant cosheaf with costalk $\bS$. In particular, the $!$-pullback functors on sheaves are quite explicit using the cosheaf perspective.   

\item[(iii)] In summary, the six functors that we consider are all for sheaves (we don't consider operations separately for cosheaves). On the other hand, the cosheaf perspective gives a convenient way to describe certain sheaves in $\Shv(VG;\Sp)$, e.g. $\varpi_{VG}$, and makes the $!$-pullback functors much more explicit. Furthermore, this offers us a convenient transition from viewing $!$-pullback between $\Loc(X;\Sp)$ for maps inside $\Slch$ to $!$-pullback between $\Loc(X;\Sp)$ for the corresponding homotopy classes of maps inside $\Spc$\footnote{To avoid any technical issues, when we make such a transition (to further identify $\Loc(X;\Sp)$ with $\Fun(X, \Sp)$), we always assume that $X$ is both in $\Slch$ and is a CW-complex. This holds for all the explicit $X$ considered in this paper, e.g. $VG_N$ and $G_N$.} (which is, for example, needed for our formulation of the $J$-homomorphism). The latter makes sense by identifying $\Loc(X;\Sp)$ with $\Fun(X, \Sp)$ (so then it becomes the standard pullback functor), and it has a left adjoint, denoted as $!$-pushforward between local systems.  

\end{itemize}
\end{remark}

Now let 
\begin{align}\label{eq: Loc_*up!}
(\Loc)^!_*: \Corr(\Slch)_{\fib, \all}&\rightarrow \PrstL,\\
\nonumber X&\mapsto \Loc(X;\Sp)
\end{align}
be the functor induced from $\ShvSp^!_*$, which restricts the vertical arrows in correspondences to locally trivial fibrations and takes each $X$ to the full subcategory of locally constant cosheaves. 
Note that here we can write the target as $\PrstL$ instead of $\PrstR$, since all the functors involved are continuous. 
In our applications, we will mainly use the functor $\Loc^!_*$, which is symmetric monoidal. Then (as mentioned before), the limit (\ref{eq: intro LocVG}) and (\ref{eq: intro, LochatQ}) are taken with respect to $!$-pullbacks.

\subsection{Applications of the category of correspondences to quantizations}\label{subsec: intro, application, correspondences}
For simplicity, we only explain the application of correspondences to define the symmetric monoidal convolution structure on $\Loc(VG;\Sp)$ and the module structure on $\Loc(\widehat{Q};\Sp)$. Other applications follow from a similar line of ideas. In order to apply the functor $\Loc_*^!$ whose source is $\Corr(\Slch)_{\fib,\all}$, we need to first work with $VG_N$ and $\widehat{Q}_{N,M}$ and then show the compatibility of the structures with taking limit over $N$ and $(N,M)$ respectively. In fact, using our approach we can show that the symmetric monoidal structure and module structure already exist at finite levels, i.e. on $\Loc(VG_N;\Sp)$ and $\Loc(\widehat{Q}_{N,M};\Sp)$ respectively.  

The convolution structures at finite levels come from the following correspondences
\begin{align}\label{eq: intro convolution}
\xymatrix{VG_{N}^{\lng 2\rng}\ar@{^{(}->}[r]^{\iota_N\ \ \ \ }\ar[d]_{a_N}&VG_{N}\times VG_{N}\\
VG_{N}&
},\ \xymatrix{\widehat{Q}_{N,M}^{\lng 1\rng_\dagg}\ar@{^{(}->}[r]^{\jmath_{N,M}\ \ \ \ }\ar[d]_{a_{N,M}}&VG_{N}\times \widehat{Q}_{N,M}\\
\widehat{Q}_{N,M}&
},
\end{align} 
where (i) $\iota_{N}$ is the embedding of the subvariety of pairs $(A_1,\bp_1;A_2,\bp_2), \bp_i\in A_i$  where $A_1\perp A_2$, and $a_{N}$ is the natural addition $(A_1\oplus A_2, \bp_1+\bp_2)$; (ii) $\jmath_{N,M}$ is defined similarly as $\iota_N$ and 
\begin{align*}
a_{N,M}: (A_1,\bp_1; A_0,\bq, s)\mapsto (A_1\oplus A_0, \bq+2\bp_1,s+A_1(\bp_1)), \bp_1\in A_1, A_1\perp A_0.
\end{align*}
The structure functors are given by 
\begin{align*}
&(a_{N})_*\iota_{N}^!\circ\boxtimes:\Loc(VG_N;\Sp)\otimes \Loc(VG_N;\Sp)\rightarrow \Loc(VG_N;\Sp), \\
&(a_{N,M})_*\jmath_{N,M}^!\circ\boxtimes: \Loc(VG_N;\Sp)\otimes \Loc(\widehat{Q}_{N,M})\rightarrow \Loc(\widehat{Q}_{N,M}).
\end{align*}
However, to give a precise definition of the symmetric monoidal structure on $\Loc(VG;\Sp)$ and the module structure on $\Loc(\widehat{Q};\Sp)$, it is not enough to just list the functors above: there are coherent homotopies that encode the commutativity and associativity properties and their compatibility with taking limit, which is an infinite amount of data. 

The category of correspondences is an ideal tool to resolve the above issues. The upshot is that all the homotopy coherence data can be packaged inside the category of correspondences where things can be checked on the nose. For example, to present the desired symmetric monoidal structure on $\Loc(VG_N;\Sp)$, we can construct a commutative algebra object in $\Corr(\Slch)_{\fib,\all}$ whose underlying object is $VG_N$ and whose multiplication rule is given by (\ref{eq: intro convolution}). To show that the inclusion $\iota_{N,N'}: VG_N\hookrightarrow VG_{N'}$, $N'\geq N$ induces a symmetric monoidal functor $\iota_{N,N'}^!: \Loc(VG_{N'};\Sp)\rightarrow \Loc(VG_N;\Sp)$, we just need to show that $\iota_{N,N'}$ can be upgraded to an algebra homomorphism from  $VG_{N'}$ to  $VG_N$ as commutative algebra objects. 
In view of the following theorem (stated informally), these can be encoded by discrete data.

Let $\Fin_*$ denote the ordinary 1-category of pointed finite sets, where any morphism between two pointed sets preserves the marked point. 
\begin{thm}
Let $\cC$ be an $\infty$-category endowed with the Cartesian symmetric monoidal structure. 
\item[(a)] Any simplicial object (resp. $\Fin_*$-object) $C^\bullet$ in $\cC$ naturally determines an associative (resp. commutative) algebra object in $\Corr(\cC)$ if a prescribed class of diagrams are Cartesian in $\cC$. 

\item[(b)] Let $C^\bullet, W^\bullet, D^\bullet$ be simplicial objects (resp. $\Fin_*$-objects) in $\cC$ that satisfy the condition in (a). Any correspondence 
\begin{align*}
\xymatrix{W^\bullet\ar[r]\ar[d] &C^\bullet\\
D^\bullet&
}
\end{align*}
naturally induces a (right-lax) homomorphism from the associative (resp. commutative) algebra object corresponding to $C^\bullet$ to that corresponding to $D^\bullet$ in $\bCorr(\cC)$ if a prescribed class of diagrams are Cartesian in $\cC$ respectively. 
\end{thm}

Indeed, we construct $\Fin_*$-objects $VG_N^\bullet$ and morphisms $VG_N^\bullet\hookrightarrow VG_{N'}^\bullet$ and apply the above theorem to show the symmetric monoidal structure on $\Loc(VG;\Sp)$. The module structure on $\Loc(\widehat{Q})$ can be obtained appealing to a similar theorem as above for modules. We make a couple of remarks. First, one can also construct a simplicial object $VG_N^\bullet$ that represents an associative algebra object in $\Corr(\Slch)_{\fib,\all}$, but it is not a Segal object (it is not even a Segal object in $\Spc$), so the above theorem strictly generalizes the construction from Segal objects in \cite{Nick}.  Second, the vertical map $a_{N,M}$ of the right correspondence in (\ref{eq: intro convolution}) is not proper, so one cannot regard the correspondence as in $\Spc$.

\subsection{Related results and future work}
Claim \ref{claim: 1} in \cite{JT} was motivated by the Nadler-Zaslow theorem (\cite{NZ}, \cite{Nadler-selecta}) and the notion of brane structures on an exact Lagrangian submanifold $L\subset T^*X$ in Floer theory (cf. \cite{Seidel}), whose obstruction classes are given by the Maslov class in $H^1(L,\bZ)$ and the second (relative) Stiefel-Whitney class in $H^2(L;\bZ/2\bZ)$. The obstruction classes are exactly the obstruction to the triviality of the composite map
\begin{align*}
L\overset{\gamma}{\rightarrow}U/O\overset{BJ}{\rightarrow} B\Pic(\bS)\overset{\varphi_{\bZ}}{\rightarrow} B\Pic(\bZ)\simeq S^1\times B^2(\bZ/2\bZ),
\end{align*}
where $\varphi_{\bZ}$ is induced from the canonical commutative ring spectrum homomorphism $\bS\rightarrow\bZ$.

For $L$ a compact exact Lagrangian submanifold in the cotangent bundle of a compact manifold, it is a prediction from Arnold's nearby Lagrangian conjecture that the classifying map $L\rightarrow B\Pic(\bS)$ is trivial.  This is confirmed in \cite{AbouzaidKragh} whose approach is based on Floer (homotopy) theory (cf. \cite{CJS}). 
Using sheaf quantization over $\bZ$, it is proved in \cite{Guillermou} that this is the case when $\bS$ is replaced by $\bZ$. The sheaf of brane structures in our setting is called the Kashiwara-Schapira stack there. In the subsequent work \cite{J-J}, we prove the triviality of the classifying map using microlocal sheaf theory over $\bS$. 

The $E_\infty$-structure of $BJ$ in Claim \ref{claim: 1} is used when the Gauss map $\gamma$ has an $E_\infty$-structure or more generally $E_n$-structure. An example that is crucial to the sheaf-theoretic approach to symplectic topology is the $E_\infty$-map $U\rightarrow U/O$ that comes from taking stabilization of the linear symplectic group actions on $T^*\bR^N$. In \cite{Lurie-circle}, it is conjectured that the obstruction to define a Fukaya category over $\bS$ on a symplectic $2n$-manifold $M$ is the composition 
\begin{align}\label{eq: Lurie obstruction}
M\longrightarrow BU\overset{B^2J_{\bC}}{\longrightarrow} B^2\Pic(\bS),
\end{align}
where the first map is the stabilization of the classifying map for the tangent bundle as an $U(n)$-bundle, and $J_{\bC}: \bZ\times BU\rightarrow \Pic(\bS)$ is the complex $J$-homomorphism.  By the compatibility of the complex $J$-homomorphism with the real $J$-homomorphism, one has $B^2J_{\bC}$ is homotopic to the delooping of the classifying map of brane structures $U\rightarrow U/O\rightarrow B\Pic(\bS)$. 

The analogue of the Fukaya category in sheaf theory for a Weinstein manifold $M$ with a choice of Lagrangian skeleton $\Lambda$ is the microlocal sheaf category along $\Lambda$. We will show in a future work that the obstruction to defining a microlocal sheaf category along $\Lambda$ is exactly the obstruction given by (\ref{eq: Lurie obstruction}) and this is based on the $E_1$-structure on the classifying map $U\rightarrow B\Pic(\bS)$. There is a different approach to this using $h$-principle given in  \cite{Shende, Nadler-Shende}.

\subsection{Acknowledgement}
I would like to thank David Nadler, Dima Tamarkin, David Treumann and Eric Zaslow for their interest and helpful discussions related to this paper. I am also grateful to John Francis, Yifeng Liu, Jacob Lurie and Nick Rozenblyum for answering several questions on higher category theory and for help with references. I am in particular grateful to an anonymous referee for useful comments and suggestions which improved the paper significantly. 

This work was done under the support of an NSF grant DMS-1854232.

\section{Background and results on the $(\infty,2)$-categories of correspondences}\label{section: background}
We recall the definition of the $(\infty,2)$-categories of correspondences for an $\infty$-category $\cC$ and its (symmetric) monoidal structure when $\cC$ is (symmetric) monoidal. Then we present several concrete constructions of (commutative) algebra objects, their modules and (right-lax) homomorphisms when $\cC$ has the Cartesian symmetric monoidal structure. 

\subsection{The definition of the $(\infty,2)$-categories of correspondences}
The materials in this subsection are following \cite[Chapter A.1 and Chapter V.1]{Nick}. Let $\OneCat$ denote the $(\infty,1)$-category of (small) $(\infty, 1)$-categories. Let $\TwoCat$ be the $(\infty,1)$-category of $(\infty,2)$-categories, defined as the full subcategory of complete Segal objects in $(\OneCat)^{\Delta^{op}}$. The tautological full embedding is denoted by 
\begin{equation*}
\Seq_\bullet: \TwoCat\longrightarrow (\OneCat)^{\Delta^{op}}.
\end{equation*}
We use the convention that a category in bold face means an $(\infty,2)$-category. For example, $\bPrstL$ means the $(\infty,2)$-category of presentable stable $\infty$-categories with 2-morphisms given by natural transformations.

 Let $\cC$ be an $(\infty,1)$-category which admits finite limits. Let $([n]\times [n]^{op})^{\geq \dgnl}$ be the full subcategory of $[n]\times [n]^{op}$, consisting of objects $(i,j), j\geq i$, pictorially given by 

\begin{equation*}
\xymatrix{
(0,n)\ar[r]\ar[d]&(0,n-1)\ar[d]\ar[r]&\cdots\ar[r]\ar[d]&(0,0)\\
(1,n)\ar[r]\ar[d]&\cdots\ar[r]\ar[d]&(1,1)\\
\vdots\ar[r]\ar[d]&\vdots\\
(n,n)
}.
\end{equation*}
Let $'\Grid_n^{\geq\dgnl}(\cC)$ be the 1-full subcategory of $\Fun(([n]\times [n]^{op})^{\geq \dgnl},\cC)$ whose objects, viewed as diagrams in $\cC$, have each square a Cartesian square in $\cC$, and whose morphisms $\eta: F_1\rightarrow F_2$ satisfy that the induced morphisms $F_1(i,i)\rightarrow F_2(i,i)$ are isomorphisms for $0\leq i\leq n$. Let $''\Grid_n^{\geq\dgnl}(\cC)$ be the maximal Kan subcomplex of $'\Grid_n^{\geq\dgnl}(\cC)$, i.e. the 1-morphisms $\eta: F_1\rightarrow F_2$ in $''\Grid_n^{\geq\dgnl}(\cC)$ further satisfy that the morphism at each entry $F_1(i,j)\rightarrow F_2(i,j)$ is an isomorphism for all $0\leq i\leq j\leq n$. 

Define $\bCorr(\cC)$ to be the $(\infty,2)$-category, considered as an object in $(\OneCat)^{\Delta^{op}}$,  given by 
\begin{equation*}
\Seq_n(\bCorr(\cC))=\leftidx{'}{\Grid_n^{\geq\dgnl}(\cC)}{}.
\end{equation*}
Let $\Corr(\cC)$ be the $(\infty,1)$-category given by 
\begin{equation*}
\Seq_n(\Corr(\cC))=\leftidx{''}{\Grid_n^{\geq\dgnl}(\cC)}{}.
\end{equation*}
More generally, if we fix three classes of morphisms in $\cC$ denoted by $vert,horiz,adm$ satisfying some natural conditions (see \cite[Chapter 7, 1.1.1]{Nick} for the precise conditions), then we can define $\Corr(\cC)_{vert, horiz}^{adm}$ to be the 2-full subcategory of $\bCorr(\cC)$ with the vertical (resp. horizontal) arrow in a 1-morphism belonging to $vert$ (resp. $horiz$), and the 2-morphisms lying in $adm$. If $adm=\isom$, where $\isom$ is the class of isomorphisms, then we will also denote $\Corr(\cC)_{vert, horiz}^{\isom}$ by $\Corr(\cC)_{vert, horiz}$.

\subsection{(Symmetric) monoidal structure on $\bCorr(\cC)$}\label{subsec: monoidal Corr}

Following \cite[Chapter 9, 4.8.1]{Nick}, one can give a description of a 2-coCartesian fibration\footnote{The construction in \emph{loc. cit.} only gives a coCartesian fibration for the $(\infty,1)$-category $\Corr(\cC)$. Our construction of a 2-coCartesian fibration for $\bCorr(\cC)$ is a slight modification of it in the $(\infty,2)$-setting.} $\bCorr(\cC)^{\otimes, \Delta^{op}}\rightarrow N(\Delta)^{op}$ (resp. $\bCorr(\cC)^{\otimes, \Fin_*}\rightarrow N(\Fin_*)$) for the (symmetric) monoidal structure on $\bCorr(\cC)$ inherited from a (symmetric) monoidal structure $\cC^{\otimes}$ as follows (See \cite{Yifeng2} for a similar construction in the $(\infty,1)$-categorical setting). By the straightening theorem for $(\infty, 2)$-categories \cite[Chapter 11, Theorem 1.1.8, Corollary 1.3.3]{Nick}, 2-coCartesian fibrations over an $(\infty,2)$-category $\mathbb{D}$ are equivalent to functors from $\mathbb{D}$ to the $(\infty,2)$-category of $(\infty,2)$-categories.

Let $p: \cC^{\otimes,\Delta}\rightarrow N(\Delta)$ (resp. $p: \cC^{\otimes,(\Fin_*)^{op}}\rightarrow N(\Fin_*)^{op}$ be the \emph{Cartesian} fibration associated to the (symmetric) monoidal structure $\cC^{\otimes}$. Then we have 
\begin{align}
\label{eq: Seq_kbCorr, otimes}&\Seq_k(\bCorr(\cC)^{\otimes,\Delta^{op}}):=\Fun''(([k]\times[k]^{op})^{\geq \dgnl}, \cC^{\otimes, \Delta})\\
\nonumber&\Seq_k(\bCorr(\cC)^{\otimes,\Fin_*}):=\Fun''(([k]\times[k]^{op})^{\geq \dgnl}, \cC^{\otimes, (\Fin_*)^{op}}),
\end{align}
where $\Fun''(([k]\times[k]^{op})^{\geq \dgnl}, \cC^{\otimes, \Delta})$ is defined as follows. First, let $\Maps'(([n]\times[n]^{op})^{\geq\dgnl}, N(\Delta))$ be the full subcategory of $\Maps(([n]\times[n]^{op})^{\geq\dgnl}, N(\Delta))$ consisting of functors that map vertical arrows in $([n]\times[n]^{op})^{\geq\dgnl}$ to identity morphisms (more precisely, the contractible component of each self-morphism space containing the identity morphism) in $N(\Delta)$, and we have $\Fun'(([n]\times[n]^{op})^{\geq\dgnl}, \cC^{\otimes,\Delta})$ defined as the pullback 
\begin{equation}\label{eq: diagram Fun'}
\xymatrix{\Fun'(([n]\times[n]^{op})^{\geq\dgnl}, \cC^{\otimes,\Delta})\ar[r]\ar[d]&\Fun(([n]\times[n]^{op})^{\geq\dgnl}, \cC^{\otimes,\Delta})\ar[d]\\
\Maps'(([n]\times[n]^{op})^{\geq\dgnl}, N(\Delta))\ar[r] &\Fun(([n]\times[n]^{op})^{\geq\dgnl}, N(\Delta))
}.
\end{equation}

Then $\Fun''(([n]\times[n]^{op})^{\geq \dgnl}, \cC^{\otimes, \Delta})$ is  the 1-full subcategory of $\Fun'([n]\times[n]^{op})^{\geq\dgnl}, \cC^{\otimes,\Delta})$ whose objects $F$ and morphisms $q: F\rightarrow F'$ are those satisfying (Obj) and (Mor) below respectively:
\begin{itemize}
\item[(Obj)] For every square $\Delta^{1}_{vert}\times \Delta^{1}_{hor}$ in $([n]\times[n]^{op})^{\geq \dgnl}$, its image under $F$ is an edge $(f: x\rightarrow y)\in \Fun(\Delta_{vert}^1,\cC^{\otimes, \Delta})$, where $p(x)=id_{[k]}, p(y)=id_{[m]}$. Let $h$ be a Cartesian edge ending at $y$ of the Cartesian fibration $\widetilde{p}:\Fun(\Delta^1_{vert}, \cC^{\otimes, \Delta})\rightarrow \Fun(\Delta^1_{vert}, N(\Delta))$ over $\widetilde{p}(f)$. Then there exists a (contractible space of) morphism(s) $g$ in $\Fun(\Delta^1_{vert}, \cC^{\otimes, \Delta})$ together with a homotopy $f\simeq h\circ g$. We require that the square in $\cC^{\otimes, \Delta}_{[k]}\simeq (\cC)^{\times k}$ determined by $g$ is a Cartesian square.  

\item[(Mor)] $q((i,i)): F((i,i))\rightarrow F'((i,i))$ is an isomorphism for all $0\leq i\leq k$. 
\end{itemize}
The $\infty$-category $\Fun''(([n]\times[n]^{op})^{\geq \dgnl}, \cC^{\otimes, (\Fin_*)^{op}})$ is defined in the same way. We first prove that the natural functor $p_{\bCorr(\cC)}: \bCorr(\cC)^{\otimes, \Delta^{op}}\rightarrow N(\Delta)^{op}$ (resp. $p_{\bCorr(\cC)}^{\Comm}: \bCorr(\cC)^{\otimes, \Fin_*}\rightarrow N(\Fin_*)$) represents a monoidal (resp. symmetric monoidal) structure on $\bCorr(\cC)$. 

\begin{prop}\label{prop: 2-coCart monoidal}
The functor $p_{\bCorr(\cC)}: \bCorr(\cC)^{\otimes, \Delta^{op}}\rightarrow N(\Delta)^{op}$ (resp. $p_{\bCorr(\cC)}^{\Comm}: \bCorr(\cC)^{\otimes, \Fin_*}\rightarrow N(\Fin_*)$) is a 2-coCartesian fibration in the sense of \cite[Chapter 11, Definition 1.1.2]{Nick}. 
\end{prop}
Let's first state the definition of a 2-coCartesian fibration (see \cite[Chapter 11, Definition 1.1.2, \S 1.3.1]{Nick}).

\begin{definition}\label{def: 2-coCart arrow}
Let $F: \cD\rightarrow \cE$ be a functor of $(\infty,2)$-categories. We say $f: x\rightarrow y$, a $1$-morphism in $\cD$, is a \emph{2-coCartesian morphism}  with respect to $F$ if the induced functor
\begin{equation}\label{eq: 2-coCart morphism}
\bMaps_{\cD}(y,z)\longrightarrow \bMaps_{\cD}(x,z)\times_{\bMaps_{\cE}(F(x), F(z))} \bMaps_{\cE}(F(y), F(z))
\end{equation}
is an equivalence of $\infty$-categories for every $z\in \cD$. 
\end{definition}

\begin{definition}\label{def: 2-coCart fib}
Let $F: \cD\rightarrow \cE$ be a functor of $(\infty,2)$-categories. We say $F$ is a \emph{2-coCartesian fibration} if 
\begin{itemize}
\item[(i)] For any morphism $f: s\rightarrow t$ in $\cE$ and any object $x$ in $\cD$ over $s$, there exists a 2-coCartesian 1-morphism in $\cD$ emitting from $x$ that covers $f$;
\item[(iia)] For every objects $x,y$ in $\cD$, one has $\bMaps_{\cD}(x,y)\rightarrow \bMaps_{\cE}(F(x), F(y))$ a Cartesian fibration;
\item[(iib)] For any $1$-morphisms $f: x\rightarrow y$ and $g: z\rightarrow w$, the induced functors
\begin{equation*}
\bMaps_{\cD}(y,z)\rightarrow \bMaps_{\cD}(x,z), \ \bMaps_{\cD}(y,z)\rightarrow \bMaps_{\cD}(y,w)
\end{equation*}
sends each Cartesian morphism over $\bMaps_{\cE}(F(y), F(z))$ to a Cartesian morphism over $\bMaps_{\cE}(F(x),F(z))$ and $\bMaps_{\cE}(F(y),F(w))$,  respectively. 
\end{itemize}
\end{definition}

\begin{proof}[Proof of Proposition \ref{prop: 2-coCart monoidal}]
We claim that a 1-morphism in $\bCorr(\cC)^{\otimes,\Delta^{op}}$ represented by any element in $f\in\Fun''(([1]\times [1]^{op})^{\geq \dgnl}, \cC^{\otimes,\Delta})$ is 2-coCartesian if  both its horizontal and vertical arrows go to Cartesian arrows in $\cC^{\otimes, \Delta}$. Once this is proved, (i) becomes an immediate consequence. To see the claim, for any $f:x\rightarrow y$ and $z$ in $\bCorr(\cC)^{\otimes, \Delta^{op}}$ over $p_{\bCorr(\cC)}(f)^{op}:[m]\rightarrow [\ell]$ and $[n]$ respectively, 
\begin{equation*}
\xymatrix{&U\ar[r]\ar[d] &x\\
W\ar[r]\ar[d]&y &\\
z& & \\
[n]\ar[r]&[m]\ar[r]&[\ell]
}
\end{equation*}
where the vertical arrow in the diagram representing $f$ is an isomorphism, we have 
\begin{equation*}
\bMaps_{\bCorr(\cC)^{\otimes, \Delta^{op}}}(y,z)\simeq (\cC^{\otimes, \Delta}_{[n]})_{/z}\times_{\cC^{\otimes,\Delta}} \cC^{\otimes, \Delta}_{/y}
\end{equation*}
and the right-hand-side of (\ref{eq: 2-coCart morphism}) equivalent to
\begin{equation*}
(\cC^{\otimes, \Delta}_{[n]})_{/z}\times_{\cC^{\otimes,\Delta}} (\cC^{\otimes, \Delta}_{/x}\times_{N(\Delta)_{/[\ell]}} N(\Delta)_{/p_{\bCorr(C)}(f)^{op}}). 
\end{equation*}
The functor (\ref{eq: 2-coCart morphism}) is homotopic to the composition
\begin{align*}
&(\cC^{\otimes, \Delta}_{[n]})_{/z}\times_{\cC^{\otimes,\Delta}} \cC^{\otimes, \Delta}_{/y}\overset{\sim}{\rightarrow} (\cC^{\otimes, \Delta}_{[n]})_{/z}\times_{\cC^{\otimes,\Delta}} \cC^{\otimes, \Delta}_{/f_{hor}}\\
&\rightarrow (\cC^{\otimes, \Delta}_{[n]})_{/z}\times_{\cC^{\otimes,\Delta}} (\cC^{\otimes, \Delta}_{/x}\times_{N(\Delta)_{/[\ell]}} N(\Delta)_{/p_{\bCorr(\cC)}(f)^{op}}),
\end{align*}
where $f_{hor}$ is the horizontal arrow that represents $f$. 
Now the horizontal arrow in the graph representing $f$ is Cartesian in $\cC^{\otimes, \Delta}$ means that 
\begin{equation*}
\cC^{\otimes, \Delta}_{/f_{hor}}\simeq \cC^{\otimes, \Delta}_{/x}\times_{N(\Delta)_{/[\ell]}} N(\Delta)_{/p_{\bCorr(C)}(f)^{op}},
\end{equation*}
which implies that $f$ is a 2-coCartesian morphism.

It's clear that $p_{\bCorr(\cC)}$ satisfies (iia) and (iib), since the mapping space of any two objects in $N(\Delta)$ is discrete and then a Cartesian morphism in $\bMaps_{\bCorr(\cC)}(y,z)$ over $\bMaps_{N(\Delta)^{op}}(F(y), F(z))$ is the same as an isomorphism. 

A similar proof works for $p_{\bCorr(\cC)}^{\Comm}: \bCorr(\cC)^{\otimes, \Fin_*}\rightarrow N(\Fin_*)$.
\end{proof}

\begin{prop}
The 2-coCartesian fibration $p_{\bCorr(\cC)}: \bCorr(\cC)^{\otimes, \Delta^{op}}\rightarrow N(\Delta)^{op}$ (resp. $p_{\bCorr(\cC)}^{\Comm}: \bCorr(\cC)^{\otimes, \Fin_*}\rightarrow N(\Fin_*)$) represents a monoidal (resp. symmetric monoidal) structure on $ \bCorr(\cC)$.
\end{prop}
\begin{proof}
Firstly, the fiber of $p_{\bCorr(\cC)}$ at $[1]$, denoted as $\bCorr(\cC)^{\otimes, \Delta^{op}}_{[1]}$,
has $\Seq_k(\bCorr(\cC)^{\otimes, \Delta^{op}}_{[1]})$ the full subcategory of $\Fun''(([k]\times[k]^{op})^{\geq \dgnl}, \cC^{\otimes,\Delta})$ consisting of objects whose image in $\Map'(([k]\times[k]^{op})^{\geq \dgnl}, N(\Delta))$ are the constant functor to $[1]\in N(\Delta)$. Therefore, $\bCorr(\cC)^{\otimes, \Delta^{op}}_{[1]}$ is equivalent to $\bCorr(\cC^{\otimes,\Delta}_{[1]})\simeq \bCorr(\cC)$. 

Secondly, it is clear from unwinding the definitions that the natural functor induced from all the convex morphisms $[1]\rightarrow [n]$ in $N(\Delta)$,
\begin{equation*}
\Seq_k(\bCorr(\cC)_{[n]}^{\otimes,\Delta^{op}})\longrightarrow (\Seq_k(\bCorr(\cC))_{[1]}^{\otimes,\Delta^{op}})^{\times n}
\end{equation*}
is an equivalence for all $k>0$. It is also clear that the fiber $\bCorr(\cC)_{[0]}^{\otimes,\Delta^{op}}\simeq *$. Hence the proposition follows. 

A similar argument works for $p_{\bCorr(\cC)}^{\Comm}$.
\end{proof}

\subsubsection{A construction of Cartesian fibration $\cC^{\times,\Delta}\rightarrow N(\Delta)$ (resp. $\cC^{\times,(\Fin_*)^{op}}\rightarrow N(\Fin_*)^{op}$) for a Cartesian monoidal structure}\label{subsec: CtimesDelta}
We give a model of $\cC^{\times,\Delta}\rightarrow N(\Delta)$ representing a Cartesian monoidal structure on $\cC$ that is a modification of the one constructed in \cite[Notation 1.2.7]{DAGII}), which will be later convenient for us to construct algebra objects in $\Corr(\cC)^{\otimes, \Delta^{op}}$ from 
a collection of concise and discrete data. 

We define ordinary 1-categories $T$ and $T^1$ over $N(\Delta)$ as follows. The category $T^1$ 
consists of objects $([n], i\leq j)$ with any morphism $([n],i\leq j)\rightarrow ([m],i'\leq j')$ given by a morphism $f:[n]\rightarrow [m]$ in $\Delta$ satisfying $[i',j']\subset [f(i), f(j)]$ (note the slight difference from  $\Delta^\times$ in \cite[Notation 1.2.5]{DAGII}). Let $T$ be the full subcategory of $T^1$ consisting only of subintervals of length 1, i.e. $([n], i\leq i+1)$. It is clear that $T^1\rightarrow N(\Delta)$ is a coCartesian fibration, with a morphism $f: ([n],i\leq j)\rightarrow ([m],i'\leq j')$ being coCartesian if and only if $[i',j']=[f(i), f(j)]$. Under straightening, it corresponds to the functor that sends $[k]\in N(\Delta)$ to the poset of its nonempty convex subsets (with the opposite poset structure of inclusion). 

By the exactly same construction and argument as in \cite[Notation 1.2.7, Proposition 1.2.8]{DAGII}, for an $\infty$-category $\cC$ admitting finite products, one can define $\widetilde{\cC}^{\times, \Delta, 1}$ to be the simplicial set over $N(\Delta)$ determined by the property that for any simplicial set $K$ over $N(\Delta)$
\begin{align*}
&\Hom_{N(\Delta)}(K, \widetilde{\cC}^{\times, \Delta, 1})\simeq \Hom(K\times_{N(\Delta)}T^1, \cC),
\end{align*}
which is a Cartesian fibration over $N(\Delta)$, and define $\cC^{\times, \Delta, 1}$ to be the full subcategory of $\widetilde{\cC}^{\times, \Delta, 1}$ consisting of $F\in \widetilde{\cC}^{\times, \Delta, 1}_{[n]},[n]\in \Delta,$ such that the natural morphism
\begin{equation}\label{eq: condition C}
F([n],i\leq j)\rightarrow F([n], i\leq i+1)\times\cdots\times F([n], j-1\leq j)
\end{equation}
is an isomorphism for all $0\leq i\leq j\leq n$, which is a Cartesian fibration over $N(\Delta)$ as well. A morphism $p: F\rightarrow F'$ over $f: [n]\rightarrow [m]$ is a Cartesian morphism if and only if 
\begin{equation}
F([n], i\leq j)\rightarrow F'([m], f(i)\leq f(j))
\end{equation}
is an isomorphism for all $0\leq i\leq j\leq n$. 

Now we define $\cC^{\times,\Delta}$ over $N(\Delta)$ as the simplicial set characterized by the property that for any simplicial set $K$ over $N(\Delta)$
\begin{equation*}
\Hom_{N(\Delta)}(K, \cC^{\times,\Delta})\simeq \Hom(K\times_{N(\Delta)} T, \cC).
\end{equation*} 

\begin{prop}\label{prop: Cartesian equiv}
There is a natural equivalence $R:\cC^{\times,\Delta,1}\rightarrow \cC^{\times,\Delta}$ over $N(\Delta)$, which is surjective on vertices of $\cC^{\times,\Delta}$. Thus $\cC^{\times, \Delta}$ is a Cartesian fibration over $N(\Delta)$ and $p: F\rightarrow F'$ is a Cartesian morphism over $f:[n]\rightarrow [m]$ if and only if 
\begin{equation}\label{lemma eq: Cartesian condition}
F([n], i\leq i+1)\rightarrow \prod\limits_{[j,j+1]\subset [f(i), f(i+1)]}F'([m], j\leq j+1)
\end{equation}
is an isomorphism for all $0\leq i< n$. 
\end{prop}
\begin{proof}
The functor $R$ is induced from the compatible system of restriction morphisms
\begin{equation*}
\Hom(K\times_{N(\Delta)}T^1,\cC)\longrightarrow \Hom(K\times_{N(\Delta)}T,\cC),
\end{equation*}
for all $K$ over $N(\Delta)$. For each $\infty$-category $\cD$ over $N(\Delta)$, we have a map 
\begin{equation}\label{proof eq: Maps}
\Maps_{N(\Delta)}(\cD, \cC^{\times,\Delta,1})\simeq \Maps'(\cD\times_{N(\Delta)}T^1, \cC) \rightarrow 
\Maps(\cD\times_{N(\Delta)}T, \cC)\simeq\Maps_{N(\Delta)}(\cD, \cC^{\times,\Delta}),
\end{equation}
where $\Maps'(\cD\times_{N(\Delta)}T^1,\cC)$ is the full subcategory of $\Maps(\cD\times_{N(\Delta)}T^1,\cC)$ spanned by the functors $F$ satisfying an analogue of (\ref{eq: condition C}), i.e. for every object $x\in \cD$, let $[n_x]$ denote $p(x)$ where $p: \cD\rightarrow N(\Delta)$ is the projection, then the natural morphism 
\begin{equation*}
F(x, ([n_x],i\leq j))\rightarrow F(x, ([n_x],i\leq i+1))\times\cdots\times F(x, ([n_x],j-1\leq j))
\end{equation*}
is an isomorphism for all $0\leq i\leq j\leq n_x$.  We want to show that the middle arrow  in (\ref{proof eq: Maps}) is a trivial Kan fibration, then this proves $\cC^{\times,\Delta, 1}\overset{\sim}{\rightarrow }\cC^{\times,\Delta}$ and it is surjective on vertices.

We will apply \cite[4.3.2.15]{higher-topoi} by showing that every functor $F: \cD\times_{N(\Delta)}T\rightarrow \cC$ has a right Kan extension to $\cD\times_{N(\Delta)}T^1$, and $\Maps'(\cD\times_{N(\Delta)}T^1,\cC)$ is the full subcategory of $\Maps(\cD\times_{N(\Delta)}T^1,\cC)$ spanned by right Kan extensions of functors $F: \cD\times_{N(\Delta)}T\rightarrow \cC$. For any object $(x, ([n_x],i\leq j))$ in $\cD\times_{N(\Delta)}T^1$ with $j\geq i+2$, we have 
\begin{equation*}
(\cD\times_{N(\Delta)}T)_{(x, ([n_x],i\leq j))/}\simeq \coprod\limits_{i\leq k\leq j-1}(\cD\times_{N(
\Delta)}T)_{(x, ([n_x],k\leq k+1))/},
\end{equation*}
for every arrow $(x, ([n_x],i\leq j))\rightarrow (y, ([n_y],s\leq s+1))$ has a unique factorization 
\begin{equation*}
(x, ([n_x],i\leq j))\rightarrow (x, ([n_x],k\leq k+1))\rightarrow (y, ([n_y],s\leq s+1))
\end{equation*}
for a unique $k\in [i,j-1]$, where the first map is defined by the identity on $x$ and the inclusion $[k,k+1]\subset [i,j]$. Since  $(x, ([n_x],k\leq k+1))$ is the initial object in $(\cD\times_{N(
\Delta)}T)_{(x, ([n_x],k\leq k+1))/}$, for any $F: \cD\times_{N(\Delta)}T\rightarrow \cC$, the limit of the induced functor $F_{(x, ([n_x],i\leq j))}:(\cD\times_{N(\Delta)}T)_{(x, ([n_x],i\leq j))/}\rightarrow \cC$ is isomorphic to $\prod\limits_{i\leq k\leq j-1}F(x, ([n_x],k\leq k+1))$, it follows that the middle map in (\ref{proof eq: Maps}) is a trivial Kan fibration. 

Lastly, since $R$ is surjective on vertices, a morphism $p: F\rightarrow F'$ is Cartesian in $\cC^{\times,\Delta}$ if and only if it is the image of a Cartesian morphism through $R$. This is clearly the condition (\ref{lemma eq: Cartesian condition}). 
\end{proof}

One can define $\cC^{\times,(\Fin_*)^{op}}\rightarrow N(\Fin_*)^{op}$ in a similar way. Let $T^{1,\Comm}$ be the ordinary category consisting of objects $(\langle n\rangle,S), S\subset \langle n\rangle^{\circ}$, and morphisms $\alpha^{op}: (\langle n\rangle,S)\rightarrow (\langle m\rangle,S')$ given by a morphism $\langle m\rangle\rightarrow \langle n\rangle$ in $\Fin_*$ such that $\alpha^{-1}(S)\supset S'$. The projection $T^{1,\Comm}\rightarrow N(\Fin_*)^{op}$ is a coCartesian fibration with a morphism $\alpha^{op}: (\langle n\rangle,S)\rightarrow (\langle m\rangle,S')$ being coCartesian if and only if $S'=\alpha^{-1}(S)$. Let $T^{\Comm}$ be the full subcategory of $T^{1,\Comm}$ consisting of $(\langle n\rangle, S)$ where $S$ contains exactly one element.

Define $\cC^{\times,(\Fin_*)^{op}}$ over $N(\Fin_*)^{op}$ to be the simplicial set characterized by that for any simplicial set $K$ over $N(\Fin_*)^{op}$, we have 
\begin{equation}\label{eq: cCtimes, Finop}
\Hom_{N(\Fin_*)^{op}}(K, \cC^{\times, (\Fin_*)^{op}})\simeq \Hom(K\times_{N(\Fin_*)^{op}}T^{\Comm},\cC).
\end{equation}
Then we have the commutative analogue of Proposition \ref{prop: Cartesian equiv}.

\begin{prop}\label{prop: Cartesian equiv comm}
The projection $\cC^{\times,(\Fin_*)^{op}}\rightarrow N(\Fin_*)^{op}$ is a Cartesian fibration over $N(\Fin_*)^{op}$ representing the Cartesian symmetric monoidal structure on $\cC$.  A morphism $p: F\rightarrow F'$ is a Cartesian morphism over $\alpha^{op}:\langle n\rangle\rightarrow \langle m\rangle$ if and only if 
\begin{equation}\label{lemma eq: Cartesian condition comm}
F(\langle n\rangle,\{i\})\rightarrow \prod\limits_{j\in\alpha^{-1}(i)}F'(\langle m\rangle, \{j\})
\end{equation}
is an isomorphism for all $i\in \langle n\rangle^\circ$. 
\end{prop}

\subsection{Construction of (commutative) algebra objects of $\Corr(\cC^\times)$}\label{subsubsec: algebra objects}

From now on, we assume that $\cC$ has the Cartesian (symmetric) monoidal structure, and $\Corr(\cC)$ is endowed with the induced (symmetric) monoidal structure.  Note that the latter is \emph{not} Cartesian in general, for the final object in $\Corr(\cC)$ is the initial object in $\cC$. 

It is shown in \cite[Corollary 4.4.5, Chapter V.3]{Nick} that every Segal object $C^\bullet$ of $\cC$ gives an algebra object in $\Corr(\cC)$. Here we give a more general construction of (commutative) algebra objects in $\Corr(\cC)$ out of simplicial or $\Fin_*$-objects in $\cC$ (see Theorem \ref{thm: algebra objs}). Moreover, we provide a concise condition for defining a (right-lax) homomorphism between two (commutative) algebra objects constructed in this way (see Theorem \ref{thm: right-lax}). We expect that this construction gives ``all" the (commutative) algebra objects and (right-lax) homomorphisms between them in the sense of the equivalences in Remark \ref{alg composition}. We will address this point in a separate note. 

Let $\pi_I: I\rightarrow N(\Delta)^{op}$ (resp. $\pi_I^\Comm: I^\Comm\rightarrow N(\Fin_*)$) be the coCartesian fibration (actually a left fibration) from the Grothendieck construction for $\Seq_\bullet(N(\Delta)^{op}):N(\Delta)^{op}\rightarrow\text{Sets}$ (resp. $\Seq_\bullet(N(\Fin_*)):N(\Delta)^{op}\rightarrow\text{Sets}$). More explicitly, the set of objects of $I$ is the union of discrete sets $\bigsqcup\limits_{n}\Seq_n(N(\Delta)^{op})$, and the morphisms are induced from the maps $\Seq_{m}(N(\Delta)^{op})\rightarrow \Seq_{n}(N(\Delta)^{op})$, for $[n]\rightarrow [m]$ in $\Delta$.

The assignment 
\begin{align*}
f_I: I^{op}&\longrightarrow \OneCat^{\ord}, \\
 ([n],s)&\mapsto ([n]\times[n]^{op})^{\geq\dgnl}\times_{N(\Delta)}T,
\end{align*}
in which the map $([n]\times[n]^{op})^{\geq\dgnl}\rightarrow N(\Delta)$ is the composition of the projection of $([n]\times[n]^{op})^{\geq\dgnl}$ to the second factor with $s^{op}: (\Delta^n)^{op}\rightarrow N(\Delta)$ and $f_I(\alpha)$ is the identity on the factors $N(\Delta)$ and $T$ for any $(\alpha: ([n],s)\rightarrow ([m], t))$ in $I^{op}$, 
gives a Cartesian fibration 
\begin{align*}
\pi_\cT: \cT\rightarrow I
\end{align*}
from the Grothendieck construction. Similarly, we have a Cartesian fibration 
\begin{align*}
\pi_{\cT^\Comm}: \cT^\Comm\rightarrow I^{\Comm}. 
\end{align*}

Consider the functors
\begin{align}
\label{eq: f_I^C} f_I^\cC: I&\longrightarrow \OneCat, \\
\nonumber([n],s)&\mapsto \Fun (([n]\times[n]^{op})^{\geq\dgnl}\times_{N(\Delta)}T,\cC)
\end{align}
which is a composition of $f_I^{op}$ with $\Fun(-, \cC): (\OneCat^{\ord})^{op}\rightarrow\OneCat$, and 
\begin{align}
\label{eq: f_Delta^C}f_{\Delta^\bullet}^\cC: N(\Delta)^{op}&\longrightarrow\OneCat\\
\nonumber [n]&\mapsto \coprod\limits_{s\in \Seq_n(N(\Delta)^{op})}\Fun(([n]\times[n]^{op})^{\geq\dgnl}\times_{N(\Delta)}T,\cC)\simeq \Fun'([n]\times[n]^{op})^{\geq\dgnl}, \cC^{\times,\Delta}),
\end{align}
where $\Fun'([n]\times[n]^{op})^{\geq\dgnl}, \cC^{\times,\Delta})$ was defined above (\ref{eq: diagram Fun'}). 
It is clear that $f_\Delta^\cC$ is a left Kan extension of $f_I^\cC$ through the projection $\pi_I$. 

We state a lemma before proceeding further on the constructions. Assume $S=N(\cD)$ is the nerve of a small ordinary category $\cD$. Let $f: \cD^{op}\rightarrow \SSet^+$ be a fibrant object in $(\SSet^+)^{\cD^{op}}$. Let $N_\bullet^{+,op}(\cD): (\SSet^+)^{\cD^{op}}\rightarrow (\SSet^+)_{/S}$ and  $N_\bullet^{+}(\cD): (\SSet^+)^{\cD}\rightarrow (\SSet^+)_{/S}$ be the relative nerve functor as in \cite[\S 3.2.5]{higher-topoi}. Let $X^\natural\longrightarrow S$ denote $N_f^{+,op}(\cD)\in (\SSet^+)_{/S}$, and let $\cZ^\natural\in \SSet^+$ be an $\infty$-category. Here we use the notation following \cite[Definition 3.1.1.8]{higher-topoi}, that if $X\rightarrow S$ is a (co)Cartesian fibration, then $X^\natural$ is the marked simplicial set with all its (co)Cartesian edges being marked.  In particular, for an $\infty$-category $\cZ$, an edge is marked in $\cZ^\natural$ if and only if it is an isomorphism. Let $g: \cD\rightarrow \SSet^+$ denote the composition $\cD\overset{f^{op}}{\rightarrow} (\SSet^+)^{op}\overset{(Z^{\natural})^{(-)}}{\rightarrow} \SSet^+$. We define 
\begin{equation}\label{eq: Z^X}
(\cZ^\natural)^{X^\natural}=N_g^{+}(\cD).
\end{equation}
Following \cite[Corollary 3.2.2.12]{higher-topoi}, let $T_{\cZ}^X$ be the simplicial set defined by the property
\begin{equation*}
\Hom_S(K, T_{\cZ}^X):=\Hom_S(K\times_S X, \cZ\times S), \text{ for any }K\in (\SSet)_{/S},
\end{equation*}
then $T_{\cZ}^X\rightarrow S$ is a coCartesian fibration and $(T_{\cZ}^{X})^{\natural}\in (\SSet^+)_{/S}$ can be characterized by the property
\begin{equation*}
\Hom_S(K, (T_{\cZ}^{X})^{\natural})=\Hom_S(K\times_S X^\natural, (\cZ\times S)^\natural), \text{ for any }K\in (\SSet^+)_{/S}. 
\end{equation*}

\begin{lemma}\label{lemma: Z^X}
Under the above assumptions, there is a weak equivalence  $(\cZ^\natural)^{X^\natural}\longrightarrow (T_{\cZ}^X)^\natural$ with respect to the coCartesian model structure on $(\SSet^+)_{/S}$. 
\end{lemma}
\begin{proof}
We have a natural map in $(\SSet^+)_{/S}$ 
\begin{equation*}
N^+_g(\cD)\times_S X^\natural\longrightarrow (\cZ\times S)^\natural
\end{equation*}
induced from the evaluation maps
\begin{equation*}
(Z^\natural)^{f(j)}\times f(j)\longrightarrow \cZ^\natural
\end{equation*}
for $j\in \cD$. By the definition of $(T_\cZ^X)^\natural$, this corresponds to a map 
\begin{equation}\label{eq: N_g, Z}
N_g^+(\cD)\longrightarrow (T_\cZ^X)^\natural
\end{equation}
in $(\SSet^+)_{/S}$. Since both sides of (\ref{eq: N_g, Z}) are fibrant objects in $(\SSet^+)_{/S}$ and the map induces  isomorphisms on the fibers, (\ref{eq: N_g, Z}) is a weak equivalence with respect to the coCartesian model structure on $(\SSet^+)_{/S}$. 
\end{proof}

Now coming back to the construction of algebra objects, we let $(\cC^\natural)^{\cT^\natural}\rightarrow I$ denote the coCartesian fibration defined by $N_{f_I^\cC}^+(I)$. Then Lemma \ref{lemma: Z^X} implies that 
\begin{align}\label{eq: Maps(T,C)}
& \Maps_{I}(\cT^\natural, (\cC\times I)^\natural)\simeq\Maps_{I}(I^\natural,(\cC^{\natural})^{{\cT}^\natural}) \simeq \Maps_{(\pi_I)_!I^\natural}((\pi_I)_!I^\natural, (\pi_I)_!(\cC^\natural)^{\cT^\natural})\\
\nonumber&\simeq \Maps_{(\OneCat)^{\Delta^{op}}/_{\Seq_\bullet(N(\Delta)^{op})}}(\Seq_\bullet(N(\Delta)^{op}), f_{\Delta^\bullet}^\cC).
\end{align}

There is a totally analogous construction of $f_I^\Comm, f_I^{\cC, \Comm}$ and $f_{\Delta}^{\cC, \Comm}, \cT^\Comm$ in which one does the following replacement
\begin{align*}
&I\rightsquigarrow I^\Comm,\ N(\Delta)\rightsquigarrow N(\Fin_*)^{op},\ T\rightsquigarrow T^\Comm, \cC^{\times,\Delta}\rightsquigarrow \cC^{\times,(\Fin_*)^{op}},\\
&f_{\Delta^\bullet}^\cC\rightsquigarrow f_{\Delta^\bullet}^{\cC,\Comm}.  
\end{align*}

Recall that a morphism $f: [n]\rightarrow [m]$ in $N(\Delta)$ (resp. $f: \lng m\rng\rightarrow \lng n\rng$ in $N(\Fin_*)$) is called \emph{active} if $[f(0), f(n)]=[0,m]$ (resp. $f^{-1}(*)=\{*\}$). A morphism $f: [n]\rightarrow [m]$ in $N(\Delta)$ (resp. $f: \lng m\rng\rightarrow \lng n\rng$ in $N(\Fin_*)$) is called \emph{inert} if $f$ is of the form $f(i)=i+k, 0\leq i\leq n$ for a fixed $k$  (resp. $f^{-1}(i)$ has exactly one element for every $i\in \lng n\rng^\circ$). It is clear that every inert morphism $f: [n]\rightarrow [m]$ in $N(\Delta)$ (resp. $f: \lng m\rng\rightarrow\lng n\rng$ in $N(\Fin_*)$) corresponds to a unique embedding $[n]\hookrightarrow [m]$ (resp. $\lng n\rng^\circ\hookrightarrow\lng m\rng^\circ$). 

\begin{theorem}\label{thm: algebra objs}
\item[(i)] Given any simplicial object $C^\bullet$ in $\cC$ such that for any \emph{active} $f: [n]\rightarrow [m]$ in $N(\Delta)$ and $0\leq j\leq n-1$, the following square
\begin{equation}\label{thm: diagram Cartesian}
\xymatrix{C^{[m]}\ar[r]\ar[d]&\prod\limits_{[i,i+1]\subset [n]}C^{[f(i),f(i+1)]}\ar[d]\\
C^{[n]}\ar[r]&\prod\limits_{[i,i+1]\subset [n]}C^{[i,i+1]}
}
\end{equation}
is Cartesian in $\cC$, where the vertical morphisms are induced by $f$ and the horizontal morphisms are induced from the inert morphisms $[i,i+1]\hookrightarrow [n]$, $[f(i),f(i+1)]\hookrightarrow [m]$, then $C^\bullet$ naturally induces an associative algebra object in $\Corr(\cC^\times)$ via (\ref{eq: Maps(T,C)}). \\

\item[(ii)] Given any functor $C^\bullet: N(\Fin_*)\longrightarrow \cC$ such that for any \emph{active} $f: \langle m\rangle\longrightarrow \langle n\rangle$, the following diagram
\begin{equation}\label{thm: diagram Cartesian 2}
\xymatrix{C^{\lng m\rng}\ar[r]\ar[d]&\prod\limits_{i\in \langle n\rangle^\circ} C^{(f^{-1}(i)\sqcup\{*\})}\ar[d]\\
C^{\langle n\rangle}\ar[r]&\prod\limits_{i\in \langle n\rangle^\circ}C^{(\{i,*\})}
}
\end{equation}
is Cartesian in $\cC$, where the vertical morphisms are induced by $f$ and the horizontal morphisms are induced from the inert morphisms corresponding to the embeddings $\{i\}\hookrightarrow\lng n\rng$, $f^{-1}(i)\hookrightarrow \lng m\rng$, then $C^\bullet$ naturally induces a commutative algebra object in $\Corr(\cC^\times)$ via the commutative version of (\ref{eq: Maps(T,C)}).
\end{theorem}

\begin{proof}
(i) In view of (\ref{eq: Maps(T,C)}), we just need to construct an element in $\Hom_I(\cT^\natural, \cC^\natural\times I)$, and check that its image in the rightmost term gives a section of $\bCorr(\cC)^{\otimes,\Delta^{op}}\rightarrow N(\Delta)^{op}$ satisfying the properties required for an algebra object.

First, there is a natural functor $F_{\cT,\Delta^{op}}: \cT\longrightarrow N(\Delta)^{op}$ given by 
\begin{align}
\label{eq: object T}&([n],s,(i,j), u\leq u+1)\mapsto [s_{j,i}(u),s_{j,i}(u+1)]\\
\label{eq: morphism T}&(\alpha^{op}: [n]\leftarrow [m], s\rightarrow t, (i_1,j_1)\rightarrow (\alpha^{op}(i_2),\alpha^{op}(j_2)), u_2\leq u_2+1\subset [s_{j_1, \alpha^{op}(j_2)}(u_1)\leq s_{j_1, \alpha^{op}(j_2)}(u_1+1)])\\
\label{eq: morphism T_2}&\mapsto ([s_{j_1,i_1}(u_1), s_{j_1,i_1}(u_1+1)]\rightarrow [t_{j_2,i_2}(u_2), t_{j_2,i_2}(u_2+1)]),
\end{align}
in which $s_{j,i}: s(j)\rightarrow s(i), j\geq i$ is the morphism in $N(\Delta)$ determined by $s$, (\ref{eq: morphism T}) represents a morphism $\hat{\alpha}:([n],s,(i_1,j_1), u_1\leq u_1+1)\rightarrow ([m],t,(i_2,j_2), u_2\leq u_2+1)$ in $\cT$ consisting of the data of a morphism $\alpha^{op}$ in $N(\Delta)$, and the relations $t(j)=s(\alpha^{op}(j))$, $t_{a,b}=s_{\alpha^{op}(a), \alpha^{op}(b)}$, $i_1\leq \alpha^{op}(i_2)\leq \alpha^{op}(j_2)\leq j_1$ and $[u_2,u_2+1]\subset [s_{j_1,\alpha^{op}(j_2)}(u_1),s_{j_1,\alpha^{op}(j_2)}(u_1+1)]$. The morphism (\ref{eq: morphism T_2}) in $N(\Delta)^{op}$ comes from the composition
\begin{align*}
&[s_{\alpha^{op}(j_2),\alpha^{op}(i_2)}(u_2),s_{\alpha^{op}(j_2),\alpha^{op}(i_2)}(u_2+1)]\hookrightarrow [s_{j_1,\alpha^{op}(i_2)}(u_1),s_{j_1,\alpha^{op}(i_2)}(u_1+1)]\\
&\overset{s_{\alpha^{op}(i_2),i_1}}{\longrightarrow}  [s_{j_1,i_1}(u_1),s_{j_1,i_1}(u_1+1)]
\end{align*}
in $N(\Delta)$. 

It is not hard to check that the assignment (\ref{eq: morphism T_2}) is compatible with compositions. If we have two morphisms
\begin{equation*}
([n], s,(i_1,j_1), u_1\leq u_1+1)\overset{\hat{\alpha}}{\longrightarrow}([m], t,(i_2,j_2), u_2\leq u_2+1)\overset{\hat{\beta}}{\longrightarrow}([\ell], r,(i_3,j_3), u_3\leq u_3+1),
\end{equation*}
then we have the relations
\begin{align*}
&[u_3,u_3+1]\subset [t_{j_2,\beta^{op}(j_3)}(u_2),t_{j_2,\beta^{op}(j_3)}(u_2+1)]\subset [s_{j_1,(\beta\alpha)^{op}(j_3)}(u_1),s_{j_1,(\beta\alpha)^{op}(j_3)}(u_1+1)]. 
\end{align*}
Then $(F_{\cT,\Delta^{op}}(\hat{\beta})\circ F_{\cT,\Delta^{op}}(\hat{\alpha}))^{op}$ in $N(\Delta)$ is the same as the composition
\begin{align*}
&[t_{\beta^{op}(j_3),\beta^{op}(i_3)}(u_3), t_{\beta^{op}(j_3),\beta^{op}(i_3)}(u_3+1)]\hookrightarrow [t_{j_2,\beta^{op}(i_3)}(u_2),t_{j_2,\beta^{op}(i_3)}(u_2+1)]\\
&\overset{t_{\beta^{op}(i_3), i_2}}{\longrightarrow}[t_{j_2,i_2}(u_2),t_{j_2,i_2}(u_2+1)]=[s_{\alpha^{op}(j_2)\alpha^{op}(i_2)}(u_2),s_{\alpha^{op}(j_2),\alpha^{op}(i_2)}(u_2+1)]\\
&\hookrightarrow [s_{j_1,\alpha^{op}(i_2)}(u_1), s_{j_1,\alpha^{op}(i_2)}(u_1+1)]\overset{s_{\alpha^{op}(i_2), i_1}}{\longrightarrow}[s_{j_1,i_1}(u_1),s_{j_1,i_1}(u_1+1)],
\end{align*}
and $(F_{\cT,\Delta^{op}}(\hat{\beta}\hat{\alpha}))^{op}$ is the composition
\begin{align*}
&[t_{\beta^{op}(j_3),\beta^{op}(i_3)}(u_3), t_{\beta^{op}(j_3),\beta^{op}(i_3)}(u_3+1)]\hookrightarrow [s_{j_1,(\beta\alpha)^{op}(i_3)}(u_1), s_{j_1,(\beta\alpha)^{op}(i_3)}(u_1+1)]\\
&\overset{s_{(\beta\alpha)^{op}(i_3)},i_1}{\longrightarrow}[s_{j_1,i_1}(u_1),s_{j_1,i_1}(u_1+1)]. 
\end{align*}
Now it is clear that $(F_{\cT,\Delta^{op}}(\hat{\beta})\circ F_{\cT,\Delta^{op}}(\hat{\alpha}))^{op}=(F_{\cT,\Delta^{op}}(\hat{\beta}\hat{\alpha}))^{op}$. 

Consider the functor $F_{C^\bullet}: \cT\rightarrow \cC$ from the composition
\begin{equation}\label{eq: thm F_C}
F_{C^\bullet}: \cT\overset{F_{\cT,\Delta^{op}}}{\longrightarrow} N(\Delta)^{op}\overset{C^\bullet}{\longrightarrow} \cC.
\end{equation}
A morphism $\hat{\alpha}$ (\ref{eq: morphism T}) in $\cT$ is Cartesian over $I$ if and only if $i_1=\alpha^{op}(i_2), j_1=\alpha^{op}(j_2)$ and $u_1=u_2$, so $F_{C^\bullet}$ sends a Cartesian morphism in $\cT$ over $I$ to an isomorphism in $\cC$, and therefore, it belongs to $\Hom_I(\cT^\natural, \cC^\natural\times I)$, whose image in $\Hom_I(I^\natural, (\cC^\natural)^{\cT^{\natural}})$ will be denoted by $F_{C^\bullet}$ as well. For any $([n], s)\in I^\natural$, $F_{C^\bullet}([n], s)$ is the functor in $\Fun(([n]\times[n]^{op})^{\geq\dgnl}\times_{N(\Delta)}T,\cC)$ that sends $([n], s,(i,j), u\leq u+1)$ to $C^{[s_{j,i}(u), s_{j,i}(u+1)]}$. We check that $F_{C^\bullet}([n], s)$ satisfies the condition (Obj) in Subsection \ref{subsec: monoidal Corr} for any $([n],s)$, which would imply that the image of $F_{C^\bullet}$ in $\Hom_{(\OneCat)^{\Delta^{op}}/_{\Seq_\bullet(N(\Delta)^{op})}}(\Seq_\bullet(N(\Delta)^{op}), f_\Delta^\cC)$ corresponds to a functor from $N(\Delta)^{op}$ to $\bCorr(\cC)^{\otimes, \Delta^{op}}$ over $N(\Delta)^{op}$. 

The condition (Obj) is equivalent to that the following diagram 
\begin{equation*}
\xymatrix{C^{[s_{j_1,i_1}(u), s_{j_1,i_1}(u+1)]}\ar[r]\ar[d]&\prod\limits_{[v,v+1]\subset [s_{j_1,j_2}(u), s_{j_1,j_2}(u+1)]}C^{[s_{j_2,i_1}(v), s_{j_2,i_1}(v+1)]}\ar[d]\\
C^{[s_{j_1,i_2}(u), s_{j_1,i_2}(u+1)]}\ar[r]&\prod\limits_{[v,v+1]\subset [s_{j_1,j_2}(u), s_{j_1,j_2}(u+1)]} C^{[s_{j_2,i_2}(v), s_{j_2,i_2}(v+1)]}
}
\end{equation*}
is Cartesian in $\cC$ for all $i_1\leq i_2\leq j_2\leq j_1$ and $0\leq u\leq u+1\leq s(j_1)$. It can be further reduced to the case when $i_2=j_2$. Then the diagram fits into (\ref{thm: diagram Cartesian}) and all diagrams (\ref{thm: diagram Cartesian}) occur in this way, and thus the assumption in the theorem guarantees that $F_{C^\bullet}$ gives a section of $p_{\bCorr(\cC)}$. Since $N(\Delta)^{op}$ is a 1-category, the condition (Mor) is vacuous. 

The last thing we need to check is that the section sends any inert morphism $\alpha^{op}: [n]\rightarrow [m]$ in $N(\Delta)$ to a 2-coCartesian morphism. By construction, $F_{C^\bullet}([1], \alpha)$ is determined by the collection of diagrams
\begin{equation*}
\xymatrix{C^{[\alpha^{op}(u), \alpha^{op}(u+1)]}\ar[r]\ar[d]&\prod\limits_{[v,v+1]\subset [\alpha^{op}(u), \alpha^{op}(u+1)]}C^{[v,v+1]}\\
C^{[u,u+1]}
}
\end{equation*}
for all $0\leq u<n$. The condition that $\alpha^{op}$ is inert implies that $[\alpha^{op}(u),\alpha^{op}(u+1)]$ is of length one, so the horizontal and vertical arrows are isomorphisms in $\cC$ and it corresponds to a 2-coCartesian morphism in $\bCorr(\cC)^{\otimes,\Delta^{op}}$ as desired. 

(ii) The proof is very similar to that of (i). There is a natural functor 
\begin{align}\label{eq: F_T_Comm}
&F_{\cT^\Comm, \Fin_*}: \cT^\Comm\longrightarrow N(\Fin_*)\\
\nonumber&([n], s,(i,j), u)\mapsto s_{i,j}^{-1}(u)\sqcup\{*\}\\
\nonumber&(\alpha^{op}: [n]\leftarrow [m], s\rightarrow t, (i_1,j_1)\rightarrow (\alpha^{op}(i_2), \alpha^{op}(j_2)), u_2\in s_{\alpha^{op}(j_2),j_1}^{-1}(u_1))\\
\nonumber&\mapsto (s_{i_1,j_1}^{-1}(u_1)\sqcup\{*\}\rightarrow t_{i_2,j_2}^{-1}(u_2)\sqcup\{*\}),
\end{align}
in which $s_{i,j}: s(i)\rightarrow  s(j), i\leq j$ is a morphism in $N(\Fin_*)$, the second line represents a morphism $\hat{\alpha}: ([n], s,(i_1,j_1), u_1)\rightarrow ([m], t,(i_2,j_2), u_2)$ consisting of the data of a morphism $\alpha^{op}: [n]\leftarrow [m]$ in $N(\Delta)$, $t_{a,b}=s_{\alpha^{op}(a),\alpha^{op}(b)},\ i_1\leq \alpha^{op}(i_2)\leq \alpha^{op}(j_2)\leq j_1$ and $u_2\in s_{\alpha^{op}(j_2),j_1}^{-1}(u_1)$. The morphism $s_{i_1,j_1}^{-1}(u_1)\sqcup\{*\}\rightarrow t_{i_2,j_2}^{-1}(u_2)\sqcup\{*\}$ comes from the composition
\begin{equation*}
s_{i_1,j_1}^{-1}(u_1)\sqcup\{*\}\overset{\rho}{\longrightarrow} s_{i_1,\alpha^{op}(j_2)}^{-1}(u_2)\sqcup\{*\}\overset{s_{i_1,\alpha^{op}(i_2)}}{\longrightarrow} s_{\alpha^{op}(i_2),\alpha^{op}(j_2)}^{-1}(u_2)\sqcup\{*\},
\end{equation*} 
in which $\rho$ is the inert morphism that is the right inverse to the inclusion $s_{i_1,\alpha^{op}(j_2)}^{-1}(u_2)\sqcup\{*\}\hookrightarrow s_{i_1,j_1}^{-1}(u_1)\sqcup\{*\}$. It is straightforward to check that $F_{\cT^\Comm,\Fin_*}$ is compatible with compositions. The remaining steps of the proof are direct analogues of those for (i). 
\end{proof}

Recall that a simplicial object $C^\bullet$ in $\cC$ (that admits finite limits) is called a \emph{Segal object} if the morphism 
\begin{align*}
C^{n}\rightarrow C^{1}\underset{C^0}{\times}C^{1}\underset{C^0}{\times}\cdots\underset{C^0}{\times}C^1
\end{align*}
induced from the $n$ distinct inert morphisms $[1]\hookrightarrow [n]$ is an isomorphism in $\cC$. 

Similarly, we say a functor $C^\bullet: N(\Fin_*)\rightarrow \cC$ is a \emph{commutative Segal object}, if for any $m$, the diagram
\begin{align*}
\xymatrix{C^{\lng m\rng}\ar[r]\ar[d]& \prod\limits_{j\in \lng m\rng^\circ}C^{\{j,*\}}\ar[d]\\
C^{\lng 0\rng}\ar[r]^\Delta&\prod\limits_{j\in \lng m\rng^\circ}C^{\lng 0\rng}
}
\end{align*}
associated with the $m$ distinct inert morphisms $\alpha_j: \langle m\rangle\rightarrow \langle 1\rangle, \alpha_j^{-1}(1)=j$, is a Cartesian diagram. 
\begin{cor}\label{prop: Segal algebra}
Let $\cC$ be an $(\infty,1)$-category which admits finite limits. Then every (commutative) Segal object $C^\bullet$ in $\cC$ naturally induces an (commutative) algebra object in $\Corr(\cC^\times)$ via (\ref{eq: Maps(T,C)}). 
\end{cor}

\begin{remark}\label{rem: const alg}
For each $X\in \cC^\times$, the constant $\Fin_*$-object mapping to $X$ in $\cC$, denoted by $X^{\const, \bullet}$, gives a commutative Segal object and so a commutative algebra object in $\Corr(\cC)$, by Corollary \ref{prop: Segal algebra}.  
\end{remark}

\subsection{Construction of (right-lax) algebra homomorphisms in $\bCorr(\cC^\times)$}\label{subsec: right-lax assoc.}

Let $\Gray$ be the Gray tensor product of two $(\infty,2)$-categories (cf. \cite[Chapter A.1, Section 3.2]{Nick}), which agrees with the usual Gray product in the case of ordinary 2-categories. We have a description of $\Seq_\bullet(N(\Delta)^{op}\Gray [1])$: $\Seq_k(N(\Delta)^{op}\Gray [1])$ is an ordinary $1$-category\footnote{Here $[1]=\Delta^1$.} consisting of objects $(\alpha,\beta;\tau_0,\tau_1)$
\begin{equation*}
\xymatrix{
&&&[\xi]\ar[dl]_{\tau_0}&&&\\
[n_0]_0& \cdots \ar[l]_{\alpha_{1,0}}&[n_{\mu-1}]_0\ar[l]_{\alpha_{\mu-1,\mu-2}}&&[m_0]_{1}\ar[ll]_{\tau_0\tau_1}\ar[ul]_{\tau_1}&\ar[l]_{\beta_{1,0}}\cdots &&[m_{k-\mu}]_1\ar[ll]_{\beta_{k-\mu,k-\mu-1}},
}
\end{equation*}
where $[a]_\epsilon$ means the object $([a],\epsilon)\in \text{Obj}(N(\Delta)^{op}\times[1])= \text{Obj}([ N(\Delta)^{op}\Gray [1]),\ \epsilon=0,1$ (note that the marked morphisms $\alpha_{i,j},\beta_{i,j}$ are in $N(\Delta)$).
Here we allow $\mu=0$ and $\mu=k+1$, so then $\tau_0,\tau_1=\emptyset$. 
For two objects $(\alpha,\beta;\tau_0,\tau_1)$ and $(\alpha,\beta;\tau'_0,\tau'_1)$, a morphism from  the former to the latter is corresponding to a map $\varphi:[\xi']\rightarrow [\xi]$ in $N(\Delta)$ such that the following diagram commutes
\begin{equation*}
\xymatrix{&[\xi]\ar[dl]_{\tau_0}&\\
[n_{\mu-1}]_0& &[m_0]_{1}\ar[ul]_{\tau_1}\ar[dl]^{\tau_1'}\\
&[\xi']\ar[uu]^{\varphi}\ar[ul]^{\tau_0'}&
}
\end{equation*}

Here is a picture illustrating a morphism $(\alpha,\beta;\tau_0,\tau_1)\rightarrow (\alpha,\beta;\tau_0',\tau_1')$ (the upper right and lower left arrows indicate morphisms in $N(\Delta)^{op}$).

\begin{center}
\begin{tikzpicture}
\draw
(1,-3)--(0,-3)--(0,-1.7) ;
\draw
(1,-4)--(0,-4)--(0,-3) ;
\draw
(1,-4)--(1,-3)--(0,-3) ;
\draw[->, double equal sign distance] (0.8,-3.2)--(0.2,-3.8);
\draw
(1,-5)-- (1,-4)--(0,-4) ;
\draw[blue,<-,thick]
(0.9,-5)--(0.9,-4.2)--(-0.1, -4.2)node[left]{$[\xi']_0$}--(-0.1,-3)node[left]{$[\xi]_0$}--(-0.1,-1.7);
\draw[orange,<-,thick]
(1.1,-5)--(1.1, -4)node[right]{$[\xi']_1$}--(1.1, -2.9)node[right]{$[\xi]_1$}--(0.1,-2.9)--(0.1, -1.7);
\draw[black] (1.6,-4.5) node{$(\tau_1')^{op}$};
\draw (1.4,-3.5) node{$\varphi^{op}$};
\draw (-0.6,-2.3) node{$(\tau_0)^{op}$};
\end{tikzpicture}
\end{center}
For a map $\sigma: [\ell]\rightarrow [k]$ in $N(\Delta)$, the functor $\sigma^*: \Seq_k(N(\Delta)^{op}\Gray[1])\rightarrow \Seq_\ell(N(\Delta)^{op}\Gray[1])$ sends an object $(\alpha, \beta, \tau_0,\tau_1)$ to $\sigma^*(\alpha,\beta;\tau_0,\tau_1)$ defined by 
\begin{itemize}
\item[(1)] $(\sigma^*\alpha,\emptyset;\emptyset, \emptyset )$ if $\sigma(\ell)\leq \mu-1$,\\
\item[(2)] $((\sigma|_{[0,\nu-1]})^*\alpha, (\sigma|_{[\nu,\ell]})^*\beta; \alpha_{\mu-1,\sigma(\nu-1)}\circ\tau_0, \tau_1\circ\beta_{\sigma(\nu)-\mu, 0})$, if there exists $\nu\in [\ell]$ such that $\sigma(\nu-1)\leq \mu-1<\sigma(\nu)$,\\
\item[(3)] $(\emptyset, \sigma^*\beta;\emptyset, \emptyset )$ if $\sigma(0)\geq \mu$.
\end{itemize}

Similarly, we can describe $\Seq_k(N(\Fin_*)\Gray[1])$: an object is a four-tuple $(\alpha,\beta; \tau_0,\tau_1)$, represented by the diagram
\begin{equation*}
\xymatrix{
&&&\langle\xi\rangle\ar[dr]^{\tau_1}&&&\\
\langle n_0\rangle_0 \ar[r]^{\alpha_{0,1}}& \cdots\ar[r]^{\alpha_{\mu-2,\mu-1}}&\langle n_{\mu-1}\rangle_0\ar[ur]^{\tau_0}\ar[rr]^{\tau_1\tau_0}& &\langle m_0\rangle_{1}\ar[r]^{\beta_{0,1}}&\cdots \ar[rr]^{\beta_{k-\mu-1,k-\mu}}&&\langle m_{k-\mu}\rangle_1,
}
\end{equation*}
and a morphism from $(\alpha,\beta;\tau_0,\tau_1)$ to $(\alpha,\beta;\tau_0',\tau_1')$ corresponds to a morphism $\varphi: \langle\xi'\rangle\rightarrow \langle \xi\rangle$ in $\Fin_*$ such that the following diagram commutes
\begin{equation*}
\xymatrix{&\langle\xi\rangle\ar[dr]^{\tau_1}\ar[dd]^{\varphi}&\\
\langle n_{\mu-1}\rangle_0\ar[ur]^{\tau_0}\ar[dr]_{\tau'_0}& &\langle m_0\rangle_{1}\\
&\langle\xi'\rangle\ar[ur]_{\tau'_1}&
}
\end{equation*}
For any map $\sigma: [\ell]\rightarrow [k]$, the functor $\sigma^*: \Seq_k(N(\Fin_*)\Gray[1])\rightarrow \Seq_\ell(N(\Fin_*)\Gray[1])$ is similar as above. 

Let $\pi_\Gray: I_{N(\Delta)^{op}\Gray[1]}\rightarrow N(\Delta)^{op}$ (resp.  $\pi_\Gray^\Comm: I^\Comm_{N(\Delta)^{op}\Gray[1]}\rightarrow N(\Delta)^{op}$) be the coCartesian fibration from the Grothendieck construction for  $\Seq_\bullet (N(\Delta)^{op}\Gray[1]): N(\Delta)^{op}\rightarrow \OneCat^\ord$ (resp. $\Seq_\bullet (\Fin_*\Gray[1]): N(\Delta)^{op}\rightarrow \OneCat^\ord$), where $\OneCat^\ord$ means the full subcategory of $\OneCat$ consisting of ordinary 1-categories. 
We introduce another coCartesian fibration $\pi_{I_+}: I_+\rightarrow N(\Delta)^{op}$, which is the Grothendieck construction of the functor 
\begin{align*}
&N(\Delta)^{op}\rightarrow \Sets\\
&[n]\mapsto \{(s,e)|s: \Delta^n\rightarrow N(\Delta)^{op}, e\in\{-1,0,\cdots, n\}\}\\
&([n]\leftarrow [m]:\alpha^{op})\mapsto (f: (s,e)\mapsto ((\alpha^{op})^*s, e'=\begin{cases}-1,\text{ if }\alpha^{op}(0)>e;\\
m, \text{ if }\alpha^{op}(n)<e+1;\\
\text{the unique element }u \text{ such that }\\
\ \  [e,e+1]\subset\alpha^{op}([u,u+1]), \text{ otherwise}
\end{cases})).
\end{align*}
Here we think of $e$ as a marked edge $\{i,i+1\}$ in $\Delta^n$ when $i\in [0,n-1]$, and when $i=-1, n$, we get a degenerate marked edge at $0$ and $n$ respectively. In a similar way, we can define $\pi_{I_+}^{\Comm}: I_+^\Comm\rightarrow N(\Delta)^{op}$.

\begin{lemma}\label{lemma: pi Gray coCart}
The natural functors 
\begin{align*}
\pi_{\Gray, I_+}: I_{N(\Delta)^{op}\Gray[1]}&\longrightarrow I_+\\
(\alpha,\beta; \tau_0,\tau_1)&\mapsto ((\alpha,\tau_0\tau_1,\beta), e=\mu-1),
\end{align*}
and 
\begin{align*}
\pi_{\Gray, I_+}^\Comm: I_{N(\Fin_*)\Gray[1]}&\longrightarrow I_+^\Comm\\
(\alpha,\beta; \tau_0,\tau_1)&\mapsto ((\alpha,\tau_1\tau_0,\beta), e=\mu-1),
\end{align*}
are coCartesian fibrations. 
\end{lemma}
\begin{proof}
By construction, every edge in $I_+$ is $\pi_{I_+}$-coCartesian. By \cite[Proposition 2.4.1.3]{higher-topoi}, every $\pi_\Gray$-coCartesian edge in $I_{N(\Delta)^{op}\Gray[1]}$ is also $\pi_{\Gray, I_+}$-coCartesian. Since $\pi_{\Gray}$ is a coCartesian fibration,
the claim for $\pi_{\Gray,I_+}$ follows. The part for $\pi_{\Gray,I_+}^\Comm$ follows similarly. 
\end{proof}

\begin{lemma}\label{lemma: Fun coCar Car}
For any $\infty$-category $S$, a coCartesian fibration $X\rightarrow S$, a Cartesian fibration $Y\rightarrow S$ and an ordinary 1-category $\cC$, to give a functor $F: X\times_S Y\rightarrow  \cC$, it suffices to give the following data:
\begin{itemize}
\item[(a)] 
a functor $F_s: X_s\times Y_s\rightarrow \cC$ for each fiber over $s\in S$,
\item[(b)]
for each edge $(\alpha,\beta): (x_0,y_0)\rightarrow (x_1,y_1)$ in $X\times_SY$ over an edge $p(\alpha): s\rightarrow t$ in $S$, given by a coCartesian edge $(\alpha: x_0\rightarrow x_1)$ and a Cartesian edge $(\beta: y_0\rightarrow y_1)$ over $p(\alpha)$, a morphism in $\cC$ from $F_s(x_0,y_0)$ to $F_t(x_1,y_1)$, such that for any commutative diagrams in $X,Y$ respectively
\begin{align*}
\xymatrix{x_0\ar[r]\ar[d] &x_1\ar[d]\\
x_0'\ar[r]&x_1'
},\ \xymatrix{y_0\ar[r]\ar[d] &y_1\ar[d]\\
y_0'\ar[r]&y_1'
},\end{align*} 
where the horizontal arrows are respectively coCartesian/Cartesian, and the left (resp. right) vertical arrow  in each diagram is in $X_s, Y_s$ (resp. $X_t, Y_t$) respectively,  
the following diagram commutes
\begin{align}\label{lemma: diag coCart_Cart}
\xymatrix{F_s(x_0,y_0)\ar[r]\ar[d] &F_t(x_1,y_1)\ar[d]\\
F_s(x_0', y_0')\ar[r]&F_t(x_1',y_1')
}
\end{align}
in $\cC$. 
\item[(c)] The assignment in (b) is compatible with composition of coCartesian and Cartesian edges in $X$ and $Y$ respectively. 
\end{itemize}
\end{lemma}
\begin{proof}
Any morphism $f:(x_0,y_0)\rightarrow (x_1,y_1)$ over $s\rightarrow t$ in $S$ factors uniquely as 
\begin{align*}
\xymatrix{(x_0,y_0)\ar[d]&\\
(x_0,y_{01})\ar[r]^{(\alpha_{01}, \beta_{01})}&(x_{01},y_1)\ar[d]\\
&(x_1,y_1)
}
\end{align*}
where the vertical arrows are contained in $X_s\times Y_s$ and $X_t\times Y_t$ respectively, and the horizontal arrow $(\alpha,\beta)$ has $\alpha$ (resp. $\beta$) a coCartesian (resp. Cartesian) arrow.  
Now if in addition we have $g: (x_1,y_1)\rightarrow (x_2,y_2)$ over $t\rightarrow u$ in $S$ and $g\circ f$ factor as follows 
\begin{align*}
\xymatrix{(x_1,y_1)\ar[d]&\\
(x_1,y_{12})\ar[r]^{(\alpha_{12}, \beta_{12})}&(x_{12},y_2)\ar[d]\\
&(x_2,y_2)
},\ 
\xymatrix{(x_0,y_0)\ar[d]&\\
(x_0,y_{02})\ar[r]^{(\alpha_{02}, \beta_{02})}&(x_{02},y_2)\ar[d]\\
&(x_2,y_2)
},
\end{align*}
then we have a commutative diagram
\begin{align}\label{diagram: composition F}
\xymatrix{(x_0,y_0)\ar[d]&&&\\
(x_0,y_{01})\ar[dd]_{(id, \gamma_{12})}\ar[rr]^{(\alpha_{01}, \beta_{01})}&&(x_{01},y_1)\ar[d]\ar[dl]&\\
&(x_{01},y_{12})\ar[dr]\ar@/_/[rr]_{(\alpha_{012}, \beta_{12})}&(x_{1},y_1)\ar[d]&(x_{02,}y_2)\ar[d]^{(\eta_{01},id)}\\
(x_0, y_{02})\ar[ur]_{(\alpha_{01}, \beta_{012})}&&(x_1,y_{12})\ar[r]&(x_{12},y_2)\ar[d]\\
&&&(x_2,y_2)
}
\end{align}
in which the arrows $\beta_{012}$ (resp. $\alpha_{012}$) is Cartesian (resp. coCartesian), and $\gamma_{12}$ (resp. $\eta_{01}$) is the image under pullback (resp. pushforward) along the Cartesian (resp. coCartesian) arrows from $F_s\rightarrow F_t$ (resp. $F_t\rightarrow F_u$). 
Therefore, to define a functor $X\times_S Y\rightarrow \cC$, we just need to define functor $F_s: X_s\times Y_s\rightarrow \cC$ for every $s\in S$, and a morphism $F_s(x_0,y_0)\rightarrow F_t(x_1,y_1)$ for each edge $(\alpha_{01},\beta_{01})$ as above, so that  diagram (\ref{diagram: composition F}) is sent to a commutative diagram in $\cC$. 
It's easy to see that these are exactly the data (a), (b) and (c) in the Lemma. 
 \end{proof}

By \cite[Chapter 9, 1.4.5, Chapter 10, 3.2.7]{Nick}, 
to give a right-lax homomorphism between two associative (resp. commutative) algebra objects in $\bCorr(\cC)^{\otimes, \Delta^{op}}$ (resp. $\bCorr(\cC)^{\otimes, \Fin_*}$) is the same as to construct a functor from $N(\Delta)^{op}\Gray[1]$ (resp. $N(\Fin_*)\Gray[1]$) to $\bCorr(\cC)^{\otimes, \Delta^{op}}$ (resp. $\bCorr(\cC)^{\otimes, \Fin_*}$) over $N(\Delta)^{op}$ (resp. $N(\Fin_*)$), such that the restriction of the functor to $N(\Delta)^{op}\times\{\epsilon\}$ (resp. $N(\Fin_*)\times\{\epsilon\}$), $\epsilon=0,1$ coincides with the two associative (resp. commutative) algebra objects respectively (up to homotopy). Therefore, we are aiming to define for each $(\alpha,\beta;\tau_0,\tau_1)\in \Seq_k(N(\Delta)^{op}\Gray[1])$ (resp. $(\alpha,\beta;\tau_0,\tau_1)\in \Seq_k(N(\Fin_*)\Gray[1])$) a diagram $W_{\alpha,\beta;\tau_0,\tau_1}$ in $\Fun''(([k]\times [k]^{op})^{\geq \dgnl}, \cC^{\otimes,\Delta})$ (resp. $\Fun''(([k]\times [k]^{op})^{\geq \dgnl}, \cC^{\otimes,(\Fin_*)^{op}})$) satisfying required properties. 

By the approach in Subsection \ref{subsubsec: algebra objects}, we just need to construct an element in $\Hom_{I}(I^\natural_{N(\Delta)^{op}\Gray[1]}\underset{I}{\times}\cT^\natural, (\cC\times I)^\natural)$ (resp. $\Hom_{I}(I^\natural_{N(\Fin_*)\Gray[1]}\underset{I}{\times} \cT^\natural, (\cC\times I)^\natural)$). For this purpose, we first define a Cartesian fibration $\pi_{\cT_+}:\cT_+\rightarrow I_+$ which can be thought of as an extension of $\pi_{\cT}: \cT\rightarrow I$. 
Consider the ordinary 1-category 
\begin{align*}
&\Gamma_{[n],e}:=([0,e]\times [0,n]^{op})^{\geq \dgnl})\times\{0\}_1 \coprod\limits_{[0,e]\times [e+1,n]^{op}\times \{0\}_1}([0,e]\times[e+1,n]^{op}\times (\Delta^1_1\coprod\limits_{\Delta^{\{1\}}}\Delta^1_2))\\
&\coprod\limits_{[0,e]\times [e+1,n]^{op}\times \{0\}_2}([0,n]\times [e+1,n]^{op})\times \{0\}_2,
\end{align*}
which has a natural functor to $[n]^{op}$ from projecting to the second factor (see Figure \ref{figure: Gamma}). 
The functor $\pi_{\cT_+}$ comes from the Grothendieck construction of the following functor 
\begin{align*}
&f_{I_+}: I_+^{op}\longrightarrow\OneCat^\ord\\
&([n], (s,e))\mapsto \Gamma_{[n],e}\underset{N(\Delta)}{\times}T,
\end{align*}
for any morphism $([m], ((\alpha^{op})^*s,e'))\rightarrow([n],(s,e))$ in $I_+^{op}$ induced by $\alpha^{op}: [m]\rightarrow [n]$ in $N(\Delta)$, the corresponding morphism under $f_{I_+}$ is the natural one that sends $(i,j,\eta; u\leq u+1)\in [m]\times [m]^{op}\times \{\{0\}_1,\{0\}_2,\{1\}\}$ to $(\alpha^{op}(i), \alpha^{op}(j), \eta; u\leq u+1)$. 

\subsubsection{Construction of the main functor}\label{subsec: F_Gray}
Now we define a functor from $F_{\Gray,\cT_+}:  I_{N(\Delta)^{op}\Gray[1]}\underset{I_+}{\times}\cT_+\rightarrow N(\Delta)^{op}\times (\Delta^{\{W,C\}}\coprod\limits_{\Delta^{\{W\}}}\Delta^{\{W,D\}})$ as follows. First, on the object level, we have
\begin{align}\label{eq: F_Gray, T_+}
&F_{\Gray,\cT_+}: I_{N(\Delta)^{op}\Gray[1]}\underset{I_+}{\times}\cT_+\rightarrow N(\Delta)^{op}\times (\Delta^{\{W,C\}}\coprod\limits_{\Delta^{\{W\}}}\Delta^{\{W,D\}})\\
\nonumber&(\alpha, \beta,\tau_0,\tau_1;[n], (s,e); i,j,\{0\}_1;u\leq u+1)\mapsto ([\alpha_{ji}(u),\alpha_{ji}(u+1)], \{C\}), \text{ if }i\leq j\leq e;\\
\nonumber&(\alpha, \beta,\tau_0,\tau_1;[n], (s,e); i,e+1+j,\{0\}_1;u\leq u+1)\mapsto ([\alpha_{ei}\tau_0\tau_1\beta_{j0}(u), \alpha_{ei}\tau_0\tau_1\beta_{j0}(u+1)], \{C\}), i\leq e, j\geq 0;\\
\nonumber&(\alpha, \beta,\tau_0,\tau_1;[n], (s,e); i,e+1+j,\{1\};u\leq u+1)\mapsto ([\tau_1\beta_{j0}(u),\tau_1\beta_{j0}(u+1)], \{C\}),i\leq e, j\geq 0;\\
\nonumber&(\alpha, \beta,\tau_0,\tau_1;[n], (s,e); i,e+1+j,\{0\}_2;u\leq u+1)\mapsto ([\tau_1\beta_{j0}(u),\tau_1\beta_{j0}(u+1)], \{W\}),i\leq e, j\geq 0;\\
\nonumber&(\alpha, \beta,\tau_0,\tau_1;[n], (s,e);e+1+i,e+1+j,\{0\}_2; u\leq u+1)\mapsto ([\beta_{ji}(u),\beta_{ji}(u+1)], \{D\}),\text{ if }0\leq i\leq j.
\end{align}
By Lemma \ref{lemma: Fun coCar Car}, to give a complete definition of $F_{\Gray, \cT_+}$, we just need to give the following data\\

\noindent(a) For any $([n], (s,e))\in I_+$, 
\begin{align}\label{eq: F_Gray, T, [n], s, e}
F_{\Gray,\cT_+; [n],(s,e)}: &(I_{N(\Delta)^{op}\Gray[1]})_{[n],(s,e)}\rightarrow \Fun((\cT_+)_{[n],(s,e)}, N(\Delta)^{op}\times (\Delta^{\{W,C\}}\coprod\limits_{\Delta^{\{W\}}}\Delta^{\{W,D\}}))
\end{align}

Let $[0,n]_{e,+}$ be the ordered set $[0,n]\sqcup\{e_+\}$ with $e<e_+<e+1$. Consider the ordinary 1-category
\begin{align*}
S_{[n],e_+}=&(([0,n]_{e,+}\times [0,n]^{op})^{\geq \dgnl}\times \{0\}_1)\coprod\limits_{([e_+, n]\times [e+1,n]^{op})^{\geq \dgnl}\times \{0\}_1}\\
&(([e_+, n]\times [e+1,n]^{op})^{\geq \dgnl}\times \Delta^{\{\{0\}_2,\{0\}_1\}}),
\end{align*}
which has a natural functor to $[n]^{op}$ by projecting to the second factor (see Figure \ref{figure: S_+}). 
There is a unique functor 
\begin{align*}
\widetilde{C}_e: \Gamma_{[n], e}\rightarrow S_{[n],e_+}
\end{align*}
that 
\begin{itemize}
\item[(i)] collapses the vertical edges $(i,j,\eta)\rightarrow (i+1, j,\eta)$ to the identity morphism at $(e_+, j,\eta)$ for $i<e, j\geq e+1,\eta\in \{\{1\},\{0\}_2\}$
\item[(ii)] is the identity on the subcategories $(([0,e]\times [0,n]^{op})^{\geq \dgnl}\times \{0\}_1)$ and restricts to an equivalence of subcategories 
\begin{align*}
&([e,n]\times [e+1,n]^{op})^{\geq \dgnl}\times \{0\}_2\overset{\sim}{\rightarrow} ([e_+, n]\times [e+1,n]^{op})^{\geq\dgnl}\times\{0\}_2\\
&\{e\}\times [e+1,n]^{op}\times\Delta_2^1\overset{\sim}{\rightarrow} \{e_+\}\times[e+1,n]^{op}\times \Delta_2^1
\end{align*}
\item[(iii)] restricts to an equivalence of subcategories
\begin{align*}
&\{e\}\times [e+1,n]^{op}\times\Delta_1^{1}\overset{\sim}{\rightarrow} [e,e_+]\times [e+1,n]^{op}\times\{0\}_1.
\end{align*}
\end{itemize}
Property (ii) (resp. (iii)) is illustrated by the white wall at the back and gray faces (resp. the faces hatched by lines) in both Figure \ref{figure: Gamma} and Figure \ref{figure: S_+}. 

For any $([n],s,e)\in I_+$, let 
\begin{align*}
&C_{[n],s, e}:(\cT_+)_{([n],(s,e))}\simeq\Gamma_{[n],e}\underset{N(\Delta)}{\times}T\rightarrow S_{[n],e_+}\underset{N(\Delta)}{\times}T\end{align*}
be the functor induced from $\widetilde{C}_e$. 

Given any $s^\tau=(\alpha,\beta,\tau_0,\tau_1)\in (I_{N(\Delta)^{op}\Gray[1]})_{[n],(s,e)}$ which we will also think of as a functor $[n]_{e,+}\rightarrow N(\Delta)^{op}$, 
let 
\begin{align*}
&F_{2,\tau}: S_{[n],e_+}\underset{N(\Delta)}{\times }T\rightarrow N(\Delta)^{op}\times (\Delta^{\{W, C\}}\coprod\limits_{\Delta^{\{W\}}}\Delta^{\{W,D\}})
\end{align*}
 be the functor given by the product of $G_{1,\tau}\circ\proj$ and $G_2\circ\proj$ defined as follows.
\begin{align*}
&S_{[n],e_+}\underset{N(\Delta)}{\times}T
\overset{\proj}{\rightarrow} ([0, n]_{e,+}\times [0,n]^{op})^{\geq \dgnl}\underset{N(\Delta)}{\times}T\overset{G_{1,\tau}}{\rightarrow} N(\Delta)^{op}
\end{align*}
\begin{align*}
&S_{[n],e_+}\underset{N(\Delta)}{\times}T
\overset{\proj}{\rightarrow} ([0,n]_{e,+}\times [0,n]^{op})^{\geq \dgnl}\times \Delta^{\{\{0\}_2,\{0\}_1\}}\overset{G_2}{\rightarrow}\Delta^{\{W, C\}}\coprod\limits_{\Delta^{\{W\}}}\Delta^{\{W,D\}},\\
&\text{where }G_{1,\tau}\text{ is defined similarly as }F_{\cT,\Delta^{op}} \text{ that is induced by } s^\tau \text{ and }G_2 \text{ is the one} \\
&\text{sending }(i,j,\{0\}_1)\text{ to }C, (e_+,j,\{0\}_2)\text{ to }W, 
\text{ and }(e+i,j,\{0\}_2) \text{ to }D \text{ for }i>0. 
\end{align*}

Now for each $([n], s^\tau)\in (I_{N(\Delta)^{op}\Gray[1]})_{[n],(s,e)}$, where $s^\tau=(\alpha,\beta,\tau_0,\tau_1)$ as above, we define 
\begin{align*}
F_{\Gray,\cT_+;[n],(s,e)}([n], s^\tau)=F_{2,\tau}\circ C_{[n], s,e},
\end{align*} 
(as in (\ref{eq: F_Gray, T, [n], s, e})). For any morphism $([n],s^\tau)\mapsto ([n], s^{\tau'})$ defined by $\varphi: [\xi']\rightarrow [\xi]$ in $N(\Delta)$, 
let $[n]_{e,+,+}=[n]_{e,+}\sqcup\{e_+'\}$ with the ordering $e_+<e_+'<e+1$, and let $s^{\tau,\tau'}: [n]_{e,+,+}^{op}\rightarrow N(\Delta)$, such that $s^{\tau,\tau'}|_{[n]_{e,+}^{op}}=(s^\tau)^{op}$, $s^{\tau,\tau'}(e_+'\rightarrow e_+)=\varphi$ and  $s^{\tau,\tau'}(e+1\rightarrow e_+')=\tau_1'$. 
Consider the composition
\begin{align}\label{eq: natural transformation tau,tau'}
&\Delta^1\times (\cT_+)_{[n], s,e}\overset{(id_{\Delta^1},C_{[n],s,e})}{\longrightarrow}\Delta^1\times (S_{[n],e_+}{\underset{N(\Delta)}{\times}}T)\overset{p}{\rightarrow} \Delta^1\times (([0,n]_{e,+}\times [0,n]^{op})^{\geq \dgnl}{\underset{N(\Delta)}{\times}}T)\\
\nonumber&\overset{C_+}{\rightarrow} ([0,n]_{e,+,+}\times [0,n]^{op})^{\geq \dgnl}{\underset{N(\Delta)}{\times}}T\overset{F_{\tau,\tau'}}{\rightarrow }N(\Delta)^{op},
\end{align}
where 
\begin{itemize}
\item[(i)] $p$ is determined by the projection $S_{[n],e_+}\rightarrow ([0,n]_{e,+}\times [0,n]^{op})^{\geq \dgnl}$,  
\item[(ii)] $C_+$ is the contraction determined by sending the arrow $((0,e_+,j;[u,u+1])\rightarrow (1,e_+,j;[u,u+1]))$ to $((e_+,j;[u,u+1])\rightarrow (e_+',j;[u,u+1]))$, and sending $(\epsilon, i,j;[u,u+1])$ to $(i,j;[u,u+1])$ for $\epsilon=0,1, i\neq e_+$,
\item[(iii)] $F_{\tau,\tau'}$ is defined similarly as $F_{\cT,\Delta^{op}}$ induced by $s^{\tau,\tau'}$. 
\end{itemize}
The composite functor (\ref{eq: natural transformation tau,tau'}) gives a natural transformation 
\begin{align*}
T_{\tau,\tau'}: F_{\Gray,\cT_+;[n],(s,e)}([n], s^\tau)\rightarrow F_{\Gray,\cT_+;[n],(s,e)}([n], s^{\tau'}). 
\end{align*}
By the naturality of the definition, it is easy to see that $T_{\tau',\tau''}\circ T_{\tau,\tau'}=T_{\tau,\tau''}$. \\

\noindent(b) Over any edge $([n], (s,e_1))\rightarrow ([m], (\sigma^*(s), e_2))$ in $I_+$ induced by $\sigma: [m]\rightarrow [n]$ in $N(\Delta)$, we have any $\pi_{\Gray, I_+}$-coCartesian edge of the form $\sigma_1:([n],(\alpha,\beta;\tau_0,\tau_1))\rightarrow ([m], (\sigma^*(\alpha,\beta;\tau_0,\tau_1)))$, and any $\pi_{\cT_+}$-Cartesian edge of the form $\sigma_2(\sigma(i),\sigma(j),\eta;[u,u+1])\rightarrow (i,j,\eta;[u,u+1])$, where $\eta\in \{\{0\}_1,\{0\}_2,\{1\}\}$. For any commutative diagrams
\begin{align*}
\xymatrix{([n],(\alpha,\beta;\tau_0,\tau_1))\ar[r]\ar[d]&([m],\sigma^*(\alpha,\beta;\tau_0,\tau_1))\ar[d]\\
([n],(\alpha,\beta;\tau_0',\tau_1'))\ar[r]&([m],\sigma^*(\alpha,\beta;\tau_0',\tau_1'))
},\\
\xymatrix{(\sigma(i),\sigma(j),\eta;[u,u+1])\ar[r]\ar[d]&(i,j,\eta;[u,u+1])\ar[d]\\
(\sigma(i'),\sigma(j'),\eta;[u',u'+1])\ar[r]&(i',j',\eta;[u',u'+1])
},
\end{align*}
we denote $\sigma^*(\alpha,\beta;\tau_0,\tau_1)$ by $(\widetilde{\alpha},\widetilde{\beta};\widetilde{\tau_0},\widetilde{\tau_1})$ and similarly for $\sigma^*(\alpha,\beta;\tau'_0,\tau'_1)$. 
By the commutativity of the following diagram
\begin{align*}
\xymatrix{
\Delta^1\times (\cT_{+})_{[n],(s,e_1)}\ar[r]^{\substack{C_+p\circ (id_{\Delta^1},C_{[n],s,e_1})\\ \\ }}& ([0,n]_{e_1,+,+}\times [0,n]^{op})^{\geq \dgnl}\underset{N(\Delta)}{\times}T\ar[r]^{\ \ \ \ \ \ \ \ \ \ \ \ \ \ \ F_{\tau,\tau'}}&N(\Delta)^{op}\\
\Delta^1\times (\cT_{+})_{[m],(\sigma^*s,e_2)}\ar[u]\ar[r]^{\substack{C_+p\circ (id_{\Delta^1},C_{[m],\sigma^*s,e_2})\\ \\ }}& ([0,m]_{e_2,+,+}\times [0,m]^{op})^{\geq \dgnl}\underset{N(\Delta)}{\times}T\ar[u]\ar[ur]_{F_{\widetilde{\tau},\widetilde{\tau}'}}&
},
\end{align*}
we see that we just need to make $F_{\Gray,\cT_+}(\sigma_1,\sigma_2)=id$, then the diagrams (\ref{lemma: diag coCart_Cart}) are all commutative. So we have finished the definition of $F_{\Gray, \cT_+}$.

\begin{center}
\begin{figure}[h]
\begin{tikzpicture}
\foreach \x in {3,...,6} 
  \foreach \k in {1,...,\x}
  \draw [->] (\k-1,\x-3)--(\k,\x-3);
\foreach \x in {1,...,4}
   \foreach \k in {1,...,3}
       \draw [->] (\x-1, \k)--(\x-1, \k-1);
\foreach \x in {1,2}
    \foreach \k in {1,...,\x}
       \draw [->] (-\x+6, -\k+4)--(-\x+6,-\k+3);       
\foreach \x in {0,...,2}
   \foreach \y in {0,...,3}
     \draw[->, orange] (\x,\y)--(\x-0.4, \y-0.4);    
\foreach \x in {0,...,2}
   \foreach \y in {0,...,3}       
     \draw[->,orange](\x-0.8, \y-0.8)--(\x-0.4, \y-0.4);
 \foreach \x in {0,1}
    \foreach \y in {0,...,3}
       \draw[->, orange](\x-0.8,\y-0.8)--(\x+0.2,\y-0.8);     
 \foreach \x in {0,1}
    \foreach \y in {0,...,3}
       \draw[->, orange](\x-0.4,\y-0.4)--(\x+0.6,\y-0.4);    
 \foreach \x in {0,1}
    \foreach \y in {0,...,3}
       \draw[->, orange](\x-0.8,\y-0.8)--(\x+0.2,\y-0.8);   
 \foreach \x in {0,...,2}
    \foreach \y in {0,...,2}
       \draw[<-,orange](\x-0.8,\y-0.8)--(\x-0.8,\y+0.2);   
 \foreach \x in {0,...,2}
    \foreach \y in {0,...,2}
       \draw[<-,orange](\x-0.4,\y-0.4)--(\x-0.4,\y+0.6);   
\foreach \x in {0,1,2}
    \foreach \k in {0,...,\x} 
        \draw[<-] (-\x+1.2,-1.8-\k)--(-\x+1.2,-0.8-\k);
\foreach \y in {0,1}
     \foreach \x in {0,...,\y}
        \draw[->] (\x-0.8,\y-2.8)--(\x+0.2,\y-2.8);
 \draw (4,0) node {$(e,e,\{0\}_1)$};
 \draw (2.7, -1.8) node {$(e+1,e+1,\{0\}_2)$};
\draw[pattern=north west lines, pattern color=cyan] (0,0)--(-0.4,-0.4)--(0.6,-0.4)--(1,0)--(0,0);
\draw[pattern=north west lines, pattern color=cyan] (1,0)--(0.6,-0.4)--(1.6,-0.4)--(2,0)--(1,0);
\fill[fill=gray, opacity=0.2] (-0.4,-0.4)--(-0.8,-0.8)--(1.2,-0.8)--(1.6,-0.4); 
\fill[fill=gray, opacity=0.2] (-0.8,-0.8)--(-0.8,-1.8)--(1.2,-1.8)--(1.2,-0.8);
\fill[fill=gray, opacity=0.2] (-0.8,-1.8)--(-0.8, -2.8)--(0.2,-2.8)--(0.2,-1.8); 
\end{tikzpicture}
\caption{A picture of $\Gamma_{[n],s,e}$.}\label{figure: Gamma}
\end{figure}
\end{center}

\begin{figure}[h]
\begin{tikzpicture}
\foreach \x in {3,...,6} 
  \foreach \k in {1,...,\x}
  \draw [->] (\k-1,\x-3)--(\k,\x-3);
\foreach \x in {1,...,4}
   \foreach \k in {1,...,3}
       \draw [->] (\x-1, \k)--(\x-1, \k-1);
\foreach \x in {1,2}
    \foreach \k in {1,...,\x}
       \draw [->] (-\x+6, -\k+4)--(-\x+6,-\k+3);    
\foreach \x in {0,...,1}
    \draw[->] (\x, -1)--(\x+1, -1); 
 \foreach \x in {0,...,2}
    \draw[->] (\x, 0)--(\x, -1);            
\foreach \x in {0,1,2}
    \foreach \k in {0,...,\x} 
        \draw[<-] (-\x+1.6,-2.4-\k)--(-\x+1.6,-1.4-\k);
\foreach \x in {0,1}
   \draw [->] (\x-0.4, -1.4)--(\x+0.6, -1.4);         
\foreach \y in {0,1}
     \foreach \x in {0,...,\y}
        \draw[->] (\x-0.4,\y-3.4)--(\x+0.6,\y-3.4);
\foreach \x in {0,1,2}
    \foreach \k in {0,...,\x} 
        \draw[<-] (-\x+2,-2-\k)--(-\x+2,-1-\k);
      \foreach \y in {0,1}
     \foreach \x in {0,...,\y}
        \draw[->] (\x,\y-3)--(\x+1,\y-3);     
 \foreach \x in {0,...,2}
    \draw[->] (\x-0.4,-1.4)--(\x,-1);
 \foreach \y in {0,1,2}
    \foreach \x in {0,...,\y}   
     \draw[->] (\x-0.4, \y-4.4)--(\x, \y-4); 
\draw[pattern=north west lines, pattern color=cyan]  (0,0)--(0,-1)--(2,-1)--(2,0); 
 \fill[fill=gray, opacity=0.2] (0,-1)--(-0.4,-1.4)--(1.6,-1.4)--(2,-1); 
\fill[fill=gray, opacity=0.2] (-0.4,-1.4)--(-0.4,-2.4)--(1.6,-2.4)--(1.6,-1.4);
\fill[fill=gray, opacity=0.2] (-0.4,-2.4)--(-0.4, -3.4)--(0.6,-3.4)--(0.6,-2.4); 
 \draw (4,0) node {$(e,e,\{0\}_1)$};
  \draw (3.3,-1) node {$(e_+,e_+,\{0\}_1)$};
 \draw (3.2, -2.4) node {$(e+1,e+1,\{0\}_2)$};
\end{tikzpicture}
\caption{A picture of $S_{[n],e_+}$.}\label{figure: S_+}
\end{figure}

Let $C^\bullet, D^\bullet, W^\bullet$ be simplicial objects in $\cC$, and let $D^\bullet\leftarrow W^\bullet\rightarrow C^\bullet$ represent a fixed functor $S_W\in \Fun(N(\Delta)^{op}\times (\Delta^{\{W, C\}}\coprod\limits_{\Delta^{\{W\}}}\Delta^{\{W,D\}}), \cC)$ whose restriction to $N(\Delta)^{op}\times \Delta^{\{?\}}$ is the simplicial object $(?)^\bullet$ for $?=C, D,W$. We say $S_W$ is a \emph{correspondence} from $C^\bullet$ to $D^\bullet$. If all of $C^\bullet$, $D^\bullet$ and $W^\bullet$ are Segal objects of $\cC$, then we say $S_W$ is a \emph{correspondence of Segal objects} from $C^\bullet$ to $D^\bullet$. The composition $S_W\circ F_{\Gray, \cT_+}$ gives a functor $I_{N(\Delta)^{op}\Gray [1]}\underset{I_+}{\times}\cT_+\rightarrow \cC$, which also gives a functor $I^\natural_{N(\Delta)^{op}\Gray [1]}\underset{I_+}{\times}\cT^\natural_+\rightarrow (\cC\times I_+)^\natural$ over $I_+$. 

Next, we take the right Kan extension of $S_W\circ F_{\Gray, \cT_+}$ along the projection $p_{\cT_+}: I_{N(\Delta)^{op}\Gray [1]}\underset{I_+}{\times}\cT_+\rightarrow I_{N(\Delta)^{op}\Gray [1]}\underset{I}{\times}\cT$ defined by forgetting the coordinate $\{0\}_1,\{0\}_2,\{1\}$ in $\cT_+$. This gives the desired functor 
\begin{align}\label{eq: functor F_Gray, W}
&F_{\Gray,S_W}:  I_{N(\Delta)^{op}\Gray [1]}\underset{I}{\times}\cT\rightarrow \cC,
\end{align} 
which also lies in 
\begin{align*}
&\Maps_{I}(I^\natural_{N(\Delta)^{op}\Gray [1]}\underset{I}{\times}\cT^\natural, (\cC\times I)^\natural)\simeq \Maps_I(I^\natural_{N(\Delta)^{op}\Gray [1]}, (\cC^\natural)^{\cT^\natural})\\
\simeq&\Maps_{(\OneCat)^{\Delta^{op}}_{/\Seq_\bullet (N(\Delta)^{op})}}(\Seq_\bullet(N(\Delta)^{op}\Gray [1]),f_{\Delta^\bullet}^\cC),
\end{align*} 
where $f_{\Delta^\bullet}^\cC$ was defined in (\ref{eq: f_Delta^C}). In the following, we also think of $F_{\Gray,S_W}$ as a functor $\Seq_\bullet(N(\Delta)^{op}\Gray[1])\rightarrow f_{\Delta^\bullet}^\cC$ over $\Seq_\bullet (N(\Delta)^{op})$. In a totally similar fashion, we can define a functor
\begin{align}\label{eq: functor F_Gray, W, comm}
&F_{\Gray,S_W}^\Comm:  I_{N(\Fin_*)\Gray [1]}\underset{I^\Comm}{\times}\cT^\Comm\rightarrow \cC,
\end{align} 
from any correspondence $D^\bullet\leftarrow W^\bullet\rightarrow C^\bullet$ of $\Fin_*$-objects in $\cC$ from $C^\bullet$ to $D^\bullet$.

\subsubsection{Right-lax homomorphism from correspondence of simplicial/$\Fin_*$-objects}
\begin{thm}\label{thm: right-lax}
Let $D^\bullet\leftarrow W^\bullet\rightarrow C^\bullet$ be a correspondence of simplicial (resp. $\Fin_*$-) objects (from $C^\bullet$ to $D^\bullet$). Suppose that $C^\bullet$ and $D^\bullet$ represent associative (resp. commutative) algebra objects in $\bCorr(\cC^\times)$. Then $F_{\Gray, S_W}$ (\ref{eq: functor F_Gray, W}) (resp. $F_{\Gray, S_W}^\Comm$ (\ref{eq: functor F_Gray, W, comm})) gives a right-lax homomorphism between the associative (resp. commutative) algebra objects if and only if $W^\bullet$ represents another associative (resp. commutative) algebra object and the following diagrams are Cartesian in $\cC$
 \begin{align}\label{lemma: diagram Cart}
&\xymatrix{W^{[k]}\ar[r]\ar[d]&\prod\limits_{[v,v+1]\subset [k]} W^{[v,v+1]}\ar[d]\\
D^{[k]}\ar[r]&\prod\limits_{[v,v+1]\subset[k]}D^{[v,v+1]}
}
\end{align}
for any $[k]\in N(\Delta)^{op}$. 
(resp.
\begin{align}\label{lemma: diagram Cart comm}
&\xymatrix{W^{\lng n\rng}\ar[r]\ar[d]&\prod\limits_{j\in \langle n\rangle^\circ} W^{\{j,*\}}\ar[d]\\
D^{\langle n\rangle}\ar[r]&\prod\limits_{j\in \langle n\rangle^\circ} D^{\{j,*\}}
}
\end{align}
for any $\lng n\rng\in N(\Fin_*)$).

Moreover, if the following diagrams are also Cartesian in $\cC$
\begin{align}\label{lemma: diagram Cart 2}
\xymatrix{
W^{[2]}\ar[r]\ar[d]&C^{[2]}\ar[d]\\
W^{[1]}\ar[r]&C^{[1]}
}
\end{align}
for the unique \emph{active} $f: [1]\rightarrow [2]$ in $N(\Delta)$
(resp. 
\begin{align}\label{lemma: diagram Cart 2 comm}
\xymatrix{
W^{\lng 2\rng}\ar[r]\ar[d]&C^{\lng 2\rng}\ar[d]\\
W^{\langle 1\rangle}\ar[r]&C^{\langle 1\rangle}
}
\end{align} 
for the unique \emph{active} $f: \langle 2\rangle\rightarrow\langle 1\rangle$ in $N(\Fin_*)$),
then $F_{\Gray, S_W}$ (resp. $F_{\Gray, S_W}^\Comm$) induces an algebra homomorphism between the associative (resp. commutative) algebra objects in $\Corr(\cC)$.
\end{thm}
\begin{proof}
We will only prove the associative case; the commutative case follows in a completely similar way. Like in the proof of Theorem \ref{thm: algebra objs}, to make $F_{\Gray,S_W}$ correspond to a functor $N(\Delta)^{op}\Gray [1]\rightarrow \bCorr(\cC)^{\otimes,\Delta^{op}}$ over $N(\Delta)^{op}$, we just need $F_{\Gray,S_W}$ to send any object $([n]; \alpha,\beta,\tau_0,\tau_1)\in \Seq_n(N(\Delta)^{op}\Gray [1])$ to an element $F\in \Fun''(([n]\times[n]^{op})^{\geq \dgnl},\cC^{\times,\Delta})$, i.e. $F$ satisfying (Obj) below (\ref{eq: diagram Fun'}), and for any morphism $(([n]; \alpha,\beta,\tau_0,\tau_1)\rightarrow ([n]; \alpha,\beta,\tau'_0,\tau'_1))$ to a morphism $q: F\rightarrow F'$ satisfying (Mor) below (\ref{eq: diagram Fun'}). It is clear from the construction of $F_{\Gray,S_W}$ that the latter is satisfied by $F_{\Gray, S_W}$ regardless of the conditions on $C^\bullet, W^\bullet, D^\bullet$. 

For any $([n]; \alpha,\beta,\tau_0,\tau_1)\in \Seq_n(N(\Delta)^{op}\Gray [1])$, its image under $F_{\Gray, S_W}$ is a functor $F_{([n],s^\tau)}: ([n]\times[n]^{op})^{\geq \dgnl}\underset{N(\Delta)}{\times} T\rightarrow \cC$ in $\cC$ which takes
\begin{align}
\nonumber((i,j), [u,u+1])\mapsto &C^{[\alpha_{ji}(u), \alpha_{ji}(u+1)]}, i\leq j\leq \mu-1\\
\nonumber((i, \mu+j),[u,u+1])\mapsto &W^{[\tau_1\beta_{j,0}(u),\tau_1\beta_{j,0}(u+1)]}\underset{C^{[\tau_1\beta_{j,0}(u),\tau_1\beta_{j,0}(u+1)]}}{\times}C^{[\alpha_{\mu-1,i}\tau_0\tau_1\beta_{j,0}(u), \alpha_{\mu-1,i}\tau_0\tau_1\beta_{j,0}(u+1)]}\\
\label{eq: WCCalphabeta}&0\leq i\leq \mu-1, 0\leq j\leq n-\mu\\
\nonumber(\mu+i,\mu+j; [u,u+1])\mapsto &D^{[\beta_{ji}(u),\beta_{ji}(u+1)]}, 0\leq i\leq j\leq n-\mu.
\end{align}

Now assume $C^\bullet\leftarrow W^\bullet \rightarrow D^\bullet$ is a correspondence of simplicial objects with $C^\bullet$ and $D^\bullet$ representing algebra objects in $\Corr(\cC)$. Let $WCC^{\alpha,\beta,\tau}_{i, j;u}$ denote the fiber product on the RHS of (\ref{eq: WCCalphabeta}). 
Then to show that $F_{([n],s^\tau)}$ satisfies (Obj), it suffices to check that the following diagrams are Cartesian in $\cC$:
\begin{align}\label{thm: diagram Cart 1}
&\xymatrix{WCC^{\alpha,\beta,\tau}_{i,j;u}\ar[r]\ar[d]&
\prod\limits_{\substack{[v,v+1]\subset [\alpha_{\mu-1,i+1}\tau_0\tau_1\beta_{j,0}(u),\\
\alpha_{\mu-1,i+1}\tau_0\tau_1\beta_{j,0}(u+1)]}}C^{[\alpha_{i+1,i}(v),\alpha_{i+1,i}(v+1)]}\ar[d]\\
WCC_{i+1,j;u}^{\alpha,\beta,\tau}\ar[r]&
\prod\limits_{\substack{[v,v+1]\subset [\alpha_{\mu-1,i+1}\tau_0\tau_1\beta_{j,0}(u),\\
 \alpha_{\mu-1,i+1}\tau_0\tau_1\beta_{j,0}(u+1])]}}C^{[v,v+1]}
}\\
\nonumber&i+1\leq \mu-1,
\end{align}

 \begin{align}\label{thm: diagram Cart 2}
&\xymatrix{WCC^{\alpha,\beta,\tau}_{i,j;u}\ar[r]\ar[d]&\prod\limits_{[v,v+1]\subset [\beta_{j,j-1}(u),\beta_{j,j-1}(u+1)]} WCC^{\alpha,\beta,\tau}_{i,j-1;v}\ar[d]\\
D^{[\beta_{j,j-1}(u),\beta_{j,j-1}(u+1)]}\ar[r]&\prod\limits_{[v,v+1]\subset [\beta_{j,j-1}(u),\beta_{j,j-1}(u+1)]}D^{[v,v+1]}
}\\
\nonumber&i\leq \mu-1.
\end{align}

The first diagram (\ref{thm: diagram Cart 1}) is Cartesian directly follows from the assumption that $C^\bullet$ represents an algebra object, and the second (\ref{thm: diagram Cart 2}) is Cartesian 
is equivalent to the condition that the following diagram 
 \begin{align}
&\xymatrix{W^{[f(0),f(k)]}\ar[r]\ar[d]&\prod\limits_{[v,v+1]\subset [k]} W^{[f(v),f(v+1)]}\ar[d]\\
D^{[k]}\ar[r]&\prod\limits_{[v,v+1]\subset[k]}D^{[v,v+1]}
}
\end{align}
is Cartesian for any $f: [k]\rightarrow [\ell]$ in $N(\Delta)$. The latter is equivalent to the assumptions on $W^\bullet$ in the theorem. So this proves (a).

For the second part, we just need to check that for any morphism $([n],\alpha,\beta;\tau_0,\tau_1)\rightarrow ([n],\alpha,\beta;\tau_0',\tau_1')$ in $\Seq_n(N(\Delta)^{op}\Gray[1])$ induced by $\tau: [\xi']\rightarrow [\xi]$ in $N(\Delta)$, the morphisms
\begin{align*}
&WCC_{i,j;u}^{\alpha,\beta,\tau}
\longrightarrow WCC_{i,j;u}^{\alpha,\beta,\tau'}
\end{align*}
are isomorphisms for $0\leq i\leq \mu-1, 0\leq j\leq n-\mu$. It then suffices to show that for any $[n]\overset{f}{\rightarrow} [m]\overset{g}{\rightarrow} [k]$ in $N(\Delta)$, the following diagram is Cartesian
\begin{align}\label{diagram: W, C Cartesian}
\xymatrix{
W^{[gf(u),gf(u+1)]}\ar[r]\ar[d]&C^{[gf(u),gf(u+1)]}\ar[d]\\
W^{[f(u), f(u+1)]}\ar[r]&C^{[f(u),f(u+1)]}
}.
\end{align}
These are exactly the diagrams
\begin{align*}
\xymatrix{
W^{[\ell]}\ar[r]\ar[d]&C^{[\ell]}\ar[d]\\
W^{[1]}\ar[r]&C^{[1]}
}
\end{align*}
for the unique \emph{active} $f: [1]\rightarrow [\ell]$ for each $[\ell]$ in $N(\Delta)$. This can be easily reduced to the case for $\ell=2$. Alternatively, one can directly get to this from the characterization of algebra homomorphisms out of right-lax homomorphisms. So the second part is established. 
 \end{proof}

Now Theorem \ref{thm: right-lax} immediately implies the following. 
\begin{cor}\label{cor: right-lax Segal}
\begin{itemize}
\item[(a)]
Any correspondence of Segal objects (resp. commutative Segal objects) $D^\bullet\leftarrow W^\bullet \rightarrow C^\bullet$ gives a right-lax homomorphism from the associative algebra object (resp. commutative algebra object) in $\bCorr(\cC^\times)$ represented by $C^\bullet$ to that represented by $D^\bullet$. \\
\item[(b)] In the same setting of (a), if the correspondence further satisfies that the following diagram is Cartesian in $\cC$,
\begin{equation}\label{thm: Cartesian}
\xymatrix{W^{[2]}\ar[r]\ar[d]&C^{[2]}\ar[d]\\
W^{[1]}\ar[r]&C^{[1]}.}
\end{equation}
(resp. 
\begin{equation}\label{thm: Cartesian comm}
\xymatrix{W^{\langle 2\rangle}\ar[r]\ar[d]&C^{\lng 2\rng}\ar[d]\\
W^{\lng 1\rng}\ar[r]&C^{\lng 1\rng}.}),
\end{equation}
where the vertical maps are induced by the unique active map $[1]\rightarrow [2]$ in $N(\Delta)$ (resp. $\lng 2\rng\rightarrow\lng 1\rng$ in $N(\Fin_*)$),
then it gives a homomorphism in $\Corr(\cC^\times)$ between the corresponding associative (resp. commutative) algebra objects. 
\end{itemize}
\end{cor}

\begin{remark}\label{alg composition}
From the above description, it is easy to see that the composition of two (right-lax) homomorphisms between (commutative) algebra objects given by correspondences 
\begin{align}\label{eq: 2 corresp}
D^\bullet\leftarrow W^\bullet\rightarrow C^\bullet,\ E^\bullet\leftarrow Y^\bullet\rightarrow D^\bullet
\end{align}
of simplicial objects (resp. $\Fin_*$-objects). 
The composition is homotopic to the one given by the correspondences 
\begin{align}\label{eq: comp corresp}
E^\bullet\leftarrow W^\bullet\underset{D^\bullet}{\times}Y^\bullet\rightarrow C^\bullet.
\end{align}
Moreover, using the machinery that we have developed, one can construct  functors 
\begin{align*}
&\Corr(\Fun'(N(\Fin_*),\cC^\times))_{\inert, \all}\rightarrow \CAlg(\bCorr(\cC^\times))^{\rightlax},\\
&\Corr(\Fun'(N(\Fin_*),\cC^\times))_{\inert, \act}\rightarrow \CAlg(\bCorr(\cC^\times)),
\end{align*}
where $\Fun'(N(\Fin_*),\cC^\times)$ is the full subcategory of $\Fun(N(\Fin_*),\cC^\times)$ consisting of $\Fin_*$-objects satisfying the condition in Theorem \ref{thm: algebra objs}, $\inert$ (resp. $\act$) is the class of morphisms satisfying that (\ref{lemma: diagram Cart comm}) (resp. (\ref{lemma: diagram Cart 2 comm}))) is Cartesian. There is an obvious analogue for the associative case. it is reasonable to expect that the functors are equivalences and we will address this in a separate note.
\end{remark}

Let $\Fun_{\inert}(N(\Delta)^{op}, \cC)$ (resp. $\Fun_{\inert}(N(\Fin_*), \cC)$) be the 1-full subcategory of $\Fun(N(\Delta)^{op}, \cC)$ (resp. $\Fun(N(\Fin_*), \cC)$) consisting of simplicial objects (resp. $\Fin_*$-objects) in $\cC$ that satisfy the condition in Theorem \ref{thm: algebra objs} (i) (resp. (ii)), and whose morphisms from $W^\bullet$ to $D^\bullet$ consisting of those satisfying  (\ref{lemma: diagram Cart}) (resp. (\ref{lemma: diagram Cart comm})). It is not hard to deduce from the above constructions the following proposition. 
\begin{prop}\label{prop: Fun_Cart}
There are natural functors 
\begin{align*}
&\Fun_{\inert}(N(\Delta)^{op}, \cC)\rightarrow \Assoc\Alg(\Corr(\cC^\times)),\\
&\Fun_{\inert}(N(\Fin_*), \cC)\rightarrow \Comm\Alg(\Corr(\cC^\times)), 
\end{align*}
that send every object $C^\bullet$ on the left-hand-side to the associative/commutative algebra object assigned in Theorem \ref{thm: algebra objs}. 
\end{prop}
\begin{proof}
We will prove the commutative case; the associative case is entirely similar. Let $I_{\Delta^n}\rightarrow N(\Delta)^{op}$ (resp. $I_{\Delta^n\times N(\Fin_*)}\rightarrow N(\Delta)^{op}$) be the coCartesian fibration of ordinary 1-categories from the Grothendieck construction of $\Seq_\bullet(\Delta^n): N(\Delta)^{op}\rightarrow \OneCat^{\ord}$ (resp. $\Seq_\bullet(N(\Fin_*)\times \Delta^n): N(\Delta)^{op}\rightarrow \OneCat^{\ord}$). Then we have an isomorphism of coCartesian fibrations over $I^\Comm$. 
\begin{align*}
I_{\Delta^n\times N(\Fin_*)}\simeq I_{\Delta^n}\underset{N(\Delta)^{op}}{\times}(I^\Comm).
\end{align*}
We will define the sought-for functor by first exhibiting a morphism in $\Spc^{\Delta^{op}}$:
\begin{align}\label{eq: proof Seq I_Delta}
\Seq_\bullet(\Fun(N(\Fin_*),\cC))\rightarrow \Maps_{I^{\Comm}}(I_{\Delta^\bullet}\underset{N(\Delta^{op})}{\times} \cT^\Comm,\cC)
\end{align}
To this end, we just need to construct a natural morphism
\begin{align}\label{eq: prop I_delta_n}
I_{\Delta^\bullet}\underset{N(\Delta)^{op}}{\times} \cT^\Comm\rightarrow \Delta^\bullet\times N(\Fin_*)
\end{align}
in $(\OneCat)^{\Delta}$. The construction is very similar to that of (\ref{eq: F_T_Comm}). For any $[n]\in N(\Delta)^{op}$, the fiber of 
\begin{align*}
I_{\Delta^\bullet}\underset{N(\Delta)^{op}}{\times} \cT^\Comm\rightarrow N(\Delta)^{op}
\end{align*}
over $[n]$ is $\Seq_n(\Delta^\bullet)\times (\bigsqcup\limits_{s\in \Seq_n(N(\Fin_*))}\cT^\Comm_{([n],s)})$. Now we define 
\begin{align}
\label{eq: I_Delta to Delta}&\Seq_n(\Delta^\bullet)\times (\bigsqcup\limits_{s\in \Seq_n(N(\Fin_*))}\cT^\Comm_{([n],s)})\rightarrow \Delta^\bullet\times N(\Fin_*)\\
\nonumber&(\vartheta^{(m)}; (s, (i,j), u))\mapsto (\vartheta^{(m)}(i); s_{i,j}^{-1}(u)\sqcup\{*\}),
\end{align}
in which $\vartheta^{(m)}: \Delta^n\rightarrow \Delta^m$ represents an element in $\Seq_n(\Delta^m)$ and the functor on morphisms is the obvious one. With the input from $F_{\cT^\Comm, \Fin_*}$,  the functors  (\ref{eq: I_Delta to Delta}) for different $[n]\in N(\Delta)^{op}$ assemble to be the desired morphism (\ref{eq: prop I_delta_n}). It is easy to see that the image under (\ref{eq: proof Seq I_Delta}) of any object in $\Seq_\bullet(\Fun(N(\Fin_*),\cC))$ lies in $\Maps_{I^\Comm}(I^\natural_{\Delta^\bullet\times N(\Fin_*)}\underset{I^\Comm}{\times}(\cT^\Comm)^\natural, (\cC\times I^\Comm)^\natural)$, therefore we have a well defined morphism in $\Spc^{\Delta^{op}}$:
\begin{align}\label{eq: proof Seq bullet star}
\Seq_\bullet(\Fun(N(\Fin_*),\cC))\rightarrow\Maps_{(\OneCat)^{\Delta^{op}}_{/\Seq_\clubsuit(N(\Fin_*)) }}(\Seq_\clubsuit(\Delta^\bullet\times N(\Fin_*)), f_{\Delta^\clubsuit}^{\cC,\Comm}).
\end{align}

Lastly, one just need to check that the image of any object in (\ref{eq: proof Seq bullet star}) gives rise to a functor $\Delta^k\times N(\Fin_*)\rightarrow \Corr(\cC)^{\otimes, \Fin_*}$ over $N(\Fin_*)$ for any $k$. Like in the proof of Theorem \ref{thm: algebra objs}, this is guaranteed by the assumptions on objects and morphisms in $\Fun_{\inert}(N(\Fin_*),\cC)$. 
\end{proof}

An immediate corollary of Proposition \ref{prop: Fun_Cart} is the following. 
\begin{cor}\cite[Chapter 9, Corollary 4.4.5]{Nick}
For any $c\in \cC$, there is a canonically defined functor
\begin{align*}
\Seg(c)\rightarrow \Assoc\Alg(\Corr(\cC)).
\end{align*}
\end{cor}

Let $\Fun_{\act}(N(\Delta)^{op}, \cC)$ (resp. $\Fun_{\act}(N(\Fin_*), \cC)$) be the 1-full subcategory of $\Fun(N(\Delta)^{op},\cC)$ (resp. $\Fun(N(\Fin_*), \cC)$) consisting of the simplicial  ($\Fin_*$-) objects in $\cC$ that satisfy the condition  in Theorem \ref{thm: algebra objs} (i) (resp. (ii)), whose morphisms from $W^\bullet$ to $C^\bullet$ consist of those satisfying (\ref{lemma: diagram Cart 2}) (resp. (\ref{lemma: diagram Cart 2 comm})). 
\begin{prop}\label{prop: Fun_active, Alg}
There are natural functors 
\begin{align}
\label{eq: prop Fun_act, assoc}&\Fun_{\act}(N(\Delta)^{op}, \cC)^{op}\rightarrow \Assoc\Alg(\Corr(\cC^\times)),\\
\label{eq: prop Fun_act, comm}&\Fun_{\act}(N(\Fin_*), \cC)^{op}\rightarrow \Comm\Alg(\Corr(\cC^\times)),
\end{align}
that send every object $C^\bullet$ on the left-hand-side to the associative/commutative algebra object assigned in Theorem \ref{thm: algebra objs}. 
\end{prop}
\begin{proof}
The idea is similar to that of Proposition \ref{prop: Fun_Cart}.  
In the commutative case, we define a functor 
\begin{align*}
&I_{\Delta^\bullet}\underset{N(\Delta)^{op}}{\times}\cT^\Comm\rightarrow (\Delta^\bullet)^{op}\times N(\Fin_*)\\
&([n]; \vartheta^{(m)}; (s,(i,j),u))\mapsto ((\vartheta^{(m)})^{op}(j); s_{i,j}^{-1}(u)\sqcup\{*\}),
\end{align*} 
where $(\vartheta^{(m)}; (s,(i,j),u))$ is in the fiber of $I_{\Delta^\bullet}\underset{N(\Delta)^{op}}{\times}\cT^\Comm$ over $[n]$ (cf. (\ref{eq: I_Delta to Delta})). This induces a functor
\begin{align*}
\Seq_\bullet(\Fun(N(\Fin_*),\cC)^{op})\rightarrow \Maps_{I^\Comm}(I^\natural_{\Delta^\bullet\times N(\Fin_*)}\underset{I^\Comm}{\times}(\cT^\Comm)^\natural, (\cC\times I^\Comm)^\natural)
\end{align*}
 for any $\cC$, and hence a morphism in $\Spc^{\Delta^{op}}$:
 \begin{align}\label{eq: proof clubsuit}
 \Seq_\bullet(\Fun(N(\Fin_*),\cC)^{op})\rightarrow\Maps_{(\OneCat)^{\Delta^{op}}_{/\Seq_\clubsuit(N(\Fin_*)) }}(\Seq_\clubsuit(\Delta^\bullet\times N(\Fin_*)), f_{\Delta^\clubsuit}^{\cC,\Comm}).
\end{align}
It is straightforward to see that the restriction of the morphism (\ref{eq: proof clubsuit}) to the simplical subspace $\Seq_\bullet(\Fun_{\act}(N(\Fin_*),\cC)^{op})$ induces the desired functor (\ref{eq: prop Fun_act, comm}).

The proof for the associative case proceeds in the same vein. 
\end{proof}

\begin{prop}\label{prop: K times L CAlg}
Let $K,L$ be two $\infty$-categories. Assume $F: K\times L^{op}\rightarrow \Fun(N(\Fin_*), \cC)$ is a functor satisfying
\begin{itemize}
\item[(i)] For any $\ell\in L$ (resp.  $k\in K$), $F|_{K\times\{\ell\}}$ (resp. $F|_{\{k\}\times L^{op}}$) factors through the 1-full subcategory $\Fun_{\inert}(N(\Fin_*), \cC)$ (resp. $\Fun_{\act}(N(\Fin_*), \cC)$);
\item[(ii)] The induced functor $ev_{\lng 1\rng}\circ F: K\times L^{op}\rightarrow \cC$ satisfies that for any square $\alpha: [1]\times [1]^{op}\rightarrow K\times L^{op}$ determined by a functor $\alpha\in \Fun([1],K)\times \Fun([1]^{op},L^{op})$, the diagram $ev_{\lng 1\rng}\circ F\circ \alpha: [1]\times [1]^{op}\rightarrow \cC$ is a Cartesian square,
\end{itemize}
then $F$ naturally determines a functor $F_{\CAlg}: K\times L\rightarrow \CAlg(\Corr(\cC^\times))$. 
\end{prop}
\begin{proof}
Let $I_{K\times L\times N(\Fin_*)}$ be the coCartesian fibration over $N(\Delta)^{op}$ from the Grotherdieck construction of 
\begin{align*}
&\Seq_\bullet(K\times L\times N(\Fin_*)): N(\Delta)^{op}\rightarrow \OneCat^\ord.
\end{align*}
Define $I_K$ and $I_L$ in the same way, and we have 
\begin{align*}
I_{K\times L\times N(\Fin_*)}\simeq I_K\underset{N(\Delta)^{op}}{\times}I_L\underset{N(\Delta)^{op}}{\times}I^\Comm.
\end{align*}
We define a functor 
\begin{align*}
&I_{K\times L\times N(\Fin_*)}\underset{I^\Comm}{\times} \cT^\Comm\simeq I_K\underset{N(\Delta)^{op}}{\times}I_L\underset{N(\Delta)^{op}}{\times}\cT^\Comm\longrightarrow K\times L^{op}\times N(\Fin_*)\\
&([n]; \vartheta^K,\vartheta^L;(s,(i,j),u))\mapsto (\vartheta^K(i),(\vartheta^L)^{op}(j);s^{-1}_{i,j}(u)\sqcup\{*\}). 
\end{align*}
The composition of the above with $F$ determines a morphism in $(\OneCat)^{\Delta^{op}}_{/\Seq_\bullet(N(\Fin_*))}$
\begin{align}\label{eq: prop proof K, L}
\Seq_\bullet(K\times L\times N(\Fin_*))\rightarrow f_{\Delta^\bullet}^{\cC,\Comm}.
\end{align}
To make (\ref{eq: prop proof K, L}) into a functor $K\times L\times N(\Fin_*)\rightarrow \Corr(\cC)^{\otimes,\Fin_*}$ over $N(\Fin_*)$, we need for every square $\alpha\in \Fun([1], K)\times \Fun([1]^{op},L^{op})\rightarrow \Fun([1]\times [1]^{op},K\times L^{op})$ in $K\times L^{op}$ and every active morphism $f: \lng m\rng\rightarrow \lng n\rng$ in $N(\Fin_*)$, the following diagram 
\begin{align*}
\xymatrix{(F\circ\alpha(0,1))^{\lng m\rng}\ar[r] \ar[d]&\prod\limits_{j\in\lng n\rng^\circ}(F\circ\alpha(0,0))^{(f^{-1}(j)\sqcup\{*\})}\ar[d]\\
(F\circ\alpha(1,1))^{\lng n\rng}\ar[r] &\prod\limits_{j\in\lng n\rng^\circ}(F\circ\alpha(1,0))^{\{j,*\}}
}
\end{align*}
is Cartesian in $\cC$. Assumption (i) in the proposition reduces the above family of Cartesian diagrams to the following Cartesian diagram 
\begin{align*}
\xymatrix{(F\circ\alpha(0,1))^{\lng 1\rng}\ar[r] \ar[d]&(F\circ\alpha(0,0))^{\lng 1\rng}\ar[d]\\
(F\circ\alpha(1,1))^{\lng 1\rng}\ar[r] &(F\circ\alpha(1,0))^{\lng 1\rng}
}, 
\end{align*}
for each $\alpha$.  This is exactly condition (ii) in the proposition and the above determines the desired functor $F_\CAlg: K\times L\rightarrow\CAlg(\Corr(\cC^\times))$.  
\end{proof}

\subsection{Module objects of an algebra object $C^\bullet$ in $\Corr(\cC^\times)$}

Let $e: N(\Delta)^{op}\rightarrow N(\Delta)^{op}$ be the natural functor taking $[n]$ to $[n]\star[0]=[n+1]$, and let $\alpha: e\rightarrow id$ be the obvious natural transformation from $e$ to $id$. These notions have been used in the path space construction for a simplicial space (cf. \cite{Seg2}).  
\begin{definition}\label{def: left module}
For any monoidal $\infty$-category $\cC^{\otimes,\Delta^{op}}$ over $N(\Delta)^{op}$, and an algebra object $C$ defined by a section $s\in \Hom_{N(\Delta)^{op}}(N(\Delta)^{op}, \cC^\otimes)$, a \emph{left $C$-module} $M$ is given by a functor 
\begin{align*}
M: \Delta^1\times N(\Delta)^{op}\rightarrow \cC^{\otimes,\Delta^{op}}
\end{align*}
over $N(\Delta)^{op}$, where the projection $\Delta^1\times N(\Delta)^{op}\rightarrow N(\Delta^{op})$ is given by the natural transformation $\theta: e\rightarrow id$, satisfying that 
\begin{itemize}
\item[(i)] it sends every edge $\Delta^1\times \{[k]\}$ to a coCartesian edge in $\cC^{\otimes,\Delta^{op}}$;
\item[(ii)] the restriction of $M$ to $\{1\}\times N(\Delta)^{op}$ is equivalent to the algebra object $C$;
\item[(iii)] For any $[n]\in N(\Delta)^{op}$, the inclusion $\iota_{(0,[j,n])}: (0,[j,n])\hookrightarrow (0,[n])$ in $N(\Delta)$
is sent to a coCartesian morphism in $\cC^{\otimes,\Delta^{op}}$ over $\alpha(\iota_{(0,[j,n])}^{op }):[0,n+1]\rightarrow [j,n+1]$ in $N(\Delta)^{op}$. 
\end{itemize}
\end{definition}

\begin{definition}
For any symmetric monoidal $\infty$-category $\cC^{\otimes, \Fin_*}$ over $N(\Fin_*)$, and a commutative algebra object $C$ given by a section $s: N(\Fin_*)\rightarrow \cC^{\otimes, \Fin_*}$ over $N(\Fin_*)$, a $C$-module $M$ is defined to be a functor 
\begin{align*}
M: N(\Fin_*)_{\langle 1\rangle/}\rightarrow \cC^{\otimes, \Fin_*}
\end{align*}
over $N(\Fin_*)$ satisfying that 
\begin{itemize}
\item[(i)] its restriction to the null morphisms from $\langle 1\rangle$ (as objects in $N(\Fin_*)_{\lng 1\rng/}$) gives $C$;
\item[(ii)] it sends every inert morphism to an inert morphism (i.e. coCartesian morphism) in $\cC^{\otimes, \Fin_*}$. 
\end{itemize}
\end{definition}

Let $I_{\Delta^1\times N(\Delta)^{op}}$ be the 1-category from the Grothendieck construction of the functor $\Seq_\bullet(\Delta^1\times N(\Delta)^{op}): N(\Delta)^{op}\rightarrow \Sets$. Let $\pi
_\theta:I_{\Delta^1\times N(\Delta)^{op}}\rightarrow I$ be the functor corresponding to $\theta: \Delta^1\times N(\Delta)^{op}\rightarrow N(\Delta)^{op}$ which is a coCartesian fibration by the same argument as in Lemma \ref{lemma: pi Gray coCart}. 

Given a morphism of simplicial objects $M^\bullet\rightarrow C^\bullet$ in $\cC^\times$, we are aiming to construct a functor 
\begin{align}\label{eq: F_{M,C}}
F_{M,C}: I^\natural_{\Delta^1\times N(\Delta)^{op}}\underset{I}{\times} \cT^\natural\rightarrow (\cC\times I)^\natural
\end{align}
over $I$. We will represent every object in $I_{\Delta^1\times N(\Delta^{op})}$ by $(\alpha, \beta;\tau)$
\begin{align*}
\xymatrix{
[n_0]_0& \cdots \ar[l]_{\alpha_{1,0}}&[n_{\mu-1}]_0\ar[l]_{\alpha_{\mu-1,\mu-2}}&[m_0]_{1}\ar[l]_{\tau}&\ar[l]_{\beta_{1,0}}\cdots &[m_{k-\mu}]_1\ar[l]_{\beta_{k-\mu,k-\mu-1}},
}\end{align*}
where as before, the subscripts indicate the vertices in $\Delta^1$ and  the arrows indicate the morphisms in $(\Delta^1)^{op}\times N(\Delta)$. The image of $(\alpha, \beta;\tau)$ in $I$ is 
\begin{align*}
\xymatrix{
[n_0]\star [0]& \cdots \ar[l]_{\theta(\alpha_{1,0})}&[n_{\mu-1}]\star[0]\ar[l]_{\theta(\alpha_{\mu-1,\mu-2})}&[m_0]\ar[l]_{\theta([n_{\mu-1}])\circ\tau}&\ar[l]_{\beta_{1,0}}\cdots &[m_{k-\mu}]\ar[l]_{\beta_{k-\mu,k-\mu-1}},
}
\end{align*}

Now we define a functor $F_{\theta}: I_{\Delta^1\times N(\Delta)^{op}}\underset{I}{\times} \cT\rightarrow \Delta^1\times N(\Delta)^{op}$ in a similar way as we defined for $F_{\cT,\Delta^{op}}$ in the proof of Theorem \ref{thm: algebra objs}. First, we have the composition $\widetilde{F}_\theta: I_{\Delta^1\times N(\Delta)^{op}}\underset{I}{\times} \cT\rightarrow \cT\overset{F_{\cT,N(\Delta)^{op}}}{\rightarrow} N(\Delta)^{op}$.  
On the object level, 
\begin{itemize}
\item[(i)] for any object $(\alpha,\beta,\tau; (\mu+i,\mu+j), [u,u+1])\in I^\natural_{\Delta^1\times N(\Delta)^{op}}\underset{I}{\times} \cT^\natural$, $\widetilde{F}_\theta$ sends it to $[\beta_{ji}(u),\beta_{ji}(u+1)]\in N(\Delta)^{op}$;
\item[(ii)] for any object $(\alpha,\beta,\tau; (i,j), [u,u+1])$ with $i,j\leq \mu-1$, if $u<n_j$, then $F_\theta$ sends it to $[\alpha_{ji}(u), \alpha_{ji}(u+1)]$. If $u=n_j$, then $\widetilde{F}_\theta$ sends it to $[\theta(\alpha_{ji})(n_j), n_i+1]$;
\item[(iii)] for any object $(\alpha,\beta,\tau; (i,\mu+j), [u,u+1])$ with $i\leq \mu-1$, $\widetilde{F}_\theta$ sends it to
$[\alpha_{\mu-1,i}\tau\beta_{j,0}(u), \alpha_{\mu-1,i}\tau\beta_{j,0}(u+1)]$. 
\end{itemize}
Now we do the following modification of $\widetilde{F}_\theta$ on the object level to the above as follows: 
\begin{itemize}
\item[(i)] $[\beta_{ji}(u),\beta_{ji}(u+1)]\in N(\Delta)^{op}\rightsquigarrow(1, [\beta_{ji}(u),\beta_{ji}(u+1)])\in \Delta^1\times N(\Delta)^{op}$\\
\item[(ii)] If $u<n_j$, $[\alpha_{ji}(u), \alpha_{ji}(u+1)]\in N(\Delta)^{op}\rightsquigarrow (1, [\alpha_{ji}(u), \alpha_{ji}(u+1)])\in \Delta^1\times N(\Delta)^{op}$.
If $u=n_j$, then  $[\theta(\alpha_{ji})(n_j), n_i+1]\in N(\Delta)^{op}\rightsquigarrow (0, [\theta(\alpha_{ji})(n_j), n_i])=(0, [\alpha_{ji}(n_j), n_i])\in \Delta^1\times N(\Delta)^{op}$\\
\item[(iii)] $[\alpha_{\mu-1,i}\tau\beta_{j,0}(u), \alpha_{\mu-1,i}\tau\beta_{j,0}(u+1)]\in N(\Delta)^{op}\rightsquigarrow (1, [\alpha_{\mu-1,i}\tau\beta_{j,0}(u), \alpha_{\mu-1,i}\tau\beta_{j,0}(u+1)])\in \Delta^1\times N(\Delta)^{op}$. 
\end{itemize}
Since $\theta(\alpha_{ji})$ maps $n_{j}+1$ to $n_i+1$ and $n_j$ to a number less than or equal to $n_i$, it is easy to see that the above modification defines a functor $ I_{\Delta^1\times N(\Delta)^{op}}\underset{I}{\times} \cT\rightarrow \Delta^1\times N(\Delta)^{op}$ and that is $F_\theta$. 

Now viewing $M^\bullet\rightarrow C^\bullet$ as a functor $MC: \Delta^1\times N(\Delta)^{op}\rightarrow \cC$, the composition $MC\circ F_\theta$ gives the functor $F_{M,C}$ (\ref{eq: F_{M,C}}). 

\begin{thm}\label{thm: left module}
Given a morphism of simplicial objects $M^\bullet \rightarrow C^\bullet$ in $\cC^\times$, if $C^\bullet$ represents an algebra object in $\Corr(\cC)$, i.e. it satisfies the conditions in Theorem \ref{thm: algebra objs} (i), and $M^\bullet$ satisfies that for any morphism $f: [m]\rightarrow [n]$ in $N(\Delta)$, we have the following diagram 
 \begin{align}\label{diagram: thm left module}
\xymatrix{M^{[f(0), n]}\ar[d]\ar[r]&\prod\limits_{[u,u+1]\subset [0,m])}C^{[f(u), f(u+1)]}\times M^{[f(m), n]}\ar[d]\\
M^{[0, m]}\ar[r] &\prod\limits_{[u,u+1]\subset [0,m]}C^{[u,u+1]}\times M^{\{m\}}
}
\end{align}
is Cartesian in $\cC$, then the data naturally determine a left module of $C$ in $\Corr(\cC)$. Moreover, the condition on $M^\bullet$ is both necessary and sufficient. 
\end{thm}

\begin{proof}
Assuming $C^\bullet$ gives an algebra object in $\Corr(\cC)$, to make $F_{M,C}$ a functor from $\Delta^1\times N(\Delta)^{op}$ to $\Corr(\cC)^{\otimes, \Delta^{op}}$, we need for every object $(\alpha,\beta;\tau)$ in $I_{\Delta^1\times N(\Delta)^{op}}$ and 
any $i\leq j\leq \mu-1$, the diagram 
\begin{align*}
\xymatrix{M^{[\alpha_{j,i-1}(n_j), n_{i-1}]}\ar[d]\ar[r]&\prod\limits_{[u,u+1]\subset [\alpha_{ji}(n_j),n_i])}C^{[\alpha_{i,i-1}(u),\alpha_{i,i-1}(u+1)]}\times M^{[\alpha_{i,i-1}(n_i), n_{i-1}]}\ar[d]\\
M^{[\alpha_{j,i}(n_j), n_{i}]}\ar[r] &\prod\limits_{[u,u+1]\subset [\alpha_{ji}(n_j),n_i])}C^{[u,u+1]}\times M^{[n_{i},n_i]}
}
\end{align*}
is Cartesian in $\cC$. But this is exactly the condition that (\ref{diagram: thm left module}) is Cartesian in $\cC$. By construction, $F_{M,C}$ certainly satisfies (ii) in Definition \ref{def: left module}. Condition (i) and (iii) can be easily checked by the characterization of a 2-coCartesian morphism in $\Corr(\cC)^{\otimes,\Delta^{op}}$ at the beginning of the proof of Proposition \ref{prop: 2-coCart monoidal} and the characterization of a Cartesian morphism in $\cC^{\times, \Delta}$ in Proposition \ref{prop: Cartesian equiv}.
\end{proof}

In a similar fashion, for the commutative case, introduce $\Fin_{*,\dagger}$ to be the ordinary 1-category consisting of $\langle n\rangle_\dagger:= \langle n\rangle\sqcup\{\dagger\}$, and morphisms $\langle n\rangle_\dagger\rightarrow \langle m\rangle_\dagger$ being maps of sets that send $*$ to $*$ and $\dagger$ to $\dagger$. Note that there is a natural functor $\pi_\dagger: N(\Fin_{*,\dagger})\rightarrow N(\Fin_*)$ that sends $\langle n\rangle_\dagger$ to $\langle n\rangle$, and sends $f: \langle n\rangle_\dagger\rightarrow \langle m\rangle_\dagger$ to $\widetilde{f}: \langle n\rangle\rightarrow \langle m\rangle$ where $\widetilde{f}$ sends $f^{-1}(\{\dagger,*\})\backslash \{\dagger\}$ to $*$ and others in the same way as $f$. Let $I_{\pi_\dagger}$ be the Grothendieck construction of $\pi_\dagger: \Delta^1\rightarrow \OneCat^\ord$. We will denote each object in $I_{\pi_\dagger}$ by $(0,\langle n\rangle_\dagger)$ or $(1, \langle n\rangle)$ if it projects to 0 or 1 in $\Delta^1$ respectively. Note that there is a natural equivalence $N(\Fin_*)_{\langle 1\rangle/}\rightarrow I_{\pi_\dagger}$ that takes the object $b: \langle 1\rangle\rightarrow \langle m\rangle$ to $(1,\langle m\rangle)$ if $b$ is null and to $(0, \langle m-1\rangle_\dagger)$ otherwise. 

 We will define a natural functor 
$
 F_\theta^\Comm: I_{N(\Fin_*)_{\langle 1\rangle/}}\underset{I^\Comm}{\times} \cT^\Comm\rightarrow 
 I_{\pi_\dagger}$ as follows, where the functor $N(\Fin_{*})_{\langle 1\rangle/}\rightarrow N(\Fin_*)$ is the natural projection. First, we have the composition functor 
 \begin{align}\label{eq: tilde F_theta}
 \widetilde{F}_\theta^\Comm: I_{N(\Fin_*)_{\langle 1\rangle/}}\underset{I^\Comm}{\times} \cT^\Comm\rightarrow \cT^\Comm\overset{F_{\cT^\Comm,\Fin_*}}{\longrightarrow}  N(\Fin_*).
 \end{align}
 We denote every object in $I_{N(\Fin_*)_{\langle 1\rangle/}}$ by $(\alpha,\beta;\tau)$ represented by 
\begin{align*}
\xymatrix{
\langle n_0\rangle_0\ar[r]^{\alpha_{01}}& \cdots \ar[r]^{\alpha_{\mu-2,\mu-1}}&\langle n_{\mu-1}\rangle_0\ar[r]^{\tau}&\langle m_0\rangle_{1}\ar[r]^{\beta_{01}}&\cdots \ar[r]^{\beta_{k-\mu-1,k-\mu}}&\langle m_{k-\mu}\rangle_1,\\
\langle1\rangle\ar[u]^{\gamma_0}\ar[urrr]&&&&
}
\end{align*}
where the morphism $\langle 1\rangle\rightarrow \langle \bullet\rangle_\epsilon$ is null if and only if $\epsilon=1$. We will abuse notation and denote the image of $(\alpha,\beta,\tau)$ in $I^{\Comm}$ under the obvious projection by the same notation. On the object level,
\begin{itemize}
\item[(i)] For any object $(\alpha,\beta,\tau; (\mu+i,\mu+j), t)\in I_{N(\Fin_*)_{\langle 1\rangle/}}\underset{I^\Comm}{\times} \cT^\Comm$, for $0\leq i\leq j\leq k-\mu$, $\widetilde{F}_\theta^\Comm$ sends it to $\beta_{ij}^{-1}(t)\sqcup\{*\}$;

\item[(ii)] For any object $(\alpha,\beta,\tau; (i,j), t)$, $0\leq i\leq j\leq \mu-1$, $\widetilde{F}_\theta^\Comm$ sends it to $\alpha_{ij}^{-1}(t)\sqcup\{*\}$;

\item[(iii)] For any object $(\alpha,\beta,\tau; (i,\mu+j), t)$, for $0\leq i\leq \mu-1$, $\widetilde{F}_\theta^\Comm$ sends it to $(\beta_{0j}\tau\alpha_{i,\mu-1})^{-1}(t)\sqcup\{*\}$;
\end{itemize}
We will modify $\widetilde{F}_\theta^\Comm$ on the object level in each of the above cases 
\begin{itemize}
\item[(i)] $\beta_{ij}^{-1}(t)\sqcup\{*\}\rightsquigarrow (1,\beta_{ij}^{-1}(t)\sqcup\{*\})\in I_{\pi_\dagger}$;

\item[(ii)] If $t$ is the image of $1\in \langle 1\rangle$, then $\alpha_{ij}^{-1}(t)\sqcup\{*\}\rightsquigarrow (0,(\alpha_{ij}^{-1}(t)\backslash (\alpha_{0i}\gamma_0)(1))\sqcup\{\dagger, *\})\in I_{\pi_\dagger}$, i.e. we replace $\alpha_{0i}\gamma_0(1)$ by $\dagger$.  If $t$ is not in the image of $1\in \langle 1\rangle$, then $\alpha_{ij}^{-1}(t)\sqcup\{*\}\rightsquigarrow (1, \alpha_{ij}^{-1}(t)\sqcup\{*\})$; 

\item[(iii)]$(\beta_{0j}\tau\alpha_{i,\mu-1})^{-1}(t)\sqcup\{*\}\rightsquigarrow (1,(\beta_{0j}\tau\alpha_{i,\mu-1})^{-1}(t)\sqcup\{*\})$.
\end{itemize}
This modification uniquely determines a functor $F_\theta^\Comm: I_{N(\Fin_*)_{\langle 1\rangle/}}\underset{I^\Comm}{\times} (\cT^\Comm)\rightarrow I_{\pi_\dagger}$: if we think of the coCartesian  morphism $(0,\langle n\rangle_\dagger)\rightarrow (1,\langle n\rangle)$ in $I_{\pi_\dagger}$ as an inert morphism sending $\dagger$ to $*$, then the modification of $\widetilde{F}_{\theta}^\Comm$ is just replacing the image of $1\in \langle 1\rangle$ in each $\langle n_i\rangle_0$ (and its subsets containing the image) by $\dagger$ and mark it with $0\in \Delta^1$; since $\dagger$ can be only mapped to $\dagger$ or $*$, whose projection to $\Delta^1$ is $id_{0}$ and $0\rightarrow 1$ respectively, $F_\theta^\Comm$ is well defined.

Given a diagram
\begin{align*}
\xymatrix{ N(\Fin_{*,\dagger})\ar[r]^{M^{\bullet,\dagger}}\ar[d]_{\pi_\dagger}&\cC\\
N(\Fin_*)\ar[ur]_{C^\bullet}&
},
\end{align*} 
and a natural transformation $\eta: M^{\bullet,\dagger}\rightarrow C^\bullet\circ \pi_\dagger$, which we can view as a functor $MC^\Comm:I_{\pi_\dagger}\rightarrow \cC$, the composition $MC^\Comm\circ F_\theta^\Comm$ induces a functor between coCartesian fibrations
\begin{align*}
F_{M,C}^\Comm: I^\natural_{N(\Fin_*)_{\langle 1\rangle/}}\underset{I^\Comm}{\times} (\cT^\Comm)^\natural\rightarrow (\cC\times I^\Comm)^\natural
\end{align*}
over $I^\Comm$. 

\begin{thm}\label{thm: module}
Assume we are given a $\Fin_{*,\dagger}$-object $M^{\bullet,\dagger}$ and a $\Fin_*$-object $C^\bullet$  in $\cC^\times$, together with a natural transformation $\eta: M^{\bullet,\dagger}\rightarrow C^\bullet\circ \pi_\dagger$.  If $C^\bullet$ represents a commutative algebra object in $\Corr(\cC)$, i.e. it satisfies the conditions in Theorem \ref{thm: algebra objs} (ii), and $M^{\bullet,\dagger}$ satisfies that for any active morphism $f: \langle n\rangle_\dagger \rightarrow \langle m\rangle_\dagger $, i.e. $f^{-1}(*)=*$, we have the following diagram 
 \begin{align}\label{diagram: thm module}
\xymatrix{M^{\lng n\rng_\dagg}\ar[d]\ar[r]&\prod\limits_{s\in  \langle m\rangle^\circ}C^{f^{-1}(s)\sqcup\{*\}}\times M^{f^{-1}(\dagger)\sqcup\{*\}}\ar[d]\\
M^{\langle m\rangle_\dagger}\ar[r] &\prod\limits_{s\in \langle m\rangle^\circ}C^{\{s,*\}}\times M^{\{\dagger,*\}}
}
\end{align}
is Cartesian in $\cC$, then the data naturally determine a module of $C$ in $\Corr(\cC)$.  Moreover, the condition on $M^{\bullet,\dagg}$ is both necessary and sufficient.
\end{thm}
\begin{proof}
The proof is completely similar to that of Theorem \ref{thm: left module}. For any object $(\alpha,\beta,\tau)\in I_{N(\Fin_*)_{\langle1\rangle/}}$, we just need to make sure that for any $t=\alpha_{0j}\gamma_0(1)$,  the diagram
\begin{align*}
\xymatrix{M^{(\alpha_{i-1,j}^{-1}(t)\backslash \alpha_{0,i-1}\gamma_0(1))\sqcup\{\dagger, *\}}\ar[r]\ar[d]&\prod\limits_{s\in \alpha_{ij}^{-1}(t)\backslash \alpha_{0i}\gamma_0(1)}C^{\alpha_{i-1,i}^{-1}(s)\sqcup\{*\}}\times M^{(\alpha_{i-1,i}^{-1}(\alpha_{0i}\gamma_0(1))\backslash \alpha_{0,i-1}\gamma_{0}(1))\sqcup\{\dagger,*\}}\ar[d]\\
M^{(\alpha_{ij}^{-1}(t)\backslash \alpha_{0i}\gamma_0(1))\sqcup\{\dagger,*\}}\ar[r]&\prod\limits_{s\in \alpha_{ij}^{-1}(t)\backslash \alpha_{0i}\gamma_0(1)}C^{\{s,*\}}\times M^{\{\dagger,*\}}
}
\end{align*}
is Cartesian in $\cC$. 

The condition is equivalent to that 
for any morphism $f: \langle n\rangle \rightarrow \langle m\rangle $ and any $a\in \langle n\rangle^\circ$ such that $f(a)\neq *$ in $N(\Fin_*)$, we have the following diagram 
 \begin{align}\label{diagram: thm pf module}
\xymatrix{M^{(f^{-1}\langle m\rangle^\circ\backslash \{a\})\sqcup\{\dagger, *\}}\ar[d]\ar[r]&\prod\limits_{s\in  \langle m\rangle^\circ\backslash\{f(a)\}}C^{f^{-1}(s)\sqcup\{*\}}\times M^{(f^{-1}(f(a))\backslash \{a\})\sqcup\{\dagger,*\}}\ar[d]\\
M^{(\langle m\rangle\backslash\{f(a)\})\sqcup\{\dagger\}}\ar[r] &\prod\limits_{s\in \langle m\rangle^\circ\backslash\{f(a)\}}C^{\{s,*\}}\times M^{\{\dagger, *\}}
}
\end{align}
is Cartesian in $\cC$. This is equivalent to (\ref{diagram: thm module})
\end{proof}

Recall the definition of 
\begin{align}\label{eq: Mod^O(C)}
\cMod^{\cO}(\cC)^{\otimes}\rightarrow \cO^\otimes\times \Alg_{/\cO}(\cC) 
\end{align}
given in \cite[Definition 3.3.3.8]{higher-algebra} for any map $\cC^{\otimes}\rightarrow \cO^{\otimes}$ of generalized $\infty$-operads. For $\cO^\otimes=N(\Fin_*)$, we use $\cMod^{N(\Fin_*)}(\cC)$ to denote the fiber product 
\begin{align*}
\cMod^{N(\Fin_*)}(\cC)^{\otimes}\underset{N(\Fin_*)\times \Alg_{/N(\Fin_*)}(\cC)}{\times} (\{\langle 1\rangle\}\times \Alg_{/N(\Fin_*)}(\cC)),
\end{align*}
and any object is represented by a pair $(A,M)$, where $A$ is a commutative algebra object in $\cC$ and $M$ is an $A$-module. In the following, we represent every object $(\lng 1\rng\rightarrow \lng n\rng)$ in $N(\Fin_*)_{\langle 1\rangle/}$ by a pair $(\lng n\rng, s)$, where $s\in \lng n\rng$ is the image of $1\in \lng 1\rng$ and  is regarded as the \emph{marking}. 

When $\cC^{\otimes}=(\PrstL)^{\otimes}$, any object in $\cMod^{N(\Fin_*)}(\PrstL)$ can be represented by a coCartesian fibration 
\begin{align}\label{eq: coCart module}
\cM\rightarrow N(\Fin_*)_{\langle 1\rangle/}
\end{align} satisfying that 
\begin{itemize}
\item[(i)] The pullback of (\ref{eq: coCart module}) along the inclusion of null morphisms $N(\Fin_*)\hookrightarrow N(\Fin_*)_{\lng 1\rng/}$ gives a coCartsian fibration $C^\otimes\rightarrow N(\Fin_*)$ that represents a symmetric monoidal stable $\infty$-category $C\simeq \cM_{(\lng 1\rng, *)}$;
\item[(ii)] For any object $(\lng n\rng, s)\in N(\Fin_*)_{\langle 1\rangle/}$, the coCartesian morphisms over the $n$ inert morphisms $\rho^i: (\lng n\rng, s)\rightarrow (\lng 1\rng,\rho^i(s))$ gives an equivalence
\begin{align}\label{eq: M (ii)}
\cM_{(\lng n\rng, s)}\overset{\sim}{\rightarrow} \prod\limits_{i=1}^n \cM_{(\lng 1\rng, \rho^i(s))}
\end{align} 
of $\infty$-categories;
\item[(iii)] The fiber $\cM_{(\lng 1\rng, 1)}$ is stable, and any morphism in $\cM$ over $\alpha: (\lng m\rng,s)\rightarrow (\lng n\rng, \alpha(s))$, given by $n$ functors via (\ref{eq: M (ii)}),
\begin{align*}
F_i: \prod\limits_{j\in \alpha^{-1}(i)} \cM_{(\lng 1\rng, \rho^i(s))}\rightarrow \cM_{(\lng 1\rng, \rho^i\alpha(s))},\ i\in \lng n\rng^\circ,
\end{align*} 
satisfies that each $F_i$ is exact and continuous in each variable. 
\end{itemize}
Moreover, a \emph{right-lax} morphism (resp. morphism) between two objects $\cM\rightarrow N(\Fin_*)_{\lng1\rng/}$ and $\cN\rightarrow N(\Fin_*)_{\lng1\rng/}$ in $\cMod^{N(\Fin_*)}(\PrstL)$ is given by a commutative diagram
\begin{align*}
\xymatrix{\cM\ar[rr]\ar[dr]&&\cN\ar[dl]\\
&N(\Fin_*)_{\lng1\rng/}
}
\end{align*}
that restricts to exact and continuous functors $\cM_{(\lng 1\rng, 1)}\rightarrow \cN_{(\lng 1\rng, 1)}$ and $\cM_{(\lng 1\rng,*)}\rightarrow \cN_{(\lng 1\rng,*)}$ and that sends coCartesian morphisms in $\cM$ over inert morphisms (resp. all morphisms) in $N(\Fin_*)_{\lng1\rng/}$ to coCartesian morphisms in $\cN$.

Given two pairs $(C^\bullet, M^{\bullet,\dagg})$, $(D^\bullet, N^{\bullet,\dagg})$ of a module over a commutative algebra in $\bCorr(\cC^\times)$, similarly to the construction of (right-lax) algebra homomorphisms in $\bCorr(\cC^\times)$,  we can construct a (right-lax) morphism between the two pairs in the sense of  \cite[Chapter 8, 3.5.1]{Nick} from a correspondence
\begin{align}\label{diagram: corr of modules}
\xymatrix{N^{\bullet,\dagg}\ar[d]&\ar[l] W^{\bullet,\dagg}\ar[r]\ar[d]&M^{\bullet,\dagg}\ar[d]\\
D^\bullet\circ \pi_\dagg&P^\bullet\circ \pi_\dagg\ar[r]\ar[l]&C^\bullet\circ \pi_\dagg
}
\end{align}
of modules over commutative algebras. Let $I_{\widetilde{\pi}_{\dagg}}$ be the ordinary 1-category from the Grothendieck construction of the functor 
\begin{align*}
\widetilde{\pi}_{\dagg}: N(\Fin_{*,\dagg})\times(\Delta^{\{W,M\}}\coprod\limits_{\Delta^{\{W\}}}\Delta^{\{W,N\}})\rightarrow N(\Fin_{*})\times (\Delta^{\{P,C\}}\coprod\limits_{\Delta^{\{P\}}}\Delta^{\{P,D\}}),
\end{align*}
which is given by $(\pi_\dagg,(W\mapsto P, M\mapsto C, N\mapsto D))$. 
 We can regard (\ref{diagram: corr of modules}) as a functor 
\begin{align}\label{eq: MWN}
MWN: I_{\widetilde{\pi}_{\dagg}}\rightarrow \cC. 
\end{align}

Now we are going to construct a functor 
\begin{align}\label{eq: F_theta,module}
F_{\theta, \module}^{\Comm}: I_{N(\Fin_*)_{\langle 1\rangle/}\Gray \Delta^1}\underset{I_+^\Comm}{\times} \cT_+^\Comm\rightarrow I_{\widetilde{\pi}_{\dagg}}.
\end{align}
First, we have the composite functor
\begin{align*}
&\widetilde{F}_{\theta, \module}^{\Comm}: I_{N(\Fin_*)_{\langle 1\rangle/}\Gray \Delta^1}\underset{I_+^\Comm}{\times} \cT_+^\Comm\rightarrow I_{N(\Fin_*)\Gray \Delta^1}\underset{I_+^\Comm}{\times} \cT_+^\Comm\\
&\rightarrow N(\Fin_*)\times (\Delta^{\{W,M\}}\coprod\limits_{\Delta^{\{W\}}}\Delta^{\{W,N\}}),
\end{align*}
in which the first functor comes from the projection $N(\Fin_*)_{\langle 1\rangle/}\rightarrow N(\Fin_*)$ and the second functor is the commutative version of $F_{\Gray, \cT_+}$ defined in Subsection \ref{subsec: F_Gray} (note that we have changed $C,D$ to $M,N$ respectively in the target). 

Next, we modify $\widetilde{F}_{\theta, \module}^{\Comm}$ to get $F_{\theta, \module}^{\Comm}$ in a similar way as we did for $\widetilde{F}_\theta^\Comm$ (\ref{eq: tilde F_theta}). The objects in $N(\Fin_*)_{\langle 1\rangle/}\Gray\Delta^1$ can be divided into \emph{unmarked} ones and \emph{marked} ones, depending on whether its projection to $N(\Fin_*)_{\langle 1\rangle/}$ is null or not null. If an object $\alpha: \lng 1\rng\rightarrow \lng n\rng$ in $N(\Fin_*)_{\langle 1\rangle/}$ is not null, then we say $\alpha(1)$ is the \emph{marking} or \emph{marked element} in $ \lng n\rng$. For any object $x$ of in $I_{N(\Fin_*)_{\langle 1\rangle/}\Gray \Delta^1}\underset{I_+^\Comm}{\times} \cT_+^\Comm$ over an object of the form
\begin{align*}
\xymatrix{&&&(\lng \xi\rng,1)\ar[r]^{\tau_1}&(\langle m_0, 1)\rangle\ar[r]&\cdots \ar[r]&\lng (m_{k-\mu}\rng,1)\\
(\lng n_0\rng,0)\ar[r]&\cdots\ar[r]&(\lng n_{\mu-1}\rng,0)\ar[r]^{\tau_0}&(\lng \xi\rng,0)\ar[u]_{id_{\lng \xi\rng}}&&&\\
\langle 1\rangle\ar[u]
}
\end{align*}
in $I_{N(\Fin_*)_{\langle 1\rangle/}\Gray \Delta^1}$, the projection of $\widetilde{F}_{\theta, \module}^{\Comm}(x)$ to $N(\Fin_*)$ is by definition a subset of one of the above $\lng n_i\rng, 0\leq i\leq \mu-1$ and $\lng m_j\rng,0\leq j\leq k-\mu$. If it contains the image of $\lng 1\rng$ in the latter, i.e. the marked element, then we say the image $\widetilde{F}_{\theta, \module}^{\Comm}(x)$ is \emph{marked}, otherwise we say it is \emph{unmarked}. 

 Now for any  object $x$ in $I_{N(\Fin_*)_{\langle 1\rangle/}\Gray \Delta^1}\underset{I_+^\Comm}{\times} \cT_+^\Comm$, if $\widetilde{F}_{\theta, \module}^{\Comm}(x)$ is unmarked in the above sense, we define $F_{\theta, \module}^{\Comm}(x)$ to be the corresponding element of $\widetilde{F}_{\theta, \module}^{\Comm}(x)$  in $I_{\widetilde{\pi}_\dagg}\underset{\Delta^1}{\times}\{1\}$ under the identification 
 \begin{align*}
 I_{\widetilde{\pi}_\dagg}\underset{\Delta^1}{\times}\{1\}\simeq N(\Fin_{*})\times (\Delta^{\{P,C\}}\coprod\limits_{\Delta^{\{P\}}}\Delta^{\{P,D\}})\simeq  N(\Fin_{*})\times (\Delta^{\{W,M\}}\coprod\limits_{\Delta^{\{W\}}}\Delta^{\{W,N\}}),
 \end{align*}
 where in the latter identification, we send $P\mapsto W, M\mapsto C, N\mapsto D$. 
 If $\widetilde{F}_{\theta, \module}^{\Comm}(x)$ is marked, then we define $F_{\theta, \module}^{\Comm}(x)$ to be the corresponding element in 
 \begin{align*}
 I_{\widetilde{\pi}_\dagg}\underset{\Delta^1}{\times}\{0\}\simeq N(\Fin_{*,\dagg})\times (\Delta^{\{W,M\}}\coprod\limits_{\Delta^{\{W\}}}\Delta^{\{W,N\}}), 
 \end{align*}
 that makes the marking in $\lng n_x\rng$ to be $\dagg$. It is easy to see that these extend to a well defined functor 
$F_{\theta, \module}^{\Comm}$ (\ref{eq: F_theta,module}). 

The right Kan extension of the composition  $MWN\circ F_{\theta, \module}^{\Comm}$ (\ref{eq: MWN}) along 
\begin{align*}
p_{\cT^\Comm_+}: I_{N(\Fin_*)_{\langle 1\rangle/}\Gray \Delta^1}\underset{I_+^\Comm}{\times} \cT_+^\Comm\rightarrow I_{N(\Fin_*)_{\langle 1\rangle/}\Gray \Delta^1}\underset{I^\Comm}{\times} \cT^\Comm
\end{align*}
gives a functor between coCartesian fibrations
\begin{align*}
F_{MWN}^\Comm: I^\natural_{N(\Fin_*)_{\langle 1\rangle/}\Gray \Delta^1}\underset{I^\Comm}{\times} (\cT^\Comm)^\natural\rightarrow (\cC\times I^\Comm)^\natural.
\end{align*}
over $I^\Comm$. 

\begin{prop}\label{prop: right-lax module}
Assume we are given a commutative diagram (\ref{diagram: corr of modules}) which exhibits $M^{\bullet,\dagg}$, $W^{\bullet,\dagg}$ and $N^{\bullet,\dagg}$ as modules over the commutative algebra objects $C^\bullet$, $P^\bullet$ and $D^\bullet$ in $\bCorr(\cC^\times)$, respectively. Assume that the correspondence $D^\bullet\leftarrow P^\bullet\rightarrow C^\bullet$ determines a right-lax algebra homomorphism from the commutative algebra $C^\bullet$ to $D^\bullet$ in the sense of Theorem \ref{thm: right-lax}. Then the correspondence  (\ref{diagram: corr of modules}) of the pairs naturally determines a right-lax  
morphism from the pair $(C^\bullet, M^{\bullet,\dagg})$ to $(D^\bullet, N^{\bullet,\dagg})$ if  the diagram
\begin{align}\label{diagram: W, N, P, D}
&\xymatrix{W^{\lng n\rng_\dagg}\ar[r]\ar[d]&\prod\limits_{j\in \langle n\rangle^\circ} P^{\{j,*\}}\times W^{\{\dagg, *\}}\ar[d]\\
N^{\langle n\rangle_\dagg}\ar[r]&\prod\limits_{j\in \langle n\rangle^\circ} D^{\{j,*\}}\times N^{\{\dagg, *\}}
}
\end{align}
is Cartesian for every $\lng n\rng_\dagg$ in $\Fin_{*,\dagg}$. If in addition the correspondence $D^\bullet\leftarrow P^\bullet\rightarrow C^\bullet$ determines an algebra homomorphism and the following diagram
\begin{align}\label{eq: prop W, M, active}
\xymatrix{
W^{\lng 1\rng_\dagg}\ar[r]\ar[d]&M^{\langle 1\rangle_\dagg}\ar[d]\\
W^{\lng 0\rng_\dagg}\ar[r]&M^{\langle 0\rangle_\dagg}
}
\end{align}
is Cartesian for the active morphism $f: \langle 1\rangle_\dagg\rightarrow \langle 0\rangle_\dagg$ in $N(\Fin_{*,\dagg})$, i.e. $f^{-1}(*)=*$, 
then the same correspondence naturally determines a morphism between the pairs.
\end{prop}
\begin{proof}
The proof is almost the same as the proof of Theorem \ref{thm: right-lax}, in which we only need to check the Cartesian property of the (commutative version of the) diagrams (\ref{thm: diagram Cart 1}), (\ref{thm: diagram Cart 2}) and (\ref{diagram: W, C Cartesian}) in the \emph{marked} case. These are exactly the conditions listed in the proposition. 
\end{proof}

Proposition \ref{prop: right-lax module} immediately implies the following. 
\begin{cor}\label{cor: right-lax module C}
Assume we are given a commutative algebra object $C$ in $\bCorr(\cC^\times)$, determined by a $\Fin_*$-object $C^\bullet$ in $\cC$, and a correspondence of $\Fin_{*,\dagg}$-objects 
\begin{align*}
\xymatrix{
N^{\bullet,\dagg}\ar[dr]&\ar[l] W^{\bullet,\dagg}\ar[d]\ar[r] &M^{\bullet,\dagg}\ar[dl]\\
&C^\bullet\circ\pi_\dagg&
}
\end{align*}
in $\cC$ as modules over $C$ in $\bCorr(\cC^\times)$. Then the correspondence naturally determines a right-lax $C$-module morphism from $M^{\bullet,\dagg}$ to $N^{\bullet,\dagg}$ if the diagram
\begin{align*}
&\xymatrix{W^{\lng n\rng_\dagg}\ar[r]\ar[d]&\prod\limits_{j\in \langle n\rangle^\circ}  C^{\{j,*\}}\times W^{\{\dagg, *\}}\ar[d]\\
N^{\langle n\rangle_\dagg}\ar[r]&\prod\limits_{j\in \langle n\rangle^\circ} C^{\{j,*\}}\times N^{\{\dagg, *\}}
}
\end{align*}
is Cartesian for every $\lng n\rng_\dagg$ in $\Fin_{*,\dagg}$. If in addition the following diagram
\begin{align*}
\xymatrix{
W^{\lng 1\rng_\dagg}\ar[r]\ar[d]&M^{\lng 1\rng_\dagg}\ar[d]\\
W^{\langle 0\rangle_\dagg}\ar[r]&M^{\langle 0\rangle_\dagg}
}
\end{align*}
is Cartesian for the \emph{active} $f: \langle 1\rangle_\dagg\rightarrow \langle 0\rangle_\dagg$ in $\Fin_{*,\dagg}$,
then the same correspondence naturally determines a $C$-module morphism.
\end{cor}

\begin{remark}\label{remark: composition of right-lax Mod}
As one would expect, the composition of two (right-lax) morphisms between pairs given by two correspondences 
\begin{align*}
&\xymatrix{N^{\bullet,\dagg}\ar[d]&\ar[l] W^{\bullet,\dagg}\ar[r]\ar[d]&M^{\bullet,\dagg}\ar[d]\\
D^\bullet\circ \pi_\dagg\ar[r]&P^\bullet\circ \pi_\dagg\ar[r]&C^\bullet\circ \pi_\dagg
},\\
&\xymatrix{R^{\bullet,\dagg}\ar[d]&\ar[l] Y^{\bullet,\dagg}\ar[r]\ar[d]&N^{\bullet,\dagg}\ar[d]\\
E^\bullet\circ \pi_\dagg&H^\bullet\circ \pi_\dagg\ar[r]\ar[l]&D^\bullet\circ \pi_\dagg
}
\end{align*}
is homotopic to the one given by the correspondence
\begin{align*}
\xymatrix{R^{\bullet,\dagg}\ar[d]&\ar[l] W^{\bullet,\dagg}\underset{N^{\bullet,\dagg}}{\times}Y^{\bullet,\dagg}\ar[r]\ar[d]&N^{\bullet,\dagg}\ar[d]\\
E^\bullet\circ \pi_\dagg&(P^{\bullet}\underset{D^\bullet}{\times}H^\bullet)\circ \pi_\dagg\ar[r]\ar[l]&C^\bullet\circ \pi_\dagg
}
\end{align*}
\end{remark}

Let $\Fun_{\inert}(I_{\pi_\dagg}, \cC)$ (resp. $\Fun_{\act}(I_{\pi_\dagg},\cC)$) be the 1-full subcategory of $\Fun(I_{\pi_\dagg},\cC)$ consisting of objects satisfying the condition in Theorem \ref{thm: module} and whose morphisms from $W^{\bullet,\dagg}\rightarrow P^{\bullet}\circ\pi_\dagg$ to $N^{\bullet,\dagg}\rightarrow D^{\bullet}\circ\pi_\dagg$  (resp. $W^{\bullet,\dagg}\rightarrow P^{\bullet}\circ\pi_\dagg$ to $M^{\bullet,\dagg}\rightarrow C^{\bullet}\circ\pi_\dagg$) satisfy that the diagrams (\ref{diagram: W, N, P, D}) (resp. (\ref{eq: prop W, M, active})) are Cartesian for every  $\lng n\rng_\dagg\in \Fin_{*,\dagg}$. One can show with the same argument as in Proposition \ref{prop: Fun_Cart} for the following. 
\begin{prop}\label{prop: Fun_Cart, Mod}
There are natural functors
\begin{align*}
&\Fun_{\inert}(I_{\pi_\dagg}, \cC)\rightarrow \cMod^{N(\Fin_*)}(\Corr(\cC^\times)),\\
&\Fun_{\act}(I_{\pi_\dagg},\cC)^{op}\rightarrow \cMod^{N(\Fin_*)}(\Corr(\cC^\times)). 
\end{align*}
\end{prop}

\begin{prop}\label{prop: K times L Mod}
Let $K,L$ be two $\infty$-categories. Assume $F: K\times L^{op}\rightarrow \Fun(I_{\pi_\dagg}, \cC)$ is a functor satisfying
\begin{itemize}
\item[(i)] For any  $\ell\in L$ (resp. $k\in K$), $F|_{K\times\{\ell\}}$ (resp. $F|_{\{k\}\times L^{op}}$) factors through the 1-full subcategory $\Fun_{\inert}(I_{\pi_\dagg}, \cC)$ (resp. $\Fun_{\act}(I_{\pi_\dagg}, \cC)$);
\item[(ii)] The induced functor $ev_{\lng 1\rng}\circ F: K\times L^{op}\rightarrow \cC$ (resp. $ev_{\lng 0\rng_\dagg}\circ F$) satisfies that for any square $\alpha: [1]\times [1]^{op}\rightarrow K\times L^{op}$ determined by a functor $\alpha\in \Fun([1],K)\times \Fun([1]^{op},L^{op})$, the diagram $ev_{\lng 1\rng}\circ F\circ \alpha: [1]\times [1]^{op}\rightarrow \cC$ (resp. $ev_{\lng 0\rng_\dagg}\circ F\circ \alpha$) is a Cartesian square,
\end{itemize}
then $F$ naturally determines a functor $F_{\cMod}: K\times L\rightarrow \cMod^{N(\Fin_*)}(\Corr(\cC^\times))$. 
\end{prop}

\begin{remark}
There is a complete analogue of Proposition \ref{prop: right-lax module}, Corollary \ref{cor: right-lax module C}, Proposition \ref{prop: Fun_Cart, Mod}, Proposition \ref{prop: K times L Mod} for the associative case. We leave the details to the reader. 
\end{remark}

\section{The functor $\ShvSp$ out of the category of correspondences}\label{sec: ShvSp}

In this section, we will define the functor 
\begin{align*}
\ShvSp: &\Corr(\Slch)\rightarrow \PrstL\\
&X\mapsto \Shv(X;\Sp)
\end{align*}
from Theorem \ref{thm: ShvSp} and give a proof of the theorem. Moreover, we will show that $\ShvSp$ is canonically symmetric monoidal. These follow from a direct application of the approach in \cite{Nick} to define the functor of taking ind-coherent sheaves out of the category of correspondences of schemes, and the application is based on the foundations of $\infty$-topoi (\cite{higher-topoi}) and stable $\infty$-categories (\cite{higher-algebra}). We have provided the references for each step explicitly.

\subsection{The functor $\ShvSp_!: \Slch\rightarrow \PrstL$ }
We recall some basic definitions and useful facts about the $\infty$-category of sheaves on a topological space with values in an (stable) $\infty$-category from \cite{higher-topoi} and \cite{higher-algebra}. For any topological space $X$, let $\cU(X)$ be the poset of open subsets in $X$. Let $\cC$ be any $\infty$-category that admits small limits. Let $\cP(X;\cC)=\Fun(N(\cU(X))^{op}, \cC)$ be the $\infty$-category of $\cC$-valued presheaves on $X$. For any subset $S\subset X$, let $\cF(S)=\varinjlim_{S\subset U}\cF(U)$. 

\begin{definition}
A \emph{$\cC$-valued sheaf} on $X$ is an object $\cF$ in $\cP(X;\cC)$ satisfying that if $\{U_i\subset U\}_{i\in I}$ is an open cover of $U$, then the natural morphism
\begin{align*}
\cF(U)\rightarrow \varprojlim\limits_{\emptyset\neq J\subset I}\cF(\bigcap\limits_{j\in J}U_j)
\end{align*}
is an isomorphism in $\cC$. We denote the full subcategory of $\cP(X;\cC)$ consisting of $\cC$-valued sheaves by $\Shv(X;\cC)$.
\end{definition}

Any continuous map $f: X\rightarrow Y$ induces a functor $f^{-1}: N(\cU(Y))^{op}\rightarrow N(\cU(X))^{op}$, which gives rise to the functor $f_*: \cP(X;\cC)\rightarrow \cP(Y;\cC)$. It is straightforward from the definition that $f_*$ sends a sheaf to a sheaf. So we have a well defined functor $\Shv\cC: \STop\rightarrow \OneCat$ from the ordinary 1-category of topological spaces to the $\infty$-category of $\infty$-categories.

Now following notations in \cite{higher-topoi}, let $\Shv(X)$ (resp. $\cP(X)$) denote the $\infty$-category of $\Spc$-valued sheaves (resp. presheaves) on $X$. The functor $f_*$ admits a left adjoint $f^*$.

In particular, we see that  $f^*$ preserves small colimits and finite limits, and $f_*$ preserves small limits and finite colimits, using the adjunction and 
\begin{itemize}
\item $f^*$ on presheaves preserves finite limits and $f_*$ on presheaves preserves small colimits; 
\item sheafification commutes with finite limits and small colimits (\cite[Proposition 6.2.2.7]{higher-topoi});
\item the inclusion $\Shv(X)\hookrightarrow \cP(X)$ preserves small limits and finite colimits. 
\end{itemize}

When $f$ is an open embedding, $f^*$ preserves small limits, so it admits a left adjoint denoted by $f_!$. When $f$ is proper, $f_*$ preserves small colimits. In the following, we will focus on $\Shv(X;\Sp)$, the stable $\infty$-category of sheaves of spectra on $X$, which can be obtained by taking stabilization of $\Shv(X)$ (cf. \cite[\S1.4.2]{higher-algebra}).

Let $\Slch$ be the ordinary 1-category of locally compact Hausdorff spaces. There are two ways to define $f_!: \Shv(X;\Sp)\rightarrow \Shv(Y;\Sp)$ for any morphism $f:X\rightarrow Y$ in $\Slch$. One is to use the fact that $f_!$ is the left (resp. right) adjoint of $f^*$ for $f$ an open embedding (resp. proper), and decompose every morphism into the composition of an open embedding followed by a proper map. The construction needs to invoke the machinery developed in \cite{Nick}. The other way is to use the self-duality of $\Shv(X;\Sp)$ as a dualizable object in $\PrstL$ proved in \cite[Proposition 21.1.6.12, Proposition 21.1.7.1, Appendix D.7.3]{SAG} 
and $f_!$ is just the dual of $f^*$, and we also get $f^!$ directly as the right adjoint of $f_!$. We will adopt the latter definition, for it gives us a convenient way to follow the steps in \cite{Nick} to define the functor $\ShvSp$ of taking sheaves of spectra out of the category of correspondences in $\Slch$. We will denote the resulting functor as
\begin{align}\label{eq: ShvSp_!}
\ShvSp_!: & \Slch\rightarrow \bPrstL\\
\nonumber&X\mapsto \Shv(X; \Sp),\\
\nonumber& (f: X\rightarrow Y)\mapsto (f_!: \Shv(X;\Sp)\rightarrow \Shv(Y;\Sp)).
\end{align}
We state a lemma which will be used soon. 
\begin{lemma}\label{lemma: base change}
Let 
\begin{align*}
\xymatrix{W\ar[r]^{q}\ar[d]_{p} &X\ar[d]^{p'}\\
Y\ar[r]^{q'} &Z
}
\end{align*}
be any Cartesian diagram of locally compact Hausdorff spaces. 
\item[(i)]If  $q'$ is an open embedding (then so is $q$), then the natural transformation
\begin{align*}
p_!q^*\rightarrow (q')^*(p')_!: \Shv(X;\Sp)\rightarrow \Shv(Y;\Sp)
\end{align*}
arising from applying adjunctions to the natural isomorphism
\begin{align*}
(q')_!p_!\simeq p'_!q_!
\end{align*}
is an isomorphism of functors. 
\item[(ii)] If  $q'$ is proper (then so is $q$), then the natural transformation
\begin{align*}
(q')^*(p')_!\rightarrow p_!q^*: \Shv(X;\Sp)\rightarrow \Shv(Y;\Sp)
\end{align*}
arising from applying adjunctions to the natural isomorphism
\begin{align*}
(q')_*p_!\simeq p'_!q_*
\end{align*}
is an isomorphism of functors. 
\end{lemma}
\begin{proof}
\item[(i)]  Since $q'$ and $q$ are open embeddings, this is clear from the definition of the functors. 
\item[(ii)] Applying the self-duality on the sheaf categories, it suffices to prove that the natural transformation
\begin{align*}
&(p')^*q'_!\rightarrow q_!p^*
\end{align*}
is an isomorphism of functors. 
Since $q'_!\simeq q'_*, q_!\simeq q_*$, this is just the stable version of the Nonabelian Proper Base Change Theorem \cite[Corollary 7.3.1.18]{higher-topoi}. 
\end{proof}

\subsection{Definition of $\ShvSp_{\all,\all}^{\propmap}: \Corr(\Slch)_{\all,\all}^{\propmap}\rightarrow (\bPrstL)^{\twoop}$}

We follow the steps in \cite[Chapter 5, Section 2]{Nick} to define the canonical functor
\begin{align*}
\ShvSp_{\all,\all}^{\propmap}: \Corr(\Slch)_{\all,\all}^{\propmap}\rightarrow (\bPrstL)^{\twoop},
\end{align*}
where the superscript $\twoop$ means reversing the 2-morphisms. Moreover, we will show that $\ShvSp_{\all,\all}^{\propmap}$ is canonically symmetric monoidal. 

First, let 
\begin{align*}
vert=horiz=\all,\ adm=\propmap, \ co\text{-}adm=\open,
\end{align*}
which satisfy the conditions in \cite[Chapter 7, 5.1.1-5.1.2]{Nick}. 
We start from the functor (\ref{eq: ShvSp_!})
\begin{align*}
\ShvSp_!:& \Slch\rightarrow \bPrstL
\end{align*}
First, by Lemma \ref{lemma: base change} (i), $\ShvSp_!$ satisfies the left Beck-Chevalley condition (cf. \cite[Chapter 7, Definition 3.1.2]{Nick}) with respect to open maps, therefore by Theorem 3.2.2 in \emph{loc. cit.}, it uniquely determines a functor
\begin{align*}
\ShvSp_{\all,\open}^{\open}: \Corr(\Slch)_{\all,\open}^{\open}\rightarrow \bPrstL.
\end{align*}
Now by Theorem 4.1.3 in \emph{loc. cit.}, the restriction of $\ShvSp_{\all,\open}^{\open}$ to 
\begin{align*}
\ShvSp_{\all,\open}^{\isom}: \Corr(\Slch)_{\all,\open}^{\isom}\rightarrow \PrstL
\end{align*}
loses no information. 

Next, we apply \cite[Chapter 7, Theorem 5.2.4]{Nick} to give the canonical extension of $\ShvSp_{\all,\open}^{\isom}$ to 
\begin{align*}
\ShvSp_{\all, \all}^{\propmap}: \Corr(\Slch)_{\all,\all}^{\propmap}\rightarrow (\bPrstL)^{\twoop}.
\end{align*}
Note that the restriction $\ShvSp_{\all,\open}^{\isom}|_{(\Slch)_{vert}}$ satisfies the left Beck-Chevalley condition for $adm=\propmap\subset vert=\all$ if we view the target as $(\bPrstL)^{\twoop}$. This follows from Lemma \ref{lemma: base change} (ii). Hence to apply the theorem,  we only need to check the conditions in Chapter 7, 5.1.4 and  5.2.2 in \emph{loc. cit.}. 

The condition in 5.1.4 in our setting says the following. For any $\alpha: X\rightarrow Y$ in $horiz=\all$, consider the ordinary 1-category $\Factor(\alpha)$ of factorizations of $\alpha$ into an open embedding followed by a proper map: 
\begin{itemize}
\item[(i)] An object in $\Factor(\alpha)$ is a sequence of maps 
\begin{align*}
X\overset{f}{\rightarrow }Z\overset{g}{\rightarrow }Y
\end{align*}
such that $g\circ f=\alpha$, $f\in \open$ and $g\in \propmap$. 
\item[(ii)] A morphism from 
\begin{align*}
X\overset{f}{\rightarrow }Z\overset{g}{\rightarrow }Y
\end{align*}
to 
\begin{align*}
X\overset{f'}{\rightarrow }Z'\overset{g'}{\rightarrow }Y
\end{align*}
is a morphism $\beta: Z\rightarrow Z'$ that makes the following diagram commutes
\begin{align*}
\xymatrix{&Z\ar[dr]^{g}\ar[dd]^{\beta}&\\
X\ar[ur]^{f}\ar[dr]_{f'}&&Y\\
&Z'\ar[ur]_{g'}&
}.
\end{align*}
\end{itemize}
We need to show that $\Factor(\alpha)$ is contractible. Recall that a 1-category is \emph{contractible} if the geometric realization of its nerve is contractible.  
\begin{lemma}
For any $\alpha$, $\Factor(\alpha)$ is contractible.
\end{lemma}
\begin{proof}
The proof follows the same line as the proof of \cite[Chapter 5, Proposition 2.1.6]{Nick}, but significantly simpler than the scheme setting there. First, one shows that $\Factor(\alpha)$ is nonempty. Choose a compactification $\overline{X}$ of $X$, say the one-point compactification,
 and let 
\begin{align*}
Z=\overline{\Graph(\alpha)}\subset \overline{X}\times Y.
\end{align*} 
Then the natural embedding $X\hookrightarrow Z$ is open and the projection from $Z$ to $Y$ is proper, and their composition is $\alpha$. 

Let $\Factor(\alpha)_{\dense}$ be the full subcategory of $\Factor(\alpha)$ whose objects are factorizations $X\overset{f}{\rightarrow} Z\overset{g}{\rightarrow} Y$ with $\overline{f(X)}=Z$. Because of the transitivity of taking closures for any sequence of maps $A\overset{u}{\rightarrow} B\overset{v }{\rightarrow} C$, i.e. $\overline{v\circ u(A)}=\overline{v(\overline{u(A)})}$. The inclusion $\Factor(\alpha)_{\dense}\hookrightarrow \Factor(\alpha)$ admits a right adjoint, so one just need to show that  $\Factor(\alpha)_{\dense}$ is contractible. 

Lastly, one shows that $\Factor(\alpha)_{\dense}$ admits products, so then we have a functor $\Factor(\alpha)_{\dense}\times \Factor(\alpha)_{\dense}\rightarrow \Factor(\alpha)_{\dense}$ that admits a left adjoint, and this would imply $\Factor(\alpha)_{\dense}$ is contractible. For any two factorizations in $\Factor(\alpha)_{\dense}$,
\begin{align}\label{eq: product of two}
X\overset{f}{\rightarrow }Z\overset{g}{\rightarrow }Y,\ 
X\overset{f'}{\rightarrow }Z'\overset{g'}{\rightarrow }Y,
\end{align}
let $W$ be the closure of the image of $X$ in $Z\underset{Y}{\times }Z'$ under $(f,f')$. One can easily check that the factorization $X\rightarrow W\rightarrow Y$ serves as the product of the two in (\ref{eq: product of two}).
\end{proof}

Next, we check the condition of Chapter 7, 5.2.2 in \emph{loc. cit.}, which says that for any Cartesian diagram
\begin{align*}
\xymatrix{X_1\ar[r]^{j_X}\ar[d]_{p_1}&X_2\ar[d]^{p_2}\\
Y_1\ar[r]^{j_Y}&Y_2
}
\end{align*}
with $j_X,j_Y\in \open$ and $p_1,p_2\in \propmap$, the natural transformation
\begin{align*}
&p_1^*j_Y^!\rightarrow j_X^!p_2^*
\end{align*}
arising from applying adjunctions to the isomorphism from base change
\begin{align*}
 j_Y^!(p_2)_*\simeq (p_1)_*j_X^!
 \end{align*}
needs to be an isomorphism. Since $j_X^!\simeq j_X^*, j_Y^!\simeq j_Y^*$, the condition obviously holds. 

Having finished the check of the conditions for applying \cite[Chapter 7, Theorem 5.2.4]{Nick}, we now obtain the canonical functor 
\begin{align*}
\ShvSp_{\all, \all}^{\propmap}: \Corr(\Slch)_{\all,\all}^{\propmap}\rightarrow (\bPrstL)^{\twoop}.
\end{align*}

\subsection{Symmetric monoidal structure on $\ShvSp_{\all,\all}^{\propmap}$}
\subsubsection{The symmetric monoidal structure on $\ShvSp_!: \Slch\rightarrow \PrstL$.}
Let 
\begin{align*}
\Shv^*: &\Slch^{op}\rightarrow \PrL\\
&X\mapsto \Shv(X)\\
&(f:X\rightarrow Y)\mapsto (f^*: \Shv(Y)\rightarrow \Shv(X))
\end{align*}
be the functor associated to $*$-pullback.  It is proven in \cite{higher-topoi} that the functor $\Shv^*$ is symmetric monoidal (\cite[Proposition 7.3.1.11]{higher-topoi}). Since $\ShvSp_!$ is the composition of $\Shv^*$ with the functor $\PrL\rightarrow \PrstL$ of taking stabilizations, followed by the functor of taking the dual category on the full subcategory of dualizable objects in $\PrstL$,  we have 
\begin{align*}
\ShvSp_!: \Slch\rightarrow \PrstL
\end{align*}
is symmetric monoidal as well. 

Now applying the same argument as in \cite[Chapter 9, Proposition 3.1.2 and 3.1.5]{Nick}, we get that the symmetric monoidal structure on $\ShvSp_!$ uniquely extends to a symmetric monoidal structure on $\ShvSp_{\all,\open}^{\isom}$, which further uniquely extends to a symmetric monoidal structure on $\ShvSp_{\all, \all}^{\propmap}$. In summary, we have proved the following:

\begin{thm}\label{thm: ShvSp symm. monoidal}
 The functor 
\begin{align}\label{eq: functor ShvSp prop}
\ShvSp_{\all, \all}^{\propmap}: \Corr(\Slch)_{\all,\all}^{\propmap}\rightarrow (\bPrstL)^{\twoop}
\end{align}
is equipped with a canonical symmetric monoidal structure. 
\end{thm}

\begin{proof}[Proof of Theorem \ref{thm: ShvSp}]
$\ShvSp^*_!$ is defined as the restriction of $\ShvSp_{\all, \all}^{\propmap}$ to $adm=\isom$. 
\end{proof}

By an entire similar process, one gets a canonical symmetric monoidal functor $\Corr(\Slch)_{\all,\all}^{\propmap}\rightarrow \bPrstR$, by taking $!$-pullback and $*$-pushforward. Alternatively, by taking $\oneop$ on both sides of (\ref{eq: functor ShvSp prop}), using that $(\bPrstL)^{\onetwoop}\simeq \bPrstR$ and $(\Corr(\Slch)_{\all,\all}^{\propmap})^{\oneop}\simeq \Corr(\Slch)_{\all,\all}^{\propmap}$, we get $\ShvSp^!_*$ as well. Restricting $adm=\isom$, this is exactly Corollary \ref{cor: ShvSp!*}.

\subsubsection{Symmetric monoidal structure on taking local system categories out of correspondences}
Now we restrict our vertical arrows to locally trivial fibrations, denoted by $\fib$, and consider the functor induced from $\ShvSp_{\all,\all}^{\propmap}$ by restricting the image to the full subcategory of locally constant sheaves:  
\begin{align}\label{eq: Loc,!,*}
(\Loc')_{!}^*: &\Corr(\Slch)_{\fib,\all}\rightarrow \PrstL\\
\nonumber&X\mapsto \Loc'(X;\Sp).
\end{align}
It is a symmetric monoidal functor, since $\boxtimes: \Loc'(X;\Sp)\otimes \Loc'(Y;\Sp)\to \Loc'(X\times Y;\Sp)$ is an equivalence. 

On the other hand, we have the symmetric monoidal functor $\Loc_*^!$ (\ref{eq: Loc_*up!}) already defined in Subsection \ref{subsec: intro correspondences}. 

Later, we also use the natural symmetric monoidal functor $\Loc\Sp_!: \Spc\rightarrow \PrstL$ that takes a CW-complex $X$ to $\Loc(X;\Sp)$ and $f: X\rightarrow Y$ to $f_!$ (with right adjoint $f^!$). Here again, we regard a local system on $X$ as a locally constant cosheaf, through the equivalence $\Loc(X, \Sp)\simeq \Fun(X; \Sp)$ (cf. Remark \ref{remark:locconstcoshv}). When $X\in \Slch$, the constant cosheaf with fiber $\bS$ is exactly the dualizing sheaf $\varpi_X:=f^!\bS_{pt}$, for the map $f: X\rightarrow\pt$. Then $f_!\varpi_X$ is the homology of $X$ (with coefficient in $\bS$), i.e. $\Sigma_+^\infty X$, while $f_*\varpi_X$ is the Borel-Moore homology of $X$, i.e. $\Sigma^\infty X^c$, where $X^c$ is the one-point compactification of $X$ as a pointed space. For example, if $X=V$ is a finite dimensional vector space, then $f_!$ and $f^!$ are inverse equivalences between $\Loc(V;\Sp)$ and $\Loc(pt;\Sp)\simeq \Sp$, so $f_!\varpi_V\simeq f_!f^!\bS\simeq \bS$. On the other hand, let $j: V\hookrightarrow V^c$ and $i: \{x\}\hookrightarrow V^c$ be the open and closed inclusions, where $\{x\}$ is the complement of $V$ in $V^c$. Let $\overline{f}: V^c\to pt$ be the obvious map.  Using the standard fiber sequence 
\begin{align*}
i_*\bS_{x}\to \varpi_{V^c}\to j_*\varpi_{V}
\end{align*} 
in $\Shv(V^c;\Sp)$ 
and applying $\overline{f}_*\simeq \overline{f}_!$, we get a fiber sequence in $\Sp$, 
\begin{align*}
\bS\to \overline{f}_!\varpi_{V^c}\to f_*\varpi_V
\end{align*}
which says that $f_*\varpi_V$ is exactly calculating the reduced homology of $V^c$, i.e. $\Sigma^\infty V^c$.

\section{Quantization of the Hamiltonian $\coprod\limits_nBO(n)$-action and the $J$-homomorphism}\label{sec: quantization}
In this section, we use sheaves of spectra to quantize the Hamiltonian $\coprod\limits_nBO(n)$-action and  its ``module" generated by the stabilization of $L_0=\Graph(d(-\frac{1}{2}|\bq|^2))$ in $T^*\bR^M$.  This is  an application of the various results that we have developed in Section \ref{section: background} concerning (commutative) algebra/module objects in $\Corr(\cC)$ and the morphisms between them, for the case $\cC=\Slch$. We first list the relevant algebra/module objects in $\Corr(\Slch)_{\fib,\all}$ and morphisms between them, and then we give a model of the $J$-homomorphism based on correspondences, and the quantization result will be an immediate consequence of these.

\subsection{The relevant algebra/module objects in $\Corr(\Slch)_{\fib,\all}$}
Recall that we let $A$ denote a quadratic form on $\bR^N$ whose symmetric matrix relative to the standard basis of $\bR^N$ is idempotent. We equally view $A$ as the eigenspace of eigenvalue $1$ of its  symmetric matrix and in this way $A$ is also regarded as an element of $\coprod\limits_nBO(n)$. 

We first define the $\Fin_*$-object $G_N^\bullet$ in $\Slch$. Let
\begin{align*}
G_N^{\lng n\rng}=\{(A_1,\cdots,A_n): A_i\subset \bR^N, A_j\perp A_k\text{ for }k\neq j\}.
\end{align*}
Here $G_N^{\lng 0\rng}=\pt$. 
For any $f: \lng n \rng\rightarrow\lng m\rng$ in $N(\Fin_*)$, define the associated morphism
\begin{align*}
G_N^{\lng n\rng}&\rightarrow G_N^{\lng m\rng}\\
(A_i)_{i\in\lng n\rng^\circ}&\mapsto (\bigoplus\limits_{i\in f^{-1}(j)}A_i)_{j\in\lng m\rng^\circ}. 
\end{align*}
Here we take the convention that if $ f^{-1}(j)=\emptyset$, then $\bigoplus\limits_{i\in f^{-1}(j)}A_i=0$. It is easy to see that $G_N^\bullet$ is a well defined $\Fin_*$-object in $\Slch$. 

\begin{lemma}
The $\Fin_*$-object $G_N^{\bullet}$ represents a commutative algebra object in\\
 $\Corr(\Slch)_{\fib,\all}$ in the sense of Theorem \ref{thm: algebra objs}. 
\end{lemma}
\begin{proof}
By Theorem \ref{thm: algebra objs}, we just need to check that every diagram (\ref{thm: diagram Cartesian 2}) is Cartesian and the vertical maps are in $\fib$, for any active morphism $f: \lng m\rng\rightarrow\lng n\rng$, but this is straightforward. 
\end{proof}
\begin{remark}
We can also use a simplicial object $G_N^\bullet$ to represent the associative algebra structure on $G_N$ in $\Corr(\Slch)_{\fib,\all}$. The $\Fin_*$-object  (resp. simplicial object) $G_N^{\bullet}$ is not a commutative Segal object (resp. Segal object) in $\Slch$, but it represents a commutative (resp. associative) algebra object in $\Corr(\Slch)_{\fib,\all}$. 
\end{remark}

Similarly to $G_N^\bullet$, let 
\begin{align*}
\widehat{G}_{N,M}^{\langle n\rangle}=&\{(A_1,\cdots, A_n; \bq,\bq+\sum\limits_{i=1}^n \partial_{\bp} A_i(\bp_i); t=\sum\limits_{i=1}^n A_i(\bp_i)): \bp_i\in \bR^M\}\\
&\subset G_N^{\langle n\rangle}\times \bR^M\times \bR^M\times \bR
\end{align*}
Equivalently, we can denote an element in $\widehat{G}_{N,M}^{\langle n\rangle}$ by
\begin{align*}
(A_1,\cdots,A_n;\bp_1,\cdots,\bp_n;\bq,t), \bp_i\in A_i, t=\sum\limits_{i=1}^n|\bp_i|^2. 
\end{align*}
Note that $\varinjlim\limits_{N,M}\widehat{G}_{N,M}^{\lng 1\rng}$ is the front projection of the (stabilized) conic Lagrangian lifting of the graph of the Hamiltonian action of $\coprod\limits_nBO(n)$ (see the beginning of Subsection \ref{subsec: Morse transformations} for the definition of conic Lagrangian liftings). 
For any $f: \langle n\rangle\rightarrow \langle m\rangle$ in $N(\Fin_*)$, we define
\begin{align*}
&\widehat{G}_{N,M}^{\langle n\rangle}\rightarrow \widehat{G}_{N,M}^{\langle m\rangle}\\
&((A_i)_{i\in \lng n\rng^\circ}; \bq,\bq+\sum\limits_{i=1}^n \partial_{\bp} A_i(\bp_i); t=\sum\limits_{i=1}^n A_i(\bp_i)))\mapsto\\
& ((\bigoplus\limits_{i\in f^{-1}(j)}A_i)_{j\in \langle m\rangle^\circ}; \bq,\bq+\sum\limits_{i\in f^{-1}(\langle m\rangle^\circ)}\partial_{\bp}A_i(\bp_i); t=\sum\limits_{i\in f^{-1}(\langle m\rangle^\circ)}A_i(\bp_i)).
\end{align*}
It is direct to check that $\widehat{G}_{N,M}^\bullet$ defines a $\Fin_*$-object in $\Slch$. 
\begin{lemma}
The $\Fin_*$-object $\widehat{G}_{N,M}^\bullet$ defines a commutative algebra object in $\Corr(\Slch)_{\fib,\all}$ in the sense of Theorem \ref{thm: algebra objs}, and the natural projection $\widehat{G}_{N,M}^\bullet\rightarrow G_N^\bullet$, viewed as a correspondence from $\widehat{G}_{N,M}^\bullet$ to $G_N^\bullet$ defines an algebra homomorphism from $\widehat{G}_{N,M}^\bullet$ to $G_N^\bullet$ in $\Corr(\Slch)_{\fib,\all}$ in the sense of Theorem \ref{thm: right-lax}.
\end{lemma}
\begin{proof}
For the first part of the lemma, we just need to check that for any active map $f: \lng m\rng\rightarrow\lng n\rng$ the following diagram 
\begin{align*}
\xymatrix{
\widehat{G}_{N,M}^{\lng m\rng}\ar[r]\ar[d]&\prod\limits_{i\in \langle n\rangle^\circ}\widehat{G}_{N,M}^{f^{-1}(i)\sqcup\{*\}}\ar[d]\\
\widehat{G}_{N,M}^{\langle n\rangle}\ar[r]&\prod\limits_{i\in \langle n\rangle^\circ}\widehat{G}_{N,M}^{\{i,*\}}
}
\end{align*}
is Cartesian in $\Slch$. 
This is easy to see, for the image of the bottom horizontal map consists of $(A_i,\bq, \bq+\partial_\bp A_i(\bp_i), t=A_i(\bp_i))_{1\leq i\leq n}$, where $A_i, 1\leq i\leq n$ are mutually orthogonal, and the right vertical arrow further decomposes each $A_i$ into orthogonal pieces (or makes it 0 if $f^{-1}(i)=\emptyset$).  

For the second part of the lemma, we only need to check the diagram 
\begin{align*}
\xymatrix{
\widehat{G}_{N,M}^{\lng n\rng}\ar[r]\ar[d]&\prod\limits_{j\in\lng n\rng^\circ}\widehat{G}_{N,M}^{\{j,*\}}\ar[d]\\
G_{N}^{\langle n\rangle}\ar[r]&\prod\limits_{j\in\lng n\rng^\circ}G_N^{\{j,*\}}
}
\end{align*}
is Cartesian in $\Slch$ for each $n$. This is also straightforward.  
\end{proof}

Let $VG_{N}^\bullet: N(\Fin_*)\rightarrow \Slch$ be the $\Fin_*$-object that takes $\langle n\rangle$ to the tautological vector bundle on $G_N^{\langle n\rangle}$, and we have an obvious isomorphism
\begin{align*}
VG_{N}^{\langle n\rangle}\cong\{\bq=0\}\subset \widehat{G}_{N,M}^{\lng n\rng}.
\end{align*}
Note that $ \widehat{G}_{N,M}^{\bullet}\cong VG_{N}^\bullet\times (\bR^M)^{\const,\bullet}$, where $(\bR^M)^{\const,\bullet}$ is the constant $\Fin_*$-object mentioned in Remark \ref{rem: const alg} that gives a commutative algebra object in $\Corr(\Slch)$. Since there is a natural equivalence
\begin{align*}
\Loc(\widehat{G}_{N,M};\Sp)^{\otimes_c}\simeq \Loc(VG_N;\Sp)^{\otimes_c},
\end{align*}
where $\otimes_c$ represents the symmetric monoidal convolution structure induced after taking $\Loc_*^!$, in the following quantization process, we will replace $\widehat{G}_{N,M}^{\bullet}$ by $VG_N^\bullet$ without losing any information. 

\begin{lemma}
The correspondence of $\Fin_*$-objects $G_N^\bullet\leftarrow VG_N^\bullet\rightarrow G_N^\bullet$ determines an algebra endomorphism of the commutative algebra object in $\Corr(\Slch)_{\fib,\all}$ corresponding to $G_N^\bullet$. 
\end{lemma}
\begin{proof}
This directly follows from Theorem \ref{thm: right-lax}. 
\end{proof}

\subsubsection{A module object $\widehat{Q}_{N, M}^{\bullet,\dagger}$ of $VG_{N}^\bullet$} \label{subsubsec: VG module widehatQ}
Let $\widehat{Q}_{N,M}^{\langle n\rangle_{\dagger}}$ be the subvariety of $G^{\langle n\rangle}_{N}\times (G_{N}\times \bR^M\times \bR^M\times \bR)$ consisting of points 
\begin{align}\label{eq: point in Q}
(A_1,\cdots, A_n; A, \bq,\bq+\sum\limits_{i=1}^n \partial_{\bp} A_i(\bp_i);s),\ s< -\frac{1}{2}|\bq|^2+A(\bq), A\perp \bigoplus\limits_{i\in \langle n\rangle^\circ}A_i.
\end{align}
Here and after, we use the convention that for a simplicial object (resp. $\Fin_*$-object) $C^\bullet$ in $\cC$ that represents an associative (resp. commutative) algebra object in $\Corr(\cC)$, we regard $C^{[1]}$ (resp. $C^{\lng 1\rng}$) as the underlying object of the algebra object and denote it by $C$. Similar convention applies to module objects. 
For any $f: \langle n\rangle_\dagger\rightarrow \langle m\rangle_\dagger$ in $N(\Fin_{*,\dagger})$, we let 
\begin{align}\label{eq: def Q_M}
&\widehat{Q}_{N,M}^{\langle n\rangle_\dagger}\rightarrow \widehat{Q}_{N,M}^{\langle m\rangle_\dagger}\\
\nonumber&(A_1,\cdots, A_n; A, \bq,\bq+\sum\limits_{i=1}^n \partial_{\bp} A_i(\bp_i); s)\mapsto \\
\nonumber&((\bigoplus\limits_{i\in f^{-1}(j)}A_i)_{j\in \langle m\rangle^\circ}; A+\bigoplus\limits_{i\in f^{-1}(\dagger)\backslash \{\dagger\}}A_i, \bq+\sum\limits_{i\in f^{-1}(\dagger)\backslash \{\dagger\}}\partial_{\bp}A_i(\bp_i), \\
\nonumber&\bq+\sum\limits_{i\in f^{-1}(\langle m\rangle^\circ\cup \dagger)\backslash \dagger}\partial_{\bp}A_i(\bp_i); s+\sum\limits_{i\in f^{-1}\langle \dagger\rangle\backslash\{\dagger\}}A_i(\bp_i)).
\end{align}

In the following, for any $A\in \coprod\limits_nBO(n)$, we use $\bq^A$ (resp. $\bq^{A^\perp}$) to denote the orthogonal projection of $\bq$ onto $A$ (resp. the orthogonal complement of $A$). 
\begin{lemma}\label{lemma: s, A}
\item[(a)]
The map (\ref{eq: def Q_M}) is well defined, i.e. we have
\begin{align}
s+\sum\limits_{i\in f^{-1}(\dagger)\backslash\{\dagger\}}A_i(\bp_i)< &-\frac{1}{2}|\bq+\sum\limits_{i\in f^{-1}(\dagger)\backslash \{\dagger\}}\partial_{\bp}A_i(\bp_i)|^2+A(\bq+\sum\limits_{i\in f^{-1}(\dagger)\backslash \{\dagger\}}\partial_{\bp}A_i(\bp_i))\\
\nonumber&+\sum\limits_{i\in f^{-1}(\dagger)\backslash \{\dagger\}}A_i(\bq+\sum\limits_{j\in f^{-1}(\dagger)\backslash \{\dagger\}}\partial_{\bp}A_j(\bp_i)). 
\end{align}
\item[(b)]
For any active map $f: \langle n\rangle_\dagg\rightarrow\lng 0\rng_\dagg$ in $N(\Fin_{*,\dagg})$, the map 
\begin{align*}
\widehat{Q}_{N,M}^{\lng n\rng_\dagg}\rightarrow \widehat{Q}_{N,M}^{\lng 0\rng_\dagg}
\end{align*}
is in $\fib$. 
\end{lemma}
\begin{proof}
For (a), we only need to show the case for $f: \lng n\rng_\dagg\rightarrow\lng 0\rng_\dagg$ an active map in  $N(\Fin_{*,\dagg})$. We will prove (a) and (b) simultaneously. For any fixed $(\widetilde{A},\widetilde{\bq},\widetilde{s})\in G_N\times \bR^M\times\bR_{\widetilde{s}}$, we are going to solve for $(A_1,\cdots, A_n,\bp_1,\cdots,\bp_n; A,\bq, s)$ satisfying
\begin{align*}
&s+\sum\limits_i A_i(\bp_i)=\widetilde{s},\ A\oplus\bigoplus\limits_iA_i=\widetilde{A},\ \bq+\sum\limits_i2\bp_i=\widetilde{\bq},\\
&s<-\frac{1}{2}|\bq|^2+A(\bq). 
\end{align*}
This amounts to the inequality for $A_i,\bp_i, i=1,\cdots,n$ 
\begin{align}\label{ineq: s, A, p_i}
\nonumber&\widetilde{s}-\sum\limits_iA_i(\bp_i)<-\frac{1}{2}|\widetilde{\bq}-\sum\limits_i 2\bp_i|^2+|\widetilde{\bq}^A|^2\\
\Leftrightarrow&\widetilde{s}+\frac{1}{2}|\widetilde{\bq}|^2-\widetilde{A}(\widetilde{\bq})<-\sum\limits_i|\widetilde{\bq}^{A_i}-\bp_i|^2.
\end{align}
Only when 
\begin{align*}
\widetilde{s}+\frac{1}{2}|\widetilde{\bq}|^2-\widetilde{A}(\widetilde{\bq})<0
\end{align*}
the inequality (\ref{ineq: s, A, p_i}) has a solution.  So (a) is established. The fiber over $(\widetilde{A},\widetilde{\bq},\widetilde{s})\in \widehat{Q}_{N,M}^{\lng 0\rng_\dagg}$ is the union of (open) disc bundles (or a point if $A=\widetilde{A}$) over the partial flag varieties of $\widetilde{A}$ 
\begin{align*}
\{(A_1,\cdots, A_n,A): A\oplus \bigoplus\limits_{i}A_i=\widetilde{A}\},
\end{align*}
so (b) follows. 
\end{proof}

Now it is easy to see that $\widehat{Q}_{N,M}^{\bullet,\dagger}$ is a $\Fin_{*,\dagger}$-object. We have a natural map 
\begin{align*}
&\widehat{Q}_{N,M}^{\langle n\rangle_\dagger}\rightarrow VG_{N}^{\langle n\rangle}\\
&(A_1,\cdots, A_n; A, \bq,\bq+\sum\limits_{i=1}^n \partial_{\bp} A_i(\bp_i);s)\mapsto (A_1,\cdots,A_n; \sum\limits_{i=1}^n \partial_{\bp} A_i(\bp_i); t=\sum\limits_{i=1}^n A_i(\bp_i)),
\end{align*}
for each $n$, that defines a natural transformation $\widehat{Q}_{N,M}^{\bullet,\dagger}\rightarrow VG_{N}^{\bullet}\circ \pi_\dagger$.  
\begin{lemma}\label{lemma: hat Q module}
The natural transformation $\widehat{Q}_{N,M}^{\bullet,\dagger}\rightarrow VG_{N}^{\bullet}\circ \pi_\dagger$ exhibits $\widehat{Q}_{N,M}^{\bullet,\dagger}$ as a module of $VG_{N}^{\bullet}$ in $\Corr(\Slch)_{\fib,\all}$. 
\end{lemma}
\begin{proof}
By Theorem \ref{thm: module}, we just need to show that for any active $f: \langle n\rangle_\dagg\rightarrow\langle m\rangle_\dagg$,  the diagram
\begin{align*}
\xymatrix{\widehat{Q}_{N,M}^{\lng n\rng_\dagg}\ar[d]\ar[r]&(\prod\limits_{k\in  \langle m\rangle^\circ}VG_{N}^{f^{-1}(k)\sqcup\{*\}})\times \widehat{Q}_{N,M}^{f^{-1}(\dagg)\sqcup\{*\}}\ar[d]\\
\widehat{Q}_{N,M}^{\langle m\rangle_\dagg}\ar[r] &(\prod\limits_{k\in \langle m\rangle^\circ}VG_{N}^{\{k,*\}})\times \widehat{Q}_{N,M}^{\{\dagg, *\}}
}
\end{align*}
is Cartesian in $\Slch$. Writing things out explicitly, the bottom horizontal map is 
\begin{align*}
&(A_1,\cdots, A_m; A, \bq,\bq+\sum\limits_{i=1}^m \partial_{\bp} A_i(\bp_i); s)\mapsto\\
&((A_k, \partial_\bp A_k(\bp_k), t_k=A_k(\bp_k))_{k\in \langle m\rangle^\circ}, (A, \bq,\bq; s))
\end{align*}
and the right vertical map is 
\begin{align*}
&((A_{k,j_k})_{j_k\in f^{-1}(k)}; \sum\limits_{j_k\in f^{-1}(k)}\partial_\bp A_{k,j_k}(\bp_{k,j_k}), t_k)_{k\in \langle m\rangle^\circ}, \\
&(A_\alpha)_{\alpha\in f^{-1}(\dagg)\backslash \dagg};
 A, \bq, \bq+\sum\limits_{\alpha\in f^{-1}(\dagg)\backslash \dagg}\partial_{\bp}A_{\alpha}(\bp_\alpha); s)\mapsto\\
&((\bigoplus\limits_{j_k\in f^{-1}(k)}A_{k,j_k}, \sum\limits_{j_k\in f^{-1}(k)}\partial_\bp A_{k,j_k}(\bp_{k,j_k}), t_k)_{k\in \langle m\rangle^\circ}, \\
&A+\sum\limits_{\alpha\in f^{-1}(\dagg)\backslash\dagg}A_\alpha, \bq+\sum\limits_{\alpha\in f^{-1}(\dagg)\backslash \dagg}\partial_{\bp}A_{\alpha}(\bp_\alpha), \bq+\sum\limits_{\alpha\in f^{-1}(\dagg)\backslash \dagg}\partial_{\bp}A_{\alpha}(\bp_\alpha), s+\sum\limits_{\alpha\in f^{-1}(\dagg)\backslash\dagg}A_\alpha(\bp_\alpha)).
\end{align*}
It is then straightforward to see that the diagram is Cartesian. 
\end{proof}

Let 
\begin{align*}
Q^0_M=\{(\bq, s): s<-\frac{1}{2}|\bq|^2\}\subset \bR^M\times \bR. 
\end{align*}
Let $(F_NQ^0_M)^{\bullet,\dagg}$ be the free $VG^\bullet_{N}$-module generated by $Q^0_M$. By the universal property of free modules, a $VG_{N}^\bullet$-module homomorphism from $(FQ^0_M)^{\bullet,\dagg}$ to $\widehat{Q}_{N,M}^{\bullet,\dagg}$ is determined by a morphism $Q^0_M\rightarrow \widehat{Q}_{N,M}^{\langle 0\rangle_\dagg}$ in $\Corr(\Slch)$. Let 
\begin{align}\label{eq: psi_M}
\psi_{N,M}: (F_NQ^0_M)^{\bullet,\dagg}\rightarrow \widehat{Q}_{N,M}^{\bullet,\dagg}
\end{align}
 denote the $VG_{N}^\bullet$-module homomorphism corresponding to the 1-morphism \begin{align}\label{diagram: W, Q^0}
\xymatrix{\{A=0\}\times Q^0_M\ar[r]^{ }\ar[d]&Q^0_M\\
\widehat{Q}_{N,M}^{\langle 0\rangle_\dagg}&
},
\end{align}  
where the horizontal map is the identity map and the vertical map is the natural embedding of the connected component of $\widehat{Q}_{N,M}^{\langle 0\rangle_\dagg}$ defined by $A=0$ (the unit of the commutative algebra $VG_{N}^\bullet$). In particular, on the underlying module we have 
\begin{align}\label{eq: psi_N,M,0}
(F_NQ^0_M)^{\lng 0\rng_\dagg}=VG_N^{\lng 0\rng_\dagg}\times Q_M^0&\longrightarrow \widehat{Q}_{N,M}^{\lng 0\rng_\dagg}\\
\nonumber (A,\partial_{\bp}A(\bp); \bq,s)&\mapsto (A, \bq+\partial_{\bp}A(\bp),  \bq+\partial_{\bp}A(\bp), s+A(\bp)).
\end{align}
Let $\pi_{N,M}: \widehat{Q}^{\lng 0\rng_\dagg}_{N,M}\rightarrow G_N$ be the obvious projection. 

The following statement is an immediate consequence of Lemma \ref{lemma: s, A}.
\begin{lemma} \label{lemma: disc fiber}
The map (\ref{eq: psi_N,M,0}) is a fiber bundle on each component of $\widehat{Q}_{N,M}^{\lng 0\rng_\dagg}$ isomorphic to the pullback bundle $\pi_{N,M}^{-1}VG_N$.  
\end{lemma}

\subsubsection{Stabilization of the local system categories of $G_{N}$, $VG_N$ and $\widehat{Q}_{N,M}$}

\begin{lemma}\label{lemma: inductive Alg, Mod}
\item[(a)] The  inductive system of $\Fin_*$-objects in $\Slch$
\begin{equation*}
\begin{tikzcd}
X_{0,0}^\bullet\ar[r, hook]\ar[d]&\cdots X_{0,N}^\bullet\ar[r, hook] \ar[d]&X_{0,N+1}^\bullet\ar[r, hook]\ar[d]&\cdots\\
X_{1,0}^\bullet\ar[r, hook]&\cdots X_{1,N}^\bullet\ar[r, hook] &X_{1,N+1}^\bullet\ar[r, hook]&\cdots
\end{tikzcd}
\end{equation*}
for $X_0=VG, X_1=G$ 
determines a functor
\begin{align*}
&X: \Delta^1\times \bN^{op}\rightarrow \CAlg(\Corr(\Slch)_{\fib,\all}). 
\end{align*}

\item[(b)] The  inductive system of $I_{\pi_\dagg}$-objects over $(\bN\times \bN)^{\geq \dgnl}$ in $\Slch$
\begin{align*}
&\xymatrix{\cdots\ar[r]&Y_{N,M}^{\bullet,\dagg}\ar[d]\ar[r]&Y_{N',M'}^{\bullet,\dagg}\ar[d]\ar[r]&\cdots\\
\cdots\ar[r]&X_N^\bullet\circ \pi_\dagg\ar[r]&X_{N'}^\bullet\circ \pi_\dagg\ar[r]&\cdots
}
\end{align*}
for $(X,Y)=(G, FQ^0), (VG, FQ^0), (VG, \widehat{Q})$ (where $FQ^0$ means the  corresponding free module generated by $Q^0_M$)
defines a functor
\begin{align*}
&(X_\clubsuit, Y_{\clubsuit, \spadesuit}): ((\bN\times\bN)^{\geq \dgnl})^{op} \rightarrow \cMod^{N(\Fin_*)}(\Corr(\Slch)_{\fib,\all}). 
\end{align*}
\end{lemma}
\begin{proof}
Part (a) follows from Proposition \ref{prop: K times L CAlg}, and part (b) follows from Proposition \ref{prop: Fun_Cart, Mod}. 
\end{proof}

By Lemma \ref{lemma: inductive Alg, Mod}, after applying the symmetric monoidal functor 
\begin{align*}
\Loc_*^!: \Corr(\Slch)_{\fib,\all}\rightarrow \PrstL, 
\end{align*}
we can form the limit in $\CAlg(\PrstL)$:
\begin{align*}
&\Loc(X;\Sp)^{\otimes^!_*}:=\varprojlim\limits_{N}\Loc(X_N;\Sp)^{\otimes^!_*}
\end{align*}
for $X=G, VG.$ 
Similarly, we can form the limit in $\cMod^{N(\Fin_*)}(\PrstL)$
\begin{align*}
&(\Loc(X;\Sp)^{\otimes^!_*}, \Loc(Y;\Sp)):=\varprojlim\limits_{N,M}(\Loc(X_N;\Sp)^{\otimes^!_*},\Loc(Y_{N,M};\Sp))
\end{align*}
for $(X,Y)=(G,FQ^0),\ (VG, FQ^0),\ (VG, \widehat{Q}),$ 

Note that for $X=G, VG$, 
\begin{align*}
&X^\bullet=\varinjlim\limits_N X_N^\bullet
\end{align*}
gives a commutative algebra object in $\Spc$ (here only the homotopy type is relevant), so under the natural symmetric monoidal functor $\Loc\Sp_!: \Spc\rightarrow \PrstL$, the local system category $\Loc(X,\Sp)$ carries a natural symmetric monoidal structure, denoted by $\Loc(X;\Sp)^{\otimes_c}$. Since for any active map $f: \lng m\rng\rightarrow \lng n\rng$ in $N(\Fin_*)$, the morphism
\begin{align*}
&X_N^{\lng m\rng}\rightarrow X_N^{\lng n\rng}
\end{align*}
is a proper locally trivial fibration for both $X=G, VG$,  the symmetric monoidal structure on 
$\Loc(X;\Sp)^{\otimes_*^!}$ agrees with the one on $\Loc(X;\Sp)^{\otimes_c}$. 
Indeed, this is a direct consequence of Proposition \ref{prop appendix: bT_VTopCorr, bT_VSpcCorr}. By restricting to the full subcategories of vector bundles of rank $0$ and composing with the natural functor $(\bPrstL)^{\Sp//}\to \bPrstL$, we get a commutative diagram from (\ref{diagram: prop bT_VTopCorr, bT_VSpcCorr}): 
\begin{equation}
\begin{tikzcd}[column sep=3em]
\Top^{op}\ar[r]\ar[d]& \Corr(\Top)_{\propmap\fib, \all}\ar[d] \ar[r, "\Loc_*^!\simeq \Loc^!_!"]&\bPrstL\\
\Spc^{op}\ar[r]&\Corr(\Spc)_{\hpf',\all}\ar[ur, "\Loc^!_!"'] & \\
\Spc_{\hpf'}\ar[ur]&\ &\ 
\end{tikzcd}. 
\end{equation}
where all functors are symmetric monoidal. Now using \emph{Step 2} in the proof of Proposition \ref{prop: equiv model of J} (from Appendix \ref{Appendix subsec: proof of J}), both $X=G,VG$ give commutative algebra objects in $\Corr(\Top)_{\propfib, \all}$ and $\Corr(\Spc)_{\hpf', \all}$. As objects in the latter, both lie in the essential image of $\CAlg(\Spc_{\hpf'})\to \CAlg(\Corr(\Spc)_{\hpf',\all})$, corresponding to the classical commutative topological monoid structure on $G$ ($VG\simeq G$ in $\CAlg(\Spc^\times)$). Then the above claim about 
\begin{align}\label{eq: Locotimes*!simeqotimesc}
\Loc(X;\Sp)^{\otimes_*^!}\simeq \Loc(X;\Sp)^{\otimes_c}, X=VG, G
\end{align}
 follows.

\begin{prop}\label{prop: C,X;D,Y}
The commutative diagram over $(\bN\times \bN)^{\geq \dgnl}\times \Delta^2$
\begin{align}\label{diagram: prop C,X;D,Y}
\xymatrix{&(F_{N'}Q_{M'}^0)^{\bullet,\dagg}\ar[dd]\ar[rr]^{\psi_{N',M'}}&&\widehat{Q}_{N',M'}^{\bullet,\dagg}\ar@/^2pc/[ddll]\\
{(F_NQ_M^0)}^{\bullet,\dagg}\ar[dd]\ar[rr]^{\psi_{N,M}}\ar[ur]&&\widehat{Q}_{N,M}^{\bullet,\dagg}\ar[ur]\ar@/^2pc/[ddll]&\\
& VG_{ N'}^{\bullet}\circ \pi_\dagg&&\\
VG_{N}^{\bullet}\circ\pi_\dagg\ar[ur]&&&}
\end{align}
in $\Fun(N(\Fin_{*,\dagg}), \Slch)$  determines a functor 
\begin{align*}
F_{\psi}: &((\bN\times \bN)^{\geq\dgnl})^{op}\times \Delta^1\rightarrow \cMod^{N(\Fin_*)}(\Corr(\Slch)_{\fib,\all})\\
 &(N,M; 0)\mapsto  (VG^\bullet_{N}, (F_NQ_M^0)^{\bullet,\dagg})\\
 &(N,M;1)\mapsto (VG^\bullet_{N}, \widehat{Q}^{\bullet,\dagg}_{N,M}).
\end{align*}
\end{prop}
\begin{proof}
The proposition follows immediately from Lemma \ref{lemma: disc fiber} and Proposition \ref{prop: K times L Mod}. 
\end{proof}

By Proposition \ref{prop: C,X;D,Y} and passing to $\cMod^{N(
\Fin_*)}(\PrstL)$ through $\Loc_*^!$, we have an isomorphism (using Lemma \ref{lemma: disc fiber} again)
\begin{align}\label{eq: module psi}
(id,\psi_*): (\Loc(VG;\Sp)^{\otimes_c}, \Loc(VG\times Q_0;\Sp))\overset{\sim}{\rightarrow} (\Loc(VG;\Sp)^{\otimes_c},\Loc(\widehat{Q};\Sp)). 
\end{align}

\subsection{The $J$-homomorphism and quantization result}\label{subsec: J-hom}
Recall that the correspondence (\ref{diagram: W, Q^0}) induces a $VG_N^\bullet$-module homomorphism in the category $\cMod^{N(\Fin_*)}(\Corr(\Slch)_{\fib,\all})$ from the free module generated by $Q_M^0$, denoted by $(F_NQ_M)^{\bullet,\dagg}$, to the module $\widehat{Q}_{N,M}^{\bullet,\dagg}$. 
We rewrite the resulting isomorphism (\ref{eq: module psi}) in $\cMod^{N(
\Fin_*)}(\PrstL)$ as follows,
\begin{align*}
F_{Q^0}: (\Loc(VG;\Sp)^{\otimes_c}, \Loc(VG;\Sp)\otimes \Loc(Q^0;\Sp))&\overset{\sim}{\longrightarrow} (\Loc(VG;\Sp)^{\otimes_c}, \Loc(\widehat{Q};\Sp))\\
\cF\otimes \cR&\mapsto \psi_*(\cF\boxtimes \cR)
\end{align*}
in which $\Loc(VG\times Q_0;\Sp)\simeq \Loc(VG;\Sp)\otimes \Loc(Q^0;\Sp)$ is represented as the free module of $\Loc(VG;\Sp)^{\otimes_c}$ generated by $\Loc(Q^0;\Sp)$.

Let 
\begin{align*}
j_{Q^0_M, N}: G_N\times Q^0_M\hookrightarrow G_M\times \bR^M\times\bR \text{ and }  j_{\widehat{Q}_{N,M}}: \widehat{Q}_{N,M}\hookrightarrow G_N\times \bR^M\times \bR
\end{align*} 
be the obvious (open) inclusions. Let 
\begin{align*}
p_*: \Loc(VG;\Sp)^{\otimes_c}\overset{\sim}{\to} \Loc(G;\Sp)^{\otimes_c} ,
\end{align*}
be the natural equivalence induced from applying $\Loc^!_*$ to Lemma \ref{lemma: inductive Alg, Mod} (a), taking horizontal inverse limits, and using (\ref{eq: Locotimes*!simeqotimesc}). 
Consider the diagram of objets in $\cMod^{N(\Fin_*)}(\PrstL)$ 
\begin{align*}
\xymatrix{
(\Loc(VG;\Sp)^{\otimes_c}, \Loc(VG;\Sp)^{\otimes_c}\otimes \Loc(Q^0;\Sp))\ar[d]^{p_*}_\sim\ar[r]^{\ \ \ \ \ \ \ \ \ F_{Q^0}}&(\Loc(VG;\Sp)^{\otimes_c}, \Loc(\widehat{Q};\Sp))\\
(\Loc(G;\Sp)^{\otimes_c}, \Loc(G;\Sp)^{\otimes_c}\otimes \Loc(Q^0;\Sp))&(\Loc(VG;\Sp)^{\otimes_c}, \Loc(G\times \varinjlim\limits_{M}\bR^M\times \bR;\Sp))\ar[u]_{j^!_{\widehat{Q}}:=\varprojlim\limits_{N,M}j_{\widehat{Q}_{N,M}}^!}^{\sim}\\
(\Loc(G;\Sp)^{\otimes_c}, \Loc(G\times \varinjlim\limits_{M}\bR^M\times \bR;\Sp))\ar[u]_{j^!_{Q^0}:=\varprojlim\limits_{N, M}j_{Q^0_M, N}^!}^\sim
}.
\end{align*}
The composition $\Phi_{Q}:=(j_{\widehat{Q}})^{-1}\circ F_{Q^0}\circ (p_*)^{-1}\circ j^!_{Q^0}$ is an equivalence in $\cMod^{N(\Fin_*)}(\PrstL)$. Moreover, we have 
\begin{lemma}
The equivalence $\Phi_Q$ on the underlying module category $\Loc(G\times \varinjlim\limits_{M}\bR^M\times \bR;\Sp)$ is canonically homotopy equivalent to the identity functor. 
 \end{lemma}
 \begin{proof}
 Consider the following diagram
 \begin{align*}
 \xymatrix{&VG_N\times \bR^M\times\bR\ar[dl]_{\widetilde{\psi}_{N,M}}\ar[dr]^{p_{N,M}}&\\
G_N\times \bR^M\times\bR &&G_N\times\bR^M\times\bR
 }
 \end{align*}
 where $p_{N,M}$ is the natural projection and $\widetilde{\psi}_{N,M}$ is defined in the same way as $\psi_{N,M}$ above. 
 To show the claim, we just need to show that the functor $(p_{N,M})_*$ is canonically isomorphic to $(\widetilde{\psi}_{N,M})_*$, that is compatible with $N,M$. Let
 \begin{align*}
\widetilde{\Psi}_{N,M}: &[0,1]\times VG_N\times \bR^M\times \bR\rightarrow [0,1]\times G_N\times \bR^M\times \bR\\
&(\alpha, A,\bp, \bq, t)\mapsto (\alpha, A, \bq+2\alpha\bp, t+\alpha^2 |\bp|^2).
 \end{align*}
 Then $\widetilde{\Phi}_{N,M}$ gives an equivalence between the local system categories, whose restriction to $\alpha=0,1$ are $(p_{N,M})_*$ and $(\widetilde{\psi}_{N,M})_*$ respectively, and this induces the canonical isomorphism between the two functors. 
  \end{proof}

In summary, we have:
\begin{prop}\label{prop: F_Q, p}
There is a natural equivalence in $\cMod^{N(\Fin_*)}(\PrstL)$
\begin{align}\label{eq: prop F_Q, p}
\Phi_Q: (\Loc(G;\Sp)^{\otimes_c}, \Loc(\widehat{Q};\Sp))\rightarrow (\Loc(VG;\Sp)^{\otimes_c}, \Loc(\widehat{Q};\Sp)), 
\end{align}
which induces the identity functor on the common underlying module category $\Loc(\widehat{Q};\Sp)$.
\end{prop}

\subsubsection{An equivalent model of the $J$-homomorphism}\label{subsubsec: equiv model of J}

We state a well known result (cf. \cite[Proposition 2.12]{Glasman}). 
\begin{prop}\label{prop: Alg Day convolution}
For any $K^{\otimes}\in \CAlg(\Spc^\times)$, we have a natural equivalence
\begin{align}
\label{prop eq: CAlg(Loc)}\CAlg(\Loc(K;\Sp)^{\otimes_c})&\overset{\sim}{\rightarrow} \Fun^{\rightlax}(K^{\otimes},\Sp^{\otimes})\\
\nonumber\cF&\mapsto (x\mapsto \iota_x^!\cF).
\end{align}
\end{prop}

In the Appendix Corollary \ref{cor: CAlgRunK;Sp}, we give a proof of a weaker version of the proposition using the formalism of adjoint functors. Since we are not trying to compare the equivalence (\ref{prop eq: CAlg(Loc)}) with the weaker version from Corollary \ref{cor: CAlgRunK;Sp}, we will use the latter throughout the paper. 

Let $\varpi_{VG}$ be the commutative algebra object in $\Loc(VG;\Sp)^{\otimes_c,!}$ that corresponds to the symmetric monoidal functor 
\begin{align*}
VG^\otimes\rightarrow \pt^{\otimes}\rightarrow \Sp^{\otimes},
\end{align*}
where the latter map takes $\pt^{\otimes}$ to the monoidal unit, the sphere spectrum.

\begin{prop}\label{prop: equiv model of J}
Under the natural equivalence
\begin{align*}
p_*: \Loc(VG;\Sp)^{\otimes_c}\overset{\sim}{\to} \Loc(G;\Sp)^{\otimes_c} ,
\end{align*}
the commutative algebra object $p_*\varpi_{VG}$ corresponds to the $J$-homomorphism
\begin{align*}
J: G^\otimes\rightarrow \Pic(\bS)^\otimes
\end{align*}
through (\ref{eq: cor: CAlgRunK;Sp}). 
\end{prop}
The proposition is proved in Appendix \ref{Appendix subsec: proof of J}.

\subsubsection{Twisted equivariant local systems}

For any $H^\otimes\in \CAlg(\Spc^\times)$ and $X\in \Spc$ an $H$-module, we have $\Loc(X;\Sp)$ a $\Loc(H;\Sp)^{\otimes_c}$-module. Let $a: H\times X\rightarrow X$ denote the action map. 
We call any object in $\Fun(H^\otimes, \Pic(\bS)^\otimes)$ a \emph{character} of $H$ (in $\bS$-lines). 
\begin{definition}\label{def: chi-equivariant}
For any $\chi\in \Fun(H^\otimes, \Pic(\bS)^\otimes)$, we define $\Loc(X;\Sp)^{\chi}$ to be the full (stable) subcategory of $\Mod_{\chi}(\Loc(X;\Sp))$ consisting of $\chi$-modules $\cL$ satisfying that the structure map
\begin{align*}
\chi\boxtimes \cL\rightarrow a^!\cL
\end{align*}
is an equivalence in $\Loc(H\times X;\Sp)$. We call any object in $\Loc(X;\Sp)^{\chi}$ a $\chi$\emph{-equivariant} local system on $X$. 
\end{definition}

\begin{lemma}\label{lemma: chi, fully faithful}
If $X\simeq H$ is an $H$-torsor, i.e the free module of $H$ generated by a point, then there is a natural equivalence between the (stable $\infty$-)category of $\chi$-equivariant local systems on $X$ and the category of local systems on a point, i.e. $\Sp$, via
\begin{align*}
&\Loc(\pt;\Sp)\rightarrow \Loc(X;\Sp)^{\chi},\\
&\cR\mapsto \chi\underset{\bS}{\otimes}\cR
\end{align*}
whose inverse is pullback along the inclusion $\iota_{\pt}: \pt\hookrightarrow X$ that exhibits $X$ as a free $H$-module. 
\end{lemma}
\begin{proof}
First, consider the sequence of $H$-module maps
\begin{align*}
&H\times \pt\overset{id\times \iota_{\pt}}{\hookrightarrow} H\times X\overset{a}{\rightarrow} X
\end{align*}
which exhibits $X$ as an $H$-torsor. Given a $\chi$-equivariant local system $\cM$ on $H$, from definition, we have 
\begin{align*}
(id\times \iota_{\pt})^!a^!\cM\simeq \chi\boxtimes \iota_{\pt}^!\cM,
\end{align*}
which implies that $\cM$ is isomorphic to $\chi\underset{\bS}{\otimes}\cR$ where $\cR$ is the costalk of $\cM$ at the image of $\iota_{\pt}$. 

To see that 
$$\Hom_{\Loc(X;\Sp)^\chi}(\chi\underset{\bS}{\otimes}\cR_1, \chi\underset{\bS}{\otimes}\cR_2)\cong \Hom_{\Sp}(\cR_1, \cR_2),$$ 
we use the results about free modules from \cite[4.2.4]{higher-algebra} (and Corollary 4.5.1.6 in \emph{loc. cit.} for the equivalence to the commutative setting). Let $i_e: \{e\}\hookrightarrow H$ be the unity of $H$ (up to a contractible space of choices), then $(i_{e})_!\bS$ is the monoidal unity in $\Loc(H;\Sp)^{\otimes_c}$ (recall $(i_{e})_!$ is the left adjoint of $(i_{e})^!$ on local system categories). Using $i_{\pt}$, we identify $\Loc(X;\Sp)^\chi$ with $\Loc(H;\Sp)^\chi$. Now $\chi\in \Loc(H;\Sp)$ is a free module of $\chi$ generated by $(i_{e})_!\bS$ in the sense of Definition 4.2.4.1 in \emph{loc. cit.}. 
By $\bS$-linearity, $\chi\otimes_{\bS}\cR$ is isomorphic to the free module generated by $(i_e)_!\cR$. 
Using Corollary 4.2.4.6 in \emph{loc. cit.}, we have 
\begin{align*}
&\Hom_{\Loc(H;\Sp)^\chi}(\chi\underset{\bS}{\otimes}\cR_1, \chi\underset{\bS}{\otimes}\cR_2)\cong \Hom_{\Loc(H;\Sp)}((i_e)_!\cR_1, \chi\underset{\bS}{\otimes}\cR_2)\\
&\cong \Hom(\cR_1, i_e^!(\chi\underset{\bS}{\otimes}\cR_2))\cong \Hom(\cR_1, \cR_2). 
\end{align*}
Then the lemma follows. 
\end{proof}

A direct consequence of Proposition \ref{prop: F_Q, p} and the above identification of $J$ with $p_*\varpi_{VG}$ is the following:
\begin{cor}\label{cor: equiv J-equivariant}
\item[(a)]
For any $\bS$-module $\cR$, the local system $\psi_*(\varpi_{VG}\boxtimes \cR)\in \Loc(\widehat{Q};\Sp)$ as an $\varpi_{VG}$-module is corresponding to the $J$-equivariant local system $J\boxtimes \cR$, under the natural equivalence $\Phi_Q$ (\ref{eq: prop F_Q, p}).  
\item[(b)] The correspondence 
\begin{align*}
\xymatrix{&VG\times Q^0\ar[dr]^{\pi_{Q^0}}\ar[dl]_{\psi}&\\
\widehat{Q}&&Q^0
}
\end{align*}
induces a canonical equivalence 
\begin{align*}
\Loc(Q^0;\Sp)\overset{\sim}{\longrightarrow} \Loc(\widehat{Q};\Sp)^{J}
\end{align*}
through the functor $\psi_*\pi_{Q^0}^!$.
\end{cor}
Thus, we have achieved the desired quantization results on the Hamiltonian $\coprod\limits_n BO(n)$-action on $\varinjlim\limits_NT^*\bR^N$ and its ``module" generated by the stabilization of $L_0$: the Hamiltonian $\coprod\limits_n BO(n)$-action is canonically quantized by the dualizing sheaf $\varpi_{VG}$ as a commutative algebra object in $\Loc(VG;\Sp)^{\otimes_c}$, and the ``module" generated by the stabilization of $L_0$ is quantized by the objects in the stable $\infty$-category of $J$-equivariant local systems on $\widehat{Q}$.

\section{Morse transformations and applications to stratified Morse theory}
In this section, we introduce the notion of \emph{Morse transformations} associated to a (germ of a) smooth conic Lagrangian, which is a special class of contact transformations that has intimate relation with stratified Morse functions. We give a classification of (stabilized) Morse transformations, then we apply the quantization results from the previous section to give an enrichment of the main theorem in stratified Morse theory by the $J$-homomorphism. 

\subsection{Morse transformations}\label{subsec: Morse transformations}
For any smooth manifold $X$, let $T^{*,<0}(X\times \bR_t)$ denote the open half of $T^*(X\times \bR_{t_0})$ consisting of covectors whose components in $dt$ are strictly negative. Let $T^{*,\geq 0}(X\times\bR_{t_0})$ by its complement in $T^*(X\times \bR_{t_0})$.  We implicitly assume that all the geometric objects and maps are taken in an analytic-geometric category, e.g. we can assume that they are subanalytic. 
For an exact (connected) Lagrangian submanifold $L\subset T^*X$ (not necessarily closed), a \emph{conic} Lagrangian lifting $\bL$ of $L$ is a conic Lagrangian in $T^{*,<0}(X\times\bR_t)$ such that its projection to $T^*X$ is $L$ and the 1-form $-dt+\alpha|_{\bL}$ vanishes. Note that the conic Lagrangian liftings of $L$ are essentially unique up to a shift of a constant in the $t$-coordinate. Moreover, if the exact Lagrangian $L$ is in general position, i.e. its projection to $X$ is finite-to-one, then any conic Lagrangian lifting $\bL$ is determined by its front projection in $X\times \bR_t$,  i.e. its image under the projection $T^{*,<0}(X\times \bR_t)\rightarrow X\times \bR_t$ to the base. For example, if the front projection of $\bL$ is a smooth hypersurface, then $\bL$ is just half of the conormal bundle of the hypersurface contained in $T^{*,<0}(X\times \bR_t)$, and we say it is the \emph{negative conormal bundle} of the smooth hypersurface.

Let $(L,(0,p_0))$ be a germ of a smooth Lagrangian (in general position) in $T^*\bR^M$ with center $(0,p_0)$ and let $(\bL, (0,p_0; t_0=0, -dt_0))$ be the germ of a conic Lagrangian lifting of $(L,(0,p_0))$ in $T^{*,<0}(\bR^M\times \bR_{t_0})$. Following \cite{KS}, a germ of a \emph{contact transformation with respect} $(0,p_0; t_0=0, -dt_0)$ is a germ of a conic Lagrangian $\bL_{01}\subset (T^{*,<0}(\bR^M\times \bR_{t_0}))^-\times \mathring{T}^*(\bR^M\times \bR_{t_1})$ such that $\bL_{01}$ induces a conic symplectomorphism from a germ of an open set in $T^{*,<0}(\bR^M\times \bR_{t_0})$ centered at $(0,p_0; t_0=0, -dt_0)$ to a germ of an open set in $\mathring{T}^{*}(\bR^M\times \bR_{t_1})$, where $\mathring{T}^*X$ means the cotangent bundle with the zero-section deleted for any $X$. In the following, any (conic) Lagrangian is understood as a germ of a (conic) Lagrangian, and any contact transformation is understood as a germ of a contact transformation, unless otherwise specified. 
 
It is proved in \cite{KS} that there exists a contact transformation $\bL_{01}$ such that $\bL_{01}$ is locally the conormal bundle of a smooth hypersurface near the center, and it corresponds to a conic symplectomorphism which takes $(\bL,(0,p_0; t_0=0, -dt_0))$ to a conic Lagrangian that is locally the conormal bundle of a smooth hypersurface in $\bR^M\times \bR_{t_1}$. We call such a contact transformation a \emph{Morse transformation with respect to} $(\bL, (0,p_0; t_0=0, -dt_0))$. The space of Morse transformations form an open dense subset in the space of all contact transformations with respect to $(0,p_0; t_0=0, -dt_0)$.  There is an obvious action by the group of diffeomorphisms of $\bR^M\times \bR_{t_1}$ on the space of Morse transformations. 

We state some useful facts about the space of Morse transformations modulo the action by $\Diff(\bR^M\times \bR_{t_1})$. We will assume that the germ $(\bL,(0,p_0; t_0=0, -dt_0))$ is sent to a germ $(\bL_1, (0,p_1; t_1=0, -dt_1))$ in $T^{*,<0}(\bR^M\times \bR_{t_1})$ through $\bL_{01}$. 
For any cotangent bundle involved, we let $\pi$ denote its projection to the base. Let $A_{L_{(0,p_0)}}$ be the quadratic form on $\pi_*T_{(0,p_0)}L$ determined by the linear Lagrangian $T_{(0,p_0)}L$ in $T_{(0,p_0)}(T^*\bR^M)$. We will equally view $A_{L_{(0,p_0)}}$ as a quadratic form on $\pi_*T_{(0,p_0; t_0=0, -dt_0)}\bL$ as a linear subspace in $\bR^M\times \bR_{t_0}$. 

\begin{prop}\label{prop: classification contact}
\item[(a)]
The space of Morse transformations $\bL_{01}$ with respect to $(\bL, (0,p_0; t_0=0, -dt_0))$ modulo the action of $\Diff(\bR^M\times \bR_{t_1})$ is canonically homotopy equivalent to the space of quadratic forms $A_S$ on $\pi_*T_{(0,p_0)}L$ satisfying that $A_S-A_{L_{(0,p_0)}}$ is nondegenerate. 
\item[(b)] For each quadratic form $A_S$ as above, viewed as a quadratic form on $\pi_*T_{(0,p_0; t_0=0, -dt_0)}\bL$ as well, it  corresponds to a Morse transformation given by the negative conormal bundle (negative in $dt_1$) of 
\begin{align}\label{eq: prop contact A_S}
t_1-t_0-p_1\cdot \bq_1+p_0\cdot \bq_0+\frac{1}{2}A_S(\bq_0)-\bq_0\cdot\bq_1=0,
\end{align}
where $p_0,p_1$ are fixed and $A_S(\bq_0)$ is from extending by zero along the orthogonal complement of $\pi_*T_{(0,p_0; t_0=0, -dt_0)}\bL$ in $\bR^M\times \bR_{t_0}$.
\item[(c)] For two (globally defined) quadratic forms $A_{S,1}$ and $A_{S,2}$ on $\bR^M$, the negative conormal bundle of the hypersurface (\ref{eq: prop contact A_S}) for $A_{S,i}$ determines a symplectomorphism $\Phi_i$ on $T^*\bR^M$, respectively. Let $H_{12}(\bq_1,\bp_1)=\frac{1}{2}(A_{S,2}-A_{S,1})(\bp_1-p_1)$ and $\varphi_{H_{12}}^1$ be the time-1 map of the Hamiltonian flow of $H_{12}$. Then 
\begin{align*}
\Phi_2=\varphi_{H_{12}}^1\circ\Phi_1.
\end{align*}
\end{prop} 
\begin{proof}
Since a germ of a contact transformation is determined by the tangent space of its center up to a contractible space of homotopies, (a) and (b) are linear problems and can be solved using a similar consideration as in the proof of \cite[Proposition A.2.6]{KS}. For the reader's convenience, we sketch the proof here. 

Without loss of generality, we may reduce the case to $p_1=0$ and $p_0=0$. For the general case, we just need to apply an linear transformation $(\bq_i, t_i)\mapsto (\bq_i, t_i+p_i\cdot \bq_i)$ for $i=0,1$. 
We fix a basis in $\bR^M_{\bq_0}$ (as a vector space) such that the first $s$ elements is the basis of $\pi_*T_{(0,0)}L$. Given any $\bL_{01}$, the fixed basis in  $\bR^M_{\bq_0}$ together with the standard basis on $\bR_{t_0}$ determines a unique basis in $\bR^M_{\bq_1}\times \bR_{t_1}$ (the projection to the cotangent base of both the image of the dual basis in $(\bR_{\bq_0}^M)^*$ and the image of $\partial_{t_0}$ under $\bL_{01}$), with respect to which we can write down the tangent space of $\bL_{01}$ at $(0, 0;t_0=0, dt_0; 0, 0;t_1=0, -dt_1)$ in the standard form as 
\begin{equation*}
\left(\begin{array}{ccc|ccc|cc|ccc|ccc|cc}
1&         &  & 0&         & & 0      &0        & &  &       &  &  &  &0  &0\\
 &\ddots&  &  &\ddots& &\vdots&\vdots& &A&       &  &I&  & \vdots &\vdots \\
 &          &1&  &          &0& 0     & 0       & &   &      &  & &  &0 &0\\
 \hline\\
 0&         &  & 1&         & & 0      &0       & &  &       &  &  &  & 0 &y_1 \\
 &\ddots&  &  &\ddots& &\vdots&\vdots& &I &       &  &C&  & \vdots &\vdots \\
 &          &0&  &          &1& 0      & 0      & &  &       &  & &  &0  &y_M\\
 \hline\\
 0&\cdots&0&0&\cdots&0&1     &1      &0&\cdots&0 &y_1&\cdots&y_M&0&x\\
  0&\cdots&0&0&\cdots&0&0     &0      &0&\cdots&0 &0&\cdots&0&1&-1
\end{array}\right),
\end{equation*}
in which (i) a basis of the tangent space is displayed in row vectors with respect to the chosen basis of $\bR^M_{\bq_0}, \bR^M_{\bq_1}, \bR_{t_0}, \bR_{t_1}$ (for $\bR_{t_i}$ we use the standard basis $\partial_{t_i}$) and their respective dual basis in the dual vector spaces, (ii) $A=A^T, C=C^T$ and $y_1,\cdots, y_M, x$ are arbitrary real numbers. The requirement on $(\bL_1, (0,p_1;0,-dt_1))$ is exactly the condition that the top left $s\times s$-submatrix in the symmetric matrix $A$ satisfies $(A-A_{L_{(0,p_0)}})|_{\pi_*T_{(0,p_0)}L}$ is nondegenerate. Thus the the top left $s\times s$-submatrix in $A$ exactly gives $A_S|_{\pi_*T_{(0,0)}L}$. Then part (a) follows directly. Part (b) follows from taking $C=0$ and $y_1=\cdots=y_M=x=0$.

For (c), if we write down the negative conormal of (\ref{eq: prop contact A_S}) explicitly, 
\begin{align*}
(t_0, dt_0, \bq_0, (-p_0-\frac{1}{2}\partial_{\bq_0} A_S(\bq_0)+\bq_1)\cdot d\bq_0; t_1,-dt_1, \bq_1, (p_1+\bq_0)\cdot d\bq_1),
\end{align*} 
then we see that it corresponds to the symplectomorphism on $T^*\bR^M$ given by 
\begin{align*}
&(\bq_0,\bp_0)\mapsto (-\bp_0
+p_0+\frac{1}{2}\partial_{\bq_0}A_S(\bq_0), \bq_0+p_1).
\end{align*}
Thus (c) follows. 
\end{proof}
Note that when $p_0,p_1$ and $A_S(\bq_0)$ are all zero, the contact transformation (\ref{eq: prop contact A_S}) gives the Fourier transform on $T^*\bR^M$.

\subsection{Relation to stratified Morse theory}\label{subsec: relationtoStratifiedMorse}

There is a close relation between Morse transformations and stratified Morse functions (and their stabilizations) as follows. Now assume we are in the setting of Subsection \ref{subsec: SMT}. Let $(x_0,\xi_0)\in \Lambda_{\fS}$ be a smooth point with $\xi_0\neq 0$. If $\xi_0=0$, we will embed $X$ into $X\times \bR$ by the graph of the constant function $0$ on $X$, and reduce the case to $\xi_0\neq 0$.  Let $\bk$ be any commutative ring spectrum.

\subsubsection{Passing from local (stratified) Morse functions to $\Omega$-lenses}\label{subsubsec: local Morse, Omega lenses}

Recall in stratified Morse theory, to define the microlocal stalk at $(x_0,\xi_0)$, we choose a local Morse function $f$ in the following sense: $f(x_0)=0$, $df_{x_0}=\xi_0$ and $f|_{S_0}$ is Morse at $x_0$ (here $S_0$ is the stratum containing $x_0$). In particular, the space of local Morse functions associated to $(x_0,\xi_0)\in \Lambda_{\fS}^{sm}$ is naturally homotopy equivalent to the space of non-degenerate quadratic forms on the tangent space $T_{x_0}S_0. $
Since $\xi_0\neq 0$, using a subanalytic change of local coordinates, we identify an open neighborhood of $x_0$ with an open neighborhood of $U$ of $(0,0)$ in $\bR^r\times \bR_{t}$, where $r=\dim X$, such that $f$ becomes the function $-t$ and $\xi_0$ becomes $-dt$. In this way, an open neighborhood of $(x_0,\xi_0)$ in $\Lambda_{\fS}$ is identified with an open neighborhood of $(0,p=0;t=0, -dt)$ in a smooth conic Lagrangian in $T^{*,<0}(\bR^r\times \bR_{t})$. Let $H=\{t=0\}$, and let $B_\delta((x,t))$ be the standard open $\delta$-ball centered at $(x, t)$. 
Fix a Whitney stratification $\fS'$ of $U$ that is a refinement of both $\fS\cap U$ and $H\cap U$. 
By assumption, for any $S_\alpha\in \fS$, $H\cap S_\alpha$ is transverse except at $(0,0)$ in a sufficiently small neighborhood of $(0,0)$. By shrinking $U$ if necessary, we may assume this is always true in $U-\{(0,0)\}$. In particular, $\fS'$ is a refinement of $\{S_\alpha\cap H\cap (U-\{(0,0)\}), \{(0,0)\}: S_\alpha\in \fS\}$. 
There exists a continuous function $\rho: (0,1)\rightarrow \bR_+$ with $\rho(\epsilon)>\epsilon$ and $\frac{\rho(\epsilon)}{\epsilon}\rightarrow 1, \epsilon\rightarrow 0^+$, such that  $\partial B_\delta((0,\epsilon))$ is transverse to $\fS'$ for all $\epsilon<\delta< \rho(\epsilon)$ and $0<\epsilon\ll 1$. 
Then the microlocal stalk of a given sheaf $\cG\in \Shv_{\cS}(X)$ at $(x_0, \xi_0)$, with respect to the local Morse function $f$, can be calculated by 
\begin{equation*}
\mu_{(x_0, \xi_0;f)}(\cG):=\Cone(\Gamma(B_{\rho(\epsilon)}((0,\epsilon)); \cG)\to \Gamma(B_{\rho(\epsilon)}((0,\epsilon))\cap\{t>0\};\cG)). 
\end{equation*}

By assumption, for any $(x,t)\in \partial B_{\rho(\epsilon)}((0,\epsilon))\cap \{t=0\}\cap S_\alpha$, where $S_\alpha\in \fS$ and $0<\epsilon\ll 1$, any nontrivial linear combination of the conormal vectors of $\partial B_{\rho(\epsilon)}((0,\epsilon))$ and $-dt$ is \emph{not} in $\Lambda_{\fS}$. 
Hence, we can choose a standard (sufficiently local) smoothing $H_\epsilon$ of the hypersurface $H_\epsilon':=(\partial B_{\rho(\epsilon)}((0,\epsilon))\cap \{t\leq 0\})\cup (H\cap (U- B_{\rho(\epsilon)}((0,\epsilon)))$, so that 
\begin{itemize}
\item Let $f_\epsilon$ be a defining function of $H_\epsilon$, i.e. $H_\epsilon=\{f_\epsilon=0\}$ and $df_\epsilon|_{H_\epsilon}$ is nowhere $0$, so that $f_\epsilon>0$ for $t>0$. Then $df_\epsilon (-\partial_t)<0$ along $H_\epsilon$;
\item $df_\epsilon|_{H_\epsilon}\cap \Lambda_\fS=\emptyset$.  
\end{itemize}
See Figure \ref{figure: S_0H} for an illustration. 
Let $H_0=H\cap U$ and $f_0=-t$. Let $U_\epsilon=\{f_\epsilon>0\}$ for all $0\leq \epsilon\ll 1$. For $\epsilon>0$ sufficiently small, there exists a sufficiently small open cone $\Omega$ in $T^*U$ containing $(0,0;0,-dt)$, such that 
$U_0$ and $U_\epsilon$ can be made as $H_{a, 0}^\dagger$ and $H_{a,\epsilon}^\dagger$ for an $\Omega$-lense, in the sense of \cite[Subsection 2.7.2]{JT}. 
Let $\cF_{\epsilon,f}:=\Cone(j_{0, !}\bk_{U_0}\to j_{\epsilon, !}\bk_{U_\epsilon})$. We call these \emph{the standard sheaves} associated to the $\Omega$-lenses.

\begin{figure}[htbp]
\begin{tikzpicture}[scale=1.4, domain=-1.4:1.4 ]
\tikzmath{
\a=0.5;
\ra=1.5;
\b=0.3;
\rb=1;
\c=0.2;
\rc=0.5;
}
\draw[thick] (-3,0)--(3,0) node[right] {$H$};
\draw[dashed, thick, black] (0,\a) circle [radius=\ra];
\draw[rounded corners=5mm, thick, black, dashed] (-\ra-0.3, 0) -- (-\ra,0) -- ({\ra*cos(210)}, {\a+ \ra*sin(210)});
\draw[rounded corners=5mm, thick, dashed, black] (\ra+0.3, 0) -- (\ra,0) -- ({\ra*cos(-30)}, {\a+ \ra*sin(-30)});
\draw[dashed, thick, black] (0,\b) circle [radius=\rb];
\draw[rounded corners=4mm, thick, dashed, black] (-\rb-0.2, 0) -- (-\rb,0) -- ({\rb*cos(210)}, {\b+ \rb*sin(210)});
\draw[rounded corners=4mm, thick, dashed, black] (\rb+0.2, 0) -- (\rb,0) -- ({\rb*cos(-30)}, {\b+ \rb*sin(-30)});
\draw[dashed, thick, black] (0,\c) circle [radius=\rc];
\draw[rounded corners=2mm, thick, dashed, black] (-\rc-0.1, 0) -- (-\rc,0) -- ({\rc*cos(220)}, {\c+ \rc*sin(220)});
\draw[rounded corners=2mm, thick, dashed, black] (\rc+0.1, 0) -- (\rc,0) -- ({\rc*cos(-40)}, {\c+ \rc*sin(-40)});
\fill[fill=gray, opacity=0.5] (-3,0) rectangle (3,-1.7);
 \draw[color=blue, thick]   plot (\x,{\x*\x}) node[right] {$S_0$};
 \draw[->] (0,0)--(0,-0.5) node[right] {$(x_0,\xi_0)$};
\end{tikzpicture}
\caption{An illustration of $S_0, H=H_0, (x_0,\xi_0)=(0,0;0, -dt)$, the curves $\partial B_{\rho(\epsilon)}((0,\epsilon))$ (the dashed circles) and the smoothings $H_\epsilon$ (the dashed curves that have tangency with $H$). The shaded region enclosed by the dashed $H_\epsilon$ and the solid $H_0$ is $U_\epsilon-U_0$. 
}
\label{figure: S_0H}
\end{figure}

\begin{lemma}\label{lemma: mux0xi0, lense}
There is a canonical isomorphism 
\begin{align}\label{eqlemma: mux0xi0, lense}
\mu_{(x_0, \xi_0;f)}(\cG) \to \Hom(\cF_{\epsilon,f}, \cG), \forall 0<\epsilon\ll 1. 
\end{align}
\end{lemma}
\begin{proof}
The isomorphism comes from a direct application of the non-characteristic deformation lemma (cf. \cite[Proposition 2.7.2]{KS}, \cite[Proposition 2.7.1]{JT}). 
For a fixed $\delta>0$, we can choose $U_{\epsilon}, \epsilon<\delta$, as above satisfying $B_{\rho(\delta)}((0,\delta))=\bigcup_{\epsilon<\delta}(U_\epsilon\cap B_{\rho(\delta)}((0,\delta)))$. By the assumption that $df_\epsilon|_{H_\epsilon}\cap \Lambda_\fS=\emptyset$, the non-characteristic deformation lemma implies that the restriction morphism
\begin{align*}
\Gamma(B_{\rho(\delta)}((0,\delta)); \cG)\to \Gamma(U_\epsilon\cap B_{\rho(\delta)}((0,\delta));\cG)
\end{align*}
is an isomorphism for all $0<\epsilon<\delta$. Now taking the sections relative to $\Gamma(\{t>0\}\cap B_{\rho(\delta)}((0,\delta));\cG)$ on both sides, we get exactly (\ref{eqlemma: mux0xi0, lense}). 
\end{proof}

\subsubsection{Simplification of sheaves.}\label{subsubsec: simplification sheaves}

Lemma \ref{lemma: mux0xi0, lense} allows us to calculate microlocal stalks using $\cF_{\epsilon,f}$ satisfying $\SS(\cF_{\epsilon,f})\cap \mathring{T}^*(U)\subset \Omega$, where $\Omega$ can be an arbitrarily small conic neighborhood of $(x_0,\xi_0)=(0,0;0,-dt)$. This allows us to use certain localization process (which is the basic idea in \cite{KS}) to ``simplify" any sheaf $\cG\in \Shv_{\cS}(X)$, that has no effect on calculating $\mu_{(x_0, \xi_0;f)}(\cG)$. To make this work for all choices of local Morse functions, we will assume $S_0=\{t=x_{r_0+1}=\cdots=x_r=0\}$ in $U$, where $r_0=\dim S_0$. 
Then for different choices of $f$, $\{f=0\}$ give different hypersurfaces and the sheaves $\cF_{\epsilon, f}$ will be depending on $f$. In the following, without loss of generality (since all calculations are local), we will assume $X=\bR^r\times \bR$ and let $S_0'=\{t=x_{r_0+1}=\cdots=x_r=0\}$ (now the picture differs from Figure \ref{figure: S_0H} by a diffeomorphism). 

For any closed conic subset $Z\subset T^*X$, write $\Shv_Z(X;\bk)$ for the full subcategory of $\Shv(X;\bk)$ with singular support contained in $Z$. Write $\Shv(X;\bk)/\Shv_Z(X;\bk)$ for the right orthogonal complement of $\Shv_Z(X;\bk)$, and for any closed conic subset $W$ of $T^*X-Z$, write $(\Shv(X;\bk)/\Shv_Z(X;\bk))_W$ for the full subcategory of $\Shv(X;\bk)/\Shv_Z(X;\bk)$ whose objects $\cF$ satisfy $\SS(\cF)\cap (T^*X-Z)\subset W$. 
Following Tamarkin, let  $\Shv_{\geq 0}(\bR^r\times \bR_{t};\bk)$ be the full subcategory of sheaves on $\bR^r\times \bR_t$, whose singular support is contained in $T^{*, \geq 0}(\bR^r\times \bR_t)$. 
Let $\Shv^{<0}(\bR^r\times \bR;\bk)$ be the left orthogonal complement of  $\Shv_{\geq 0}(\bR^r\times \bR;\bk)$, which is generated by the standard sheaves associated to $T^{*,<0}(\bR^r\times\bR)$-lenses under small colimits. For any closed conic $\Lambda\subset T^{*,<0}(\bR^r\times \bR_t)$, let $\Shv^{<0}_{\Lambda}(\bR^r\times \bR_t;\bk)$ be the full subcategory of $\Shv^{<0}(\bR^r\times \bR_{t};\bk)$ 
consisting of $\cF$ with $\SS(\cF)\cap  T^{*,<0}(\bR^r\times \bR_t)\subset \Lambda$. 
Let $\Lambda_{S'_0}^{<0}=T^{*}_{S'_0}(\bR^r\times \bR)\cap T^{*,<0}(\bR^r\times \bR_t)$. 

\begin{lemma}\label{lemma: cG, cG', mu}
For any sheaf $\cG\in \Shv_{\cS}(\bR^r\times \bR;\bk)$, there exists a sheaf $\cG'\in \Shv_{\Lambda_{S'_0}^{<0}}^{<0}(\bR^r\times \bR_t;\bk)$ (unique up to isomorphisms) with a zig-zag of morphisms $\cG\to \cM\leftarrow \cG'$ in $\Shv(\bR^r\times\bR;\bk)$, such that for any local Morse function $f$ as above, the zig-zag of morphisms induces isomorphisms 
\begin{align}\label{eq: lemma cG, cG', mu}
\mu_{(x_0, \xi_0;f)}(\cG)\overset{\sim}{\to}  \Hom(\cF_{\epsilon,f},\cM) \overset{\sim}{\leftarrow} \mu_{(x_0, \xi_0;f)}(\cG'), 
\end{align}
for all $\epsilon>0$ sufficiently small. 
\end{lemma}

\begin{proof}
Note that the natural functor 

\begin{align}\label{eq: MShpLambdaS0'}
\Shv^{<0}_{\Lambda_{S'_0}^{<0}}(\bR^r\times \bR_t;\bk)\to (\text{MSh}^{\text{p}}_{\Lambda_{S_0'}^{<0}})_{(x_0,\xi_0)}\simeq \Mod(\bk)
\end{align}
(cf. \cite[Subsection 2.8.2]{JT} for the notation $\text{MSh}^{\text{p}}_Z$; for a more thorough discussion about $\text{MSh}^{\text{p}}_Z$, we refer the reader to \cite{KS}, \cite[Section 6]{Nadler-Shende}) is an equivalence. There exists a decreasing sequence of open conic neighborhoods $\{\Omega_n\}_{n\geq 0}$ of $(x_0,\xi_0)$ whose intersection is $\bR_+\cdot (x_0, \xi_0)$ such that 
\begin{align}\label{eq: Omega_nMShp}
 \big(\Shv(\bR^r\times \bR_t;\bk)/\Shv_{T^*(\bR^r\times \bR_t)-\Omega_n}(\bR^r\times \bR_t)\big)_{\Lambda_{S'_0}^{<0}\cap \Omega_n}\to (\text{MSh}^{\text{p}}_{\Lambda_{S_0'}^{<0}})_{(x_0,\xi_0)}
\end{align}
is an equivalence for all $n$. For any $n$ sufficiently large, set the $\Omega$ in Subsection \ref{subsubsec: local Morse, Omega lenses} to be $\Omega_n$. Let $\cM$ be the image of $\cG$ under the localization functor 
\begin{align*}
\Shv(\bR^r\times \bR_t;\bk)\to  \Shv(\bR^r\times \bR_t;\bk)/\Shv_{T^*(\bR^r\times \bR_t)-\Omega_n}(\bR^r\times \bR_t)
\end{align*}
so then $\cM$ lands in the LHS of (\ref{eq: Omega_nMShp}). 
Since (\ref{eq: MShpLambdaS0'}) factors through (\ref{eq: Omega_nMShp}), there is $\cG'\in \Shv^{<0}_{\Lambda_{S'_0}^{<0}}(\bR^r\times \bR_t;\bk)$ (unique up to isomorphisms) whose image in the LHS of $(\ref{eq: Omega_nMShp})$ is isomorphic to $\cM$. Thus, we have a zig-zag of morphisms  $\cG\to \cM\leftarrow \cG'$. Since the fiber of each map has singular support outside of $\Omega$ (so they are right orthogonal to any standard sheaf associated to an $\Omega$-lense \cite[Subsection 2.7.2]{JT}), for all $\epsilon$ sufficiently small, we get the isomorphisms (\ref{eq: lemma cG, cG', mu}) by taking $\Hom(\cF_{\epsilon, f}, -)$. 
\end{proof}

\begin{remark}\label{remark: colim, Hom, cFepsilonf}
Using the equivalences (\ref{eq: Omega_nMShp}), 
we get a well defined equivalence 
\begin{equation*}
\begin{tikzcd}[column sep=7em]
(\mathrm{MSh}^{\mathrm{p}}_{\Lambda_{S_0'}^{<0}})_{(x_0,\xi_0)}\ar[r, " {\varinjlim\limits_{\epsilon\to 0^+}\Hom(\cF_{\epsilon,f}, \text{--})} ", "\sim"']&\Mod(\bk). 
\end{tikzcd}
\end{equation*}
Using this equivalence for (the non-canonical) $\simeq$ in (\ref{eq: MShpLambdaS0'}), the composition is identified with $\Hom(\cF_{\epsilon,f}\text{, })$ for any $\epsilon>0$ sufficiently small. The same holds when $\Shv^{<0}_{\Lambda_{S'_0}^{<0}}(\bR^r\times \bR_t;\bk)$ is replaced by $\Shv_{\fS}(\bR^r\times\bR_t;\bk)$. 
\end{remark}

\subsubsection{Quantized (stabilized) Morse transformations as (stabilized) Morse kernels.}

We will see soon that the correspondence given by a Morse transformation is exactly calculating the microlocal stalks $\mu_{(x_0,\xi_0;f)}$, for the local Morse function $f$ given by $-t_0+p_0\cdot \bq_0+\frac{1}{2}A_S(\bq)$ in Proposition \ref{prop: classification contact} (a) (up to adding a factor that forms a convex space).  This is established in Corollary \ref{cor: Morse transf, stratified Morse} below. 

Both Morse transformations and stratified Morse functions can be stabilized in the following way. Using the notations from Proposition \ref{prop: classification contact}, we stabilize a (germ of a) Lagrangian $L\subset T^*\bR^{M}$ by taking $L\times T^*_{\bR}\bR_{y_1}\times T^*_{\bR}\bR_{y_2}\times \cdots $ in $T^*\bR^M\times T^*\bR_{y_1}\times T^*\bR_{y_1}\times \cdots$. Then we stabilize the Morse transformation given in (\ref{eq: prop contact A_S}), by adding $-(y_1^2+y_2^2+\cdots)$ to $A_S$. For a stratification $\fS=\{S_\alpha\}$, we stabilize it in the usual way: $\{S_\alpha\times \bR_{y_1}\times \bR_{y_2}\cdots\}$. 
For a local Morse function $f$ as in Subsection \ref{subsubsec: local Morse, Omega lenses}, we stabilize $f$ by  adding $-\frac{1}{2}(y_1^2+y_2^2+\cdots)$.

\subsection{The relevant localization of sheaf categories}\label{subsec: relevantLocalization}
For a germ of a smooth conic Lagrangian $(\bL, (x_0,\xi_0)=(0,p_0;t_0=0,-dt_0))$ in $T^{*,<0}(\bR^M\times \bR_{t_0})$ as above, we can choose a diffeomorphism between a sufficiently small ball containing the front projection of $\bL$ with $\bR^M\times \bR_{t_0}$ which is the identity near the origin and sends negative covectors to negative covectors, so then the germ of front is sent to a closed front in $\bR^M\times \bR_{t_0}$. Let $\widetilde{\bL}$ be the conic Lagrangian determined by the resulting front. 
For any interval $I\subset (-\infty,\infty)$, we have a correspondence 
\begin{align*}
\xymatrix{&\bR^M\times\bR_t\times I\ar[dr]^{\pi_{12}}\ar[dl]_{m_I}&\\
\bR^M\times\bR_t&&\bR^M\times\bR_t
}
\end{align*}
where $\pi_{12}$ is the projection to the first two factors and $m_I$ is the addition operation $\bR_t\times I\rightarrow \bR_t$ and projection in $\bR^M$. We define the convolution functor $T_I$ to be
\begin{align*}
&T_I=(m_I)_!\pi_{12}^*: \Shv(\bR^M\times\bR_t;\Sp)\rightarrow \Shv(\bR^M\times \bR_t;\Sp).
\end{align*}
The convolution functor $T_{(-\infty,0]}$ defines a localization functor which corresponds to the localizing subcategory  $\Shv_{\geq 0}(\bR^M\times\bR_{t_0};\Sp)$, whose essential image is $\Shv^{<0}(\bR^M\times\bR_{t_0};\Sp)$ (cf. \cite{Tamarkin2, JT}). 
The above produces (canonically) equivalent localized sheaf categories
\begin{align}\label{eq: Shv_bL localized}
\Shv_{\widetilde{\bL}\cup T^{*,\geq 0}(\bR^M\times\bR_{t_0})}(\bR^M\times\bR_{t_0};\Sp)/ \Shv_{\geq 0}(\bR^M\times\bR_{t_0};\Sp)
\end{align}
regardless of the choice of the diffeomorphism. Hence, we will abuse notation and use $\Shv_{\bL}^0(\bR^M\times\bR_{t_0};\Sp)$ to denote the localized category in (\ref{eq: Shv_bL localized}) without reference to any particular choice of $\widetilde{\bL}$. We note that $\Shv_{\bL}^0(\bR^M\times\bR_{t_0};\Sp)$ is naturally equivalent to $(\text{MSh}_{\bL}^{\text{p}})_{(x_0, \xi_0)}$. 
In the following, we reserve the notation $\Shv_{\widetilde{\bL}}^{<0}(\bR^r\times \bR_t;\bk)$ for the left orthogonal complement of $\Shv_{\geq 0}(\bR^M\times\bR_{t_0};\Sp)$ in $\Shv_{\widetilde{\bL}\cup T^{*,\geq 0}(\bR^M\times\bR_{t_0})}(\bR^M\times\bR_{t_0};\Sp)$. 

Given any Morse transformation $\bL_{01}$ sending $(\bL, (0,p_0;t_0=0,-dt_0))$ to $(\bL_1, (0,p_1;t_1=0,-dt_1))$,  it  determines a correspondence (where $\pi(\bL_{01})$ can be globalized to a smooth hypersurface)
\begin{align}\label{diagram: L01, p1, p0}
\xymatrix{&\pi(\bL_{01})\ar[dr]^{p_0}\ar[dl]_{p_1}&\\
\bR^M\times\bR_{t_1}&&\bR^M\times\bR_{t_0}
}
\end{align}
that induces an equivalence of localized sheaf categories 
\begin{align}\label{eq: L01,functor, L0L1}
(p_1)_*p_0^!: \Shv_{\bL}^0(\bR^M\times\bR_{t_0};\Sp)\simeq \Shv_{\bL_1}^0(\bR^M\times\bR_{t_1};\Sp).
\end{align}
Since by assumption $\bL_1$ is the negative conormal bundle of a germ of a smooth hypersurface defined by $f(\bq_1,t_1)=0$, the latter category is 
equivalent to 
\begin{align}\label{eq: Upsilon, Loc}
\Upsilon: \Loc(\{f(\bq_1,t_1)<0\}\cap B_\epsilon(0,0);\Sp)&\overset{\sim}{\longrightarrow} \Sp\\
\cL&\mapsto \Gamma(\{\bq_1=0\},\Gamma_{\{\bq_1=0\}}\cL)\simeq \Sigma i_{\{(0, -K)\}}^!\cL
\end{align}
for any $K\gg 1$. Here the added $\Sigma$ is to make Lemma \ref{lemma: Psi_bL01} and Corollary \ref{cor: Morse transf, stratified Morse} below true without any shift by $\Sigma$.  
For a more detailed discussion of the equivalences induced by Morse transformations for microlocal sheaves of spectra, we refer the reader to \cite[Section 2.10]{JT} and \cite[Section 6]{Nadler-Shende}.

\subsubsection{An example when $\bL=\Lambda_{S_0'}^{<0}$.}
Assume we are in the setting of Subsection \ref{subsubsec: simplification sheaves}. Let $\bL=\Lambda_{S_0'}^{<0}$, and write $S_0'=S_{0,\bq}'\times\{t_0=0\}\subset\bR^M_\bq\times \bR_{t_0}$. Then there is a natural equivalence 
\begin{align*}
\Sp&\overset{\sim}{\to} \Shv_{\bL}^{<0}(\bR^M\times \bR_{t_0};\Sp)\simeq  \Shv_{\bL}^0(\bR^M\times\bR_{t_0};\Sp)\\
\cM&\mapsto i_*\cM_{S_{0,\bq}'\times (-\infty, 0)},
\end{align*}
where $i: S_{0,\bq}'\times (-\infty, 0)\hookrightarrow \bR^M\times\bR_{t_0}$ is the inclusion. Let $\bL_{01}$ be the Morse transformation defined by (\ref{eq: prop contact A_S}), with $p_0=p_1=0$ and $A_S$ a non-degenerate quadratic form on $\pi_*T_{(0,0;0,-dt)}\bL=S_0'$ (the identification is through the obvious way). 

\begin{lemma}\label{lemma: Psi_bL01}
Let $\bL=\Lambda_{S_0'}^{<0}$. Then the equivalence  (\ref{eq: L01,functor, L0L1}), composed with the natural equivalence $\Upsilon: \Shv_{\bL_1}^0(\bR^M\times\bR_{t_1};\Sp)\overset{\sim}{\to} \Sp$ (\ref{eq: Upsilon, Loc}), is explicitly given by 
\begin{equation*}
\begin{tikzcd}
\Psi_{\bL_{01}}: \Shv_{\bL}^{<0}(\bR^M\times\bR_{t_0};\Sp)\ar[r, "\sim"]&\Shv_{\bL_1}^0(\bR^M\times\bR_{t_1};\Sp)\ar[r, "\sim", "\Upsilon"']& \Sp\\
\cG'\ar[rr, mapsto] &\ &\mu_{(x_0,\xi_0;-t_0+\frac{1}{2}A_S(\bq))}(\cG').
\end{tikzcd}
\end{equation*}
\end{lemma}

\begin{proof}
First, the correspondence (\ref{diagram: L01, p1, p0}) induces an equivalence 
\begin{align*}
\Shv_{\bL}^{<0}(\bR^M\times\bR_{t_0};\Sp)\simeq \Shv_{\bL_1}^{<0}(\bR^M\times\bR_{t_1};\Sp)
\end{align*}
where $\bL_1$ is the negative conormal bundle of a \emph{global} smooth hypersurface (not just a germ). Indeed, let $L$ be the conormal bundle of $S'_{0,\bq}$ in $T^*\bR^M$, then 
using Proposition \ref{prop: classification contact} (c), we know that $\bL_1$ is the cone over a Legendrian lifting of the linear Lagrangian graph in $T^*\bR^M$, which comes from first doing a Fourier transform of $T^*\bR^M$ on $L$ and then doing $\varphi_{\frac{1}{2}A_S(\bp)}^1$. Let $H\subset \bR^M\times \bR_{t_1}$ be the projection of $\bL_1$, and let $U_-$ be the open subset below $H$ (with respect to the orientation from $\bR_{t_1}$).  
Then we have the obvious equivalence 
\begin{align*}
\Sp&\to \Shv_{\bL_1}^{<0}(\bR^M\times\bR_{t_1};\Sp)\\
\cN&\mapsto j_{U_-,*}\cN_{U_-}.
\end{align*}

Let $\fri_{\bR_{t_1}}: \{0\}\times \bR_{t_1}\hookrightarrow \bR^M\times \bR_{t_1}$ (resp. $\fri_{\bR_{t_1}^+}: \{0\}\times \bR_{t_1}^+\hookrightarrow \bR^M\times \bR_{t_1}$) be the inclusion. Given $\cG'=i_*\cM_{S_{0,\bq}'\times (-\infty, 0)}\in  \Shv_{\bL}^{<0}(\bR^M\times \bR_{t_0};\Sp)$, to see $\Psi_{\bL_{01}}(\cG')|_{U_-}$, it suffices to calculate the fiber of 
\begin{align*}
\Gamma(\bR_{t_1}, i_{\bR_{t_1}}^!(p_1)_*p_0^!\cG')\to \Gamma(\bR_{t_1}^+, i_{\bR_{t_1}^+}^!(p_1)_*p_0^!\cG'). 
\end{align*}
By base change, the LHS can be directly identified with $\Gamma(\bR^M\times \bR_{t_0}, \cG')\cong \cM$, and the RHS can be directly identified with 
\begin{align*}
\Gamma(\{-t_0+\frac{1}{2}A_S(\bq)<0\}, \cG')\cong \Gamma(\{\bq: \frac{1}{2}A_S(\bq)<0\}, \cM_{\bR^M_\bq})
\end{align*}
Thus the fiber is canonically identified with $\mu_{(x_0,\xi_0;-t_0+\frac{1}{2}A_S(\bq))}(\cG')$ as desired. 
\end{proof}

\begin{cor}\label{cor: Morse transf, stratified Morse}
Let $(\bL, (x_0,\xi_0)=(0,0;t_0=0,-dt_0))$ be the germ of a conic open neighborhood of $(x_0,\xi_0)\in \Lambda_\fS^{sm}$. Then for any Morse transformation $\bL_{01}$ defined by (\ref{eq: prop contact A_S}), we have a natural commutative diagram 
\begin{equation*}
\begin{tikzcd}
\Shv_{\fS}(\bR^M\times\bR_{t_0};\Sp)\ar[d, "\mu_{(x_0, \xi_0; f)}"']\ar[r, "\mu"]&(\mathrm{MSh}^{\mathrm{p}}_{\bL})_{(x_0,\xi_0)} \ar[dl, "{\substack{\varinjlim_{\epsilon\to 0^+}\Hom(\cF_{\epsilon,f}\text{, })\\ \  \\ \  \\ \ \\ \ }}" very near start, "\sim"']&\ar[l,"\sim"']\Shv_{\bL}^{0}(\bR^M\times\bR_{t_0};\Sp)\ar[d, "\Psi_{\bL_{01}}", "\sim"']\\
\Sp&\ &\Shv_{\bL_1}^{0}(\bR^M\times\bR_{t_1};\Sp)\ar[ll,"\sim"', "\Upsilon"]
\end{tikzcd}
\end{equation*}
where $\mu$ is the microlocalization functor, 
$f=-t_0+\frac{1}{2}A_S(\bq)$ and 
$\varinjlim_{\epsilon\to 0^+}\Hom(\cF_{\epsilon,f}\text{, })$ is as in Remark \ref{remark: colim, Hom, cFepsilonf}. 
\end{cor}

\begin{proof}
This is a direct consequence of Lemma \ref{lemma: Psi_bL01}, Lemma \ref{lemma: cG, cG', mu} and Remark \ref{remark: colim, Hom, cFepsilonf}. 
\end{proof}

It is clear that one can enlarge the space of $f$ occurring in Corollary \ref{cor: Morse transf, stratified Morse} to the space of $f$ whose graph of differential has tangent space at $(x_0, \xi_0)$ transverse to $T_{(x_0,\xi_0)}\bL$, which yields a homotopy equivalent space of (local) functions.

\begin{remark}\label{remark: Morse transf, stratified Morse}
\item[(i)] Corollary \ref{cor: Morse transf, stratified Morse} establishes the desired bridge between the effect of Morse transformations and stratified Morse functions on quasi-constructible sheaves (i.e. sheaves that are in $\Shv_{\fS}(X;\Sp)$ for some $\fS$), which is clearly compatible with stabilizations.  

\item[(ii)] When $(x,p)\not\in \Lambda_{\fS}^{sm}$, to calculate the correct microlocal stalk, the conditions on the local Morse functions are not sufficient. If $(x,p)\in \Lambda^{sm}$ for some conic Lagrangian $\Lambda$ (e.g. $(x,p)\in \SS(\cF)^{sm}$ for some $\cF\in \Shv_\fS(X;\Sp)$), then the condition on a ``local Morse function" $f$ is that $f(x)=0$ and $\text{Graph}(df)$ intersects $\Lambda^{sm}$ transversely at $(x,p)$. It is easy to see (using the proof of Proposition \ref{prop: classification contact} (a)) that the space of such $f$ is homotopy equivalent to the space of Morse transformations with respect to $(\Lambda^{sm}, (x, p))$ in the obvious way. With a little more work, one gets a direct generalization of Corollary \ref{cor: Morse transf, stratified Morse} in this setting, without the involvement of $\Shv_{\fS}(\bR^M\times\bR_{t_0};\Sp)$. 
\end{remark}

\subsubsection{A reformulation of Proposition \ref{prop: F_Q, p}}

If we drop the condition $s<-\frac{1}{2}|\bq|^2+A(\bq)$ in (\ref{eq: point in Q}), then we can view $G_N\times \bR^M\times \bR_s$ as the underlying space of a $VG_N^\bullet$-module in $\Corr(\Slch)_{\fib,\all}$, and we have a natural morphism of $VG_N^\bullet$-modules 
\begin{align*}
\iota_{N,M}: \widehat{Q}_{N,M}\hookrightarrow G_N\times \bR^M\times \bR.
\end{align*}
Let $\bL_{\widehat{Q}_{N,M}}$ be the negative conormal of the boundary of $\widehat{Q}_{N,M}$ in $G_N\times \bR^M\times \bR_s$. Then $\Shv^0_{\bL_{\widehat{Q}_{N,M}}}(G_N\times \bR^M\times \bR_s;\Sp)$ is equivalent to the essential image of $(\iota_{N,M})_*$. Since $\iota_{N,M}$ is open and for any $(N',M')\geq (N,M)$ the diagram
\begin{align*}
\xymatrix{\widehat{Q}_{N,M}\ar[r]\ar[d]&\widehat{Q}_{N',M'}\ar[d]\\
G_N\times \bR^M\times \bR_s\ar[r]&G_{N'}\times \bR^{M'}\times \bR_s
}
\end{align*}
is Cartesian, a direct application of Proposition \ref{prop: K times L Mod} gives an isomorphism
\begin{align*}
&\varprojlim\limits_{N,M}(\iota_{N,M})_*:\\
&(\Loc(VG;\Sp)^{\otimes_c}, \Loc(\widehat{Q};\Sp))\overset{\sim}{\rightarrow} (\Loc(VG;\Sp)^{\otimes_c}, \varprojlim\limits_{N,M}\Shv^0_{\bL_{\widehat{Q}_{N,M}}}(G_N\times \bR^M\times \bR_s;\Sp))
\end{align*}
in $\cMod^{N(\Fin_*)}(\PrstL)$. 
Similarly, let $\bL_{Q_M^0}$ denote the negative conormal of the boundary of $Q_0^M$. Then we have a canonical equivalence 
\begin{align*}
\Loc(Q^0;\Sp)\overset{\sim}{\rightarrow} \varprojlim\limits_{M} \Shv_{\bL_{Q^0_{M}}}^0(\bR^M\times \bR_s;\Sp).
\end{align*}

Now we can rewrite Proposition \ref{prop: F_Q, p} in the following form. 
\begin{cor}\label{prop: F_Q, p, quotient}
There is a natural isomorphism in $\cMod^{N(\Fin_*)}(\PrstL)$
\begin{align}\label{eq: prop F_Q, p, quotient}
&F_{Q^0}\circ p^*: (\Loc(G;\Sp)^{\otimes_c}, \varprojlim\limits_{N,M} \Shv_{\bL_{\widehat{Q}_{N,M}}}^0(G_N\times \bR^M\times \bR_s;\Sp))\\
\nonumber&\overset{\sim}{\rightarrow} (\Loc(VG;\Sp)^{\otimes_c}, \varprojlim\limits_{N,M} \Shv_{\bL_{\widehat{Q}_{N,M}}}^0(G_N\times \bR^M\times \bR_s;\Sp)). 
\end{align}
that induces the identity functor on the common underlying module category \\
$\varprojlim\limits_{N,M} \Shv_{\bL_{\widehat{Q}_{N,M}}}^0(G_N\times \bR^M\times \bR_s;\Sp)$.
\end{cor}

\subsection{Composite correspondences}
Let $(L,(0,p_0))$ be a germ of smooth Lagrangian in $T^*\bR^{M+k}$, whose tangent space at $(0,p_0)$ intersect the tangent space of the cotangent fiber in dimension $k$, equivalently $\dim\pi_*T_{(0,p_0)}L=M$. We will view $\pi_*T_{(0,p_0)}L\cong \bR^M$ as a subspace of $\bR^{M+k}$ and we will think of $\bR^{M+k}$ as the product $(\pi_*T_{(0,p_0)}L)\times (\pi_*T_{(0,p_0)}L)^{\perp}$.

In the following, to make the exposition simpler, we assume without loss of generality that $p_0=0$ and $p_1=0$. Note from (\ref{eq: prop contact A_S}), the shift of $p_0$ and $p_1$ to $0$ only changes (\ref{eq: prop contact A_S}) by a linear function in $\bq_0, \bq_1$, so it does not cause any essential difference. Assume that we have chosen $A_S$ as in Proposition \ref{prop: classification contact} (a) such that 
\begin{align}\label{eq: basic A_S}
A_S-A_{L_{(0,0)}}=-\frac{1}{2}|\bq_0|^2
\end{align}
 for $\bq_0$ in $\pi_*T_{(0,0)}L$, and let $\bL_{01}$ be the corresponding Morse transformation (i.e. the negative conormal of (\ref{eq: prop contact A_S})).  
It is easy to check that the resulting germ of Lagrangian $(L_1,(0,0))$ has tangent space $T_{(0,0)}L_1$ equal to the tangent space at $(0,0)$ of the graph of the differential of $-\frac{1}{2}|\proj_{\bR^M}\bq_0|^2$, where $\proj_{\bR^M}$ means the orthogonal projection to $\bR^M\cong \pi_*T_{(0,0)}L$. Motivated by Proposition \ref{prop: classification contact} (c), we consider the following diagram of composite correspondences (from right to left):
\begin{align}\label{diagram: composite corr}
\xymatrix@C=0.1em{&&\scriptstyle{X_{012}=VG_N\times\pi(\bL_{01})}\ar[dl]\ar[dr]&&\\
&\scriptstyle{VG_N\times \bR^{M+k}_{\bq_1}\times \bR_{t_1}}\ar[dr]^{p_{0,VG_N}}\ar[dl]_{p_{1,VG_N}}&&\scriptstyle{\pi(\bL_{01})}\ar[dr]^{p_{0,A_S}}\ar[dl]_{p_{1,A_S}}&\\
\scriptstyle{G_N\times \bR^{M+k}_{\widetilde{\bq}_1}\times \bR_{\widetilde{t}_1}}&&\scriptstyle{\bR^{M+k}_{\bq_1}\times \bR_{t_1}}&&\scriptstyle{\bR^{M+k}_{\bq_0}\times \bR_{t_0}},
}
\end{align}
where $p_{0,A_S}, p_{1,A_S}$ are the natural projections, and
\begin{align*}
&p_{0,VG_N}: (A, \bp, \bq_1, t_1)\mapsto (\bq_1, t_1),\\
&p_{1,VG_N}: (A, \bp, \bq_1,t_1)\mapsto (A, \widetilde{\bq}_1=\bq_1+2\bp, \widetilde{t}_1=t_1+A(\bp)).  
\end{align*} 

Let 
\begin{align*}
&\widetilde{\pi}_{N,M}: X_{012}=VG_N\times\pi(\bL_{01})\rightarrow G_N\times \bR^{M+k}\times \bR_{t_0}\times \bR^{M+k}\times \bR_{\widetilde{t}_1}\\
&(A, \bp, t_0,t_1, \bq_0,\bq_1)\mapsto (A, t_0,\widetilde{t}_1=t_1+A(\bp), \bq_0, \widetilde{\bq}_1=\bq_1+2\bp)
\end{align*}
be the projection. 
Let $H_{A_S,G_N}$ denote the hypersurface
\begin{align}\label{H_A_S, G_N}
\widetilde{t}_1-t_0+(\frac{1}{2}A_S+A)(\bq_0)-\widetilde{\bq}_1\cdot\bq_0=0, A\in G_N
\end{align}
in $G_N\times \bR^{M+k}\times \bR_{t_0}\times \bR^{M+k}\times \bR_{\widetilde{t}_1}$, and consider the (global) correspondence 
\begin{align}\label{eq: corresp H_A_S, G_N}
\xymatrix{&H_{A_S,G_N}\ar[dr]^{p_{0,H_{A_S,G_N}}}\ar[dl]_{p_{1,H_{A_S,G_N}}}&\\
G_N\times \bR^{M+k}_{\widetilde{\bq}_1}\times \bR_{\widetilde{t}_1}&&\bR^{M+k}_{\bq_0}\times \bR_{t_0},
}
\end{align}
where $p_{0,H_{A_S,G_N}}$ and $p_{1,H_{A_S,G_N}}$ are the obvious projections. 
Let $\bL_{1,G_N}$ be the conic Lagrangian in $T^{*,<0}(G_N\times \bR^{M+k}\times \bR_{\widetilde{t}_1})$ assembled from the family of (germs of) conic Lagrangians in $T^{*,<0}(\bR^{M+k}\times \bR_{\widetilde{t}_1})$ which are the image of $\bL$ under the family of Morse transformations given by the negative conormal of (\ref{H_A_S, G_N}) over $A\in G_N$.

\begin{lemma}\label{lemma: tilde pi N,M}
We have $\widetilde{\pi}_{N,M}$ is proper, and 
\begin{align*}
(\widetilde{\pi}_{N,M})_!\bS_{X_{012}}\simeq \bS_{H_{A_S,G_N}}
\end{align*}
in 
\begin{align*}
\Shv(G_N\times \bR_{\widetilde{\bq}_1}^{M+k}\times \bR_{\widetilde{t}_1}\times \bR^{M+k}_{\bq_0}\times \bR_{t_0};\Sp)/\Shv_{\leq 0}(G_N\times \bR_{\widetilde{\bq}_1}^{M+k}\times \bR_{\widetilde{t}_1}\times \bR^{M+k}_{\bq_0}\times \bR_{t_0};\Sp).
\end{align*}
In particular, we have a canonical isomorphism of functors 
\begin{align*}
&(p_{1,H_{A_S,G_N}})_*p_{0,H_{A_S,G_N}}^!\simeq (p_{1,VG_N})_*p_{0,VG_N}^!(p_{1,A_S})_*p_{0,A_S}^!:\\
&\Shv_{\bL}^0(\bR^{M+k}_{\bq_0}\times \bR_{t_0})\rightarrow \Shv_{\bL_{1,G_N}}^0(G_N\times \bR_{\widetilde{\bq}_1}^{M+k}\times \bR_{\tilde{t}_1}). 
\end{align*}
\end{lemma}
\begin{proof}
We look at the image and the fibers of $\widetilde{\pi}_{N,M}$. Fixing $(A,t_0,\widetilde{t}_1, \bq_0, \widetilde{\bq}_1)$ in the target, the fiber over it consists of points satisfying the equation
\begin{align*}
&\widetilde{t}_1-t_0+\frac{1}{2}A_S(\bq_0)-\widetilde{\bq}_1\cdot\bq_0+2\bp\cdot \bq_0-A(\bp)=0\\
\Leftrightarrow&|\bp-\bq_0^A|^2=\widetilde{t}_1-t_0+\frac{1}{2}A_S(\bq_0)-\widetilde{\bq}_1\cdot\bq_0+|\bq_0^A|^2.
\end{align*}
So the image of $\widetilde{\pi}_{N,M}$ is along 
\begin{align}\label{eq: image tilde pi}
\widetilde{t}_1-t_0+\frac{1}{2}A_S(\bq_0)-\widetilde{\bq}_1\cdot\bq_0+|\bq_0^A|^2\geq 0,
\end{align}
and the fiber is a point (resp. a sphere) when (\ref{eq: image tilde pi}) is an equality (resp. inequality). Hence the lemma easily follows.
\end{proof}

Note that $(\bL_1, (0,0;t_1=0,-dt_1))$ has the same tangent space as $\bL_{Q_M^0}\times T_{\bR^k}^*\bR^{k}$ at $(0,0;t_1=0,-dt_1)$, thus we have a commutative diagram
\begin{align}\label{diagram: L_1, L_Q, N, M}
\xymatrix{
\Shv_{\bL_1}^0(\bR_{\bq_1}^{M+k}\times\bR_{t_1};\Sp)\ar[d]_{(p_{1,VG_N})_*p_{0,VG_N}^!}\ar[r]^\sim&\Shv_{\bL_{Q_M^0}}^0(\bR^{M}\times\bR;\Sp)\ar[d]^{(\psi_{N,M})_*\pi_{Q^0_M}^!}\\
\Shv_{\bL_{1,G_N}}^0(G_N\times \bR^{M+k}_{\widetilde{\bq}_1}\times \bR_{\widetilde{t}_1};\Sp)\ar[r]^\sim& \Shv_{\bL_{\widehat{Q}_{N,M}}}^0(G_N\times \bR^M\times \bR;\Sp),
}
\end{align}
which represents a \emph{canonical} isomorphism between the left vertical arrow and the right vertical arrow as objects in $\Fun(\Delta^1, \PrstL)$. 

\subsection{The process of stabilization}
Now consider the family of (\ref{diagram: composite corr}) over $(\bN\times\bN)^{\geq \dgnl}_{(N,M)/}$:\\
\begin{itemize}
\item[(i)] For each $(N',M')\geq (N,M)$, set
\begin{align*}
&L^{M'}=L\times T_{\bR^{M'-M}}^*\bR^{M'-M}\subset T^*\bR^{M'+k},\\
&A^{M'}_S(\bq_0)=A_S(\proj_{\bR^{M+k}}\bq_0)-|\proj_{\bR^{M'-M}}\bq_0|^2.
\end{align*} 
and let 
\begin{align*}
\bL_{01}^{M'}\subset (T^{*,<0}(\bR_{\bq_0}^{M'+k}\times \bR_{t_0}))^-\times T^{*,<0}(\bR_{\bq_1}^{M'+k}\times \bR_{t_1})
\end{align*}
be the negative conormal bundle of 
\begin{align*}
t_1-t_0+\frac{1}{2}A_S^{M'}(\bq_0)-\bq_0\cdot\bq_1=0.
\end{align*} 
The factors  $\bR^{M+k}$, $VG_N$ and $G_N$  in the entries of (\ref{diagram: composite corr}) are replaced by $\bR^{M'+k}$, $VG_{N'}$ and $G_{N'}$, respectively. The morphisms $p_{i,VG_N}$ and $p_{i,A_S}$ in the diagram are replaced by their obvious extensions, denoted by $p_{i,VG_{N'}}$ and $p_{i,A^{M'}_S}$ respectively.

\item[(ii)] For each morphism $(N_1,M_1)\leq (N_2,M_2)$ in $(\bN\times\bN)^{\geq \dgnl}_{(N,M)/}$, the connecting morphisms between the corresponding entries in the diagrams (\ref{diagram: composite corr}) are the obvious embeddings induced from 
\begin{align*}
\bR^{M_1+k}\hookrightarrow \bR^{M_2+k},\ VG_{N_1}\hookrightarrow VG_{N_2},\ G_{N_1}\hookrightarrow G_{N_2}. 
\end{align*}
\end{itemize}

First, for each $(N_1,M_1)\leq (N_2, M_2)$, the diagram
\begin{align*}
\xymatrix{VG_{N_1}\times \bR_{\bq_1}^{M_1+k}\times\bR_{t_1}\ar[d]_{p_{1,VG_{N_1}}}\ar@{^{(}->}[r]&VG_{N_2}\times \bR_{\bq_1}^{M_2+k}\times\bR_{t_1}\ar[d]^{p_{1,VG_{N_2}}}\\
G_{N_1}\times\bR_{\widetilde{\bq}_1}^{M_1+k}\times\bR_{\widetilde{t}_1}\ar@{^{(}->}[r]&G_{N_2}\times\bR_{\widetilde{\bq}_1}^{M_2+k}\times\bR_{\widetilde{t}_1}
}
\end{align*}
is Cartesian.

Second, we look at the following diagram for $(N_1,M_1)\leq (N_2,M_2)$: 
\begin{align}\label{diagram: deficiency Cart 1}
\xymatrix{\pi(\bL^{M_1}_{01})\ar@{^{(}->}[drr]^{\iota_{01,M_1}}\ar[ddr]_{p_{1,A^{M_1}_S}}\ar[dr]_{\iota_{M_1}}&&\\
&Y_{M_1,M_2}\ar[r]\ar[d]\pb&\pi(\bL^{M_2}_{01})\ar[d]^{p_{1,A^{M_2}_S}}\\
&\bR_{\bq_1}^{M_1+k}\times \bR_{t_1}\ar@{^{(}->}[r]_{\jmath_{\bq_1,M_1}}&\bR_{\bq_1}^{M_2+k}\times \bR_{t_1}
}
\end{align}
in which 
\begin{align*}
Y_{M_1,M_2}&=\{t_1-t_0+\frac{1}{2}A_S^{M_1}(\proj_{\bR^{M_1+k}}\bq_0)-\frac{1}{2}|\proj_{\bR^{M_2-M_1}}\bq_0|^2-\bq_1\cdot \proj_{\bR^{M_1+k}}\bq_0=0\}\\
&\subset \bR^{M_2+k}_{\bq_0}\times \bR^{M_1+k}_{\bq_1}\times \bR_{t_0}\times\bR_{t_1},
\end{align*}
and $\iota_{M_1}: \pi(\bL^{M_1}_{01})\rightarrow Y_{M_1,M_2}$ represents the deficiency of the outer square from being Cartesian. 
Since the projection
\begin{align*}
\widetilde{\pi}_{Y, M_1,M_2}: &Y_{M_1,M_2}\rightarrow \bR_{\bq_0}^{M_1+k}\times \bR_{q_1}^{M_1+k}\times \bR_{t_0}\times \bR_{t_1}\\
&(\bq_0,\bq_1,t_0,t_1)\mapsto (\proj_{\bR^{M_1+k}}\bq_0,\bq_1, t_0,t_1)
\end{align*}
behaves similarly as $\widetilde{\pi}_{N,M}$ (see the proof of Lemma \ref{lemma: tilde pi N,M}) in the sense that its image is 
\begin{align*}
\{t_1-t_0+\frac{1}{2}A_S^{M_1}(\bq_0)-\bq_1\cdot\bq_0=r,r\geq 0\}\subset \bR_{\bq_0}^{M_1+k}\times \bR_{\bq_1}^{M_1+k}\times \bR_{t_0}\times \bR_{t_1},
\end{align*}
with fiber either a point or a sphere given by $\frac{1}{2}|\proj_{\bR^{M_2-M_1}}\bq_0|^2=r$, 
we have 
\begin{align*}
(\widetilde{\pi}_{Y, M_1,M_2})_!\bS_{Y_{M_1,M_2}}\simeq \bS_{\pi(\bL_{01}^{M_1})}
\end{align*}
in $\Shv(\bR_{\bq_0}^{M_1+k}\times \bR_{q_1}^{M_1+k}\times \bR_{t_0}\times \bR_{t_1};\Sp)/\Shv_{\leq 0}(\bR_{\bq_0}^{M_1+k}\times \bR_{q_1}^{M_1+k}\times \bR_{t_0}\times \bR_{t_1};\Sp)$. 
Therefore, we have a canonical isomorphism of functors
\begin{align*}
&(p_{1,A_S^{M_1}})_*\iota_{01,M_1}^!\simeq \jmath_{\bq_1,M_1}^!(p_{1,A_S^{M_2}})_*:\\
&\Loc(\pi(\bL_{01}^{M_2});\Sp)\simeq \Shv_{\bL_{01}^{M_2}}^0(\bR_{\bq_0}^{M_2+k}\times\bR_{\bq_1}^{M_2+k}\times\bR_{t_0}\times\bR_{t_1};\Sp)\rightarrow\Shv_{\bL_1^{M_1}}^0(\bR_{\bq_1}^{M_1+k}\times\bR_{t_1};\Sp).
\end{align*}

Similarly, we can consider the family of diagrams (\ref{eq: corresp H_A_S, G_N}) over $(\bN\times \bN)^{\geq\dgnl}_{(N,M)/}$, where for each $(N',M')$, $H_{A_S, G_N}$ is replaced by $H_{A_S^{M'},G_{N'}}$ (defined in (\ref{H_A_S, G_N})),  $\bR_{\widetilde{\bq}_1}^{M+k}$ is replaced by  $\bR_{\widetilde{\bq}_1}^{M'+k}$, and $G_N$ is replaced by $G_{N'}$. The connecting morphisms over $(N_1,M_1)\rightarrow (N_2, M_2)$ are directly induced from the embeddings $G_{N_1}\hookrightarrow G_{N_2}$ and $\bR^{M_1+k}\hookrightarrow \bR^{M_2+k}$. Similar to (\ref{diagram: deficiency Cart 1}), the diagram
\begin{align*}
\xymatrix{H_{A_S^{M_1}, G_{N_1}}\ar@{^{(}->}[r]^{\iota_{H,M_1,M_2}}\ar[d]_{p_{1, H_{\scriptscriptstyle{A_S^{M_1}, G_{N_1}}}}}&H_{A_S^{M_2}, G_{N_2}}\ar[d]^{p_{1, H_{\scriptscriptstyle{A_S^{M_2}, G_{N_2}}}}}\\
G_{N_1}\times\bR^{M_1+k}_{\widetilde{\bq}_1}\times \bR_{\widetilde{t}_1}\ \ \ \ \ \ \ar@{^{(}->}[r]_{\scriptstyle{\jmath_{G, N_1, N_2}}}&G_{N_2}\times\bR^{M_2+k}_{\widetilde{\bq}_1}\times \bR_{\widetilde{t}_1}
}
\end{align*}
is not Cartesian, but the deficiency of it from being Cartesian induces an invertible 2-morphism on the localized sheaf categories, hence we have a canonical isomorphism of functors 
\begin{align*}
&(p_{1,H_{\scriptscriptstyle{A_S^{M_1}}, G_{N_1}}})_*\iota_{H,M_1,M_2}^!\simeq  \jmath_{G,N_1,N_2}^!(p_{1,H_{\scriptscriptstyle{A_S^{M_2}}, G_{N_2}}})_*: \\
&\Loc(H_{A_S^{M_2}, G_{N_2}};\Sp)\rightarrow \Shv_{\bL_{\widehat{Q}_{N_2, M_2}}}^0(G_{N_2}\times\bR^{M_2+k}_{\widetilde{\bq}_1}\times \bR_{\widetilde{t}_1}; \Sp). 
\end{align*}

Now applying the dual version of $\ShvSp_{\all,\all}^{\propmap}$ (whose target is $\bPrstR$) to the family of diagrams over $(\bN\times \bN)^{\geq \dgnl}_{(N,M)/}$, passing to the localized sheaf categories and taking limits over $(\bN\times\bN)^{\geq \dgnl}_{(N,M)/}$, we get a canonical isomorphism of functors 
\begin{align*}
&(p_{1,H_{A^{\st}_S,G}})_*p_{0,H_{A^{\st}_S,G}}^!\simeq (p_{1,VG})_*p_{0,VG}^!(p_{1,A^{\st}_S})_*p_{0,A^{\st}_S}^!:\\
&\varprojlim\limits_{M'}\Shv_{\bL^{M'}}^0(\bR^{M'+k}_{\bq_0}\times \bR_{t_0};\Sp)\rightarrow \varprojlim\limits_{N',M'}\Shv_{\bL_{1,G_{N'}}}^0(G_{N'}\times \bR_{\widetilde{\bq}_1}^{M'+k}\times \bR_{\tilde{t}_1};\Sp). 
\end{align*}

Lastly, using diagram (\ref{diagram: L_1, L_Q, N, M}) and 
 Corollary \ref{cor: equiv J-equivariant}, we immediately get the following: 
\begin{thm}\label{thm: L, J-equivariant}
The correspondences (\ref{eq: corresp H_A_S, G_N}) for all $M'\geq N'$ give rise to a canonical equivalence
\begin{align*}
\varprojlim\limits_{M'}\Shv_{\bL^{M'}}^0(\bR^{M'+k}\times \bR_{t_0};\Sp)\overset{\sim}{\longrightarrow} &\Loc(\widehat{Q};\Sp)^J\simeq \\
&(\varprojlim\limits_{N,M} \Shv_{\bL_{\widehat{Q}_{N,M}}}^0(G_N\times \bR^M\times \bR_s;\Sp))^J.
\end{align*}
\end{thm}

\subsection{Compatibility among different choices of $A_S$}\label{subsection: compatibility A_S}
In this subsection, we briefly discuss the compatibility of the above results for the particular choice of $A_S$ (\ref{eq: basic A_S}), especially Theorem \ref{thm: L, J-equivariant}, with other choices of $A_S$ by adding to it some $A\in G_N$. We will go into more details of this in \cite{J-J}. 

For any $A\in G_N$, we define the $\Fin_*$-object $(G_N^{A^\perp})^\bullet$  in $\Slch$ to be 
\begin{align}\label{eq: G_N, Aperp, n}
(G_N^{A^\perp})^{\lng n\rng}=\{(A_1,\cdots, A_n)\in G_N^{\lng n\rng}: A_i\perp A, i=1,\cdots, n\}.
\end{align} 
Similarly, we define
\begin{align*}
(VG_N^{A^\perp})^\bullet=(VG_N)^\bullet\underset{G_N^\bullet}{\times} (G_N^{A^\perp})^\bullet.
\end{align*}
It is clear that the obvious inclusions
\begin{align*}
&\iota_{G_N^{A^\perp}}: (G_N^{A^\perp})^\bullet\hookrightarrow G_N^\bullet \\
&\iota_{VG_N^{A^\perp}}: (VG_N^{A^\perp})^\bullet\hookrightarrow VG_N^\bullet
\end{align*}
thought as correspondences 
\begin{align*}
&(G_N^{A^\perp})^\bullet\overset{id}{\leftarrow}(G_N^{A^\perp})^\bullet\hookrightarrow G_N^\bullet \\
&(VG_N^{A^\perp})^\bullet\overset{id}{\leftarrow}(VG_N^{A^\perp})^\bullet\hookrightarrow VG_N^\bullet
\end{align*}
induce morphisms in $\CAlg(\Corr(\Slch)_{\fib,\all})$, which induce a commutative diagram of equivalences
\begin{align*}
\xymatrix{
\Loc(VG;\Sp)^{\otimes_c}\ar[r]^{\iota_{VG^{A^\perp}}^!\ \ \ \ \ \ \ \ \ \ \ \ \ \ }\ar[d]_{\pi_*} &\Loc(VG^{A^\perp};\Sp)^{\otimes_c}:=\varprojlim\limits_{N}\Loc(VG_N^{A^\perp};\Sp)^{\otimes_c}\ar[d]^{\pi_*^{A^\perp}}\\
\Loc(G;\Sp)^{\otimes_c}\ar[r]^{\iota_{G^{A^\perp}}^!\ \ \ \ \ \ \ \ \ \ \ \ \ \ \ } &\Loc(G^{A^\perp};\Sp)^{\otimes_c}:=\varprojlim\limits_{N}\Loc(G_N^{A^\perp};\Sp)^{\otimes_c}.\\
}
\end{align*}

Let 
\begin{align*}
&\widehat{Q}_{N,M}^{A^{\perp}}=\{(A_1,\bq,s)\in G_N^{A^\perp}\times \bR^M\times \bR_s: s<-\frac{1}{2}|\bq|^2+(A_1+A)(\bq)\},\\
&Q_M^{0,A}=\{(\bq,s)\in \bR^M\times\bR_s: s<-\frac{1}{2}|\bq|^2+A(\bq)\}. 
\end{align*}

Consider the following Cartesian diagram 
\begin{align}\label{diagram: VGA_perp}
\xymatrix{VG^{A^\perp}_N\times A\times Q^0_M\ar@{^{(}->}[r]^{\jmath_{M,A}^0}\ar[d]_{\psi_{M,A}}&VG_N\times Q^0_M\ar[d]^{\psi_M}\\
\widehat{Q}_{N,M}^{A^\perp}\ar@{^{(}->}[r]^{\widehat{\jmath}_{M,A}}&\widehat{Q}_{N,M}
}, 
\end{align}
where 
\begin{align*}
&\jmath^0_{M,A}: (A_1,\bp_1,\bp\in A, \bq,s)\mapsto (A_1, \bp_1, \bq,s),\\
&\widehat{\jmath}_{M,A}: (\bp, A_1, \bq, s)\mapsto (A_1\oplus A, \bq, s+A(\bp))\\
&\psi_{M,A}: (A_1,\bp_1,\bp\in A, \bq,s)\mapsto (\bp,A_1\oplus A, \bq+\bp_1, s+A_1(\bp_1)).
\end{align*}
We view $VG^{A^\perp}_N\times A\times Q^0_M$ (resp. $VG_N\times Q_M^0$) as a free module of $(VG^{A^\perp}_N)^\bullet$ (resp. $VG_N^\bullet$) in $\Corr(\Slch)_{\fib,\all}$ generated by $A\times Q^0_M$ (resp. $Q_M^0$), and $\widehat{Q}^{A^\perp}_{N,M}$ (resp. $\widehat{Q}_{N,M}$) as a module of $(VG^{A^\perp})^\bullet$ (resp. $VG^\bullet$) as in Lemma \ref{lemma: hat Q module}. The inductive system of (\ref{diagram: VGA_perp}) under inclusions over $(\bN\times \bN)^{\geq \dgnl}$ gives rise to a functor $[1]\times [1]^{op}\rightarrow \Fun(I_{\pi_\dagg},\cC)$ (with $[1]^{op}$ corresponding to the horizontal arrows) that satisfies the properties in Proposition \ref{prop: K times L Mod}, hence applying the symmetric monoidal functor $\Loc_*^!$ to the diagram and taking limit over $(\bN\times \bN)^{\geq \dgnl}$, we get a commutative diagram in $\cMod^{N(\Fin_*)}(\PrstL)$:
\begin{align*}
\xymatrix{(\Loc(VG;\Sp)^{\otimes_c}, \Loc(VG\times Q^0))\ar[r]^{(\jmath^0_A)^!\ \ \ \ \ \ }\ar[d]_{(\psi_{A})_*}&(\Loc(VG^{A^\perp};\Sp)^{\otimes_c}, \Loc(VG^{A^\perp}\times A\times Q^0))\ar[d]^{\psi_*}\\
(\Loc(VG;\Sp)^{\otimes_c},\Loc(\widehat{Q}))\ar[r]^{\widehat{\jmath}_A^!}&(\Loc(VG^{A^\perp};\Sp)^{\otimes_c},\Loc(\widehat{Q}^{A^\perp};\Sp))
}.
\end{align*}
This canonically induces  a commutative diagram of isomorphisms in $\cMod^{N(\Fin_*)}(\PrstL)$
\begin{align*}
\xymatrix@C=0.6em{(\substack{\Loc(VG;\Sp)^{\otimes_c},\\
\Loc(VG;\Sp)^{\otimes_c}\otimes \Loc(Q^0;\Sp)})\ar[r]\ar[d]&(\substack{\Loc(VG^{A^\perp};\Sp)^{\otimes_c},\\
\Loc(VG^{A^\perp};\Sp)^{\otimes_c}\otimes \Loc(A\times Q^{0};\Sp)})\ar[d]\\
(\Loc(VG;\Sp)^{\otimes_c}, \Loc(\widehat{Q};\Sp))\ar[r]&(\Loc(VG^{A^\perp};\Sp)^{\otimes_c}, \Loc(\widehat{Q}^{A^\perp};\Sp))
},
\end{align*}
which further induces a commutative diagram of equivalences
\begin{align*}
\xymatrix{\varprojlim\limits_{M'}\Shv_{\bL^{M'}}^0(\bR^{M'+k}\times \bR_{t_0};\Sp)\ar[r]^{\ \ \ \ \ \ \ \ \ \sim}&\Loc(Q^0;\Sp)\ar[r]^\sim\ar[dr]_{\sim}&\Loc(\widehat{Q};\Sp)^J\ar[d]^{\sim}\\
&&\Loc(\widehat{Q}^{A^\perp};\Sp)^J
}.
\end{align*}

\appendix
\section{The Thom construction via correspondences}

The goal of the appendix is to give a realization of the Thom construction using the category of correspondences, and to conclude with a proof of Proposition \ref{prop: equiv model of J} that the proposed model of the $J$-homomorphism from correspondences is equivalent to a standard model.

\subsection{$\cV\Spc$ and the Thom construction}\label{subsec: cVSpc}
Let $\cV\Spc$ be the $\infty$-category of pairs $(V, X)$, where $V$ is a (real finite dimensional) vector bundle on $X\in \Spc$, and any morphism $(V, X)$ to $(V', Y)$ is given by $f: X\rightarrow Y$ and a vector bundle isomorphism $V\overset{\sim}{\rightarrow} f^*V'$ over $X$. Equivalently (and more precisely), $\cV\Spc$ is the right fibration (in particular the Cartesian fibration) over $\Spc$ from the unstraightening of the natural functor $\bV: \Spc^{op}\to \Gpd_\infty\footnote{Here $\Gpd_\infty\simeq \Spc$. We write the target as $\Gpd_\infty$ to make the exposition clearer.}$, $X\mapsto \Maps_{\Spc}(X,\coprod_nBO(n))$. Equivalently, $\cV\Spc\simeq \Spc_{/\coprod_nBO(n)}$. Clearly, $\cV\Spc$ has the natural (non-Cartesian) symmetric monoidal structure given by taking product of vector bundles, which corresponds to the right-lax symmetric monoidal functor $\bV$ using the commutative topological monoid structure on $G=\coprod_nBO(n)$.

Let $\bP: \Spc^{op}\to \Gpd_\infty, X\mapsto  \Maps(X,\Pic(\bS))$. 
The Thom construction naturally defines a symmetric monoidal functor 
\begin{align*}
\bT^{\inv}: \cV\Spc^{op}\to (\Gpd_\infty)^{*/}, (V,X)\mapsto (\Maps(X, \Pic(\bS)), \cL_V), 
\end{align*}
where $\cL_V$ is the local system of $\bS$-lines corresponding to the stable sphere fibration associated to $V$ (which is equivalent to the $!$-pullback of the ``universal" stable sphere fibration over $G$ classified by $J$). 
Equivalently, using straightening for Cartesian fibrations over $\Spc$, $\bT^{\inv}$ is corresponding to the natural morphism $\bV\rightarrow \bP$ in $\Fun^{\rightlax}((\Spc^\times)^{op}, \Gpd_\infty^\times)$, induced by the $J$-homomorphism $\coprod_{n}BO(n)\rightarrow \Pic(\bS)$ (as an $E_\infty$-map).
We will mostly consider the following version: 
\begin{align}\label{eq: bT}
\bT: \cV\Spc^{op}\to (\PrstL)^{\Sp/}, (V,X)\mapsto (\Loc(X;\Sp), \bS\mapsto \cL_V),
\end{align}
which of course contains all the information about $\bT^{\inv}$.

\subsection{A realization of the Thom construction using correspondences}\label{subsec: Thom, corr}
Let $\Top$ be the ordinary 1-category of all CW-complexes.  
Let $\VSlch$ be the ordinary 1-category of pairs $(V, X)$, where $X\in \Slch$ and $V$ is a vector bundle over $X$, and a morphism $(V, X)\rightarrow (V',Y)$ is given by a continuous map $f: X\rightarrow Y$ together with a vector bundle isomorphism $\varphi: V\rightarrow f^*V'$ over $X$. It is clear that $\VSlch\rightarrow \Slch$ is a Cartesian fibration with fibers discrete groupoids. Similarly, we define $\VTop$ over $\Top$.

\subsubsection{Definition of $\bT_{\VTop}$.}

We first define a symmetric monoidal functor $\bT_{\VTop}: \VTop^{op}\to (\PrstL)^{\Sp/}$ using correspondences. We will prove that after passing to a natural localization of $\VTop$, $\bT_{\VTop}$ becomes isomorphic to $\bT$. 
We record the following useful statement. 
 
\begin{lemma}\label{lemma: CorrFunKcC}
For any symmetric monoidal $\infty$-category $\cC^\otimes$ (which admits finite limits) and any $\infty$-category $K$, there is a natural symmetric monoidal functor 
\begin{align}\label{eq: lemma CorrFunScC}
F_{K, \cC}: \Corr(\Fun(K, \cC)^{\all}_{\all, \Cart})^{\otimes}\to (\Fun(K, \Corr(\cC)^{\all}_{\all, \all}))^\otimes. 
\end{align}
where $\Cart$ is the class of morphisms $K\times[1]\to \cC$ such that for every functor $[1]\to K$, the composition $[1]\times[1]\to \cC$ is a Cartesian square in $\cC$. 
\end{lemma}
\begin{proof}
A symmetric monoidal functor (\ref{eq: lemma CorrFunScC}) is equivalent to a morphism 
\begin{align*}
K\to \Fun_{\CAlg(\TwoCat^\times)}(\Corr(\Fun(K, \cC)_{\all, \Cart}^{\all})^{\otimes},  (\Corr(\cC)_{\all, \all}^{\all})^{\otimes}).
\end{align*}
This boils down to defining an appropriate functor in $\OneCat^{\Delta^{op}}$ over $\Seq_\bullet N(\Fin_*)$
\begin{align*}
\Seq_\bullet K\times \Seq_\bullet \Corr(\Fun(K, \cC)_{\all, \Cart}^{\all})^{\otimes, N(\Fin_*)}\to \Seq_\bullet (\Corr(\cC)_{\all,\all}^{\all})^{\otimes, N(\Fin_*)}.
\end{align*}
Using (\ref{eq: Seq_kbCorr, otimes}) and that
\begin{align*}
\Fun(K,\cC)^{\otimes,\Fin_*^{op}}\simeq \Fun(K,\cC^{\otimes, \Fin_*^{op}})\underset{\Fun(K, N(\Fin_*)^{op})}{\times} N(\Fin_*)^{op},
\end{align*} 
this amounts to defining an appropriate functor
\begin{align}\label{eq: K,Fun''Fin}
\Maps([\bullet], K)\times \Fun''(([\bullet]\times[\bullet]^{op})^{\geq \dgnl}\times K,\cC^{\otimes, \Fin_*^{op}})\to \Fun''([\bullet]\times[\bullet]^{op})^{\geq \dgnl}, \cC^{\otimes, \Fin_*^{op}}),
\end{align}
where 
\begin{align*}
&\Fun''(([\bullet]\times[\bullet]^{op})^{\geq \dgnl}\times K,\cC^{\otimes, \Fin_*^{op}}):=\\
&\Fun(K, \Fun''(([\bullet]\times[\bullet]^{op})^{\geq \dgnl},\cC^{\otimes, \Fin_*^{op}}))\underset{\Fun(K\times ([\bullet]\times[\bullet]^{op})^{\geq \dgnl}, N(\Fin^*)^{op})}{\times} \Fun([\bullet]^{op}, N(\Fin^*)^{op}).
\end{align*}
There is an obvious candidate induced from the natural map
\begin{align}\label{eq: evaluation, K}
\Maps([\bullet], K)\times ([\bullet]\times[\bullet]^{op})^{\geq \dgnl}\to K\times ([\bullet]\times[\bullet]^{op})^{\geq \dgnl}
\end{align}
that is doing the obvious evaluation map (after doing the projection $([\bullet]\times[\bullet]^{op})^{\geq \dgnl}\to [\bullet]$) and projection to the second factor. It is clear that if we use (\ref{eq: evaluation, K}) and restrict to the full subcategory of the second factor of the LHS of (\ref{eq: K,Fun''Fin}) defined by the condition that:
\begin{itemize}
\item[] for every square $\Delta_{hor}^1\times \Delta^1_K$ in $\{s\}\times [n]^{op}\times K$, it satisfies the (Obj) condition in Section \ref{subsec: monoidal Corr} with $\Delta_{vert}^1$ replaced by $\Delta^1_K$. 
\end{itemize} 
then we get the desired functor. But that full subcategory is exactly given by restricting the horizontal arrows to $\Cart$. 
\end{proof}

\begin{example}\label{example, F_Delta1, cC}
If $K=\Delta^1$, then on the object and 1-morphism level, $F_{\Delta^1, \cC}$ sends
\begin{equation*}
\text{(Obj)}:\ \begin{tikzcd}
x\ar[d]\\
y
\end{tikzcd}\mapsto \begin{tikzcd}
x\ar[d]\ar[r, "id_x"]&x\\
y
\end{tikzcd}
\end{equation*}
\begin{equation*}
\text{(1-Mor)}:\ \begin{tikzcd}
x_{11}\ar[d]&x_{01}\ar[d]\ar[r]\ar[l]&x_{00}\ar[d]\\
y_{11}&y_{01}\ar[r]\ar[l]&y_{00}\arrow[ul, phantom, "\lrcorner", very near start]
\end{tikzcd}\mapsto 
\begin{tikzcd}
x_{11}&x_{01}\ar[r]\ar[l]&x_{00}\\
x_{11}\ar[d]\ar[u, "id_{x_{11}}"]&x_{01}\ar[d]\ar[u, "id_{x_{01}}"]\ar[r]\ar[l]&x_{00}\ar[d]\ar[u,"id_{x_{00}}"]\\
y_{11}&y_{01}\ar[r]\ar[l]&y_{00}\arrow[ul, phantom, "\lrcorner", very near start]
\end{tikzcd},
\end{equation*}
where the right diagram in (1-Mor) represents a $\Delta^1\times\Delta^1$ in $\Corr(\cC)_{\all,\all}^{\all}$, with the outer boundary giving the two standard horns in the boundary of $\Delta^1\times\Delta^1$. 
The assignment from $F_{\Delta^1, \cC}$ on 2-morphisms is the obvious one. 

\end{example}

Recall for an $(\infty,2)$-category $\bC$ and an object $c\in \bC$, the right-lax coslice $(\infty,2)$-category $\bC^{c//}$ is defined as $\Fun([1], \bC)_{\rightlax}\underset{\Fun(\{0\},\bC)}{\times}\{c\}$. If $\bC$ is an ordinary $2$-category, then $\bC^{c//}$ is explicitly given by 
\begin{itemize}
\item An object is given by a pair $(x, \phi)$: $x\in \bC$ and a 1-morphism $\phi: c\to x $;\\
\item A 1-morphism from $(x, \phi)$ to $(y, \eta)$ is given by $(f, \alpha)$
\begin{equation*}
\begin{tikzcd}
c\ar[r, "\phi"]\ar[dr, "\eta"', " "{name=U, below}]&x\ar[d, "f"]\ar[Rightarrow, to=U, "\alpha"']\\
\ & y
\end{tikzcd}
\end{equation*}

\item A 2-morphism from $(f,\alpha)$ to $(g,\beta)$ is given by a 2-morphism $\nu: f\to g$ such that $\alpha=\beta\circ((\nu)\circ \phi)$
\begin{equation*}
\begin{tikzcd}[column sep=4em]
c\ar[r, "\phi\ \ \ \ \ \ "]\ar[dr, "\eta"', " "{name=U, below}]&x\ar[Rightarrow, to=U, bend right=100, "\beta"']\ar[d, "f"', ""{name=W, right}]\ar[Rightarrow, to=U, "\alpha"']\ar[d, bend left=90, "g", ""{name=X, left}]\ar[Rightarrow, from=W, to=X, "\nu"]\\
\ & y
\end{tikzcd}
\end{equation*}

\end{itemize}

In the following, for any class of 1-morphisms $vert$ in $\cC$ and any $(\infty,1)$-category $K$, we will denote the class of morphisms $\Delta^1\to \Fun(K,\cC)$  with $\Delta^1\times\{k\}, k\in K$ all in $vert$ as $vert'$ in $\Fun(K, \cC)$. If $\cC=\Slch, \Top, \cV\Spc$, then we let $vert'$ be the class of morphisms in  $\VSlch, \VTop, \cV\Spc$, respectively, whose image under the projection to $\cC$ are in $vert$. Assume the classes of morphisms $(vert, horiz=\all, adm=vert)$ satisfy \cite[Chapter 7, 1.1.1]{Nick} and $vert$ is closed under the tensor product functor on $\cC$. Then it is clear from Lemma \ref{lemma: CorrFunKcC} that (\ref{eq: lemma CorrFunScC}) restricts to a symmetric monoidal functor (which by some abuse of notations, we still denote by $F_{K,\cC}$)
\begin{align*}
F_{K, \cC}: \Corr(\Fun(K, \cC)^{vert'}_{vert', \Cart})^{\otimes}\to (\Fun(K, \Corr(\cC)^{vert}_{\all, \all}))^\otimes. 
\end{align*}

\begin{lemma}\label{lemma: bCorr, star, R-lax slice}
For any $(\infty,1)$-category $\cC$ which admits finite limits and given classes of morphisms $(vert, horiz=\all, adm=vert)$ satisfying \cite[Chapter 7, 1.1.1]{Nick} and $vert$ is closed under taking products, there is a natural symmetric monoidal $(\infty,2)$-functor 
\begin{align}\label{eq: lemma: bCorr, star, R-lax slice}
F_{\Delta^1, \cC}^{\star//}: \Corr(\Fun(\Delta^1, \cC))_{vert', \Cart}^{vert'}\to  (\Corr(\cC)_{\all, \all}^{vert})^{\star/ / }
\end{align}
with respect to the symmetric monoidal structure inherited from $\cC^\times$, such that 
\begin{itemize}
\item  on the object level, every object $(f: x\to y)\in \Fun(\Delta^1, \cC)$ is sent to  
\begin{align*}
\xymatrix{ x\ar[r]\ar[d]_{f}&\star \\
y&}. 
\end{align*}

\item any 1-morphism in $\Corr(\Fun(\Delta^1, \cC))_{vert, \Cart}^{vert}$ given by 
\begin{equation*}
\begin{tikzcd}
x_{11}\ar[d]&x_{01}\ar[d]\ar[r]\ar[l]&x_{00}\ar[d]\\
y_{11}&y_{01}\ar[r]\ar[l]&y_{00}\arrow[ul, phantom, "\lrcorner", very near start]
\end{tikzcd}
\end{equation*}
is sent to 
\begin{equation}\label{eq: xijyijCart}
\begin{tikzcd}
x_{11}\ar[dd]\ar[rr]&\ &\star\\
\ &x_{01}\ar[ul, thick]\ar[ur]\ar[dl]\ar[d]\ar[r]&x_{00}\ar[d]\ar[u]\\
y_{11}&y_{01}\ar[r]\ar[l]&y_{00}\arrow[ul, phantom, "\lrcorner", very near start]
\end{tikzcd}
\end{equation}
where (1) the lower right half of the rectangle in (\ref{eq: xijyijCart}) represents a $\Delta^2$ in $\Corr(\cC)_{vert, \all}^{vert}$ with vertices $0,1,2$ going to $\star, y_{00}, y_{11}$ respectively; (2) the thickened upper left arrow represents a 2-morphism in $\bMaps_{\Corr(\cC)_{vert, \all}^{vert}}(\star, y_{11})$. 

\item  any 2-morphism in $\Corr(\Fun(\Delta^1, \cC))_{vert', \Cart}^{vert'}$, given by a commutative diagram
\begin{equation*}
\begin{tikzcd}[row sep=0.4em]
x_{11}\ar[dd]&\ &x_{01}\ar[dr, thick]\ar[dd]\ar[rr]\ar[ll]&\ &x_{00}\ar[dd]\\
\ &\ &\ & x_{01}'\ar[ur]\ar[ulll]\ar[dd, crossing over]&\ \\
y_{11}&\ &y_{01}\ar[dr, thick]\ar[rr]\ar[ll]&\ &y_{00}\arrow[ul, phantom, "\lrcorner", very near start]\\
\ &\ &\ & y_{01}'\ar[ur]\ar[ulll]&\ 
\end{tikzcd}
\end{equation*}
we assign the following commutative diagram that represents a 2-morphism in $(\Corr(\cC)_{vert, \all}^{vert})^{\star/ / }$
\begin{equation*}
\begin{tikzcd}[row sep=0.4em]
x_{11}\ar[ddd]\ar[rrrr]&\ &\ &\ & \star\\
\ &\ &x_{01}\ar[urr]\ar[ull, thick]\ar[dr, thick]\ar[dd]\ar[rr]\ar[ddll]&\ &x_{00}\ar[dd]\ar[u]\\
\ &\ &\ & x_{01}'\ar[uulll, thick]\ar[uur]\ar[ur]\ar[dlll]\ar[dd, crossing over]&\ \\
y_{11}&\ &y_{01}\ar[dr, thick]\ar[rr]\ar[ll]&\ &y_{00}\arrow[ul, phantom, "\lrcorner", very near start]\\
\ &\ &\ & y_{01}'\ar[ur]\ar[ulll]&\ 
\end{tikzcd}.
\end{equation*}

\end{itemize}

\end{lemma}

\begin{proof}
We will define a symmetric monoidal functor
\begin{align*}
F_\cC^{\star//}: \Corr(\cC)_{vert, \all}^{vert}\to  (\Corr(\cC)_{\all, \all}^{vert})^{\star/ / } 
\end{align*}
below, then we let 
\begin{equation*}
\begin{tikzcd}[column sep=5em]
&F_{\Delta^1, \cC}^{\star//}:=  \mu^{\rightlax}_{1,1}(F_{\Delta^1,\cC}, F_{\cC}^{\star//}\circ R_0): \Corr(\Fun(\Delta^1, \cC))_{vert', \Cart}^{vert'}\\
\ar[r, "{(F_{\Delta^1,\cC},\ F_{\cC}^{\star//}\circ R_0)}"] &\Fun\left(\Delta^{\{1,2\}}, \Corr(\cC)_{\all,\all}^{vert}\right)\underset{\Fun\left(\Delta^{\{1\}}, \Corr(\cC)_{vert,\all}^{vert}\right)}{\times}\Fun\left(\Delta^{\{0,1\}}, \Corr(\cC)_{\all,\all}^{vert}\right)_{\rightlax}\\
\ar[r, "\mu^{\rightlax}_{1,1}"] &\Fun(\Delta^{\{0,2\}}, \Corr(\cC)_{\all,\all}^{vert})_{\rightlax}, \hspace{3.1in} 
\end{tikzcd}
\end{equation*}
where 
\begin{align*}
&R_0: \Corr(\Fun(\Delta^{\{0,1\}}, \cC))_{vert',\Cart}^{vert'}\to \Corr(\cC)_{vert,\all}^{vert}
\end{align*}
is the natural (symmetric monoidal) functor induced from $\Delta^{\{0\}}\hookrightarrow \Delta^{\{0,1\}}$, and $\mu^{\rightlax}_{1,1}$ is the composition functor. It is clear that $F_{\Delta^1, \cC}^{\star//}$ factors through $(\Corr(\cC)_{\all, \all}^{vert})^{\star/ / }$ (so it goes to the correct target). 

We will define $(F_\cC^{\star//})_n: \Seq_n(\Corr(\cC)_{vert, \all}^{vert})\to  \Seq_n((\Corr(\cC)_{\all, \all}^{vert})^{\star/ / })$ as follows. 
For any $n$-simplex in $\Corr(\cC)_{vert, \all}^{vert}$ represented by a functor $C: ([n]\times [n]^{op})^{\geq \dgnl}\to \cC$, we can explicitly write down the functor $[n]\Gray [1]\to \Corr(\cC)_{\all,\all}^{vert}$ representing $(F_{\cC}^{\star//})_n(C)$. Namely, for any $(s,t;u)\in \Seq_k([n]\Gray[1])$: 
\begin{equation*}
\begin{tikzcd}
         \    &\            &\         &\                        & u\ar[dr]\\   
s_0\ar[r]& s_1\ar[r]&\cdots\ar[r]&s_{\mu-1}\ar[ur]\ar[rr]&            &t_\mu\ar[r]&\cdots\ar[r]&t_n
\end{tikzcd}
\end{equation*}
where arrow means $\leq$, the corresponding diagram in $\Seq_k(\Corr(\cC)_{\all,\all}^{vert})$, denoted by $(V_{i,j})_{0\leq i\leq j\leq k}$, has 
\begin{align*}
&V_{i,j}=\star, 0\leq i\leq j\leq \mu-1\\
&V_{i, \mu+j}=C_{u, t_{\mu+j}}, 0\leq i\leq \mu-1, 0\leq j\leq n-\mu\\ 
&V_{\mu+i, \mu+j}=C_{t_{\mu+i}, t_{\mu+j}}, 0\leq i\leq j\leq n-\mu, \\
\end{align*}
and the morphisms between $V_{i,j}$ are determined by $C$ in the obvious way. For any $(s,t;u)\to (s,t;u')$, the corresponding 1-morphism in $\Seq_k(\Corr(\cC)_{\all,\all}^{vert})$ is also the obvious one. For any $m$-simplex in $\Seq_n(\Corr(\cC))_{vert, \all}^{vert})$, we get a natural functor 
\begin{align*}
\Delta^m\to \Fun_{(\OneCat)^{\Delta^{op}}}(\Seq_\bullet([n]\Gray[1]), \Seq_\bullet(\Corr(\cC)_{\all, \all}^{vert})), 
\end{align*}
which gives an element in 
\begin{align*}
\Maps(\Delta^m, \Seq_n((\Corr(\cC)_{\all, \all}^{vert})^{\star/ / }))\overset{\text{full}}{\subset} \Maps_{\TwoCat}(\Delta^m, \Fun([n]\Gray[1], \Corr(\cC)_{\all, \all}^{vert})_{\rightlax}). 
\end{align*}
These assemble to be a well defined functor 
\begin{align*}
\Sq_{m,n}^{\tilde{}}(\Corr(\cC)_{vert, \all}^{vert})\to \Sq_{m,n}^{\tilde{}}((\Corr(\cC)_{\all, \all}^{vert})^{\star/ / }), 
\end{align*}
where $\Sq_{m,n}^{\tilde{}}:=\Seq_m\Seq_n$ (cf. \cite[Chapter 10, 2.6]{Nick}). 
Clearly, everything is functorial in $m,n$, hence this gives the full description of $F_\cC^{\star//}$.

The assertion about the symmetric monoidal structure on $F_\cC^{\star//}$ can be obtained similarly by upgrading $(F_\cC^{\star//})_n$ to 
\begin{align*}
(F_\cC^{\star//})^{\otimes}_n: \Seq_n\left((\Corr(\cC)_{vert, \all}^{vert})^{\otimes, N(\Fin_*)}\right)\to  \Seq_n\left((\Corr(\cC)_{\all, \all}^{vert})^{\star/ / })^{\otimes, N(\Fin_*)}\right)
\end{align*} 
over $\Seq_n\left(N(\Fin_*)\right)$.  We omit the details since the argument is similar to many arguments explicitly carried out in Section \ref{subsec: right-lax assoc.}.
\end{proof}

\begin{prop}\label{prop: Dbullet, Wbullet, ptbullet}
Let  $\cC$ be an ordinary 1-category which admits finite limits. 
Let $W^\bullet\to D^\bullet$ be an $N(\Fin_*)$-object in $\Fun(\Delta^1, \cC)$ that satisfies the conditions in Theorem \ref{thm: right-lax} as a correspondence $D^\bullet\leftarrow W^\bullet\to W^\bullet$. 
Then 
\begin{itemize}
\item[(i)]
$W^\bullet\to D^\bullet$ represents a commutative algebra object in $\bCorr(\Fun(\Delta^1, \cC))$;

\item[(ii)]
the right-lax homomorphism determined by the correspondence $D^\bullet\leftarrow W^\bullet\to \star^\bullet$ (using Theorem \ref{thm: right-lax}) is isomorphic to the image of the commutative algebra object in (i) under $F_{\Delta^1, \cC}^{\star//}$ (for $vert=\all$). 
\end{itemize}
\end{prop}

\begin{proof}
(i) follows directly from Theorem \ref{thm: right-lax}. 

(ii) It reduces to the case when $W^\bullet=D^\bullet$ (and the morphism is the identity). Since  $\cC$ is assumed to be ordinary, one can check this directly using the proof of Theorem \ref{thm: algebra objs} and Theorem \ref{thm: right-lax}. We sketch the steps below. 

\emph{Step 1.} the data of the commutative algebra $W\in \Corr(\cC)_{vert, \all}^{vert}$ that is determined by $W^\bullet$ is exactly encoded by the ``multiplication rule": for any $\alpha\in \Seq_n(N(\Fin_*))$, we get a functor $([n]\times[n]^{op})^{\geq \dgnl}\times_{N(\Fin_*)^{op}}T^{\Comm}\to \cC$, and by taking the right Kan extension along the projection $([n]\times[n]^{op})^{\geq \dgnl}\times_{N(\Fin_*)^{op}}T^{\Comm}\to ([n]\times[n]^{op})^{\geq \dgnl}$, we get 
a diagram $([n]\times[n]^{op})^{\geq \dgnl}\to \cC$ that encodes the ``multiplication rule" of $W$ with respect to $\alpha$. 

Under $F_{\cC}^{\star//}$, the above diagram as an $n$-simplex in $\Corr(\cC)_{vert, \all}^{vert}$ gives a functor $[n]\Gray [1]\to \Corr(\cC)_{\all, \all}^{vert}$, 
where $[n]\times \{0\}$ is sent to $\star$ and $[n]\times \{1\}$ is sent to the given $n$-simplex. This functor is explicitly given in the proof of Lemma \ref{lemma: bCorr, star, R-lax slice}. 

\emph{Step 2.}
On the other hand, unwinding the definition, an element in $\CAlg((\Corr(\cC)_{\all, \all}^{vert})^{\star//})$ is equivalent to a right-lax homomorphism from $\star$ (with the trivial commutative algebra structure) to a commutative algebra in $\Corr(\cC)_{\all, \all}^{vert}$. This is again encoded by the family of functors $[n]\Gray[1]\to \Corr(\cC)_{\all, \all}^{vert}$
associated to $\alpha\in \Seq_n(N(\Fin_*))$. Such functors associated to $W^\bullet\overset{id}{\leftarrow}W^\bullet\to \star^\bullet$
from Theorem \ref{thm: right-lax} are also explicitly given in the proof of that theorem. We just need to compare these functors with the ones from \emph{Step 1}. Again, using that all 1 and 2-morphisms are discrete, and both assignments to any $\alpha\in \Seq_n(N(\Fin_*))$ are \emph{identical}, the proof is complete.
 
 \end{proof}

Let $\pi_{X,pt}: X\to pt$ be the map to the terminal object in $\Slch$. 
\begin{cor}\label{cor: CorrVSlchpropallbPRRSp}
There is a natural symmetric monoidal functor\footnote{Here we write $\Shv(X;\Sp)\ni (p_V)_*\varpi_{V}$ for the left adjointable functor $\Sp\to \Shv(X;\Sp)$ since in this case it is also right adjointable. } 
\begin{align*}
\ShvSp^!_{*;\Sp //}: \Corr(\VSlch)_{\propmap', \all}^{\propmap'}&\to (\bPrstR)^{\Sp/ /}\\
(p_V: V\to X)&\mapsto (\Shv(X;\Sp)\ni (p_V)_*\varpi_{V})
\end{align*}
that sends a correspondence 
\begin{equation}\label{diagram: VijXijpij}
\begin{tikzcd}
V_{11}\ar[d, "p_{11}"]&V_{01}\ar[d, "p_{01}"]\ar[r, "\tilde{f}"]\ar[l, "\tilde{g}"']&V_{00}\ar[d, "p_{00}"]\\
X_{11}\arrow[ur, phantom, "\llcorner", very near start]&X_{01}\ar[r, "f"]\ar[l, "g"']&X_{00}\arrow[ul, phantom, "\lrcorner", very near start]
\end{tikzcd}
\end{equation}
to 
\begin{equation}\label{diagram: SpShvX00ShvX11}
\begin{tikzcd}[column sep=4 em]
\Sp\ar[r, "(p_{00})_*\pi_{V_{00}, pt}^!"]\ar[dr, "(p_{11})_*\pi_{V_{11}, pt}^!\ \ \ \ \ \ "' {name=U, below}]&\Shv(X_{00};\Sp)\ar[d, "g_*f^!"]\ar[dr, Rightarrow, "\alpha", to=U]\\
\ &\Shv(X_{11};\Sp)
\end{tikzcd}
\end{equation}
where $\alpha: g_*f^!(p_{00})_*\pi_{V_{00}, pt}^!\cong (p_{11})_*\tilde{g}_!\tilde{g}^!\pi_{V_{11},pt}^!\to (p_{11})_*\pi_{V_{11}, pt}^!$ is induced from the adjunction from $\tilde{g}_!\tilde{g}^!\to id_{\Shv(V_{11};\Sp)}$.
\end{cor}

\begin{proof}
Using Lemma \ref{lemma: bCorr, star, R-lax slice} and $\ShvSp_*^!$, the sought-for functor is defined as the composition 
\begin{align*}
\Corr(\VSlch)_{\propmap', \all}^{\propmap'}\to \Corr(\Fun(\Delta^1, \Slch))_{\propmap', \Cart}^{\propmap'}\to (\Corr(\Slch)_{\all, \all}^{\propmap})^{pt / / }\to (\bPrstR)^{\Sp/ /}.
\end{align*}
\end{proof}

Let $\bT_{\Shv}^!: \VSlch^{op}\to (\bPrstR)^{\Sp//}$ be the restriction of $\ShvSp^!_{*;\Sp //}$ to the 1-full subcategory $\VSlch^{op}$. Since any morphism in $\VSlch^{op}$ is a correspondence (\ref{diagram: VijXijpij}) with $V_{01}=V_{11}$ and $X_{01}=X_{11}$ (and $g$ and $\tilde{g}$ the identity morphisms),  $\bT_{\Shv}^!$ factors through the $(\infty,1)$-category $(\PrstR)^{\Sp/}$. Recall the \emph{right Beck-Chevalley condition} defined in \cite[Chapter 7, Definition 3.1.5]{Nick}.

\begin{cor}\label{cor: bTShv, right B-C}
The functor $\bT_{\Shv}^!$ satisfies the right Beck-Chevalley condition with respect to $\propmap'$. In particular, for any $(f, \varphi): (V, X)\to (V',Y)$ with $f$ proper, $\bT_{\Shv}^!(f,\varphi)$ in $(\bPrstR)^{\Sp/ /}$ is left adjointable. 
\end{cor}

\begin{proof}
This follows immediately from the universal property of functors out of correspondences, established in \cite[Chapter 7, Theorem 3.2.2 (b)]{Nick}.
\end{proof}

\begin{definition}(Chapter 11, 1.2.2 \cite{Nick})
We say an $(\infty,2)$-functor $F: \bC\to\bD$ is a \emph{1-Cartesian} (resp. \emph{1-coCartesian}) \emph{fibration} if the following hold:
\begin{itemize}
\item[(1)] The induced functor $\bC^{\OneCat}\to \bD^{\OneCat}$ is a Cartesian (resp. coCartesian) fibration;

\item[(2)] For every $c',c\in \bC$, the functor
\begin{align*}
\bMaps_{\bC}(c',c)\to \bMaps_{\bD}(F(c'), F(c))
\end{align*}
is a coCartesian (resp. Cartesian) fibration in spaces. 
\end{itemize} 
\end{definition}

By \cite[Chapter 11, Lemma 1.2.5]{Nick}, $F: \bC\to \bD$ is a 1-Cartesian (resp. 1-coCartesian) fibration  if and only if $F$ is a 2-Cartesian (resp. 2-coCartesian) fibration whose fiber over any $d\in \bD$ is an $(\infty,1)$-category. 

\begin{lemma}\label{lemma: Cc,over,coCartesian}
For any $(\infty,2)$-category $\bC$ and any $c\in \bC$, $\bC^{c/ / }\to \bC$ is a 1-coCartesian fibration.
\end{lemma}

\begin{proof}
This is the dual version of \cite[Chapter 11, Lemma 5.1.4]{Nick}. Indeed, we have 
\begin{align*}
(\bC^{c/ / })^{\onetwoop}\simeq (\bC^{\onetwoop})^{/ / c} 
\end{align*}
and it is proved in \emph{loc. cit.} that $(\bC^{\onetwoop})^{/ / c}\to \bC^{\onetwoop}$ is a 1-Cartesian fibration. 
\end{proof}

\begin{remark}\label{remark: adjoints in Cc//}
Assume $\bC$ is an ordinary 2-category and $c\in \bC$. Let $(\phi:c\to a)\in \bC^{c//}$ and $L: a \rightleftarrows b: R$ be an adjoint pair in $\bC$ with unit and co-unit given by $\varepsilon: id_a\to RL$ and $\vartheta: LR\to id_b$, respectively. 
Let $\widetilde{R}=(R, id_{R\circ\phi}): \phi\to R\circ \phi$ be a 1-coCartesian (equivalently 2-coCartesian) 1-morphism in $\bC^{c//}$ over $R$. Then it is easy to see that $\widetilde{R}$ has a left adjoint $\widetilde{L}$ represented by $(L, (\vartheta)\circ\phi): R\circ\phi\to \phi$. 

As explained in \cite[Chapter 12, Section 1.1]{Nick}, the notion of adjoint pairs in an $(\infty, 2)$-category only depends on the underlying ordinary 2-category (cf. Chapter 10, 2.2.5 for the notion of $(-)^{\ordn}$ in \emph{loc. cit.}). However, we note that for an $(\infty,2)$-category $\bC$, 
\begin{align*}
(\bC^{c//})^{\ordn}\not\simeq (\bC^{\ordn})^{c//} 
\end{align*}
in general, since a 2-morphism (always invertible) in a fiber of $(\bC^{c//})^{\ordn}\to \bC^{\ordn}$ might involve a 3-morphism in $\bC$ that is \emph{not} isomorphic to the identity 3-morphism. Hence it is not clear whether the above fact for ordinary 2-categories holds for an arbitrary $(\infty,2)$-category. This means Corollary \ref{cor: bTShv, right B-C} doesn't follow directly from the above observation for ordinary 2-categories. On the other hand, it is clear that for any $(\infty,2)$-category $\bC$,  if $\widetilde{R}$ as above admits a left adjoint $\widetilde{L}$, then $\widetilde{L}$ must be represented by $(L, (\vartheta)\circ\phi): R\circ\phi\to \phi$ (up to isomorphisms), for if $\widetilde{L}$ were different from $(L, (\vartheta)\circ\phi): R\circ\phi\to \phi$, then passing from $\bC$ to $\bC^{\ordn}$ would give a different left adjoint to $\widetilde{R}$ considered now in $(\bC^{\ordn})^{c//}$, which is absurd.

\end{remark}

Let $\propmap\fib=\fib\cap \propmap$ be the class of morphisms of proper locally trivial fibrations in $\Top$, and let $\propmap\fib'$ be the preimage of $\propmap\fib$ under $\VTop\to \Top$. With some abuse of notations, we also use $\propfib$ and $\propfib'$ to denote the similarly defined classes of morphisms in $\Slch$ and $\VSlch$, respectively.

\begin{lemma}\label{lemma: LocSp//}
The restriction of $\ShvSp^!_{*;\Sp //}$ to $\Corr(\VSlch)_{\propfib, \all}$ induces a natural symmetric monoidal functor
\begin{align*}
\Loc^!_{!;\Sp //}: \Corr(\VSlch)_{\propfib', \all}&\to (\bPrstL)^{\Sp/ /}\\
(p_V: V\to X)&\mapsto \Loc(X;\Sp)\ni (p_V)_*\varpi_{V}.
\end{align*}
\end{lemma}

\begin{proof}
For any morphism $(f, \tilde{f})$ in $(\VSlch)_{\propfib'}$ as in (\ref{diagram: VijXijpij}), its corresponding 1-morphism under $\ShvSp^!_{*;\Sp//}$ in $(\bPrstR)^{\Sp//}$ is shown in (\ref{diagram: SpShvX00ShvX11}) with $g=id_{X_{11}}$ and $\tilde{g}=id_{V_{11}}$, and its left adjoint is (\ref{diagram: SpShvX00ShvX11}) with $g=f, \tilde{g}=\tilde{f}$, and $f, \tilde{f}$ replaced by the identity morphisms. 

Then the lemma follows from the fact that for any $f: X\to Y$ in $\propfib$, the restriction of $f_*: \Shv(X;\Sp)\to \Shv(Y;\Sp)$ to local systems give $f_!\cong f_*: \Loc(X;\Sp)\to \Loc(Y;\Sp)$, which is the left adjoint of $f^!: \Loc(Y;\Sp)\to \Loc(X;\Sp)$.  The unit and co-unit for the above pair of adjoints in $(\bPrstR)^{\Sp//}$ induce the unit and co-unit for the restriction on the full subcategory of local systems. Moreover, all functors involved are continuous. 
\end{proof}

Let $\bT_{\VSlch}$ be the restriction $\Loc^!_{!;\Sp //}$ to $\VSlch^{op}$, which factors through $(\PrstL)^{\Sp/}$. 

\begin{cor}\label{cor: bT_Vslch, right B-C}
The functor $\bT_{\VSlch}$ satisfies the right Beck-Chevalley condition with respect to $\propfib'$. In particular, for any $(f, \varphi): (V, X)\to (V',Y)$ with $f$ a proper locally trivial fibration, $\bT_{\VSlch}(f,\varphi)$ in $(\bPrstL)^{\Sp/ /}$ is left adjointable.
\end{cor}
\begin{proof}
This follows from Corollary \ref{cor: bTShv, right B-C} and the same proof as for Lemma \ref{lemma: LocSp//}. 
\end{proof}

\begin{prop}\label{prop: bT_VTop}
There is a natural symmetric monoidal functor 
\begin{align}\label{eq: bT_VTop}
\bT_{\VTop}: \VTop^{op} &\to (\bPrstL)^{\Sp//}\\
\nonumber \big(\xymatrix{V \ar[d]_{p_V} \\ X
}\big)&\mapsto (\Loc(X, \Sp)\ni (p_V)_*\varpi_V)
\end{align}
that factors through $ (\PrstL)^{\Sp/}$, and it satisfies the right Beck-Chevalley condition with respect to $\propfib'$. 
\end{prop}

\begin{proof}
Take the restriction of $\bT_{\VSlch}$ to $(\VTop^{lc})^{op}$, where $\VTop^{lc}=\VTop\cap \VSlch$, and then define $\bT_{\VTop}$ to be the right Kan extension along the full (symmetric monoidal) embedding $(\VTop^{lc})^{op}\hookrightarrow (\VTop)^{op}$. 
It is not hard to check that 
\begin{itemize}
\item $\bT_{\VTop}(p:VG\to G)\simeq \varprojlim\limits_{N}\bT_{\VSlch}(VG_N\to G_N)\simeq 
\left(p_*\varpi_{VG}\in \Loc(G;\Sp)\right)$;
\item for any $p_V: V\to X$ classified by $\nu_{V}: X\to G$, $\bT_{\VTop}(p_V)\simeq \left(\nu_V^!p_*\varpi_{VG}\in \Loc(X;\Sp)\right)$.
\end{itemize}
For any CW-complex $X=\colim_n(X_0\subset X_1\subset X_2\cdots)$ and any vector bundle $p_V: V\to X$, let $p_{V_n}: V_n:=V|_{X_n}\to X_n$. Then $\bT_{\VTop}(p_V)\overset{\sim}{\to}\varprojlim_n\bT_{\VTop}(p_{V_n})$. For any $(f, \varphi): (V',Y)\to (V,X)$ in $\propfib'$, let $Y_n=f^{-1}(X_n)$. Now using induction on $n$, the pushout diagrams for cell attachments (on $X_n$ to get $X_{n+1}$) and Corollary \ref{cor: bT_Vslch, right B-C}, one can easily show that $\bT_{\VTop}(f, \varphi)$ is left adjointable. It also follows easily that $\bT_{\VTop}$ satisfies the right Beck-Chevalley condition with respect to $\propfib'$. 

Lastly, the symmetric monoidal structure of $\bT_{\VTop}$ follows from \cite[Chapter 9, 3.2.2]{Nick}. 
\end{proof}

\subsubsection{$\cV\Spc$ as a localization of $\VTop$.}

Let $W$ be the collection of weak equivalences in $\Top$. Let $W'$ be the corresponding class of morphisms in $\VTop$, which we refer as weak equivalences. Clearly $W'$ is closed under the symmetric monoidal structure. Then we have the following\footnote{We thank N. Rozenblyum for suggesting this statement.}: 

\begin{prop}\label{prop: VTopW'inverse}
The simplicial localization category $\VTop[(W')^{-1}]$ (by Dwyer-Kan), viewed as a symmetric monoidal $\infty$-category through the coherent nerve functor $N$, is naturally equivalent to $\cV\Spc$.
\end{prop}

\begin{proof}

There is an obvious functor 
\begin{align*}
F: \Top_{/G}&\longrightarrow \VTop\\
(\nu: X\to G=\coprod_{n\geq 0}BO(n))&\mapsto (X, \nu^*VG)\\
(X\overset{f}{\to} Y\overset{\nu}{\to}G)&\mapsto \xymatrix{
(\nu\circ f)^*VG\ar[d]\ar[r] &\nu^*VG\ar[d]\\
X\ar[r]^f&Y
},
\end{align*}
where the upper morphism in the bottom Cartesian diagram is the tautological one. 

Let $\Fib$ (resp. $\Cof$) be the class of fibrations (resp. cofibrations) in $\Top$. 
Let $\Fib'$ (resp. $\Cof'$) be the corresponding class of morphisms in $\VTop$. 
Let 
\begin{align*}
W_\Fib=\Fib\cap W,\  W_{\Cof}=\Cof\cap W,\ W'_{\Fib}=\Fib'\cap W',\ W'_{\Cof}=\Cof'\cap W'.
\end{align*}
We calculate the hammock localization $L^H(\VTop):=L^H(\VTop, W')$ as in \cite{DK3}. We will mainly follow the proof of Proposition 4.4 from \S7.2 in \emph{loc. cit.}. Since $\VTop$ is \emph{not} a model category (for it doesn't admit a terminal object), the statement of that proposition doesn't directly apply. In the following, we use $\overset{\sim}{\rightarrowtail}$, $\overset{\sim}{\twoheadrightarrow}$ to represent the classes $W'_\Cof, W'_\Fib$ (resp. $W_\Cof, W_\Fib$) in $\VTop$ (resp. $\Top$), respectively.  \\

\noindent\emph{Step 1. Set-up of $L^H(\VTop, W'_\Fib)$.}\\

Since $\VTop$ admits pushout and pullback in the obvious way, using \cite[Proposition 8.1]{DK3} and its dual version, we know that $(\VTop, W'_{\Fib})$ (resp. $(\VTop, W'_{\Cof})$) admits a homotopy calculus of right (resp. left) fractions. Moreover, if we set $W_1=W_{\Cof}', W_2=W_{\Fib}'$ in Proposition 8.2 in \emph{loc. cit.}, then we get the pair $(\VTop, W')$ admits a homotopy calculus of (two-sided) fractions. Using Proposition 6.2 in \emph{loc. cit.}, the reduction map 
\begin{align*}
(\VTop)(W'_\Fib)^{-1}\to L^H(\VTop, W'_\Fib)
\end{align*}
is a weak equivalence. For any $(V, X), (E, Y)\in \VTop$, recall that $\Hom_{(\VTop)(W'_\Fib)^{-1}}\big((V, X), (E, Y)\big)$ is by definition the nerve of all diagrams
\begin{align*}
(V,X)\overset{\sim}{\twoheadleftarrow} (V', X')\to  (E, Y). 
\end{align*}

Following \cite[\S 7.2(iii)]{DK3}, let $S=(W\downarrow X)$ be the category of the trivial fibrations in $\Top$ ending at $X$, and let $S'=(W'\downarrow (V, X))$ be the category of $(V',X')\overset{\sim}{\twoheadrightarrow} (V, X)$ ending at $(V,X)$. Let 
\begin{align*}
K': (S')^{op}&\to \Sets\subset \SSet\\
((V',X')\overset{\sim}{\twoheadrightarrow} (V, X))&\mapsto \Hom_{\VTop}((V',X'), (E, Y)). 
\end{align*}
Then $\Hom_{(\VTop)(W'_\Fib)^{-1}}\big((V, X), (E, Y)\big)$ is weakly equivalent to the (homotopy) colimit of $K'$. By Proposition 6.12 in \emph{loc. cit.}, for any special cosimplicial resolution $X^\bullet$ of $X\in \Top$, the functor $x: \Delta\to (W\downarrow X)$ which sends $[n]$ to $X^n\overset{\sim}{\twoheadrightarrow} X$ is left cofinal. Since the projection $(S')^{op}\to S^{op}$ is an equivalence, for any such a special cosimplicial resolution $X^\bullet$, we have 
\begin{align}\label{eq: hocolim K'}
&\Maps_{(\VTop)(W'_\Fib)^{-1}}\big((V, X), (E, Y)\big)\simeq \varinjlim_{(S')^{op}}K'\simeq \varinjlim_{\Delta^{op}} \Hom_{\VTop}((V^\bullet, X^\bullet), (E, Y)), 
\end{align} 
where $V^n\to X^n$ is the pullback of $V\to X$ along $X^n\overset{\sim}{\twoheadrightarrow} X$. It is clear from the definition of a special cosimplicial resolution (cf. \S 4.3 and Remark 6.8 in \emph{loc. cit.}) that to give such a  $X^\bullet$, one just needs to find a special cosimplicial resolution $C^\bullet$ of $\pt$ and then take the product with $X$. \\

\noindent\emph{Step 2. Calculation of (\ref{eq: hocolim K'}).}

Let $X^\bullet=X\times C^\bullet$ and $V^\bullet=V\times C^\bullet$. First assume that the vector bundle $E\to Y$ (resp. $V\to X$) has constant rank $r$, so then it is classified by a (continuous) map $\nu_E: Y\to BO(r)$ (resp. $\nu_V: X\to BO(r)$). Let $E_{\nu_V}O(r)\to X\times BO(r)$ be the universal $O(r)$-torsor relative to the section $(id_X,\nu_V): X\to X\times BO(r)$, i.e. $E_{\nu_V}O(r)$ is isomorphic to $(X\times BO(r))_{(id_X,\nu_V)/}$ in $\Spc_{/X}$.

By adjunction, the (diagonal of the bi)simplicial set $\Hom_{\VTop}((V^\bullet, X^\bullet), (E, Y))$ is weakly equivalent to 
\begin{align}\label{eq: E_nuVOr}
&\Maps_{\Spc}(X, Y)\underset{\Maps_{\Spc}(X,BO(r))}{\times} \Maps_{\Spc_{/X}}(X, E_{\nu_V}O(r))\\
\nonumber\simeq &\Maps_{\Spc_{/BO(r)}}(\nu_V, \nu_E)
\end{align}
where the map $\Maps_{\Spc}(X, Y)\to \Maps_{\Spc}(X,BO(r))$ is induced from $\nu_E$. 

If $E\to Y$ is not necessarily of constant rank, then (\ref{eq: E_nuVOr}) directly extends to 
\begin{align*}
\Hom_{\VTop}((V^\bullet, X^\bullet), (E, Y))\simeq \Maps_{\Spc_{/\coprod_{r\geq 0}BO(r)}}(\nu_V, \nu_E)
\end{align*}

\noindent\emph{Step 3. $L^H((\Top_{/G})^f, W_\Fib)\to L^H(\VTop, W'_\Fib)$ is a weak equivalence.}

Here $(\Top_{/G})^f$ is the full subcategory of $\Top_{/G}$ consisting of fibrant objects. 
 Using the same cosimplicial resolution $X^\bullet$ as above (now with $\nu_V: X\to G$ and $\nu_E: Y\to G$ being fibrations), we have the commutative diagram up to homotopy
\begin{equation*}
\begin{tikzcd}
\Maps_{L^H((\Top_{/G})^f, W_\Fib)}(\nu_{V}, \nu_E)\ar[r, "\sim"] \ar[d]&\Maps_{\Spc}(X, Y)\underset{\Maps_{\Spc}(X, G)}{\times}\{\nu_V\}\ar[d]\\
\Maps_{L^H(\VTop, W'_\Fib)}((V,X), (E, Y))\ar[r, "\sim"] & \Maps_{\Spc_{/\coprod_{r\geq 0}BO(r)}}(\nu_V, \nu_E).
\end{tikzcd}
\end{equation*}
Hence the functor $N(L^H((\Top_{/G})^f, W_\Fib))\to N(L^H(\VTop, W'_\Fib))$ is fully faithful. 
Using \emph{Step 2}, we also know that the functor is essentially surjective. \\
 
\noindent\emph{Step 4. $\cV\Spc\simeq N(L^H(\Top_{/G}, W))\to N(L^H(\VTop, W'))$ is an equivalence.}

From the calculation in \emph{Step 3}, we see that $N(L^H((\Top_{/G})^f, W_{\Fib})\to N(L^H(\Top_{/G}, W)))\simeq \Spc_{/G}\simeq \cV\Spc$ is an equivalence. From the same step, we have a weak equivalence $L^H(\VTop, W'_\Fib)\to L^H(\Top_{/G}, W)$. Then by the universal property of Dwyer-Kan localization, we know that 
$$L^H(\VTop, W'_\Fib)\simeq L^H(\VTop, W'),$$
and the desired equivalence follows.

Lastly, $N(L^H(\VTop, W'))\to N(L^H(\Top, W))\simeq \Spc$ is naturally a symmetric monoidal functor and it is a commutative algebra object in $\mathrm{LFib}_{\Spc}$ as in \cite[Corollary C]{Ramzi} (see also \cite{higher-algebra}, \cite[Appendix A.2]{Hinich}). Now by the monoidal version of straightening in \emph{loc. cit.} and \cite[Chapter 9, 3.2.2]{Nick},  
to see that the symmetric monoidal structure on  $N(L^H(\VTop, W'))$ is equivalent to that on $\cV\Spc$ as objects in $\mathrm{LFib}_{\Spc}$ (which are both classified by functors $\Spc^{op}\to \Spc$ that are the right Kan extension of their restriction to $\{pt\}$), it suffices to check that 
$$N(L^H(\VTop, W'))\underset{L^H(\Top, W)}{\times}\{pt\}$$ 
has the same symmetric monoidal structure on $N(\Vect_\bR^{\simeq})$. But the calculation from \emph{Step 1 and 2} for (objects in) $\VTop\underset{\Top}{\times}\{pt\}$ (which only involves objects from $\VTop\times_{\Top}\{|\Delta^n|\}, n\geq 0$) exactly gives a standard definition of $N(\Vect_\bR^{\simeq})$ and its symmetric monoidal structure.

\end{proof}

\begin{cor}\label{cor: bT_cVSpc}
The symmetric monoidal functor $\bT_{\VTop}$ (\ref{eq: bT_VTop}) descends to a symmetric monoidal functor 
\begin{align*}
\bT_{\cV\Spc}: \cV\Spc^{op}\simeq N(\VTop[(W')^{-1}])^{op}&\longrightarrow  (\bPrstL)^{\Sp//}\\
\big(\xymatrix{V \ar[d]_f \\ X
}\big)&\mapsto (\Loc(X, \Sp)\ni f_*\varpi_V)
\end{align*}
that factors through $(\PrstL)^{\Sp/}$, and it satisfies the right Beck-Chevalley condition with respect to $\hpf'$. 

\end{cor}

\begin{proof}
The corollary follows readily from Proposition \ref{prop: bT_VTop}, \ref{prop: VTopW'inverse} and the obvious fact that $\bT_{\VTop}$ sends morphisms in $W'$ to equivalences in $(\bPrstL)^{\Sp//}$. 
\end{proof}

\subsubsection{The isomorphism $\bT_{\cV\Spc}\cong \bT$}

\begin{cor}
The symmetric monoidal functor $\bT_{\cV\Spc}$ is isomorphic to $\bT$ (\ref{eq: bT}). 
\end{cor}
\begin{proof}
Since both functors are the right Kan extension from their respective restrictions to $(\cV\Spc\times_{\Spc}\{\pt\})^{op}\simeq N(\Vect_\bR^{\simeq})^{op}$, using \cite[Chapter 9, 3.2.2]{Nick} again, 
it suffices to show that they agree on this full (symmetric monoidal) subcategory. By looking at $\bT_{\VTop}|_{(\VTop\times_{\Top}\{\pt\})^{op}}$ and the calculation of the ``function complex" in the hammock localization (\emph{Step 1 and 2} in the proof of Proposition \ref{prop: VTopW'inverse}), we see that $\bT|_{(\cV\Spc\times_{\Spc}\{\pt\})^{op}}$ agrees with the definition of $J$  \cite{Lurie-circle} as a symmetric monoidal functor. 
\end{proof}

\subsection{A symmetric monoidal extension of $\bT_{\VTop}$ and $\bT_{\cV\Spc}$ to correspondences}

The class $\propfib$ in $\Top$ determines an isomorphism class of morphisms in $\Spc$ through the localization $\Top\to N(\Top[W^{-1}])\simeq \Spc$. We denote the resulting class by $\hpf$. Concretely, every morphism in $\hpf$ is a finite composition of proper locally trivial fibrations and zig-zags of weak equivalences. 
Clearly, $\hpf$ is closed under compositions, homotopy pullbacks (along any morphism) and taking Cartesian products. Let $\hpf'$ be the preimage of $\hpf$ under $\cV\Spc\to \Spc$.

In order to prove Proposition \ref{prop: equiv model of J}, we need to extend the functors $\bT_{\VTop}$ and $\bT_{\cV\Spc}$ to $\Corr(\VTop)_{\propfib',\all}$ and $\Corr(\cV\Spc)_{\hpf',\all}$, respectively. Ideally, we would like to apply the various extension results from \cite[Chapter 8,9]{Nick}. However, the results stated in \emph{loc. cit.} that we need require the class of vertical arrows to satisfy the ``2-out-of-3" condition (cf. Chapter 7, 1.1.1 (5) in \emph{loc. cit.}). Since the classes $\propfib'$ and $\hpf'$ certainly do \emph{not} satisfy ``2-out-of-3" (and we are not aware of a better candidate that satisfies ``2-out-of-3"),  we need to essentially go back to the proof of some of those extension results and make them work in our setting. We remark that ``2-out-of-3" is needed in \emph{loc. cit.} mostly for establishing uniqueness of extensions of functors. Here we are aiming at proving Proposition \ref{prop appendix: bT_VTopCorr, bT_VSpcCorr}, which doesn't involve uniqueness of the extensions. 

\subsubsection{The formalism of universal adjoint functors and symmetric monoidal upgrades}

Recall the formalism of the universal adjointable functor from \cite[Chapter 12, \S 1.2, 1.3]{Nick}. For any $(\infty,2)$-category $\bC$ and a 1-full subcategory $\cD\subset \bC^{\OneCat}$ such that $\cD^{\Spc}=\bC^{\Spc}$, there is a \emph{universal} recipient of functors left-adjointable with respect to $\cD$, denoted by $\bC^{R_{\cD}}$. This $(\infty,2)$-category $\bC^{R_{\cD}}$ has an explicit description as 
\begin{align}\label{eq: def bC, RcD}
\fL^{\Sq}((\Sq_{\bullet,\bullet}^{\Pair}(\bC^{\twoop}, \cD))^{\vertop}),
\end{align}
where $\Sq_{\bullet, \bullet}^{\Pair}: \TwoCat^{\Pair}\to \Spc^{\Delta^{op}\times\Delta^{op}}$ is defined in Chapter 10, 4.3.3 in \emph{loc. cit.} and $\fL^{\Sq}$ is the left adjoint of $\Sq_{\bullet, \bullet}: \TwoCat\to \Spc^{\Delta^{op}\times\Delta^{op}}$. In the special case that $\cD=\bC^{\OneCat}$, $\Sq_{\bullet,\bullet}^{\Pair}(\bC^{\twoop}, \bC^{\OneCat})=\Sq_{\bullet,\bullet}(\bC^{\twoop})$ from definition. 
 The construction (\ref{eq: def bC, RcD}) is clearly functorial in $(\bC, \cD)$ in the following sense:
\begin{itemize}
\item[(1)] The assignment $(\bC, \cD)\mapsto \bC^{R_\cD}$ gives a functor $R_{\Pair}: \TwoCat^{\Pair}\to \TwoCat$ (just following the formula (\ref{eq: def bC, RcD})). 

\item[(2)] There is a natural transformation $\OblvSubcat\to R_{\Pair}$, where $\OblvSubcat: \TwoCat^{\Pair}\to \TwoCat$ is the forgetful functor that sends $(\bC, \cD)\mapsto \bC$. It is given by applying $\fL^{\Sq}$ to the tautological bi-simplicial functor
\begin{align}\label{eq: bC, cD, Sq, SqPair}
\Sq^{\tilde{}}_{\bullet, \bullet}(\bC)\simeq \Sq^{\tilde{}}_{\bullet, \bullet}(\bC^{\twoop})^{\vertop}\hookrightarrow  \Sq^{\Pair}_{\bullet, \bullet}(\bC^{\twoop}, \cD)^{\vertop}, 
\end{align}

\end{itemize}

Let $\Maps_{\TwoCat}(\bC, \bX)^{R_\cD}\subset \Maps_{\TwoCat}(\bC, \bX)$  be the full subspace consisting of left-adjointable functors with respect to $\cD$ (cf.  Chapter 12, Definition 1.1.7 in \emph{loc. cit.}). Then the key statement of the formalism of the universal adjointable functor is the following:

\begin{theorem}[(Chapter 12, Theorem 1.2.4 \cite{Nick})]\label{thm: universal adjointable}
Retricting along the canonical functor $\bC\to \bC^{R_\cD}$ defines an isomorphism
\begin{align*}
\Maps_{\TwoCat}(\bC^{R_\cD},\bX)\overset{\sim}{\longrightarrow} \Maps_{\TwoCat}(\bC,\bX)^{R_\cD}.
\end{align*}
\end{theorem}

Let $\cW$ be an $(\infty,1)$-category, and let $\cW'\subset \cW$ be a full subcategory. For any morphism $w_1\to w_2$ in $\cW$, we say it is a \emph{categorical epimorphism w.r.t. $\cW'$}, if for every $w'\in \cW'$, the morphism
\begin{align*}
\Maps_{\cW}(w_2, w')\to \Maps_{\cW}(w_1, w')
\end{align*}
is a monomorphism of spaces, i.e. it is an embedding of certain connected components up to weak equivalences. 

\begin{cor}\label{cor: bC,cD, SymmetricMonoidal}
Assume $(\bC^\otimes,\cD)\in \CAlg((\TwoCat^\Pair)^\times)$. Let $\bX^\otimes\in  \CAlg(\TwoCat^\times)$. Then for any symmetric monoidal functor $\Phi: \bC\to \bX$ that is left-adjointable with respect to $\cD$ (as a plain functor), it extends uniquely to a symmetric monoidal functor 
\begin{align*}
 \Sq^{\Pair}_{\bullet, \bullet}(\bC^{\twoop}, \cD)^{\vertop}\to \Sq_{\bullet,\bullet}(\bX)
\end{align*}
in $\Spc^{\Delta^{op}\times\Delta^{op}}$. 
\end{cor}

\begin{proof}

First, both the functors $\Sq_{\bullet, \bullet}^{\tilde{}}: \TwoCat\to \Spc^{\Delta^{op}\times\Delta^{op}}$ and $\Sq_{\bullet,\bullet}^{\Pair}:  \TwoCat^{\Pair}\to \Spc^{\Delta^{op}\times\Delta^{op}}$ are fully faithful symmetric monoidal functors. The full faithfulness follows from \cite[Chapter 10, 2.6.1\footnote{There is a typo in (2.5). The source of the functor should be $\TwoCat$.}, Theorem 4.3.5]{Nick}), and the symmetric monoidal structure follows from definition. Since 
\begin{align*}
\Sq_{\bullet, \bullet}^{\tilde{}}(\bC')=\Sq_{\bullet, \bullet}^{\Pair}(\bC',(\bC')^{\Spc})
\end{align*}
for any $\bC'\in \TwoCat$. We see that (\ref{eq: bC, cD, Sq, SqPair}) for the given $(\bC^{\otimes}, \cD)$ is a symmetric monoidal bi-simplicial functor. 

Second, by Theorem \ref{thm: universal adjointable} and that all functors respect the Cartesian products, we know that for any $n\in \bZ_{\geq 0}$
\begin{align*}
\Sq^{\tilde{}}_{\bullet, \bullet}(\bC)^{\times n}\simeq (\Sq^{\tilde{}}_{\bullet, \bullet}((\bC)^{\twoop})^{\vertop})^{\times n}\hookrightarrow  (\Sq^{\Pair}_{\bullet, \bullet}((\bC)^{\twoop}, \cD)^{\vertop})^{\times n}
\end{align*}
is a categorical epimorphism w.r.t. the the essential image of $\Sq_{\bullet,\bullet}$. Hence by (an obvious generalization of) Chapter 9, Lemma 3.1.3 in \emph{loc. cit.}, restricting along (\ref{eq: bC, cD, Sq, SqPair}) defines an isomorphism between the space of symmetric monoidal functors from $ \Sq^{\Pair}_{\bullet, \bullet}(\bC^{\twoop}, \cD)^{\vertop}$ to $\Sq_{\bullet,\bullet}\bX$ and the space of symmetric monoidal functors from $\bC$ to $\bX$ that factors through $\bC^{R_\cD}$ as a plain functor. But the latter space is exactly the subspace consisting of left-adjointable functors with respect to $\cD$. 
\end{proof}

By definition, we have a symmetric monoidal functor $(\VTop, \propfib')\to (\cV\Spc,\hpf')$ in $\TwoCat^{\Pair}$.

\begin{cor}\label{cor: Sq, VTop, cVSpc}
There is a canonical commutative diagram in $\CAlg((\Spc^{\Delta^{op}\times\Delta^{op}})^{\times})$
\begin{equation}\label{diagram: Sq, VTop, cVSpc}
\begin{tikzcd}
\Sq^{\Pair}_{\bullet, \bullet}(\VTop^{op}, (\VTop_{\propfib'})^{op})^{\vertop}\ar[r]\ar[d]& \Sq_{\bullet, \bullet}(\bPrstL)^{\Sp/ / }\\
\Sq^{\Pair}_{\bullet, \bullet}(\cV\Spc^{op}, (\cV\Spc_{\hpf'})^{op})^{\vertop}\ar[ur] &
\end{tikzcd}. 
\end{equation}
where the vertical functor is induced from localization, and the horizontal and slant functors are respectively determined by $\bT_{\VTop}$ and $\bT_{\cV\Spc}$ (up to a contractible space of choices). 
\end{cor}

\begin{proof}
Consider the commutative diagram in $\CAlg((\OneCat^\Pair)^\times)$
\begin{equation*}
\begin{tikzcd}
(\VTop, \VTop^{\Spc})\ar[r]\ar[d]&(\VTop, \VTop_{\propfib'})\ar[d]\\
(\cV\Spc, \cV\Spc^{\Spc})\ar[r]&(\cV\Spc, \cV\Spc_{\hpf'})
\end{tikzcd}
\end{equation*}
Applying $\Sq_{\bullet, \bullet}^\Pair((-)^{op},(-)^{op})^{\vertop}$, we get a commutative diagram in $\CAlg((\Spc^{\Delta^{op}\times\Delta^{op}})^\times)$, where all the morphisms are categorical epimorphisms w.r.t. the essential image of $\Sq_{\bullet,\bullet}$. Indeed, for the horizontal morphisms, this follows from Theorem \ref{thm: universal adjointable} again, and for the vertical morphisms, this follows the universal property of localizations. 

Now by Corollary \ref{prop: bT_VTop}, \ref{cor: bT_cVSpc} and Corollary \ref{cor: bC,cD, SymmetricMonoidal} (and a similar argument for it), we get $\bT_{\VTop}$ determines a unique commutative diagram in $\CAlg((\Spc^{\Delta^{op}\times\Delta^{op}})^\times)$
\begin{equation*}
\begin{tikzcd}
\Sq_{\bullet, \bullet}^{\tilde{}}(\VTop^{op})\ar[r, hook]\ar[d]&\Sq_{\bullet, \bullet}(\VTop^{op}, (\VTop_{\propfib'})^{op})^{\vertop}\ar[r]\ar[d]& \Sq_{\bullet, \bullet}(\bPrstL)^{\Sp/ / }\\
\Sq_{\bullet, \bullet}^{\tilde{}}(\cV\Spc^{op})\ar[r,hook]&\Sq_{\bullet, \bullet}(\cV\Spc^{op}, (\cV\Spc_{\hpf'})^{op})^{\vertop}\ar[ur] &
\end{tikzcd}. 
\end{equation*}
In particular, we get (\ref{diagram: Sq, VTop, cVSpc}) with the prescribed property. 

\end{proof}

\subsubsection{The bi-simplicial space of grids with defect and symmetric monoidal upgrades}

Recall the bi-simplicial space of grids with defect from \cite[Chapter 7, 2.1]{Nick}. 
For any $(\infty,1)$-category $\cC$ and classes $(vert, horiz, adm)$. let $\defGrid_{\bullet, \bullet}(\cC)^{adm}_{vert, horiz}\in \Spc^{\Delta^{op}\times\Delta^{op}}$ defined as : $\defGrid_{m,n}(\cC)^{adm}_{vert, horiz}$ is the full subspace of $\Maps([m]\times[n]^{op}, \cC)$ consisting of objects $\underline{\bc}$ with the following properties:
\begin{itemize}
\item[(1)] For any $0\leq i<i+1\leq m$, the map $\bc_{i,j}\to \bc_{i+1, j}$ belongs to $vert$;
\item[(2)] For any $0\leq j-1<j\leq n$, the map $\bc_{i,j}\to \bc_{i,j-1}$ belongs to $horiz$;
\item[(3)] For any $0\leq i<i+1\leq m$ and $0\leq j-1<j\leq n$ in the commutative square
\begin{equation*}
\begin{tikzcd}
\bc_{i,j}\ar[r]\ar[d]&\bc_{i,j-1}\ar[d]\\
\bc_{i+1, j}\ar[r]&\bc_{i+1, j-1}
\end{tikzcd}
\end{equation*}
its defect of Cartesianness, i.e. the map $\bc_{i,j}\to \bc_{i+1,j}\underset{\bc_{i+1, j-1}}{\times}\bc_{i,j-1}$
belongs to $adm$. 
\end{itemize}

It is shown in Chapter 7, 2.1.2, 2.3 in \emph{loc. cit.} that there is a canonical map in $\Spc^{\Delta^{op}\times \Delta^{op}}$ 
\begin{align}\label{eq: GR 2.1.2}
\defGrid_{\bullet, \bullet}(\cC)_{vert, horiz}^{adm}\to \Sq_{\bullet, \bullet}(\Corr(\cC)_{vert, horiz}^{adm})
\end{align}
whose key feature is stated in Chapter 7, Theorem 2.1.3 in \emph{loc. cit.}
The map is explicitly given by the composition 
\begin{equation}\label{eq: GR 2.3.1}
\defGrid_{\bullet,\bullet}(\cC)_{vert, horiz}^{adm}\to \Seq_{\bullet}(\Grid_\bullet^{\geq \dgnl}(\cC)_{vert, horiz}^{adm})
\end{equation} 
followed by the map
\begin{equation}\label{eq: SeqGridSqPairSq}
\Seq_\bullet(\Grid_{\bullet}^{\geq \dgnl}(\cC)^{adm}_{vert, horiz})\simeq \Sq_{\bullet, \bullet}^{\Pair}(\Corr(\cC)^{adm}_{vert, horiz}, \cC_{vert})\to \Sq_{\bullet, \bullet}(\Corr(\cC)_{vert, horiz}^{adm})
\end{equation}
To define the functor (\ref{eq: GR 2.3.1}), one introduces for each $n$, a 1-full subcategory of 
$\bMaps([n]^{op}, \cC)$, denoted by $\bMaps([n]^{op}, \cC)_{vert, horiz}^{adm}$. We refer the reader to Chapter 7, 2.3.2 in \emph{loc. cit.} for its definition. The upshot is there is a canonical isomorphism of bi-simplicial spaces 
\begin{align*}
\defGrid_{\bullet, \bullet}(\cC)_{vert, horiz}^{adm}\simeq \Seq_{\bullet}(\Maps([\bullet]^{op}, \cC)_{vert, horiz}^{adm}), 
\end{align*}
and there is a canonically defined fully faithful embedding
\begin{align}\label{eq: Maps,bullet,Grid,vert,horiz,adm}
\Maps([\bullet]^{op}, \cC)_{vert, horiz}^{adm}\to \Grid_\bullet^{\geq \dgnl}(\cC)_{vert, horiz}^{adm}
\end{align}
that sends every $\bullet$-string in $\cC$ to the half grids with all vertical maps isomorphisms.

\begin{lemma}\label{lemma: SeqPairGrid, CAlg}
Assume $\cC\in \CAlg(\OneCat^\times)$, and the classes of morphisms $vert, horiz, adm$ in $\cC$ are closed under the tensor product. Then we get a natural commutative diagram in $\CAlg(\OneCat^\times)$
\begin{equation}\label{diagram: lemma SeqPairGrid, CAlg}
\begin{tikzcd}
\Seq_n(\Corr(\cC)_{vert, horiz}^{adm})\ar[d]\ar[r,"\sim"] &\leftidx{''}{\Grid_n^{\geq\dgnl}(\cC)}{}_{vert, horiz}^{adm}\ar[d]\\
\Seq_n^{\Pair}(\Corr(\cC)^{adm}_{vert, horiz}, \cC_{vert})\ar[r,"\sim"]&\Grid_{n}^{\geq \dgnl}(\cC)^{adm}_{vert, horiz}
\end{tikzcd}
\end{equation}
where the symmetric monoidal structure on the LHS (resp. RHS) are the ones induced from $\Corr(\cC)^{adm}_{vert, horiz}$ (resp. $\Fun(([n]\times[n]^{op})^{\geq \dgnl}, \cC)$). 
\end{lemma}

\begin{proof}
The proof follows directly from \cite[Chapter 9, 2.1; Chapter 7, 1.4]{Nick}. 
Let $\Trpl$ be the $(\infty,1)$-category consisting of objects $(\cC; vert, horiz, adm)$ as in Chapter 7, 1.1 in \emph{loc. cit.} and the space of morphisms between two objects consisting of the full subspace of functors between the $\cC$-variables that preserves the three classes of morphisms and Cartesian squares from Chapter 7, 1.1 in \emph{loc. cit}. Consider the following diagram 

\begin{equation}\label{diagram: Trpl, Corr, Grid}
\begin{tikzcd}[column sep=4em]
\ &{\color{gray}\TwoCat^{\Pair}}\ar[dr, gray, "\OblvSubcat"]\ar[dd, gray,  "\Seq_\bullet^{\Pair}", ""{name=W, below}] &\ \\
\Trpl\ar[dr, gray, bend left=10, "{\Grid_\bullet^{\geq\dgnl}}", ""{name=V}]\ar[ur, gray, "\Corr^{\Pair}"]\ar[rr, "\Corr", crossing over, bend left=10]\ar[dr, bend right, "\leftidx{''}{\Grid_\bullet^{\geq\dgnl}}{}"', ""{name=U}]&\ &\TwoCat\ar[dl,"\Seq_\bullet"]\ar[Rightarrow, from=U, to=V]\ar[Rightarrow, to=W, gray, bend left=15]\\
\ &\OneCat^{\Delta^{op}}&\ 
\end{tikzcd},
\end{equation}
in which the functors are given by the assignments
\begin{align*}
&\Corr: (\cC; vert, horiz, adm)\mapsto \Corr(\cC)_{vert, horiz}^{adm}\\
&\Corr^{\Pair}: (\cC; vert, horiz, adm)\mapsto (\Corr(\cC)_{vert, horiz}^{adm},\cC_{vert})\\
&\leftidx{''}{\Grid_\bullet^{\geq\dgnl}}{}: (\cC; vert, horiz, adm)\mapsto\leftidx{''}{\Grid_\bullet^{\geq\dgnl}(\cC)}{}_{vert, horiz}^{adm}\\
&\Grid_\bullet^{\geq\dgnl}: (\cC; vert, horiz, adm)\mapsto \Grid_\bullet^{\geq\dgnl}(\cC)_{vert, horiz}^{adm}.
\end{align*}
and the natural transformations marked are the obvious ones. 
Moreover, the same diagram has the following features: 
\begin{itemize}
\item all functors are symmetric monoidal with respect to the Cartesian symmetric monoidal structures on the categories, and both natural transformations are between the symmetric monoidal functors;

\item except for the double arrows, all three solid faces are commutative, with functors in gray located at the ``back";

\item the natural transformation $\leftidx{''}{\Grid_\bullet^{\geq\dgnl}(\cC)}{}_{vert, horiz}^{adm}\to {\Grid_\bullet^{\geq\dgnl}(\cC)}_{vert, horiz}^{adm}$ is just the pullback of the other natural transformation along $\Corr^{\Pair}$ (and using the two solid triangles containing them). 
\end{itemize}

Then the diagram (\ref{diagram: lemma SeqPairGrid, CAlg}) is obtained from applying the functors and natural transformations to $(\cC;vert, horiz, adm)$ in the obvious way. 
\end{proof}

\begin{cor}
Under the same assumption as in Lemma \ref{lemma: SeqPairGrid, CAlg}, the map (\ref{eq: GR 2.1.2}) is naturally symmetric monoidal and is a categorical epimorphism w.r.t. the essential image of $\Sq_{\bullet, \bullet}$. 
\end{cor}

\begin{proof}
Since (\ref{eq: GR 2.3.1}) and (\ref{eq: SeqGridSqPairSq}) are both symmetric monoidal (using Lemma \ref{lemma: SeqPairGrid, CAlg}), (\ref{eq: GR 2.1.2}) is symmetric monoidal as the composition. 
The last part follows from the key feature of the map  (\ref{eq: GR 2.1.2}) (\cite[Chapter 7, Theorem 2.1.3]{Nick}).
\end{proof}

We immediately get a symmetric monoidal version of \cite[Chapter 7, Theorem 2.1.3]{Nick}:

\begin{cor}\label{cor: GR Ch7 2.1.3, monoidal}
Under the same assumption as in Lemma \ref{lemma: SeqPairGrid, CAlg}, for any symmetric monoidal $(\infty,2)$-category $\bX$, the symmetric monoidal functor (\ref{eq: GR 2.1.2}) defines an isomorphism between the space of symmetric monoidal functors $\Corr(\cC)_{vert, horiz}^{adm}\to \bX$ and the subspace of bi-simplicial symmetric monoidal functors 
\begin{align*}
\defGrid_{\bullet,\bullet}(\cC)_{vert, horiz}^{adm}\to\Sq_{\bullet, \bullet}(\bX)
\end{align*}
with the following property: for every $\underline{c}\in \defGrid_{1,1}(\cC)_{vert, horiz}^{adm}$, such that the diagram 
\begin{equation*}
\begin{tikzcd}
c_{0,1}\ar[r]\ar[d]&c_{0,0}\ar[d]\\
c_{1,1}\ar[r]&c_{1,1}
\end{tikzcd}
\end{equation*}
is \emph{Cartesian}, the corresponding object 
\begin{equation*}
\begin{tikzcd}
s_{0,1}\ar[d]\ar[dr, Rightarrow]&s_{0,0}\ar[d]\ar[l]\\
s_{1,1}&s_{1,1}\ar[l]
\end{tikzcd}
\end{equation*}
in $\Sq_{1,1}(\bX)$ represents an \emph{invertible} 2-morphism. 

\end{cor}

\subsubsection{}

Consider the commutative diagram in $\CAlg(\Spc^{\Delta^{op}\times\Delta^{op}})$
\begin{equation}\label{diagram: defGrid, SqPair, VTopVSpc,hpf'}
\begin{tikzcd}
\defGrid_{\bullet,\bullet}(\VTop)_{\propfib', \all}^{\isom}\ar[r]\ar[d]& \Sq_{\bullet, \bullet}^{\Pair}((\VTop)^{op}, (\VTop_{\propfib'})^{op})^{\vertop}\ar[d]\\
\defGrid_{\bullet,\bullet}(\cV\Spc)_{\hpf', \all}^{\isom}\ar[r]& \Sq_{\bullet, \bullet}^{\Pair}((\cV\Spc)^{op}, (\cV\Spc_{\hpf'})^{op})^{\vertop}\\
\defGrid_{\bullet,\bullet}(\cV\Spc)_{\hpf', \hpf'}^{\isom}\ar[r]\ar[u]& \Sq_{\bullet, \bullet}^{\Pair}((\cV\Spc_{\hpf'})^{op}, (\cV\Spc_{\hpf'})^{op})^{\vertop}\ar[u]
\end{tikzcd},
\end{equation}
where the horizontal maps are induced from the tautological functor $[m]\Gray[n]=:[m,n]\to [m]\times[n]$ (similarly to the definition of \cite[Chapter 7, (3.3)]{Nick}).

\begin{prop}\label{prop appendix: bT_VTopCorr, bT_VSpcCorr}
There is a natural commutative diagram in $\CAlg(\TwoCat^\times)$
\begin{equation}\label{diagram: prop bT_VTopCorr, bT_VSpcCorr}
\begin{tikzcd}
\VTop^{op}\ar[r]\ar[d]& \Corr(\VTop)_{\propmap\fib', \all}\ar[d] \ar[r, "\bT_{\VTop}^{\Corr}"]&(\bPrstL)^{\Sp//}\\
\cV\Spc^{op}\ar[r]&\Corr(\cV\Spc)_{\hpf', \all}\ar[ur, "\bT_{\cV\Spc}^\Corr"'] & \\
(\cV\Spc_{\hpf'})^{op}\ar[r]\ar[u]&\Corr(\cV\Spc)_{\hpf', \hpf'}\ar[u]\ar[uur, bend right, "\bT_{\cV\Spc_{\hpf'}}^\Corr"'] &\ 
\end{tikzcd},
\end{equation}
where 
\begin{itemize}
\item the functors in the leftmost column are the localization (top) and the 1-full embedding (bottom);
\item  $\bT_{\VTop}^{\Corr}|_{\VTop^{op}}\cong \bT_{\VTop}$, $\bT_{\cV\Spc}^{\Corr}|_{\cV\Spc^{op}}\cong \bT_{\cV\Spc}$, and $\bT_{\cV\Spc_{\hpf'}}^{\Corr}=\bT_{\cV\Spc}^{\Corr}|_{\Corr(\cV\Spc)_{\hpf', \hpf'}}$. 
\end{itemize}
\end{prop}

\begin{proof}
This is a direct consequence of the diagram (\ref{diagram: defGrid, SqPair, VTopVSpc,hpf'}), Corollary \ref{cor: GR Ch7 2.1.3, monoidal} and Corollary \ref{cor: Sq, VTop, cVSpc}. 
\end{proof}

\subsection{Proof of Proposition \ref{prop: equiv model of J} (the equivalent model of $J$)}\label{Appendix subsec: proof of J}

We will need the following to connect $\bT_{\cV\Spc}(p:VG\to G)$ with $(p_*\varpi_{VG}\in \Loc(G;\Sp)^{\otimes_c})$. 
An object in $\CAlg((\bPrstL)^{\Sp//})$ is equivalent to the datum of a commutative diagram
\begin{equation}\label{diagram: NFinGraydelta1bPrstL}
\begin{tikzcd}
N(\Fin_*)\Gray\Delta^1\ar[r]\ar[dr]&(\bPrstL)^{\otimes, \Fin_*}\ar[d]\\
\ &N(\Fin_*)
\end{tikzcd}
\end{equation}
whose restriction to $N(\Fin_*)\Gray\Delta^{\{0\}}\simeq N(\Fin_*)$ (resp. $N(\Fin_*)\Gray\Delta^{\{1\}}$) gives the monoidal unit $\Sp$ (resp. an algebra $C\in \CAlg(\PrstL)$). 
Passing to $\onetwoop$, the diagram becomes 
\begin{equation*}
\begin{tikzcd}
 N(\Fin_*)^{op}\Gray (\Delta^1)^{op}\ar[r]\ar[dr]&(\bPrstR)^{\otimes, (\Fin_*)^{op}}\ar[d]\\
\ &N(\Fin_*)^{op}
\end{tikzcd}
\end{equation*}
where the right vertical functor is the 2-Cartesian fibration representing the symmetric monoidal structure on $\bPrstR$. 
Assuming the horizontal arrows in the second factor $\Delta^1$ (i.e. $\{\lng n\rng\}\times \Delta^1$) of (\ref{diagram: NFinGraydelta1bPrstL}) also admit left adjoints, then we can use \cite[Chapter 12, Corollary 3.1.7]{Nick} to (canonically) transform the diagram into 
\begin{equation}\label{diagram: NFinGraydelta1coAlgbPrstR}
\begin{tikzcd}
 \Delta^1 \Gray N(\Fin_*)^{op}\ar[r]\ar[dr]&(\bPrstR)^{\otimes, (\Fin_*)^{op}}\ar[d]\\
\ &N(\Fin_*)^{op},
\end{tikzcd}
\end{equation}
by passing to the left adjoints for morphisms from the factor $(\Delta^1)^{op}\simeq \Delta^1$, which gives a corresponding object in $\CcoAlg(((\bPrstR)^{\twoop})^{\Sp/ /})$, where $\CcoAlg$ means the $(\infty,1)$-category of commutative coalgebras. The above correspondence clearly upgrades to the following:

\begin{lemma}\label{lemma: CAlgCcoAlg, PrstLR}
There is a natural commutative diagram 
\begin{equation}\label{diagram: CAlg, CcoAlg}
\begin{tikzcd}
(\CAlg((\bPrstL)^{\Sp//}))_\cR^{\Spc}\ar[d,  "\sim"]\ar[r, "\sim"] & (\CcoAlg(((\bPrstR)^{\twoop})^{\Sp//}))_\cL^{\Spc}\ar[d,  "\sim"]\\
((\CAlg(\bPrstL)_{\rightlax})^{\Sp/ })_\cR^{\Spc}\ar[r, "\sim"]&((\CcoAlg(\bPrstR)_{\leftlax})^{\Sp/ })_\cL^{\Spc}
\end{tikzcd}
\end{equation}
where the vertical isomorphisms are the tautological ones, and the subscript $\cR$ (resp. $\cL$) means the full subspace consisting of objects with the horizontal (resp. vertical) arrows in (\ref{diagram: NFinGraydelta1bPrstL}) (resp. (\ref{diagram: NFinGraydelta1coAlgbPrstR})) admitting left adjoints (resp. right adjoints) in the fibers of the 2-coCartesian (resp. 2-Cartesian) fibrations. 
\end{lemma}

Let $ \Sp_{\perf}$ be the full subcategory of $\Sp$ consisting of perfect $\bS$-modules.
\begin{cor}\label{cor: CAlgRunK;Sp}
For any $K\in \CAlg(\Spc^\times)$, there is a natural isomorphism 
\begin{align}\label{eq: cor: CAlgRunK;Sp}
(\CAlg((\Loc(K;\Sp)^{\otimes_c})_{\perf})^{\Spc}\simeq (\Fun^{\rightlax}(K^{\otimes}, \Sp_{\perf}^{\otimes}))^{\Spc}.
\end{align}
where $\CAlg((\Loc(K;\Sp)^{\otimes_c})_{\perf}$ is the full subcategory of $\CAlg((\Loc(K;\Sp))^{\otimes_c}$ whose underlying local systems have perfect costalks. 
\end{cor}

\begin{proof}
Using Lemma \ref{lemma: CAlgCcoAlg, PrstLR}, we just need to examine the space of diagrams (\ref{diagram: NFinGraydelta1coAlgbPrstR}) whose restriction at $\{0\}\times N(\Fin_*)^{op}$ and $\{1\}\times N(\Fin_*)^{op}$ give the structure of $\Sp$ and $\Loc(K;\Sp)\simeq \Fun(K;\Sp)$ as commutative coalgebra objects in $\bPrstR$, and then restrict to local systems with perfect costalks. Unwinding the data of the diagram, one sees exactly the data of a right-lax homomorphism from $K$ to $\Sp_\perf$. For example,  
for the unique active map $\alpha: \lng 1\rng\leftarrow\lng 2\rng$ in $N(\Fin_*)$, the horizontal arrow in (\ref{diagram: NFinGraydelta1coAlgbPrstR}) sends $\Delta^1\Gray \alpha^{op}\simeq (\alpha^{op}\Gray \Delta^1)^{\twoop}$ to 
\begin{equation*}
\begin{tikzcd}
(\Sp, \Sp)\ar[r]&(\Loc(K;\Sp), \Loc(K;\Sp))\\
\Sp\simeq \Sp\otimes \Sp\ar[r]\ar[u, thick, squiggly]\ar[dr, Rightarrow]&\Loc(K;\Sp)\otimes \Loc(K;\Sp)\simeq \Loc(K\times K;\Sp)\ar[u, thick, squiggly]\\
\Sp\ar[u, "\sim"]\ar[r]&\Loc(K;\Sp)\ar[u, "m^!"]
\end{tikzcd}
\end{equation*}
in which the bottom functor is determined by a local system with perfect costalks, the upper square with squiggly vertical arrows represents a 2-Cartesian morphism in $\Fun(\Delta^1, (\bPrstR)^{\otimes, \Fin_*^{op}})$ over $\alpha^{op}\in \Fun(\Delta^1, N(\Fin_*)^{op})$ and $m: K\times K\to K$ is the multiplication map determined by the functor $N(\Fin_*)\to \OneCat$ (and essentially determined by the image of $\alpha$) representing $K\in \CAlg(\OneCat)$. Then the bottom square in $\bPrstR$ encodes exactly  a diagram $\Delta^1\Gray \Delta^1\to \bOneCat$
\begin{equation*}
\begin{tikzcd}
K\times K\ar[r]\ar[d, "m"]&\Sp_{\perf}\times\Sp_{\perf}\ar[d, "\otimes"]\ar[dl, Rightarrow]\\
K\ar[r]&\Sp_{\perf}
\end{tikzcd}.
\end{equation*}
Examining through all the $n$-simplices in $N(\Fin_*)^{op}$ and their obvious compatibility (as encoded in $\Seq_\bullet N(\Fin_*)^{op}$), gives the corresponding right-lax homomorphism $K\to \Sp_{\perf}$ encoded by a functor $N(\Fin_*)\Gray \Delta^1\to \bOneCat$. The assignment is certainly functorial in any diagram (\ref{diagram: NFinGraydelta1coAlgbPrstR}) parametrized by a simplex, i.e. replacing the source $\Delta^1 \Gray N(\Fin_*)^{op}$ by $\Delta^n\times ( \Delta^1 \Gray N(\Fin_*)^{op})$, hence leading to the isomorphism (\ref{eq: cor: CAlgRunK;Sp}).

\end{proof}

\begin{proof}[Proof of Proposition \ref{prop: equiv model of J}]
We prove in the following steps. 

\noindent \emph{Step 1.} The homomorphisms $(p_{\hat{N}}: VG_{\hat{N}}\to G_{\hat{N}})\to (p_{\hat{M}}: VG_{\hat{M}}\to G_{\hat{M}}), \hat{N}\in \bZ_{\geq 0}\cup\{+\infty\}$ \\
in $\CAlg(\Corr(\VTop)_{\propfib', \all})$. 

Here when $\hat{N}=+\infty$, $p_{+\infty}$ means $(p: VG\to G)$. 
First, using Theorem \ref{thm: algebra objs} (ii), $(p^\bullet:VG^\bullet\to G^\bullet)$ and $(p_N^\bullet: VG_N^\bullet\to G_N^\bullet)$ give commutative algebra objects in $\Corr(\Fun(\Delta^1, \Top))$. Using the proof of Theorem \ref{thm: algebra objs} and that all vertical morphisms (corresponding to active morphisms in $N(\Fin_*)$) are proper locally trivial fibrations (note that the fibrations are usually not surjective on the components of the base),  we see that $(p^\bullet:VG^\bullet\to G^\bullet)$ and $(p_N^\bullet: VG_N^\bullet\to G_N^\bullet)$ give commutative algebra objects in the 1-full symmetric monoidal subcategory $\Corr(\VTop)_{\propfib', \all}$. We remark that this further uses that for every active map $\lng n\rng\to \lng m\rng$, the Cartesian square
\begin{equation*}
\begin{tikzcd}
VG_{\hat{N}}^{\lng n\rng}\ar[r]\ar[d]& VG_{\hat{N}}^{\lng m\rng}\ar[d]\\
G_{\hat{N}}^{\lng n\rng}\ar[r]& G_{\hat{N}}^{\lng m\rng}
\end{tikzcd}
\end{equation*}
represents a morphism in $\VTop$. 

The claim about the homomorphism follows from applying Theorem \ref{thm: right-lax} and a similar consideration as above.

\noindent\emph{Step 2.} The equivalence 
\begin{align}\label{eq: bTVTopVGG, VGNGN}
\bT_{\VTop}^{\Corr}(p:VG\to G)\overset{\sim}{\to} \varprojlim_{N\geq 0}\bT_{\VTop}^{\Corr}(p_N:VG_N\to G_N). 
\end{align}
in $\CAlg((\PrstL)^{\Sp/})$. 

Applying $\bT_{\VTop}^{\Corr}$ to the homomorphisms from \emph{Step 1}, we get a natural functor (\ref{eq: bTVTopVGG, VGNGN}). The functor is an equivalence, which follows from the assertion on $\bT_{\VTop}$ (see the proof of Proposition \ref{prop: bT_VTop}).   

\noindent\emph{Step 3.} $\bT_{\VTop}^{\Corr}(p_N:VG_N\to G_N)$ is isomorphic to $(\Loc(G_N;\Sp)^{\otimes_c}\ni (p_N)_*\varpi_{VG_N})$. 

Recall $\Loc(G_N;\Sp)^{\otimes_c}\ni (p_N)_*\varpi_{VG_N}$ is obtained using the correspondence 
\begin{equation*}
\begin{tikzcd}
VG_N^{\bullet}\ar[d]\ar[r]& \pt^{\bullet}\\
G_N^{\bullet}&\ 
\end{tikzcd}
\end{equation*}
interpreted as a right-lax homomorphism from $pt$ to $G_N$ in $\bCorr(\Slch)$ by Theorem \ref{thm: right-lax}. By Proposition \ref{prop: Dbullet, Wbullet, ptbullet}, this right-lax homomorphism is the same as the image of $VG_N^\bullet\to G_N^\bullet$ (as an object in $\CAlg(\Corr(\VSlch)_{\propfib', \all})$) in $\CAlg(\Corr(\Slch)_{\propfib,\all}^{pt//})$. Hence the claimed isomorphism follows.

\noindent\emph{Step 4.} Identifying $\bT_{\VTop}^{\Corr}(p:VG\to G)$ with $\bT_{\cV\Spc}(p:VG\to G)$. 

Following the diagram (\ref{diagram: prop bT_VTopCorr, bT_VSpcCorr}), $(p:VG\to G)\in \CAlg(\Corr(\VTop))_{\propfib', \all}$ descends to an element in $\CAlg(\Corr(\cV\Spc)_{\hpf', \hpf'})$ that comes from ``the classical" $(p:VG\to G)\in \CAlg(\cV\Spc_{\hpf'})$ under the 1-full symmetric monoidal inclusion $\cV\Spc_{\hpf'}\to \Corr(\cV\Spc)_{\hpf', \hpf'}$. Let $((\bPrstL)^{\Sp//})_{\cR}$ be the 1-full subcategory of $(\bPrstL)^{\Sp//}$  consisting of left adjointable functors. 
It follows that $\bT_{\cV\Spc_{\hpf'}}^{\Corr}(p:VG\to G)\in \CAlg((\bPrstL)^{\Sp//})$ is the image of $\bT_{\cV\Spc}(p:VG\to G)\in \mathrm{CcoAlg}(((\bPrstL)^{\Sp//})_\cR)$, through the procedure of passing to the left adjoints of the 1-morphisms. 

The above fits into the broader setting in Lemma \ref{lemma: CAlgCcoAlg, PrstLR}, using the last part of Remark \ref{remark: adjoints in Cc//}. 
Since $\bT_{\cV\Spc}(p:VG\to G)\in \CcoAlg((\PrstR)^{\Sp/}))_\cL$ (here $\cL$ means the full subcategory consisting of $\Sp\to C$ whose underlying functor is right adjointable), it certainly lies in $(\CcoAlg(((\bPrstR)^{\twoop})^{\Sp//}))_\cL$, hence we can equally view the correspondence between $\bT_{\cV\Spc_{\hpf'}}^{\Corr}(p:VG\to G)$ and $\bT_{\cV\Spc}(p:VG\to G)$ through the bottom isomorphism in (\ref{diagram: CAlg, CcoAlg}).

\noindent\emph{Step 5}. Final step. 

By definition (and tautology), $J$ is represented by $\bT(p:VG\to G)\cong \bT_{\cV\Spc}(p: VG\to G)$, as a diagram (\ref{diagram: NFinGraydelta1coAlgbPrstR}) (which actually factors through $\Delta^1\times N(\Fin_*)^{op}$) whose restriction to $\{0\}\times N(\Fin_*)^{op}$ (resp. $\{1\}\times N(\Fin_*)^{op}$) gives $\Sp$ (resp. $\Loc(G;\Sp)$) as a coalgebra in $(\PrstR)^{\otimes}$. 
By transiting the diagram to (\ref{diagram: NFinGraydelta1bPrstL}), we get the corresponding commutative algebra in $\Loc(G;\Sp)\in \CAlg((\PrstL)^{\otimes})$. But by the previous step, this is exactly 
$\bT_{\cV\Spc_{\hpf'}}^{\Corr}(p:VG\to G)\cong \bT_{\VTop}^{\Corr}(p:VG\to G)$. Combining with \emph{Step 3}, the proof is complete. 
\end{proof}


\begin{thebibliography}{99}
\bibitem[AbKr]{AbouzaidKragh} M. Abouzaid and T. Kragh, ``On the immersion classes of nearby Lagrangians,'' J. Topol. 9 (2016) no. 1, 232-244.



\bibitem[CJS]{CJS} R. Cohen, J. Jones, G. Segal, ``Floer's infinite-dimensional Morse theory and homotopy theory,'' in The Floer memorial volume, 297-325 (1995).


\bibitem[DK3]{DK3} W. Dwyer, D. Kan, ``Function complexes in homotopical algebra," Topology Vol. 19 (1980), pp. 427-440. 

\bibitem[GaRo]{Nick} D. Gaitsgory and N. Rozenblyum, ``A study in derived algebraic geometry, Vol. I. Correspondences and duality.'' Mathematical Surveys and Monographs, 221. American Mathematical Society (2017). 

\bibitem[Gla]{Glasman} S. Glasman, ``Day convolution for $\infty$-categories," Math. Res. Lett. 23 (2016), no. 5, 1369-1385. 

\bibitem[GMac]{stratified-Morse-theory} M. Goresky and R. MacPherson, ``Stratified Morse theory,'' Ergebnisse der Mathematik und ihrer Grenzgebiete (3) [Results in Mathematics and Related Areas (3)], 14. Springer-Verlag, Berlin, (1988).


\bibitem[Gui]{Guillermou} S. Guillermou, ``Sheaves and symplectic geometry of cotangent bundles," Ast\'erisque(2023), no.440, x+274 pp.



\bibitem[GKS]{GKS} S. Guillermou, M. Kashiwara, and P. Schapira, ``Sheaf quantization of Hamiltonian isotopies and applications to non displaceability problems," Duke Math. Journal Volume 161, Number 2 (2012), 201-245.


\bibitem [Har]{Harris} B. Harris, ``Bott periodicity via simplicial spaces,'' Journal of Algebra 62, 450-454 (1980).

\bibitem[Hin]{Hinich} V. Hinich, ``Rectification of algebras and modules," Doc. Math., 20 (2015), 879-926, 2015.

\bibitem[KaSc]{KS} M. Kashiwara and P. Schapira, ``Sheaves on manifolds,'' with a chapter in French by Christian Houzel.  Grundlehren der Mathematischen Wissenschaften, 292. Springer-Verlag (1994).

\bibitem[Jin]{J-J} X. Jin, ``Microlocal sheaf categories and the $J$-homomorphism", {\sf arXiv: 2004.14270v4}.

\bibitem[JiTr]{JT} X. Jin, D. Treumann, ``Brane structures in microlocal sheaf theory",  J. Topol. (to appear).


\bibitem[LiZh]{Yifeng2} Y. Liu and W. Zheng, ``Enhanced six operations and base change theorem for Artin stacks," {\sf arXiv:1211.5948}.

\bibitem[Lu1]{higher-topoi} J. Lurie, ``Higher topos theory,''  Annals of Mathematics Studies, 170. Princeton University Press (2009).

\bibitem[Lu2]{DAGII} J. Lurie, "Derived Algebraic Geometry II: Noncommutative Algebra," {\sf arXiv:math/0702299}.

\bibitem[Lu3]{higher-algebra} J. Lurie, ``Higher algebra,''
 {\sf www.math.harvard.edu/~lurie/papers/HA.pdf}.

\bibitem[Lu4]{Lurie-circle} J. Lurie, ``Rotation invariance in algebraic $K$-theory,'' {\sf www.math.harvard.edu/~lurie/papers/Waldhaus.pdf}.

\bibitem[Lu5]{SAG} J. Lurie, ``Spectral Algebraic Geometry," {\sf https://www.math.ias.edu/~lurie/papers/SAG-rootfile.pdf}.



\bibitem[Nad] {Nadler-selecta} D. Nadler, ``Microlocal branes are constructible sheaves,'' Selecta Math. (N.S.) 15 (2009), no. 4, 563-619.

\bibitem[NaSh]{Nadler-Shende} D. Nadler, V. Shende, ``Sheaf quantization in Weinstein symplectic manifolds," {\sf arXiv:2007.10154}.


\bibitem[NaZa]{NZ} D. Nadler and E. Zaslow, ``Constructible sheaves and the Fukaya category," J. Amer. Math. Soc. 22 (2009), no. 1, 233-286.



\bibitem[Ram]{Ramzi} M. Ramzi ``A monoidal Grothendieck construction for $\infty$-categories," {\sf arXiv:2209.12569v1}.

\bibitem[Seg]{Seg2} G. Segal, ``Categories and cohomology theories," Topology, Vol 13, pp. 293-312 (1974).

\bibitem[Sei]{Seidel} P. Seidel, ``Fukaya categories and Picard-Lefschetz Theory," Zur. Lect. Adv. Math.
European Mathematical Society (EMS), Zurich, 2008.


\bibitem[She]{Shende} V. Shende, ``Microlocal category for Weinstein manifolds via h-principle," Publ. Res. Inst. Math. Sci. 57 (2021), no.3-4, 1041-1048.



\bibitem[Tam]{Tamarkin2} D. Tamarkin, ``Microlocal category'', {\sf arXiv:1511.08961}.
\end{thebibliography}
\end{document}